\documentclass[10pt]{article}
\usepackage{blindtext}
\usepackage[style=ieee,sorting=nty, url=false, doi=false]{biblatex}

\usepackage[margin=1in]{geometry}
\usepackage{amsmath,amsthm,amssymb}
\usepackage{comment}
\usepackage{enumitem}
\usepackage{mathtools}
\usepackage{tikz-cd}
\usepackage{graphicx}

\usepackage{xcolor}
\allowdisplaybreaks
\numberwithin{equation}{section}

\newcommand{\C}{\mathbb{C}}
\newcommand{\R}{\mathbb{R}}
\newcommand{\Z}{\mathbb{Z}}

\newcommand{\F}{\mathcal{F}}

\newcommand{\e}{\epsilon}
\newcommand{\la}{\Delta}

\newtheorem{thm}{Theorem}[section]
\newtheorem{lem}{Lemma}[section]

\newtheorem{prop}{Proposition}[section]
\newtheorem{defn}{Definition}[section]

\begin{document}
 
\title{Global Solutions For Systems of Quadratic Nonlinear Schrödinger Equations in 3D}
\author{Boyang Su}
 
\maketitle
\section{Introduction}
In this work, we prove global well-posedness and scattering for systems of quadratic nonlinear Schrödinger equations in the critical three-dimensional case, for small, localized data. For the terms corresponding to the nonlinearity $u\bar{u}$, we need to do an $\e$ regularization of this part of the nonlinearity.

In order to tackle quadratic space-time resonances, after performing a Littlewood-Paley decomposition, we use integration by parts in the Duhamel term, to take advantage of the oscillations when space-time resonances are absent.

\subsection{Background}
The existence of a global solution for the nonlinear Schrödinger equation
\begin{equation}
     i\partial_t u + \Delta u = P(u),
     \tag{NLS}
\end{equation}
where $u(t,x):\R\times\R^n\rightarrow \C$ and $P$ is a nonlinear homogeneous function, has been extensively studied. Many papers focused on the case of a gauge invariant nonlinearity, namely $P(u)=\lambda|u|^{p-1}u$ for some real constant $p>1$ and $\lambda\neq 0$, where solutions satisfy several conservation laws. In the $L^2$-subcritical cases, i.e. $1<p<1+\frac{4}{n}$, a global solution can be obtained by iterations of local well-posedness results. In the energy-subcritical case, i.e. $1<p<1+\frac{4}{n-2}$, the local result can also be extended globally in time with some smallness condition imposed on the initial data \cite{text}.

The problem of global well-posedness and scattering for a general nonlinearity $P$, homogeneous of degree $p$, turns out to be more complicated, especially as $p$ becomes smaller. An important notion related to this problem is the Strauss exponent \cite{Strauss},
$$\frac{\sqrt{n^2+12n+4}+n+2}{2n},$$
derived from the positive root of $np^2-(n+2)p-2=0$. Notably, the Strauss exponent is smaller than the mass critical exponent, $1+\frac{4}{n}$. When $p$ is above this Strauss exponent, we can get small data global well-posedness and scattering for (NLS) using a fixed point argument with dispersion properties of the solution to linear Schrödinger solution \cite{Strauss}. Some other works include Shatah \cite{Shatah}, Klainerman and Ponce \cite{KlainermanandPonce}, and Kato \cite{Kato}. The Strauss exponent serves as an essential concept within the broader scope of nonlinear dispersive equations. In particular, for the nonlinear wave equation, it is the positive root of $(n-1)p^2-(n+1)p-2=0$. For any $p$ greater than this Strauss exponent, Lindblad and Sogge \cite{Sogge} proved global existence for small initial data with spherical symmetry, where the symmetry requirement can be removed when $n\leq 8$. Later in \cite{GLS}, Georgiev, Lindblad, and Sogge removed the symmetry assumption for all dimensions.

In dimension $n=3$, the Strauss exponent for NLS is $2$, which makes quadratic NLS in 3D an interesting case to study. Adopting the vector field method in Klainerman \cite{Klainermanvector}, one can show the existence of global solutions and scattering for a quadratic nonlinearity involving $u^2$ and $\bar{u}^2$. Relevant works using this method include Hayashi \cite{Hayashi93}, Hayashi and Naumkin \cite{Hayashi}, and Kawahara \cite{Kawahara}. In \cite{germain}, Germain, Masmoudi, and Shatah developed a more transparent proof combining the vector field and normal forms methods to deal with the space-time resonance, which can be generalized to systems of quadratic NLS.

However, for the nonlinearity of form $u\bar{u}$, the space-time resonance is 3-dimensional and we only have almost global existence in Ginibre and Hayashi \cite{Ginibre} and Wang \cite{wang}. Furthermore, Wang managed to prove the existence of global solutions when the quadratic term contains \textquotedblleft$\e$\textquotedblright derivatives in the low-frequency part. But global existence remains unknown for the general case of $|u|^2$. In \cite{Ikeda}, Ikeda and Inui showed that the solution blows up in finite time if the initial data $|u_0(x)|\geq |x|^{-k}$ for $|x|> 1$  and $3/2<k<2$. Hence, a necessary condition for the small data global existence in the $|u|^2$ case is fast enough decay at spacial infinity of the initial data.

\subsection{Main Result}
This paper generalizes Germain, Masmoudi, and Shatah's result \autocite[Theorem~2]{germain} on systems of quadratic nonlinear Schrödinger equations, to incorporate the quadratic resonance cases when $\frac{1}{c_l}=\frac{1}{c_m}+\frac{1}{c_n}$ and when $c_m+c_n=0$. The case when $c_m+c_n=0$ is analogous to the $u\bar{u}$ nonlinearity, where we use Wang's idea in \cite{wang} to add an \textquotedblleft$\e$\textquotedblright derivative to the low-frequency part of this quadratic nonlinear term. 

We thus consider the system of quadratic nonlinear Schrödinger equations (see Section \ref{notation section} for the unexplained notations)
\begin{equation}
\label{pde}
    \begin{cases}
     \partial_t u_l= ic_l\Delta u_l+\sum_{c_m+c_n\neq 0}A_{lmn} u_mu_n+\sum_{c_m+c_n= 0}A_{lmn} Q(u_m,u_n)\\
     u_l|_{t=1}=u_{l0}
    \end{cases},
\end{equation}
where $1\leq l,m,n\leq N$, $u_l(t,x):[1,\infty)\times\R^3\rightarrow \C$, $u_0\in H^{10}$, $c_l$ are nonzero real numbers, and
$$Q(u_m,u_n)=\F^{-1}\int_{\R^3}q(\xi-\eta,\eta)\hat{u}_m(t,\xi-\eta)\hat{u}_n(t,\eta)d\eta,$$
for $q(\xi,\eta)=q(\eta,\xi)$. Without loss of generality, we assume $m\geq n$ for all equations in \eqref{pde} and the constants $A_{lmn}=A_{lnm}$. As in \cite{wang}, we also impose the following condition on the symbol of $Q$,
\begin{equation}
    \label{q}
    \begin{aligned}
    \sup_{k,k_1,k_2\in\Z} &\{\|\F^{-1}[q(\xi-\eta,\eta)\Tilde{\psi}_k(\xi)\Tilde{\psi}_{k_1}(\xi-\eta)\Tilde{\psi}_{k_2}(\eta)]\|_{L^1}+\\
    &+2^k\|\F^{-1}[\nabla_\xi q(\xi-\eta,\eta)\Tilde{\psi}_k(\xi)\Tilde{\psi}_{k_1}(\xi-\eta)\Tilde{\psi}_{k_2}(\eta)]\|_{L^1}+\\
    &+2^{2k}\|\F^{-1}[\nabla^2_\xi q(\xi-\eta,\eta)\Tilde{\psi}_k(\xi)\Tilde{\psi}_{k_1}(\xi-\eta)\Tilde{\psi}_{k_2}(\eta)]\|_{L^1}+\\
    &+2^{\min\{k_1,k_2\}}\|\F^{-1}[\nabla_\eta q(\xi-\eta,\eta)\Tilde{\psi}_k(\xi)\Tilde{\psi}_{k_1}(\xi-\eta)\Tilde{\psi}_{k_2}(\eta)]\|_{L^1}+\\
    &+2^{2\min\{k_1,k_2\}}\|\F^{-1}[\nabla^2_\eta q(\xi-\eta,\eta)\Tilde{\psi}_k(\xi)\Tilde{\psi}_{k_1}(\xi-\eta)\Tilde{\psi}_{k_2}(\eta)]\|_{L^1}+\\
    &+2^{k+\min\{k_1,k_2\}}\|\F^{-1}[\nabla_\eta\nabla_\xi q(\xi-\eta,\eta)\Tilde{\psi}_k(\xi)\Tilde{\psi}_{k_1}(\xi-\eta)\Tilde{\psi}_{k_2}(\eta)]\|_{L^1}\}
    \lesssim 2^{\e k_-}.
    \end{aligned}
\end{equation}
Define
\begin{equation}
\label{Znormdef}
    \|f\|_Z=\sup_{k\in\Z}2^{-\gamma k+2k_++k/2+\alpha k}\|D^{1+\alpha}_\xi\hat{f}_k(\xi)\|_{L^2},
\end{equation}
where $f_k$ is the Littlewood-Paley projection at frequence $2^k$ and $\alpha,\gamma$ are positive real constants satisfying
\begin{equation}
\label{constant conditions1}
    0<\gamma <\e/3 \ll 1
\end{equation}
and
\begin{equation}
\label{constant conditions2}
    1/2-\gamma/8<\alpha<1/2-\gamma/(16-2\gamma).
\end{equation}
The definition of $k_{\pm}$ and Littlewood-Paley projection will be provided in Section \ref{notation section}. Note that the weights in the $Z$ norm guarantee $|x|^{-2}$ decay as $|x|\rightarrow\infty$ of the initial data. Using the Sobolev inequality, we have
\begin{align*}
    \||x|^{1/2}f(x)\|_{L^2}
    &\lesssim \sum_{k\in\Z}\|D^{1/2}_\xi\hat{f}_k(\xi)\|_{L^2}\\
    &\lesssim \sum_{k\in\Z}\|\chi_{2^{k-2}\leq |\xi|\leq 2^k}\|_{L^{\frac{6}{1+2\alpha}}}\|D^{1/2}_\xi\hat{f}_k(\xi)\|_{L^{\frac{6}{5-2\alpha}}}\\
    &\lesssim \sum_{k\in\Z}2^{k/2+\alpha k}\|D^{1+\alpha}_\xi\hat{f}_k(\xi)\|_{L^2}
    \leq \sum_{k\in\Z}2^{\gamma k-2k_+}\|f\|_Z\lesssim \|f\|_Z.
\end{align*}
This rules out the blow-up possibility introduced by Ikeda and Inui \cite{Ikeda}. Next, we state the main theorem, which proves global existence with scattering for \eqref{pde} when the initial data is small and localized around the origin.
\begin{thm}
\label{mainthm}
Fix $\alpha$, $\gamma$ and $\e$. If \eqref{constant conditions1} and \eqref{constant conditions2} are satisfied, and for all $l$
\begin{equation}
\label{initial}
    \|u_{l0}\|_{H^{10}}+\|e^{-ic_l\la}u_{l0}\|_Z\leq \e_0,
\end{equation}
for a sufficiently small constant $\e_0$, then there exists a unique global solution for the IVP \eqref{pde}. Moreover, for all $l$, the following estimate holds,
\begin{equation}
\label{gse}
    \sup_{t\in[1,\infty)}\|u_l(t,x)\|_{H^{10}_x}+\|e^{-ic_lt\la}u_l(t,x)\|_{Z_x}+t^{1+\gamma/2}\|u_l(t,x)\|_{L^\infty_x}\lesssim \e_0,
\end{equation}
and the solution $u_l$ scatters.
\end{thm}
In \cite{germain}, Germain, Masmoudi, and Shatah proved global existence and scattering of small solutions to the system \eqref{pde}, assuming no quadratic resonance, i.e. $A_{lmn}=0$ when $\frac{1}{c_l}=\frac{1}{c_m}+\frac{1}{c_n}$ and when $c_m+c_n= 0$. They defined the quadratic space-time resonance as $\mathcal{T}_{l,m,n}\cap\mathcal{S}_{m,n}$, where 
$$\mathcal{T}_{l,m,n}=\{(\xi,\eta)\in\R^3\times\R^3:c_m|\xi|^2+c_n|\eta|^2=c_l|\xi+\eta|^2\}$$
is the time resonance and
$$\mathcal{S}_{m,n}=\{(\xi,\eta)\in\R^3\times\R^3:c_m\xi=c_n\eta\}$$
is the space resonance. No quadratic resonance corresponds to $\mathcal{T}_{l,m,n}\cap \mathcal{S}_{m,n}=\{(0,0)\}$. The nonlinear component in Duhamel's formula for the profile $\hat{f}_l(t,\xi)=e^{ic_lt|\xi|^2}\hat{u}_l(t,\xi)$ is given by
$$\int_1^t\int_{\R^3}e^{is(c_l|\xi|^2-c_m|\xi-\eta|^2-c_n|\eta|^2)}\hat{f}_m(s,\xi-\eta)\hat{f}_n(s,\eta)d\eta ds.$$
Germain, Masmoudi, and Shatah used integration by parts in time and frequency to take advantage of the oscillation due to the fact of no quadratic resonance. However, when $\frac{1}{c_l}=\frac{1}{c_m}+\frac{1}{c_n}$ and when $c_m+c_n= 0$, the dimension of space-time resonance goes up to three, making it hard to apply the previous method.

It is worth noting that $c_m=1,c_n=-1$ contains the case of $u\bar{u}$ nonlinearity. Wang \cite{wang} suggested using Littlewood-Paley to decompose the frequency space, enabling one to conduct a more detailed study to fully exploit the interplay between the time variable and the frequency levels. 

When one attempts to use Wang's method for the case of $u^2$, the order of derivatives in Wang's norm $\|\nabla^2_\xi\hat{f}_k(t,\xi)\|_{L^2_\xi}$ is too large and one cannot obtain enough decay in the time variable for the estimations. Interestingly, there is a term in the norm in \cite{germain} with the same weights, namely $\|\frac{x^2}{\sqrt{t}}f(t,x)\|_{L^2_x}$. Nevertheless, the presence of $\sqrt{t}$ in the denominator effectively counteracts the substantial growth in the time variable caused by these weights.

Subsequently, work by L\'{e}ger \cite{Tristan} demonstrated, in particular, small data global existence and scattering for NLS with $u^2$ nonlinearity, employing the norm $\|\nabla_\xi\hat{f}_k(t,\xi)\|_{L^2}$. L\'{e}ger's work served as the inspiration to lower the order of derivatives in Wang's norm. However, in the $u\bar{u}$ case, the oscillation in the Duhamel's formula, taking the form of $e^{i2s\xi\cdot\eta}$, does not provide enough decay when focusing on the lower frequencies. Hence, the $t^{-3/2}$ decay brought by the $e^{it\Delta}: L^1\rightarrow L^\infty$ estimate played a crucial role in this case. The weights required for bounding the $L^1$ norm of $f_k$ translate into at least $3/2$ derivatives in $\hat{f}_k$, leading to the $D^{1+\alpha}$ derivatives in our $Z$ norm. Nevertheless, taking $\alpha\geq 1/2$ will result in excessive powers of $t$ that we cannot control, so we required $\alpha<1/2$.

Next, we briefly outline the proof and the structure of this paper.

Assuming the initial condition \eqref{initial}, we know by the local well-posedness (see Proposition \ref{lwp}) that there exists some $T>1$ so that there is a unique solution $\{u_l\}$ to \eqref{pde} such that
\begin{equation}
\label{b1}
    \sup_{t\in[1,T]}\|u_l(t,x)\|_{H^{10}_x}+\|e^{-ic_lt\la}u_l(t,x)\|_{Z_x}\leq \e_1:=\e_0^{5/6}.
\end{equation}
We use this as the bootstrap assumption and Proposition \ref{main} shows that under \eqref{b1}
\begin{equation*}
    \sup_{t\in[1,T]}\|u_l(t,x)\|_{H^{10}_x}+\|e^{-ic_lt\la}u_l(t,x)\|_{Z_x}\lesssim \e_0,
\end{equation*}
which leads to a global solution for \eqref{pde} and this solution also scatters and decays in time. 

The detailed proof of Theorem \ref{mainthm} is presented in Section \ref{main proof}, while Section \ref{lwpproof} contains the proof of local well-posedness and Section \ref{prop} provides the detailed estimates for the bootstrap argument (Proposition \ref{main}). For the reader's convenience, we included some standard inequalities in Section \ref{background}. Furthermore, Section \ref{prelim} provides proofs for some estimates utilized in Section \ref{prop}.

\subsection{Notations}
\label{notation section}
We will introduce some notations used in this paper. Let
$$\F f = \hat{f}(\xi)= \frac{1}{(2\pi)^3}\int_{\R^3}e^{-ix\cdot\xi}f(x)dx$$
and
$$\F^{-1} f = \int_{\R^3} e^{ix\cdot\xi}f(\xi)d\xi$$
denote Fourier transfrom and inverse Fourier transform. In particular, 
$$\|\F^{-1}m(\xi,\eta)\|_{L^1}=\int_{\R^3}\int_{\R^3}\bigg|\int_{\R^3}\int_{\R^3}e^{ix\cdot\xi+iy\cdot\eta}m(\xi,\eta)d\xi d\eta\bigg|dxdy.$$

For the Littlewood-Paley operator, we choose an increasing smooth function $\Psi:\R^+\rightarrow\R$ such that $0\leq \Psi(x)\leq 1$, $\Psi(x)=0$ on $\{|x|<1/2\}$, and $\Psi(x)=1$ on $\{|x|\geq 1\}$. Then, $\psi:\R^3\rightarrow\R$ defined as $\psi(x)=\Psi(2|x|)-\Psi(|x|)$ is also smooth, nonnegative and supported on $\{1/4\leq|x|\leq 1\}$. Let $\psi_k(x)=\psi(x/2^k)$, we get
\begin{equation*}
    1\equiv\sum_{k\in\Z}\psi_k(x)=\sum_{k\in\Z}\Psi(|x|/2^{k-1})-\Psi(|x|/2^k).
\end{equation*}
Hence, $\psi_k$ forms a partition of unity on $\R^3$. For an integer $a\geq 3$ to be determined later in \eqref{adef}, define another Schwartz function $\Tilde{\psi}:\R^3\rightarrow [0,1]$ supported on $\{2^{-2a-1}\leq |\xi|\leq 2^{a+3}\}$ and $\Tilde{\psi}(\xi)=1$ for $2^{-2a}\leq |\xi|\leq 2^{a+2}$, so that $\psi(\xi)\Tilde{\psi}(\xi)=\psi(\xi)$. Let $\Tilde{\psi}_k(\xi)=\Tilde{\psi}(\xi/2^k)$, then we have $\psi_k(\xi)\Tilde{\psi}_k(\xi)=\psi_k(\xi)$ and we define
$$f_k=\F^{-1}\hat{f}(\xi)\psi_k(\xi).$$
The linear Schrödinger equation operator is defined as follows
$$e^{ic_lt\la}u_0=\F^{-1} e^{-ic_lt|\xi|^2}\hat{u}_0(\xi).$$
The fractional Sobolev spaces are defined as
$$H^s(\R^3)=\{f\in\mathcal{S'}(\R^3):\Lambda^s f(x)= \F^{-1}((1+|\xi|^2)^{s/2}\hat{f}(\xi))(x)\in L^2(\R^3)\}$$
and fractional derivatives
$$D^s f=\F^{-1}|\xi|^s\hat{f}(\xi).$$
For any spacial norm $X$, the mixed norm is defined as
$$\|f(t,x)\|_{L^p_t([a,b])X_x}=
\begin{cases}
    \bigg(\int_a^b\|f(t,x)\|_{X_x}^p dt \bigg)^{\frac{1}{p}}, & 1\leq p<\infty\\
    \sup_{t\in[a,b]}\|f(t,x)\|_{X_x}, & p=\infty
\end{cases},
$$
Lastly, in this paper, we denote
$$k_-=\min\{k,0\},\, k_+=\max\{k,0\},$$
the phase function
$$\phi(\xi,\eta)=c_l|\xi|^2-c_m|\xi-\eta|^2-c_n|\eta|^2,$$
and
$\chi^1_k=\{(k_1,k_2):|k-k_1|\leq 3,k_2\leq k_1-a\}$,
$\chi^2_k=\{(k_1,k_2):|k_1-k_2|< a,k< k_1-a-2\}$ and $\chi^3_k=\{(k_1,k_2):|k_1-k_2|< a,|k-k_1|\leq a+2\}$, with $a$ as above.
\section{Background Estimates}
\label{background}
In this section, we present several fundamental results concerning the Littlewood-Paley decomposition, fractional derivatives, linear Schrödinger equation, and bilinear estimations. These results will play a crucial role in the subsequent proofs, providing essential tools and techniques for our estimations.
\subsection{Tools for Fractional Derivatives}
\label{section for fractional derivatives}
The definition of the $Z$ norm involves the $L^2$ norm of fractional derivatives, hence we will present two interpolation-type results. Using these lemmas, we may obtain a bound for the fractional derivative using the bounds we found for the integer order derivatives. The first one below bounds the $\alpha$ fractional derivative by the function and its first derivative.
\begin{lem}
\label{dalpha}
For any $f\in H^1(\R^3)$, we have
\begin{align*}
    \|D^\alpha f\|_{L^2}\leq \|f\|_{L^2}^{1-\alpha}\|\nabla f\|_{L^2}^\alpha,
\end{align*}
where $\alpha\in[0,1]$.
\end{lem}
\begin{proof}
This is a direct result of Hölder's inequality and Plancherel's theorem,
\begin{align*}
    \|D^{\alpha} f\|_{L^2}
    \leq \bigg(\int_{\R^3}|\xi|^{2\alpha} |\hat{f}|^2 d\xi\bigg)^{1/2}
    &\leq \bigg(\int_{\R^3}(|\xi|^{2\alpha}|\hat{f}|^{2\alpha})^{\frac{1}{\alpha}}d\xi\bigg)^{\frac{\alpha}{2}}\bigg(\int_{\R^3}(|\hat{f}|^{2-2\alpha})^{\frac{1}{1-\alpha}}d\xi\bigg)^{\frac{1-\alpha}{2}}\\
    &=\|\xi\hat{f}(\xi)\|_{L^2}^\alpha\|\hat{f}(\xi)\|_{L^2}^{1-\alpha}
    =\|f\|_{L^2}^{1-\alpha}\|\nabla f\|_{L^2}^\alpha
\end{align*}
\end{proof}
However, we will not always have the $1-\alpha$ extra derivative that we may use to bound the $\alpha$ derivative. In those cases, we use a variation of a special case of the Stein's interpolation theorem \autocite[Theorem~4.1]{stein} tailored for the specific problem setting in this paper.

The lemma below presents the two interpolation results, whose domains of the family of operators are different. The first one is a result on space norm, while the second one involves mixed space-time norm. The proof is analogous to Stein's proof and is based on the three lines lemma \autocite[Lemma~2.1]{text}.
\begin{lem}[Interpolation]
\label{si}
Let $S=\{x+iy\in\C: 0\leq x\leq 1\}$ and $\{T_z\}_{z\in S}$ be a family of linear operators such that $(T_zf)g$ is integrable whenever $f\in X$ and $g\in L^2(\R^3)$. Furthermore,
$$z\mapsto \int_{\R^3}(T_zf)g$$
is analytic in the interior of $S$, continuous and uniformly bounded on $S$. Take $0\leq s_1\leq s_2$.
\begin{enumerate}[label=(\alph*)]
    \item When $X=C^\infty_0(\R^3)$, if
    \begin{align}
        \label{sicon.1}
        \|T_{iy} f\|_{L^2_x}\leq M_0\|f\|_{H^{s_1}_x}\qquad
        \text{and}\qquad
        \|T_{1+iy} f\|_{L^2_x}\leq M_1\|f\|_{H^{s_2}_x},
    \end{align}
    for any $y\in\R$, then
    \begin{align*}
        \|T_{\theta} f\|_{L^2_x}\lesssim M_0^{1-\theta}M_1^{\theta}\|f\|_{H^{(1-\theta)s_1+\theta s_2}_x},
    \end{align*}
    for all $0\leq \theta\leq 1$.
    \item When $X=L^\infty[1,T]C^\infty_0(\R^3)$, if
    \begin{align}
        \label{sicon.2}
        \|T_{iy} f\|_{L^2_x}\leq M_0\|f\|_{L^\infty_t([1,T])H^{s_1}_x}\qquad
        \text{and}\qquad
        \|T_{1+iy} f\|_{L^2_x}\leq M_1\|f\|_{L^\infty_t([1,T])H^{s_2}_x},
    \end{align}
    for any $y\in\R$, then 
    \begin{align*}
        \|T_{\theta} f\|_{L^2_x}\lesssim M_0^{1-\theta}M_1^{\theta}\|f\|_{L^\infty_t([1,T])H^{(1-\theta)s_1+\theta s_2}_x},
    \end{align*}
    for all $0\leq \theta\leq 1$.
\end{enumerate}
\end{lem}
\begin{proof}
We first show part $(a)$. Define a new family of operators for $z\in S$, $$T'_z=T_z\Lambda^{-s_1+(s_1-s_2)z},$$
where $\Lambda^s f(x)=\F^{-1}(1+|\xi|^2)^{s/2}\hat{f}(\xi).$ 
Then, as a result of \eqref{sicon.1},
\begin{align*}
    \|T'_{iy}g\|_{L^2_x}
    &=\|T_{iy}\Lambda^{-s_1+(s_1-s_2)iy}g\|_{L^2_x}
    \leq M_0\|\Lambda^{-s_1+(s_1-s_2)iy}g\|_{H^{s_1}_x}=M_0\|g\|_{L^2_x}
\end{align*}
and 
\begin{align*}
    \|T'_{1+iy}g\|_{L^2_x}
    &=\|T_{1+iy}\Lambda^{-s_1+(s_1-s_2)(1+iy)}g\|_{L^2_x}
    \leq M_1\|\Lambda^{-s_2+(s_1-s_2)iy}g\|_{H^{s_2}_x}=M_1\|g\|_{L^2_x}.
\end{align*}
For $g, h\in C^\infty_0(\R^3)$ with $\|g\|_{L^2_x}=\|h\|_{L^2_x}=1$, define
$$\Psi(z) = \int_{\R^3} (T'_zg)(x)h(x)dx.$$
Then the condition on $T_z$ implies $\Psi(z)$ is analytic in the interior of $S$, and continuous and uniformly bounded on $S$. From Hölder's inequality, we obtain
$$|\Psi(iy)|\leq \|T'_{iy}g\|_{L^2_x}\|h\|_{L^2_x}\leq M_0\|g\|_{L^2_x}=M_0$$
and
$$|\Psi(1+iy)|\leq \|T'_{1+iy}g\|_{L^2_x}\|h\|_{L^2_x}\leq M_1\|g\|_{L^2_x}=M_1.$$
Hence, three lines theorem \autocite[Lemma~2.1]{text} implies for any $\theta\in[0,1]$
$$|\Psi(\theta)|\leq M^{1-\theta}_0M^\theta_1.$$
Using duality and the density of $C^\infty_0(\R^3)$ in $L^2(\R^3)$, we get $T'_\theta:L^2(\R^3)\rightarrow L^2(\R^3)$ is bounded with norm $M^{1-\theta}_0M^\theta_1$.
Thus,
\begin{align*}
    \|T_\theta f\|_{L^2_x}
    =\|T'_\theta \Lambda^{s_1-(s_1-s_2)\theta} f\|_{L^2_x}
    \lesssim M_0^{1-\theta}M_1^\theta\|\Lambda^{s_1-(s_1-s_2)\theta} f\|_{L^2_x}=M_0^{1-\theta}M_1^\theta\|f\|_{H^{s_1-(s_1-s_2)\theta}_x}.
\end{align*}
For part (b), we use the same $T'_z$ as above.
From (\ref{sicon.2}), we get
\begin{align*}
    \|T'_{iy}g\|_{L^2_x}
    &=\|T_{iy}\Lambda^{-s_1+(s_1-s_2)iy}g\|_{L^2_x}
    \leq M_0\|\Lambda^{-s_1+(s_1-s_2)iy}g\|_{L^\infty_t([1,T])H^{s_1}_x}=M_0\|g\|_{L^\infty_t([1,T])L^2_x}
\end{align*}
and 
\begin{align*}
    \|T'_{1+iy}g\|_{L^2_x}
    &=\|T_{1+iy}\Lambda^{-s_2+(s_1-s_2)iy}g\|_{L^2_x}
    \leq M_1\|\Lambda^{-s_2+(s_1-s_2)iy}g\|_{L^\infty_t([1,T])H^{s_2}_x}=M_1\|g\|_{L^\infty_t([1,T])L^2_x}.
\end{align*}
Then, it suffices to show the interpolation result, $T'_\theta:L^\infty_t([1,T])L^2_x\rightarrow L^2_x$ is bounded such that
\begin{align}
\label{threelines}
    \|T'_{\theta}g\|_{L^2_x}\leq M^{1-\theta}_0M^\theta_1\|g\|_{L^\infty_t([1,T])L^2_x},
\end{align}
for any $0\leq \theta\leq 1$, since \eqref{threelines} gives
\begin{align*}
    \|T_\theta f\|_{L^2_x}
    =&\|T'_\theta \Lambda^{s_1-(s_1-s_2)\theta} f\|_{L^2_x}
    \leq  M_0^{1-\theta}M_1^\theta\|\Lambda^{s_1-(s_1-s_2)\theta} f\|_{L^\infty_t([1,T])L^2_x}=M_0^{1-\theta}M_1^\theta\|f\|_{L^\infty_t([1,T])H^{s_1-(s_1-s_2)\theta}_x}.
\end{align*}
Now, we are left to prove \eqref{threelines}.\\
Take $g\in L_t^\infty([1,T])C^\infty_0(\R^3)$ and $h\in C^\infty_0(\R^3)$ so that $\|g\|_{L^\infty_t([1,T])L^2_x}=\|h\|_{L^2_x}=1$. Define $\Psi(z)$ same as above, we get
$|\Psi(iy)|\leq M_0$
and
$|\Psi(1+iy)|\leq M_1$.
Using three lines theorem, we have for any $\theta\in[0,1]$
\begin{align*}
    |\Psi(\theta)|\leq  M_0^{1-\theta}M_1^{\theta},
\end{align*}
which shows \eqref{threelines} by duality.
\end{proof}
\subsection{Bilinear Estimates}
Since we have the quadratic nonlinear terms in \eqref{pde}, using Duhamel's principle will result in bilinear integral forms. The lemma below provides us with an estimate for such type of integral forms with a multiplier $m$.
\begin{lem}
\label{bilinear}
For $1\leq p,q,r\leq \infty$, $f\in L^p(\R^3)$, and $g\in L^q(\R^3)$. If 
$$\|\F^{-1}m(\xi,\eta)\|_{L^1}=\int_{\R^3}\int_{\R^3}\bigg|\int_{\R^3}\int_{\R^3}e^{ix\cdot\xi+iy\cdot\eta}m(\xi,\eta)d\eta d\xi\bigg|dxdy < \infty,$$
then the following bilinear estimates hold
\begin{enumerate}[label=(\alph*)]
    \item $\|\F^{-1}\int_{\R^3}m(\xi,\eta)\hat{f}(\xi-\eta)\hat{g}(\eta)d\eta\|_{L^r}\leq  \|\F^{-1}[m(\xi,\eta)]\|_{L^1}\|f\|_{L^p}\|g\|_{L^q}$,
    \item $\|\F^{-1}\int_{\R^3}m(\xi,\eta)\hat{f}(\xi-\eta)\hat{g}(\xi)d\xi\|_{L^r}\leq \|\F^{-1}[m(\xi,\eta)]\|_{L^1}\|f\|_{L^p}\|g\|_{L^q}$,
\end{enumerate}
where $1/r=1/p+1/q$.
\end{lem}
\begin{proof}
Take an arbitrary test function $\varphi(x)\in\mathcal{S}(\R^3)$ such that $\|\varphi(x)\|_{L^{r'}}\leq 1$ with $1/r + 1/r' = 1$. Suppose $f,g\in\mathcal{S}(\R^3)$, then we have
\begin{align*}
    &\int_{\R^3}\varphi(x)\F^{-1}\int_{\R^3}m(\xi,\eta)\hat{f}(\xi-\eta)\hat{g}(\eta)d\eta dx\\
    =&\int_{\R^3}\int_{\R^3} e^{ix\cdot\xi}\varphi(x)\int_{\R^3}m(\xi,\eta)\hat{f}(\xi-\eta)\hat{g}(\eta)d\eta d\xi dx\\
    =& \frac{1}{(2\pi)^6}\int_{\R^3}\int_{\R^3}e^{ix\cdot\xi}\varphi(x)\int_{\R^3}\int_{\R^3}\int_{\R^3}m(\xi,\eta)e^{-iz\cdot(\xi-\eta)}f(z)e^{-iy\cdot\eta}g(y)dzdyd\eta d\xi dx\\
    =& \frac{1}{(2\pi)^6}\int_{\R^3}\int_{\R^3}\int_{\R^3}\int_{\R^3}\int_{\R^3}e^{i(x-z)\cdot\xi+i(z-y)\cdot\eta}m(\xi,\eta) d\eta d\xi\varphi(x)f(z)g(y)dy dx dz.
\end{align*}
Perform a change of variable on $x$ and $y$. In particular, we take $\Tilde{x} = x-z$ and $\Tilde{y} = z-y$. Hence,
\begin{align*}
    &\int_{\R^3}\varphi(x)\F^{-1}\int_{\R^3}m(\xi,\eta)\hat{f}(\xi-\eta)\hat{g}(\eta)d\eta dx\\
    =& \frac{1}{(2\pi)^6}\int_{\R^3}\int_{\R^3}\int_{\R^3}\int_{\R^3}\int_{\R^3}e^{i\Tilde{x}\cdot\xi+i\Tilde{y}\cdot\eta}m(\xi,\eta) d\eta d\xi\varphi(\Tilde{x}+z)f(z)g(z-\Tilde{y})d\Tilde{y} d\Tilde{x} dz\\
    =& \frac{1}{(2\pi)^6}\int_{\R^3}\int_{\R^3}\int_{\R^3}M(\Tilde{x},\Tilde{y})\varphi(\Tilde{x}+z)f(z)g(z-\Tilde{y})d\Tilde{y} d\Tilde{x} dz,
\end{align*}
where 
$$M(x,y) = \F^{-1}m(\xi,\eta) = \int_{\R^3}\int_{\R^3}e^{ix\cdot\xi+iy\cdot\eta}m(\xi,\eta)d\eta d\xi.$$
Using Hölder's in equality and Minkowski's integral inequality, for $1/p+1/p'=1$ and $1/p'=1/r'+1/q$, we get
\begin{align*}
    \bigg|\int_{\R^3}\varphi(x)\F^{-1}\int_{\R^3}m(\xi,\eta)\hat{f}(\xi-\eta)\hat{g}(\eta)d\eta dx\bigg|
    \leq & \bigg|\int_{\R^3}\int_{\R^3}\int_{\R^3}M(\Tilde{x},\Tilde{y})\varphi(\Tilde{x}+z)g(z-\Tilde{y})d\Tilde{y} d\Tilde{x}f(z) dz\bigg|\\
    \leq & \bigg\|\int_{\R^3}\int_{\R^3}M(\Tilde{x},\Tilde{y})\varphi(\Tilde{x}+z)g(z-\Tilde{y})d\Tilde{y} d\Tilde{x}\bigg\|_{L^{p'}_z}\|f\|_{L^p}\\
    \leq & \int_{\R^3}\int_{\R^3}|M(\Tilde{x},\Tilde{y})|\|\varphi(\Tilde{x}+z)g(z-\Tilde{y})\|_{L^{p'}_z}d\Tilde{y} d\Tilde{x}\|f\|_{L^p}\\
    \leq & \int_{\R^3}\int_{\R^3}|M(\Tilde{x},\Tilde{y})|\|\varphi\|_{L^{r'}}\|g\|_{L^q}d\Tilde{y} d\Tilde{x}\|f\|_{L^p}\\
    \leq & \|M(\Tilde{x},\Tilde{y})\|_{L^1}\|g\|_{L^q}\|f\|_{L^p},
\end{align*}
which shows $(a)$ by duality. Then, we can use the density of Schwartz space in $L^p$ to extend the result to any $f\in L^p(\R^3)$ and $g\in L^q(\R^3)$ for $1\leq p,q<\infty$. For the case of $L^\infty$, we can assume without loss of generality $g\in L^\infty(\R^3)$ and perform approximation using $g_n(x)= g(x)\chi_{|x|\leq n}(x)$.

The proof for statement $(b)$ is analogous. Starting with the test function $\varphi$, $f$, and $g$ in $\mathcal{S}(\R^3)$, we have
\begin{align*}
    &\int_{\R^3}\varphi(x)\F^{-1}\int_{\R^3}m(\xi,\eta)\hat{f}(\xi-\eta)\hat{g}(\xi)d\xi dx\\
    =&\int_{\R^3}\int_{\R^3} e^{ix\cdot\eta}\varphi(x)\int_{\R^3}m(\xi,\eta)\hat{f}(\xi-\eta)\hat{g}(\xi)d\xi d\eta dx\\
    =& \frac{1}{(2\pi)^6}\int_{\R^3}\int_{\R^3}e^{ix\cdot\eta}\varphi(x)\int_{\R^3}\int_{\R^3}\int_{\R^3}m(\xi,\eta)e^{-iz\cdot(\xi-\eta)}f(z)e^{-iy\cdot\xi}g(y)dzdyd\xi d\eta dx\\
    =& \frac{1}{(2\pi)^6}\int_{\R^3}\int_{\R^3}\int_{\R^3}\int_{\R^3}\int_{\R^3}e^{i(-y-z)\cdot\xi+i(x+z)\cdot\eta}m(\xi,\eta) d\eta d\xi\varphi(x)f(z)g(y)dy dx dz.
\end{align*}
Using $\Tilde{x} = x+z$ and $\Tilde{y} = -y-z$ for change of variables, we have
\begin{align*}
    &\int_{\R^3}\varphi(x)\F^{-1}\int_{\R^3}m(\xi,\eta)\hat{f}(\xi-\eta)\hat{g}(\xi)d\xi dx\\
    =& \frac{1}{(2\pi)^6}\int_{\R^3}\int_{\R^3}\int_{\R^3}\int_{\R^3}\int_{\R^3}e^{i\Tilde{y}\cdot\xi+i\Tilde{x}\cdot\eta}m(\xi,\eta) d\eta d\xi\varphi(\Tilde{x}-z)f(z)g(-\Tilde{y}-z)d\Tilde{y} d\Tilde{x} dz\\
    =& \frac{1}{(2\pi)^6}\int_{\R^3}\int_{\R^3}\int_{\R^3}M(\Tilde{y},\Tilde{x})\varphi(\Tilde{x}-z)f(z)g(-\Tilde{y}-z)d\Tilde{y} d\Tilde{x} dz.
\end{align*}
Lastly, we obtain (b) from Hölder's in equality and Minkowski's integral inequality,
\begin{align*}
    \bigg|\int_{\R^3}\varphi(x)\F^{-1}\int_{\R^3}m(\xi,\eta)\hat{f}(\xi-\eta)\hat{g}(\eta)d\eta dx\bigg|
    \leq & \bigg|\int_{\R^3}\int_{\R^3}\int_{\R^3}M(\Tilde{y},\Tilde{x})\varphi(\Tilde{x}-z)g(-\Tilde{y}-z)d\Tilde{y} d\Tilde{x}f(z) dz\bigg|\\
    \leq & \int_{\R^3}\int_{\R^3}|M(\Tilde{y},\Tilde{x})|\|\varphi(\Tilde{x}-z)g(-\Tilde{y}-z)\|_{L^{p'}_z}d\Tilde{y} d\Tilde{x}\|f\|_{L^p}\\
    \leq & \|M(\Tilde{y},\Tilde{x})\|_{L^1}\|g\|_{L^q}\|f\|_{L^p},
\end{align*}
and extend the result to $L^p$ spaces.
\end{proof}
The next result is an immediate application of the bilinear estimates above, combined with a duality argument. Here, we chose to estimate the term $\|e^{it\la}f_k\|_{L^\infty}$ by Bernstein's inequality (see Lemma \ref{bernstein}),
$$\|e^{it\la}f_k\|_{L^\infty}\lesssim 2^{3k/2}\|e^{it\la}f_k\|_{L^2}=\|f_k\|_{L^2}.$$
However, as we will see later, there are instances where we utilize the properties of the linear Schrödinger equation to obtain different bounds, for example, $\|e^{it\la}f_k\|_{L^\infty}\lesssim t^{-3/2}\|f_k\|_{L^1}$. The key advantage of Lemma \ref{dualitycomp} becomes evident when the magnitude of $2^k$ is relatively small compared to $t^{-1}$.
\begin{lem}
\label{dualitycomp}
For any $k,k_1,k_2\in\Z$,
\begin{align*}
    &\bigg\|\int_{\R^3}m(\xi,\eta)\hat{f}_{k_1}(\xi-\eta)\hat{g}_{k_2}(\eta)\psi_k(\xi)d\eta\bigg\|_{L^2}\\
    \lesssim &2^{3\min\{k,k_2\}/2}\|\F^{-1}m(\xi,\eta)\Tilde{\psi}_k(\xi)\Tilde{\psi}_{k_1}(\eta)\Tilde{\psi}_{k_2}(\xi-\eta)\|_{L^1}\|f_{k_1}\|_{L^2}\|g_{k_2}\|_{L^2}.
\end{align*}
\end{lem}
\begin{proof}
Using Plancherel's theorem and the bilinear estimate $L^2\times L^\infty\rightarrow L^2$ in Lemma \ref{bilinear}, we have
\begin{align*}
    &\bigg\|\int_{\R^3}m(\xi,\eta)\hat{f}_{k_1}(\xi-\eta)\hat{g}_{k_2}(\eta)\psi_k(\xi)d\eta\bigg\|_{L^2}\\
    = &\bigg\|\F^{-1}\int_{\R^3}m(\xi,\eta)\Tilde{\psi}_{k}(\xi)\Tilde{\psi}_{k_1}(\xi-\eta)\Tilde{\psi}_{k_2}(\eta)\hat{f}_{k_1}(\xi-\eta) \hat{g}_{k_2}(\xi)d\eta\bigg\|_{L^2}\\
    \lesssim &\|\F^{-1}m(\xi,\eta)\Tilde{\psi}_{k}(\xi)\Tilde{\psi}_{k_1}(\xi-\eta)\Tilde{\psi}_{k_2}(\eta)\|_{L^1}\|f_{k_1}\|_{L^2}\|g_{k_2}\|_{L^\infty}\\
    \lesssim &2^{3k_2/2}\|\F^{-1}m(\xi,\eta)\Tilde{\psi}_{k}(\xi)\Tilde{\psi}_{k_1}(\xi-\eta)\Tilde{\psi}_{k_2}(\eta)\|_{L^1}\|f_{k_1}\|_{L^2}\|g_{k_2}\|_{L^2},
\end{align*}
where the last step follows from Bernstein's inequality in Lemma \ref{bernstein}.\\
On the other hand, we use duality to estimate the $L^2$ norm. Take $h\in\mathcal{S}(\R^3)$ with $\|h\|_{L^2}\leq 1$ and we have
\begin{align*}
    &\int_{\R^3}\int_{\R^3}m(\xi,\eta)\hat{f}_{k_1}(\xi-\eta)\hat{g}_{k_2}(\eta)\psi_k(\xi)d\eta h(\xi)d\xi\\
    =&\int_{\R^3}\int_{\R^3}m(\xi,\eta)\hat{f}_{k_1}(\xi-\eta)\psi_k(\xi) h(\xi)d\xi \hat{g}_{k_2}(\eta)d\eta\\
    \leq &\bigg\|\int_{\R^3}m(\xi,\eta)\hat{f}_{k_1}(\xi-\eta)\psi_k(\xi) h(\xi)d\xi\Tilde{\psi}_{k_2}(\eta)\bigg\|_{L^2}\| \hat{g}_{k_2}(\eta)\|_{L^2}\\
    = &\bigg\|\F^{-1}\int_{\R^3}m(\xi,\eta)\Tilde{\psi}_{k_2}(\eta)\hat{f}_{k_1}(\xi-\eta) h(\xi)\psi_k(\xi)d\xi\bigg\|_{L^2}\|g_{k_2}\|_{L^2}\\
    \lesssim &\|\F^{-1}m(\xi,\eta)\Tilde{\psi}_k(\xi)\Tilde{\psi}_{k_1}(\eta)\Tilde{\psi}_{k_2}(\xi-\eta)\|_{L^1}\|f_{k_1}\|_{L^2}\|\F^{-1}h(\xi)\psi_k(\xi)\|_{L^\infty}\|g_{k_2}\|_{L^2}\\
    \lesssim &\|\F^{-1}m(\xi,\eta)\Tilde{\psi}_k(\xi)\Tilde{\psi}_{k_1}(\xi-\eta)\Tilde{\psi}_{k_2}(\eta)\|_{L^1}\|f_{k_1}\|_{L^2}\|h(\xi)\psi_k(\xi)\|_{L^1}\|g_{k_2}\|_{L^2}\\
    \lesssim &\|\F^{-1}m(\xi,\eta)\Tilde{\psi}_k(\xi)\Tilde{\psi}_{k_1}(\xi-\eta)\Tilde{\psi}_{k_2}(\eta)\|_{L^1}\|f_{k_1}\|_{L^2}\|h(\xi)\|_{L^2}\|\psi_k(\xi)\|_{L^2}\|g_{k_2}\|_{L^2}\\
    \lesssim & 2^{3k/2}\|\F^{-1}m(\xi,\eta)\Tilde{\psi}_k(\xi)\Tilde{\psi}_{k_1}(\xi-\eta)\Tilde{\psi}_{k_2}(\eta)\|_{L^1}\|f_{k_1}\|_{L^2}\|g_{k_2}\|_{L^2},
\end{align*}
which finishes the proof.
\end{proof}
\subsection{Multiplier Estimates}
Recall the bilinear estimates in Lemma \ref{bilinear}. The $L^1$ norm of the inverse Fourier transform of the multiplier shows up on the right-hand side. In this section, we present the $L^1$ norm bounds on three different types of multipliers that we will encounter later in the main proof.
\begin{lem}
\label{chi,eta}
Fix $k\in\Z$. Suppose $\tau:\R^3\times\R^3\rightarrow\R$ is homogeneous of degree $n$ and smooth on $\{(\xi,\eta)\in\R^3\times\R^3:\xi\neq 0, \xi-\eta\neq 0\}$. If $|k_1-k|= O(1)$, then
\begin{equation*}
    \|\F^{-1}[\tau(\xi,\eta)m(\xi,\eta)\Tilde{\psi}^2_k(\xi)\Tilde{\psi}^2_{k_1}(\xi-\eta)\Tilde{\psi}_{k_2}(\eta)]\|_{L^1}\lesssim 2^{nk}\|\F^{-1}[m(\xi,\eta)\Tilde{\psi}_k(\xi)\Tilde{\psi}_{k_1}(\xi-\eta)\Tilde{\psi}_{k_2}(\eta)]\|_{L^1}
\end{equation*}
\end{lem}
\begin{proof}
Note that $\tau(\xi,\eta)\Tilde{\psi}_k(\xi)\Tilde{\psi}_{k_1}(\xi-\eta)\in\mathcal{S}(\R^3\times\R^3)$. Using Young's inequality, we get
\begin{align*}
    &\|\F^{-1}[\tau(\xi,\eta)m(\xi,\eta)\Tilde{\psi}^2_k(\xi)\Tilde{\psi}^2_{k_1}(\xi-\eta)\Tilde{\psi}_{k_2}(\eta)]\|_{L^1}\\
    = &\|\F^{-1}[m(\xi,\eta)\Tilde{\psi}_k(\xi)\Tilde{\psi}_{k_1}(\xi-\eta)\Tilde{\psi}_{k_2}(\eta)]*\F^{-1}[\tau(\xi,\eta)\Tilde{\psi}_k(\xi)\Tilde{\psi}_{k_1}(\xi-\eta)]\|_{L^1}\\
    \leq &\|\F^{-1}[m(\xi,\eta)\Tilde{\psi}_k(\xi)\Tilde{\psi}_{k_1}(\xi-\eta)\Tilde{\psi}_{k_2}(\eta)]\|_{L^1}\|\F^{-1}[\tau(\xi,\eta)\Tilde{\psi}_k(\xi)\Tilde{\psi}_{k_1}(\xi-\eta)]\|_{L^1}\\
    \lesssim & \|\F^{-1}[m(\xi,\eta)\Tilde{\psi}_k(\xi)\Tilde{\psi}_{k_1}(\xi-\eta)\Tilde{\psi}_{k_2}(\eta)]\|_{L^1}2^{nk}\|\F^{-1}[\tau(\xi/2^k,\eta/2^k)\Tilde{\psi}(\xi/2^k)\Tilde{\psi}(2^{k-k_1}(\xi-\eta)/2^{k})]\|_{L^1}\\
    =& \|\F^{-1}[m(\xi,\eta)\Tilde{\psi}_k(\xi)\Tilde{\psi}_{k_1}(\xi-\eta)\Tilde{\psi}_{k_2}(\eta)]\|_{L^1}2^{nk}\|\F^{-1}[\tau(\xi,\eta)\Tilde{\psi}(\xi)\Tilde{\psi}(2^{k-k_1}(\xi-\eta))]\|_{L^1}.
\end{align*}
The last line follows from the scale-invariant property $\|\F^{-1}D_\delta f(x)\|_{L^1}=\|\F^{-1}f(x)\|_{L^1}$, where $D_\delta f(x)=f(x/\delta)$. Then we are left to show,
\begin{equation*}
    \sup_{|k-k_1|= O(1)}\|\F^{-1}[\tau(\xi,\eta)\Tilde{\psi}(\xi)\Tilde{\psi}(2^{k-k_1}(\xi-\eta))]\|_{L^1}\lesssim 1.
\end{equation*}
Take an arbitrary integer $|p|= O(1)$, define $f(\xi,\eta)=\tau(\xi,\eta)\Tilde{\psi}(\xi)\Tilde{\psi}(2^p(\xi-\eta))$. It suffices to show $f\in\mathcal{S}(\R^3\times\R^3)$.\\
We know $\Tilde{\psi}(\xi)$ is supported on $\{(\xi,\eta): 2^{-2a-1}\leq|\xi|\leq 2^{a+3}\}$ and $\Tilde{\psi}(2^p(\xi-\eta))$ is supported on $\{(\xi,\eta):2^{-2a-1-p}\leq |\xi-\eta|\leq 2^{a+3-p}\}$. Hence,
$$supp\, f\subset \{(\xi,\eta): 2^{-2a-1}\leq|\xi|\leq 2^{a+3},2^{-2a-1-p}\leq |\xi-\eta|\leq 2^{a+3-p}\}\subset\{(\xi,\eta): |\eta|\leq 2^{a+4},|\xi|\leq 2^{a+3}\}.$$
Thus, $f\in C^\infty(\R^3\times\R^3)$, since $\Tilde{\psi}, \tau$ are smooth on the support of $f$. Also, $f$ has compact support, so $f\in\mathcal{S}(\R^3\times\R^3)$.
\end{proof}

\begin{lem}
\label{chi+eta}
Suppose $\tau_1:\R^3\rightarrow\R$ is homogeneous of degree $n_1$, $\tau_2:\R^3\rightarrow\R$ is homogeneous of degree $n_2$, and $\tau_1$ and $\tau_2$ are both smooth on $\R^3\setminus\{0\}$, then
\begin{equation*}
    \|\F^{-1}[\tau_1(\xi)\tau_2(\eta)m(\xi,\eta)\Tilde{\psi}^2_k(\xi)\Tilde{\psi}_{k_1}(\xi-\eta)\Tilde{\psi}^2_{k_2}(\eta)]\|_{L^1}\lesssim 2^{n_1k+n_2k_2}\|\F^{-1}[m(\xi,\eta)\Tilde{\psi}_k(\xi)\Tilde{\psi}_{k_1}(\xi-\eta)\Tilde{\psi}_{k_2}(\eta)]\|_{L^1}.
\end{equation*}
\end{lem}
\begin{proof}
Again, we use Young's inequality,
\begin{align*}
    &\|\F^{-1}[\tau_1(\xi)\tau_2(\eta)m(\xi,\eta)\Tilde{\psi}^2_k(\xi)\Tilde{\psi}_{k_1}(\xi-\eta)\Tilde{\psi}^2_{k_2}(\eta)]\|_{L^1}\\
    = &\|\F^{-1}[m(\xi,\eta)\Tilde{\psi}_k(\xi)\Tilde{\psi}_{k_1}(\xi-\eta)\Tilde{\psi}_{k_2}(\eta)]*\F^{-1}[\tau_1(\xi)\tau_2(\eta)\Tilde{\psi}_k(\xi)\Tilde{\psi}_{k_2}(\eta)]\|_{L^1}\\
    \leq &\|\F^{-1}[m(\xi,\eta)\Tilde{\psi}_k(\xi)\Tilde{\psi}_{k_1}(\xi-\eta)\Tilde{\psi}_{k_2}(\eta)]\|_{L^1}\|\F^{-1}[\tau_1(\xi)\tau_2(\eta)\Tilde{\psi}_k(\xi)\Tilde{\psi}_{k_2}(\eta)]\|_{L^1}.
\end{align*}
Then, we have
\begin{align*}
    &\|\F^{-1}[\tau_1(\xi)\tau_2(\eta)\Tilde{\psi}_k(\xi)\Tilde{\psi}_{k_2}(\eta)]\|_{L^1}\\
    = & \int_{\R^3}\int_{\R^3}\bigg|\int_{\R^3}\int_{\R^3}e^{ix\cdot\xi}e^{iy\cdot\eta}\tau_1(\xi)\tau_2(\eta)\Tilde{\psi}_k(\xi)\Tilde{\psi}_{k_2}(\eta)d\eta d\xi\bigg| dx dy\\
    \leq & \int_{\R^3}\int_{\R^3}\bigg|\int_{\R^3}e^{ix\cdot\xi}\tau_1(\xi)\Tilde{\psi}_k(\xi)d\xi \bigg| \bigg|\int_{\R^3}e^{iy\cdot\eta}\tau_2(\eta)\Tilde{\psi}_{k_2}(\eta)d\eta\bigg| dy dx\\
    = &\int_{\R^3}\bigg|\int_{\R^3}e^{ix\cdot\xi}\tau_1(\xi)\Tilde{\psi}_k(\xi)d\xi \bigg|dx \int_{\R^3} \bigg|\int_{\R^3}e^{iy\cdot\eta}\tau_2(\eta)\Tilde{\psi}_{k_2}(\eta)d\eta\bigg| dy\\
    =& \|\F^{-1}[\tau_1(\xi)\Tilde{\psi}_k(\xi)]\|_{L^1}\|\F^{-1}[\tau_2(\eta)\Tilde{\psi}_{k_2}(\eta)]\|_{L^1}\\
    =& 2^{n_1k+n_2k_2}\|\F^{-1}[\tau_1(\xi/2^k)\Tilde{\psi}(\xi/2^k)]\|_{L^1}\|\F^{-1}[\tau_2(\eta/2^{k_2})\Tilde{\psi}(\eta/2^{k_2})]\|_{L^1}\\
    =& 2^{n_1k+n_2k_2}\|\F^{-1}[\tau_1(\xi)\Tilde{\psi}(\xi)]\|_{L^1}\|\F^{-1}[\tau_2(\eta)\Tilde{\psi}(\eta)]\|_{L^1}\\
    \lesssim & 2^{n_1k+n_2k_2},
\end{align*}
where the last line follows from $\tau_1(\xi)\Tilde{\psi}(\xi),\, \tau_2(\eta)\Tilde{\psi}(\eta)\in\mathcal{S}(
\R^3)$. We know 
$\tau_1(\xi)\Tilde{\psi}(\xi)$ and $\tau_2(\eta)\Tilde{\psi}(\eta)$ are both compactly supported on $\{\xi\in\R^3:2^{-2a-1}\leq |\xi|\leq 2^{a+3}\}$, and smooth on their support.
\end{proof}

\begin{lem}
\label{chi+eta2}
Suppose $\tau_1:\R^3\rightarrow\R$ is homogeneous of degree $n_1$, $\tau_2:\R^3\rightarrow\R$ is homogeneous of degree $n_2$, and $\tau_1$ and $\tau_2$ are both smooth on $\R^3\setminus\{0\}$, then
\begin{equation}
\label{etachi1}
    \|\F^{-1}[\tau_1(2c_m(\xi-\eta)-2c_n\eta)\tau_2(\eta)\Tilde{\psi}_{k_1}(2c_m(\xi-\eta)-2c_n\eta)\Tilde{\psi}_{k_2}(\eta)]\|_{L^1}\lesssim 2^{n_1k_1+n_2k_2},
\end{equation}
 for $c_m\neq 0$.
\end{lem}
\begin{proof}
By a change of variable $\zeta=2c_m(\xi-\eta)-2c_n\eta$,
\begin{align*}
    &\|\F^{-1}[\tau_1(2c_m(\xi-\eta)-2c_n\eta)\tau_2(\eta)\Tilde{\psi}_{k_1}(2c_m(\xi-\eta)-2c_n\eta)\Tilde{\psi}_{k_2}(\eta)]\|_{L^1}\\
    =&\int_{\R^3}\int_{\R^3}\bigg|\int_{\R^3}\int_{\R^3}e^{ix\cdot\xi}e^{iy\cdot\eta}\tau_1(2c_m(\xi-\eta)-2c_n\eta)\tau_2(\eta)\Tilde{\psi}_{k_1}(2c_m(\xi-\eta)-2c_n\eta)\Tilde{\psi}_{k_2}(\eta)d\xi d\eta \bigg| dx dy\\
    =&\int_{\R^3}\int_{\R^3}\bigg|\int_{\R^3}\int_{\R^3}e^{ix\cdot(\frac{\zeta}{2c_m}+\frac{c_n}{c_m}\eta+\eta)}e^{iy\cdot\eta}\tau_1(\zeta)\Tilde{\psi}_{k_1}(\zeta)\tau_2(\eta)\Tilde{\psi}_{k_2}(\eta)d\zeta d\eta \bigg|dx dy\\
    =&\int_{\R^3}\int_{\R^3}\bigg|\int_{\R^3}\int_{\R^3}e^{i(\frac{c_n}{c_m}x+x+y)\cdot\eta}e^{ix\cdot\frac{\zeta}{2c_m}}\tau_1(\zeta)\Tilde{\psi}_{k_1}(\zeta)\tau_2(\eta)\Tilde{\psi}_{k_2}(\eta)d\zeta d\eta \bigg| dx dy\\
    =&\int_{\R^3}\int_{\R^3}\bigg|\int_{\R^3}e^{i(\frac{c_n}{c_m}x+x+y)\cdot\eta}\tau_2(\eta)\Tilde{\psi}_{k_2}(\eta)d\eta\bigg|dy\bigg|\int_{\R^3}e^{i\frac{x}{2c_m}\cdot\zeta}\tau_1(\zeta)\Tilde{\psi}_{k_1}(\zeta) d\zeta\bigg| dx.
\end{align*}
Then, let $z_1=\frac{c_n}{c_m}x+x+y$, $z_2=\frac{x}{2c_m}$ and use the homogeneity of $\tau_1,\tau_2$,
\begin{align*}
    &\|F^{-1}\tau_1(\xi-\eta)\tau_2(\eta)\Tilde{\psi}_{k_1}(\xi-\eta)\Tilde{\psi}_{k_2}(\eta)\|_{L^1}\\
    \lesssim &\int_{\R^3}\bigg|\int_{\R^3}e^{iz_1\cdot\eta}\tau_2(\eta)\Tilde{\psi}_{k_2}(\eta)d\eta\bigg| dz_1\int_{\R^3}\bigg|\int_{\R^3}e^{iz_2\cdot\zeta}\tau_1(\zeta)\Tilde{\psi}_{k_1}(\zeta)d\zeta \bigg|dz_2\\
    =&\int_{\R^3}\bigg|\int_{\R^3}e^{iz_1\cdot\eta}\tau_2(\eta)\Tilde{\psi}(\eta/2^{k_2})d\eta\bigg| dz_1\int_{\R^3}\bigg|\int_{\R^3}e^{iz_2\cdot\zeta}\tau_1(\zeta)\Tilde{\psi}(\zeta/2^{k_1})d\zeta \bigg|dz_2\\
    =&2^{n_1k_1+n_2k_2}\int_{\R^3}\bigg|\int_{\R^3}e^{iz_1\cdot\eta}\tau_2(\eta/2^{k_2})\Tilde{\psi}(\eta/2^{k_2})d\eta\bigg| dz_1\int_{\R^3}\bigg|\int_{\R^3}e^{iz_2\cdot\zeta}\tau_1(\zeta/2^{k_1})\Tilde{\psi}(\zeta/2^{k_1})d\zeta \bigg|dz_2\\
    \lesssim &2^{n_1k_1+n_2k_2}\|\F^{-1}D_{2^{k_1}}(\tau_1\Tilde{\psi})\|_{L^1}\|\F^{-1}D_{2^{k_2}}(\tau_2\Tilde{\psi})\|_{L^1},
\end{align*}
where the operator $D_\delta$ denotes dilation, i.e. $D_\delta \psi(x)=\psi(x/\delta)$. Employing the dilation invariant property $\|\F^{-1}D_\delta \psi(x)\|_{L^1}=\|\F^{-1}\psi(x)\|_{L^1}$, we have
\begin{align*}
    &\|\F^{-1}D_{2^{k_1}}(\tau_1\Tilde{\psi})\|_{L^1}\|\F^{-1}D_{2^{k_2}}(\tau_2\Tilde{\psi})\|_{L^1}=\|\F^{-1}(\tau_1\Tilde{\psi})\|_{L^1}\|\F^{-1}(\tau_2\Tilde{\psi})\|_{L^1}\lesssim 1.
\end{align*}
This follows from $\tau_1(\xi)\Tilde{\psi}(\xi),\, \tau_2(\eta)\Tilde{\psi}(\eta)\in\mathcal{S}(
\R^3)$, since
$\tau_1(\xi)\Tilde{\psi}(\xi)$ and $\tau_2(\eta)\Tilde{\psi}(\eta)$ are both compactly supported and smooth on their support. Thus, we showed \eqref{etachi1}.
\end{proof}

\subsection{Linear Schrödinger Equation}
\label{section for lse}
The lemma below is a well-known result for the group of operators $\{e^{it\la}\}_{t=-\infty}^{\infty}$, which gives the solution to the linear initial value problem
\begin{equation*}
    \begin{cases}
     \partial_t u= i\Delta u\\
     u(x,0)=u_{0}
    \end{cases},
\end{equation*}
when applied to the initial data $u_0$. Using the fact that
$$e^{it\la}u_0=\F^{-1} e^{-it|\xi|^2}\hat{u}_0(\xi)=\frac{e^{-|x|^2/4it}}{(4it\pi)^{3/2}}*u_0,$$
we observe that Plancherel's theorem implies that $e^{it\la}:L^2(\R^3)\rightarrow L^2(\R^3)$ is an isometry. Furthermore, the convolution form, in combination with Young's inequality, shows that $e^{it\la}:L^1\rightarrow L^\infty$. By interpolation between these two bounds, we obtain the properties of the operator on the $L^q$ spaces for $q\in[1,2]$.
\begin{lem}
\label{op}
If $t\neq 0$, $1/p+1/q=1$, and $q\in[1,2]$, then $e^{it\la}:L^q(\R^3)\rightarrow L^p(\R^3)$ is continuous and
\begin{equation}
    \|e^{it\la}f\|_{L^p}\lesssim |t|^{-3/2(1/q-1/p)}\|f\|_{L^q}.
\end{equation}
\end{lem}
\begin{proof}
See \autocite[Lemma~4.1]{text}.
\end{proof}
Note that in this paper, we only consider $t\geq 1$ to avoid singularity. This estimate will be used to obtain decay in the time variable in the later proofs. Another important property of the group $\{e^{it\la}\}_{t=-\infty}^{\infty}$ is its global smoothing effect, namely the Strichartz's estimates. Before we state the inequality, we need to introduce the definition of an admissible pair in our case when the dimension $n=3$.
\begin{defn}[admissible pair]
    We say $(q,r)$ is admissible, if 
    $$\frac{2}{q}=3\bigg(\frac{1}{2}-\frac{1}{r}\bigg)$$
    and $2\leq r\leq 6$.
\end{defn}
The Strichartz's estimates encompass a family of results. Here, we will focus on stating the specific inequality that is relevant to this paper.
\begin{lem}[Strichartz's estimate]
    Let $I$ be an interval of $\R$ and $t_0\in\bar{I}$. If $(\gamma,\rho)$ is an admissible pair and $f\in L_t^{\gamma'}(I)L_x^{\rho'}(\R^3)$, then for every admissible pair $(q,r)$, there exists a constant $C$ independent of $I$ such that
    $$\bigg\|\int_{t_0}^te^{i(t-s)\la}f(s)ds\bigg\|_{L^q_t(I)L_x^r(\R^3)}\leq C \|f\|_{L^{\gamma'}_t(I)L^{\rho'}_x(\R^3)}.$$
\end{lem}
\begin{proof}
    See \autocite[Theorem~2.3.3]{cln}.
\end{proof}
Using the Strichartz's estimate above, we may obtain a new type of bilinear estimate. 
\begin{lem}
\label{strichartz}
For any $2^{M-1}\leq t_1\leq t_2\leq 2^M$,
\begin{align*}
    &\bigg\|\int_{t_1}^{t_2}\int_{\R^3}e^{ic_lt|\xi|^2}m(\xi,\eta)\widehat{e^{ic_mt\la} f}(t,\xi-\eta)\widehat{e^{ic_nt\la}g}(t,\eta)d\eta dt\bigg\|_{L^2}\\
    \lesssim &2^{-M/4}\|\F^{-1}m(\xi,\eta)\|_{L^1}\|\hat{f}(t,\xi)\|_{L^\infty_t([2^{M-1},2^M])H^1_\xi}\|\hat{g}(t,\xi)\|_{L^\infty_t([2^{M-1},2^M])L^2_\xi}.
\end{align*}
\end{lem}
\begin{proof}
Let $F(t,x)=\F^{-1}\int_{\R^3}e^{ic_lt|\xi|^2}m(\xi,\eta)\widehat{e^{ic_mt\la} f}(t,\xi-\eta)\widehat{e^{ic_nt\la}g}(t,\eta)d\eta$. We rewrite the left hand side and use Plancherel's theorem,
\begin{align*}
    &\bigg\|\int_{t_1}^{t_2}\int_{\R^3}e^{ic_lt|\xi|^2}m(\xi,\eta)\widehat{e^{ic_mt\la} f}(t,\xi-\eta)\widehat{e^{ic_nt\la}g}(t,\eta)d\eta dt\bigg\|_{L^2}\\
    =&\bigg\|\int_{t_1}^{t_2}F(t,x) dt\bigg\|_{L^2}\\
    \leq &\bigg\|\int_{t_1}^t e^{ic_l(t-s)\la}e^{-ic_l(t-s)\la}F(s,x) ds\bigg\|_{L_t^\infty([t_1,t_2])L_x^2(\R^3)}\\
    \leq & \bigg\|\int_{t_1}^t e^{ic_l(t-s)\la}e^{ic_ls\la}F(s,x) ds\bigg\|_{L_t^\infty([t_1,t_2])L_x^2(\R^3)}.
\end{align*}
Since $(\infty,2)$ and $(4,3)$ are admissible pairs, we obtain from Strichartz's estimate, 
\begin{align*}
    \bigg\|\int_{t_1}^{t_2}\int_{\R^3}e^{ic_lt|\xi|^2}m(\xi,\eta)\widehat{e^{ic_mt\la} f}(t,\xi-\eta)\widehat{e^{ic_nt\la}g}(t,\eta)d\eta dt\bigg\|_{L^2}
    \lesssim & \|e^{ic_ls\la}F(s,x)\|_{L_s^{4/3}([t_1,t_2])L^{3/2}_x(\R^3)}.
\end{align*}
Then, plugging in the definition of $F$ and using the bilinear estimate $L^6\times L^2\rightarrow L^{3/2}$ give
\begin{align*}
    &\|e^{ic_ls\la}F(s,x)\|_{L_s^{4/3}([t_1,t_2])L^{3/2}_x(\R^3)}\\
    =&\bigg\|e^{ic_lt\la}\F^{-1}\int_{\R^3}e^{ic_lt|\xi|^2}m(\xi,\eta)\widehat{e^{ic_mt\la} f}(t,\xi-\eta)\widehat{e^{ic_nt\la}g}(t,\eta)d\eta\bigg\|_{L_t^{4/3}([t_1,t_2])L^{3/2}_x(\R^3)}\\
    =&\bigg\|\F^{-1}\int_{\R^3}m(\xi,\eta)\widehat{e^{ic_mt\la} f}(t,\xi-\eta)\widehat{e^{ic_nt\la}g}(t,\eta)d\eta\bigg\|_{L_t^{4/3}([t_1,t_2])L^{3/2}_x(\R^3)}\\
    \lesssim &\|\F^{-1}m(\xi,\eta)\|_{L^1}\big\|\|e^{ic_mt\la}f(t,x)\|_{L_x^6}\|e^{ic_nt\la}g(t,x)\|_{L_x^2}\big\|_{L^{4/3}_t([t_1,t_2])}\\
    \lesssim &\|\F^{-1}m(\xi,\eta)\|_{L^1}\|e^{ic_mt\la}f(t,x)\|_{L_t^{4/3}([t_1,t_2])L_x^6}\|g(t,x)\|_{L_t^\infty([t_1,t_2])L_x^2}.
\end{align*}
Lastly, \autocite[Corollary~2.5.4]{cln} gives
\begin{align*}
    \|e^{ic_mt\la}f(t,x)\|_{L_t^{4/3}([t_1,t_2])L_x^6}
    \lesssim \|t^{-1}\|_{L_t^{4/3}([t_1,t_2])}\|\hat{f}\|_{L_t^\infty([t_1,t_2])H^1_\xi}\lesssim 2^{-M/4}\|\hat{f}\|_{L_t^\infty([t_1,t_2])H^1_\xi}.
\end{align*}
\end{proof}
We compare this method with simply taking the $L^\infty$ norm in $t$ and applying the bilinear estimate $L^\infty\times L^2\rightarrow L^2$ in Lemma \ref{bilinear},
\begin{align*}
    &\bigg\|\int_{t_1}^{t_2}\int_{\R^3}e^{ic_lt|\xi|^2}m(\xi,\eta)\widehat{e^{ic_mt\la} f}(t,\xi-\eta)\widehat{e^{ic_nt\la}g}(t,\eta)d\eta dt\bigg\|_{L^2}\\
    \leq & 2^M\bigg\|\int_{\R^3}m(\xi,\eta)\widehat{e^{ic_mt\la} f}(t,\xi-\eta)\widehat{e^{ic_nt\la}g}(t,\eta)d\eta\bigg\|_{L^\infty_t([2^{M-1},2^M])L^2_\xi}\\
    \lesssim &2^M\|\F^{-1}m(\xi,\eta)\|_{L^1}\|e^{ic_mt\la}f(t,x)\|_{L^\infty_t([2^{M-1},2^M])L_x^\infty}\|g(t,x)\|_{L_t^\infty([2^{M-1},2^M])L_x^2}\\
    \lesssim &2^{-M/2}\|\F^{-1}m(\xi,\eta)\|_{L^1}\|f(t,x)\|_{L^\infty_t([2^{M-1},2^M])L_x^1}\|g(t,x)\|_{L_t^\infty([2^{M-1},2^M])L_x^2}.
\end{align*}
This other method provides a greater decay in the time variable, $t\sim 2^M$. However, bounding the $L^1$ norm of $f$ also requires more derivatives than $\|\hat{f}\|_{H^1_\xi}$, which we will see in the next section.

\subsection{Littlewood-Paley Operator}
One may recall the definition of $Z$ norm in \eqref{Znormdef}, the frequency space is decomposed into dyadic annuli by the Littlewood-Paley operator. In this section, we will present some useful estimates as results of this decomposition. Bernstein's inequality below is one basic result due to the localization in frequency space.
\begin{lem}[Bernstein's inequality]
\label{bernstein}
For any $1 \leq p\leq q\leq \infty$,
\begin{align*}
    \|f_k\|_{L^q}\lesssim 2^{3k(1/p-1/q)}\|f_k\|_{L^p}.
\end{align*}
\end{lem}
\begin{proof}
Since $\hat{f}_k=\hat{f}\psi_k=\hat{f}_k\Tilde{\psi}_k$, where $\Tilde{\psi}_k(\xi)=\Tilde{\psi}(\xi/2^k)$ is smooth and compactly supported, by Young's convolution inequality,
\begin{align*}
    \|f_k\|_{L^q}=\|\F^{-1}\Tilde{\psi}_k(\xi)*f_k\|_{L^q}\lesssim \|\F^{-1}\Tilde{\psi}_k(\xi)\|_{L^{\frac{1}{1-1/p+1/q}}}\|f_k\|_{L^{p}}.
\end{align*}
Then the computation
\begin{align*}
    \|\F^{-1}\Tilde{\psi}_k(\xi)\|_{L^{\frac{1}{1-1/p+1/q}}}
    &=\bigg(\int_{\R^3}\bigg|\int_{\R^3}e^{it\xi\cdot x}\Tilde{\psi}(\xi/2^k)dx\bigg|^{\frac{1}{1-1/p+1/q}}dx\bigg)^{1-1/p+1/q}\\
    &=2^{3k}\bigg(\int_{\R^3}\bigg|\F^{-1}\Tilde{\psi}(2^kx)\bigg|^{\frac{1}{1-1/p+1/q}}dx\bigg)^{1-1/p+1/q}\\
    &=2^{3k}\bigg(\int_{\R^3}\bigg|\F^{-1}\Tilde{\psi}(y)\bigg|^{\frac{1}{1-1/p+1/q}}2^{-3k}dy\bigg)^{1-1/p+1/q}\\
    &\lesssim 2^{3k-3k(1-1/p+1/q)}\|\F^{-1}\Tilde{\psi}\|_{L^{\frac{1}{1-1/p+1/q}}}\\
    &\lesssim 2^{3k(1/p-1/q)}
\end{align*}
implies the desired estimate.
\begin{align*}
    \|f_k\|_{L^q}\lesssim 2^{3k(1/p-1/q)}\|f_k\|_{L^{p}}.
\end{align*}
\end{proof}
In addition, the localization allows us to control the $L^2$ norm of a function by the $L^2$ norm of its derivative.
\begin{lem}
For any $\hat{f}_k\in H^1(\R^3)$,
    \begin{equation}
    \label{addder}
    \|\hat{f}_k\|_{L^2}\lesssim 2^k\|\nabla_\xi\hat{f}_k\|_{L^2}.
    \end{equation}
\end{lem}
\begin{proof}
This is an immediate result of Hölder's and Sobolev's inequality,
\begin{equation*}
    \|\hat{f}_k\|_{L^2}\leq \|\chi_{2^{k-2}\leq |\xi|\leq 2^{k+1}}\|_{L^3}\|\hat{f}_k\|_{L^6}\lesssim 2^k\|\nabla_\xi\hat{f}_k\|_{L^2}.
\end{equation*}
\end{proof}
In order to use Lemma \ref{op}, it is necessary to find a bound on $\|f_k\|_{L^p}$ for $p\in[1,2]$. We may use Hölder's inequality to control $\|f_k\|_{L^p}$ with $p<2$ by weighted $L^2$ norms. Since derivatives on the Fourier side can be thought of as weights on the space side, we then use the bounds on $\|D^\beta_\xi\hat{f}_k\|_{L^2}$ provided in the $Z$ norm. This is shown in the result below. 
\begin{lem}
\label{lpnorms}
For any $\beta\geq 0$, if $p>\frac{6}{3+2\beta}$, then
\begin{align}
\label{lpnormseq1}
    \|f_k\|_{L^p}\lesssim 2^{-3k(1/p-1/2)+\beta k}\|D^{\beta}_\xi\hat{f}_k\|_{L^2}.
\end{align}
Furthermore, if $p>\frac{6}{5+2\alpha}$, we have
\begin{align}
\label{lpnormseq2}
    \|f_k\|_{L^p}\lesssim 2^{-2k_++\gamma k+2k-3k/p}\|f\|_Z,
\end{align}
and if $p>\frac{6}{3+2\alpha}$, we have
\begin{align}
    \label{lpnormseq3}
    \|\F^{-1}\nabla_\xi\hat{f}_k\|_{L^p}\lesssim 2^{-2k_++\gamma k+k-3k/p}\|f\|_Z.
\end{align}
\end{lem}
\begin{proof}
By Hölder's inequality and Plancherel's theorem, 
\begin{align*}
    \|f_k\|_{L^p}
    &\leq\bigg(\int_{|x|\leq 2^{-k}}|f_k(x)|^pdx\bigg)^{1/p}+\bigg(\int_{|x|\geq 2^{-k}}|f_k(x)|^pdx\bigg)^{1/p}\\
    & \leq \|f_k\|_{L^2}\bigg(\int_{|x|\leq 2^{-k}}1dx\bigg)^{\frac{2-p}{2p}}+\bigg(\int_{|x|\geq 2^{-k}}|x|^{-\beta p}||x|^{\beta}f_k(x)|^pdx\bigg)^{1/p}\\
    &\lesssim 2^{-3k(1/p-1/2)}\|f_k\|_{L^2}+\||x|^{\beta}f_k(x)\|_{L^2}\bigg(\int_{|x|\geq 2^{-k}}|x|^{-\beta\frac{2p}{2-p}}dx\bigg)^{\frac{2-p}{2p}}\\
    &\lesssim 2^{-3k(1/p-1/2)}\|f_k\|_{L^2}+2^{-3k(1/p-1/2)+\beta k}\|D^{\beta}_\xi \hat{f}_k\|_{L^2},
\end{align*}
the last line requires $p>\frac{6}{3+2\beta}$ in order for the integral to converge.
Moreover, by Sobolev's inequality
\begin{align*}
    \|f_k\|_{L^2}=\|\hat{f}_k\|_{L^2}\leq \|\chi_{2^{k-2}\leq|\xi|\leq 2^{k+2}}\|_{L^{\frac{3}{\beta}}}\|\hat{f}_k\|_{L^{\frac{6}{3-2\beta}}}\lesssim 2^{\beta k}\|D_\xi^{\beta}\hat{f}_k\|_{L^2}.
\end{align*}
Hence, we have
\begin{align*}
    \|f_k\|_{L^p}
    &\lesssim 2^{-3k(1/p-1/2)+\beta k}\|D^{\beta}_\xi \hat{f}_k\|_{L^2}.
\end{align*}
Thus, \eqref{lpnormseq2} is a direct consequence by taking $\beta=1+\alpha$ and using the definition of $\|\cdot\|_Z$,
\begin{align*}
    \|f_k\|_{L^p}
    &\lesssim 2^{-3k(1/p-1/2)+(1+\alpha)k}\|D^{1+\alpha}_\xi \hat{f}_k\|_{L^2}\leq 2^{-2k_++\gamma k+2k-3k/p}\|f\|_Z.
\end{align*}
If we replace $f_k$ by $\F^{-1}\nabla_\xi\hat{f}_k$ and let $\beta=\alpha$, we can obtain \eqref{lpnormseq3}
\begin{align*}
    \|\F^{-1}\nabla_\xi\hat{f}_k\|_{L^p}
    &\lesssim 2^{-3k(1/p-1/2)+\alpha k}\|D^{\alpha}_\xi\nabla_\xi \hat{f}_k\|_{L^2}\leq 2^{-2k_++\gamma k+k-3k/p}\|f\|_Z.
\end{align*}
\end{proof}
However, the lemma only gives us a bound on $\|f_k\|_{L^p}$ for $p>1$ close to $1$. This is due to the constraint, $\alpha <1/2$, imposed on the term $\|D^{1+\alpha}\hat{f}_k\|_{L^2}$ in $Z$ norm.

\section{Preliminary Estimates}
\label{prelim}
\subsection{Sobolev and $L^p$ norms}
There are places in Sections \ref{section for fractional derivatives} and \ref{section for lse}, where we have to consider the Sobolev norms of $\hat{f}$. The $Z$ norm is not sufficient to bound the Sobolev norms. However, after Littlewood-Paley decomposition, we can find a bound on the Sobolev norms of $\hat{f}_k$, given $f$ bounded in the $H^{10}\cap Z$ space.
\begin{lem}
\label{sobolev space}
If $\|f\|_{H^{10}}+\|f\|_{Z}\leq a$, then for all $k\in\Z$
\begin{align}
    \|\hat{f}_k\|_{H^\alpha}\lesssim 2^{-2k_++\gamma k+k/2-\alpha k}a,\,\|\hat{f}_k\|_{H^1}\lesssim 2^{-2k_++\gamma k-k/2}a,\,\|\hat{f}_k\|_{H^{1+\alpha}}\lesssim 2^{-2k_++\gamma k-k/2-\alpha k}a.
\end{align}
In particular, the bootstrap assumption \eqref{b1} implies
\begin{equation}
\label{l2}
    \|f_{l,k}\|_{L^\infty_t([1,T])L^2_x}\lesssim \min\{2^{\gamma k-2k_++k/2},2^{-10k_+}\}\e_1.
\end{equation}
Under the bootstrap assumption \eqref{b1}, we also have
\begin{equation}
    \label{sobolevnorms}
    \begin{aligned}
    \|D^\alpha_\xi\hat{f}_{l,k}\|_{L^\infty_t([1,T])L^2_\xi}\leq\|\hat{f}_{l,k}\|_{L^\infty_t([1,T])H^\alpha_\xi}\lesssim 2^{-2k_++\gamma k+k/2-\alpha k}\e_1,\\
    \|\nabla_\xi\hat{f}_{l,k}\|_{L^\infty_t([1,T])L^2_\xi}\leq \|\hat{f}_{l,k}\|_{L^\infty_t([1,T])H^1_\xi}\lesssim 2^{-2k_++\gamma k-k/2}\e_1,\\
    \|D_\xi^{1+\alpha}\hat{f}_{l,k}\|_{L^\infty_t([1,T])L^2_\xi}\leq \|\hat{f}_{l,k}\|_{L^\infty_t([1,T])H^{1+\alpha}_\xi}\lesssim 2^{-2k_++\gamma k-k/2-\alpha k}\e_1.
\end{aligned}
\end{equation}
\end{lem}
\begin{proof}
We start by showing for any $k\leq 0$ and $s\geq 0$,
\begin{align}
\label{sobolevspaceeq1}
    \|\hat{f}_k(\xi)\|_{H^s}\lesssim \|D^s_\xi \hat{f}_k(\xi)\|_{L^2}.
\end{align}
To prove \eqref{sobolevspaceeq1}, it suffices to show if $s_1\leq s_2$, then $\|D^{s_1}\hat{f}_k\|_{L^2}\lesssim \|D^{s_2}\hat{f}_k\|_{L^2}$ when $k\leq 0$.
\begin{align*}
    \|D^{s_1}\hat{f}_k\|_{L^2}\leq \|\chi_{2^{k-2}\leq|\xi|\leq 2^{k+1}}\|_{L^{\frac{3}{s_2-s_1}}}\|D^{s_1}\hat{f}_k\|_{L^{\frac{6}{3-2(s_2-s_1)}}}\lesssim 2^{(s_2-s_1)k}\|D^{s_2}\hat{f}_k\|_{L^2}\leq \|D^{s_2}\hat{f}_k\|_{L^2},
\end{align*}
using Sobolev's inequality with $s=s_2-s_1$ and $p=\frac{6}{3-2(s_2-s_1)}$.\\
Since $\|f\|_{H^{10}}+\|f\|_{Z}\leq a$, by the definition of $\|\cdot\|_Z$, we obtain the following $L^2$ norm bound,
\begin{equation}
\label{l2(1+a)}
    \|D_\xi^{1+\alpha}\hat{f}_{l,k}(t,\xi)\|_{L^2_\xi}\lesssim 2^{\gamma k-2k_+-k/2-\alpha k}a.
\end{equation}
Then, Hölder's and Sobolev's inequality imply
\begin{equation}
\label{l21}
    \|f_{l,k}(t,x)\|_{L^2}
    \lesssim \|\chi_{2^{k-2}\leq |\xi|\leq 2^{k+1}}\|_{L^{\frac{3}{1+\alpha}}_\xi}\|f_{l,k}(t,\xi)\|_{L^\frac{6}{1-2\alpha}_\xi}
    \lesssim 2^{k+\alpha k} \|D^{1+\alpha}\hat{f}_{l,k}(t,\xi)\|_{L^2_\xi}
    \lesssim 2^{\gamma k-2k_++k/2}a
\end{equation}
and
\begin{equation}
\label{l2first}
    \|\nabla_\xi\hat{f}_{l,k}(t,\xi)\|_{L^2_\xi}
    \lesssim \|\chi_{2^{k-2}\leq |\xi|\leq 2^{k+1}}\|_{L^{\frac{3}{\alpha}}_\xi}\|\nabla_\xi\hat{f}_{l,k}(t,\xi)\|_{L^\frac{6}{3-2\alpha}_\xi}
    \lesssim 2^{\alpha k} \|D^{1+\alpha}\hat{f}_{l,k}(t,\xi)\|_{L^2_\xi}
    \lesssim 2^{\gamma k-2k_+-k/2}a.
\end{equation}
Since $\|e^{ic_lt\la}f_{l,k}(t,x)\|_{L^2_x}\lesssim 2^{-10k}\||\xi|^{10}e^{-ic_lt|\xi|^2}\hat{f}_l(t,\xi)\psi_k(\xi)\|_{L^2_\xi}\lesssim 2^{-10k}\|e^{ic_lt\la}f_l(t,x)\|_{H^{10}_x},$
\begin{equation}
\label{l22}
    \|f_{l,k}(t,x)\|_{L^2_x}=\|e^{ic_lt\la}f_{l,k}(t,x)\|_{L^2_x}\lesssim 2^{-10k}\| e^{ic_lt\la}f_l(t,x)\|_{H^{10}_x}\leq 2^{-10k}a.
\end{equation}
Combine \eqref{l21} and \eqref{l22}, we get
\begin{equation}
    \|f_{l,k}(t,x)\|_{L^2_x}\lesssim \min\{2^{\gamma k-2k_++k/2},2^{-10k_+}\}a,
\end{equation}
and
\begin{align*}
    \|\hat{f}_k\|_{H^1}
    \leq \|f_k\|_{L^2}+\|\nabla_\xi\hat{f}_k\|_{L^2}\lesssim \min\{2^{\gamma k-2k_++k/2},2^{-10k_+}\}a+2^{\gamma k-2k_+-k/2}a\lesssim 2^{\gamma k-2k_+-k/2}a.
\end{align*}
Then we discuss the bound on $\|\hat{f}_k\|_{H^{1+\alpha}}$ in two cases. 

When $k\leq 0$,
\begin{align*}
    \|\hat{f}_k\|_{H^{1+\alpha}}\lesssim \|D^{1+\alpha}_\xi\hat{f}_k\|_{L^2}\lesssim 2^{\gamma k-2k_+-k/2-\alpha k}a,
\end{align*}
by the definition of $\|\cdot\|_Z$. 

When $k>0$, Gagliardo-Nirenberg inequality in \cite{Brezis} gives,
\begin{align*}
    \|\hat{f}_k\|_{H^1}\lesssim \|\hat{f}_k\|_{L^2}^{\frac{\alpha}{1+\alpha}}\|\hat{f}_k\|_{H^{1+\alpha}}^{\frac{1}{1+\alpha}},
\end{align*}
so
\begin{align*}
    \|\hat{f}_k\|_{H^1}^{1+\alpha}\lesssim \|\hat{f}_k\|_{L^2}^{\alpha}\|\hat{f}_k\|_{H^{1+\alpha}}\lesssim 2^{-10\alpha k}a^\alpha (\|\hat{f}_k\|_{H^1}+\|D^{1+\alpha}\hat{f}\|_{L^2})\lesssim 2^{-10\alpha k+\gamma k-5k/2}a^{1+\alpha}.
\end{align*}
Since $\alpha\in(1/3,1/2)$,
\begin{align*}
    \|\hat{f}_k\|_{H^1}\lesssim 2^{-10k+\gamma k+\frac{15k}{2(1+\alpha)}}a\leq 2^{\gamma k-5k/2-\alpha k}a\leq 2^{\gamma k-2k_+-k/2-\alpha k}a.
\end{align*}
\\
Thus,
\begin{align*}
    \|\hat{f}_k\|_{H^{1+\alpha}}\lesssim \|\hat{f}_k\|_{H^1}+\|D^{1+\alpha}_\xi\hat{f}_k\|_{L^2}\lesssim 2^{\gamma k-2k_+-k/2-\alpha k}a,
\end{align*}
and by Gagliardo-Nirenberg inequality
\begin{align*}
    \|\hat{f}_k\|_{H^\alpha}\lesssim \|\hat{f}\|_{L^2}^{1-\alpha}\|\hat{f}_k\|_{H^1}^\alpha
    \lesssim 2^{\gamma k -2k_++k/2-\alpha k}a.
\end{align*}
\end{proof}
In the next lemma, we combine the results in Lemma \ref{lpnorms} and \ref{sobolev space} to derive some bounds on the $L^p$ norms of the solution $u_{l,k}$. Later these bounds will be used when we perform the $L^2\times L^\infty\rightarrow L^2$, $L^4\times L^4\rightarrow L^2$, $L^3\times L^6\rightarrow L^2$ bilinear estimates.
\begin{lem}
\label{estimates}
Under the assumption \eqref{b1}, for any $t\in[2^{M-1},2^M]$
\begin{equation}
\label{linfinity}
    \|e^{ic_lt\la}f_{l,k}\|_{L^\infty}\lesssim \min\{2^{-3M/2+\gamma M/8-2k_++\gamma k-k+\gamma k/4}, 2^{\gamma k-2k_++2k}\}\e_1,
\end{equation}
\begin{equation}
    \label{linfinity2}
    \|e^{ic_lt\la}f_{l,k}\|_{L^\infty}\lesssim 2^{-M/2+3k/2}\|\nabla_\xi\hat{f}_{l,k}\|_{L^2}\lesssim 2^{-M/2-2k_++\gamma k+k}\e_1,
\end{equation}
\begin{equation}
\label{l3nablaf}
    \|e^{ic_lt\la}\F^{-1}\nabla_\xi\hat{f}_{l,k}\|_{L^3}\lesssim 2^{-M/2+\gamma M/8-2k_++\gamma k-k+\gamma k/4}\e_1,
\end{equation}
and
\begin{equation}
\label{l4}
    \|e^{ic_lt\la}f_{l,k}\|_{L^4}\lesssim 2^{-3M/4+k/4}\|\nabla_\xi\hat{f}_{l,k}\|_{L^2}\lesssim 2^{-3M/4-2k_++\gamma k-k/4}\e_1.
\end{equation}
\end{lem}
\begin{proof}
Using Bernstein's inequality in Lemma \ref{bernstein}, we have
\begin{align*}
    \|e^{ic_lt\la}f_{l,k}\|_{L^\infty}
    &\lesssim 2^{\gamma k/8}\|e^{ic_lt\la}f_{l,k}\|_{L^{24/\gamma}}.
\end{align*}
Then by Lemma \ref{op},
\begin{align*}
    \|e^{ic_lt\la}f_{l,k}\|_{L^{24/\gamma}}
    \lesssim t^{-3/2(1-2\gamma/24)}\|f_{l,k}\|_{L^{\frac{24}{24-\gamma}}}
    \lesssim 2^{-3M/2+\gamma M/8-2k_++\gamma k-k+\gamma k/8}\e_1,
\end{align*}
since $\alpha>1/2-\gamma/8$ allows us to use Lemma \ref{lpnorms}.\\
Another way to control $\|e^{ic_lt\la}f_{l,k}\|_{L^\infty}$ is to use Hölder's inequality directly,
\begin{align*}
    \|e^{ic_lt\la}f_{l,k}\|_{L^\infty}
    &=\bigg\|\int e^{ix\cdot\xi}e^{-ic_lt|\xi|^2}\hat{f}_l(\xi)\psi_k(\xi)d\xi\bigg\|_{L^\infty}\\
    &\leq \int |\hat{f}_l(\xi)\psi_k(\xi)|d\xi\\
    &\leq \|\hat{f}_l(\xi)\psi_k(\xi)\|_{L^2}\|\chi_{2^{k-2}\leq|\xi|\leq 2^k}\|_{L^2}
    \leq 2^{\gamma k-2k_++2k}\e_1.
\end{align*}
Combining the two results above gives \eqref{linfinity}.\\
If we use Bernstein's inequality with a different parameter, Lemma \ref{op}, and \eqref{lpnormseq1}, we can also get the following bound
\begin{equation*}
    \|e^{ic_lt\la}f_{l,k}\|_{L^\infty}\lesssim 2^{k}\|e^{ic_lt\la}f_{l,k}\|_{L^3}\lesssim 2^{k-M/2}\|f_{l,k}\|_{L^{3/2}}\lesssim 2^{-M/2+3k/2}\|\nabla_\xi\hat{f}_{l,k}\|_{L^2}.
\end{equation*}
We can obtain \eqref{l3nablaf} in a similar fashion using Bernstein's inequality, Lemma \ref{op}, and Lemma \ref{lpnorms},
\begin{align*}
    \|e^{ic_lt\la}\F^{-1}\nabla_\xi\hat{f}_{l,k}\|_{L^3}
    &\lesssim 2^{\gamma k/8}\|e^{ic_lt\la}\F^{-1}\nabla_\xi\hat{f}_{l,k}\|_{L^{\frac{24}{8+\gamma}}}\\
    &\lesssim 2^{\gamma k/8-M/2+\gamma M/8}\|\F^{-1}\nabla_\xi\hat{f}_{l,k}\|_{L^{\frac{24}{16-\gamma}}}\lesssim 2^{-M/2+\gamma M/8-k+\gamma k/4 -2k_++\gamma k}\e_1.
\end{align*}
Next, Lemma \ref{op} and Lemma \ref{lpnorms} give
\begin{align*}
    \|e^{ic_lt\la}f_{l,k}\|_{L^4}\lesssim t^{-3/4}\|f_{l,k}\|_{L^{4/3}}\lesssim 2^{-3M/4+k/4}\|\nabla_\xi\hat{f}_{l,k}\|_{L^2}\lesssim 2^{-3M/4-2k_++\gamma k-k/4}\e_1,
\end{align*}
for $t\in[2^{M-1},2^M]$.
\end{proof}

\subsection{Time Derivatives Estimates}
An important idea in this paper is to make use of the oscillation of $e^{it\phi(\xi,\eta)}$, where $\phi(\xi,\eta)=c_l|\xi|^2-c_m|\xi-\eta|^2-c_n|\eta|^2$ is called the phase function. To capture the oscillation, we will integrate by parts using the identity 
$$(\partial_t+\frac{P(\xi,\eta)}{t}\cdot\nabla_\eta)e^{it\phi(\xi,\eta)}=iZ(\xi,\eta)e^{it\phi(\xi,\eta)}$$
where $P(\xi,\eta)$ and $Z(\xi,\eta)$ are polynomials. This introduces time derivatives inside the Duhamel's term. Hence, we shall derive bounds for the time derivatives of $\hat{f}_{l,k}(t,\xi)$, $D^\alpha_\xi\hat{f}_{l,k}(t,\xi)$, and $\nabla_\xi\hat{f}_{l,k}(t,\xi)$ in this section.

We will use the fact that $u_l(t,x)=e^{ic_lt\la}f_l(t,x)$ is the solution to \eqref{pde} such that $f_l\in\mathcal{C}([1,T];H^{10}\cap Z)$ to find bounds for the time derivatives of $\hat{f}_{l,k}(t,x)$. In addition, we will do a further Littlewood-Paley decomposition on $f_{m}$ and $f_n$ in the nonlinear term.
\begin{lem}
\label{timel2}
Under the assumption \eqref{b1}, the following estimate holds for any $t\in[2^{M-1},2^M]$
\begin{align}
\label{l2timef}
    \|\partial_t\hat{f}_{l,k}(t,\xi)\|_{L^2}\lesssim \min\{2^{-M-2k_++\gamma k+k/2},2^{-10k_+-(1+\gamma/2)M}\}\e_1^2
\end{align}
and
\begin{align}
\label{l6timef}
    \|e^{ic_lt\la}\partial_t f_{l,k}(t,\xi)\|_{L^6}\lesssim 2^{-2M-2k_++\gamma k-k/2}\e_1^2.
\end{align}
\end{lem}
\begin{proof}
Assuming \eqref{b1}, the quadratic nonlinearity allows us to write the time derivative of the profile $f_l=e^{-ic_lt\la}u_l$ as bilinear forms,
\begin{equation*}
    \begin{aligned}
       \partial_t\hat{f}_{l}(t,\xi)=&ic_l|\xi|^2e^{ic_lt|\xi|^2}\hat{u}_l(t,\xi)+e^{ic_lt|\xi|^2}\partial_t\hat{u}_l(t,\xi)\\
        =&\sum_{\substack{c_m+c_n\neq 0\\m\geq n}}A_{lmn}\int_{\R^3}e^{ic_lt|\xi|^2}\hat{u}_m(t,\xi-\eta)\hat{u}_n(t,\eta)d\eta\\
        &+\sum_{\substack{c_m+c_n= 0\\m\geq n}}A_{lmn}\int_{\R^3}e^{ic_lt|\xi|^2}q(\xi-\eta,\eta)\hat{u}_m(t,\xi-\eta)\hat{u}_n(t,\eta)d\eta\\
        =&\sum_{\substack{c_m+c_n\neq 0\\m\geq n}}A_{lmn}\int_{\R^3}e^{it(c_l|\xi|^2-c_m|\xi-\eta|^2-c_n|\eta|^2)}\hat{f}_m(t,\xi-\eta)\hat{f}_n(t,\eta)d\eta\\
        &+\sum_{\substack{c_m+c_n= 0\\m\geq n}}A_{lmn}\int_{\R^3}e^{it(c_l|\xi|^2-c_m|\xi|^2+2c_m\xi\cdot\eta)}q(\xi-\eta,\eta)\hat{f}_m(t,\xi-\eta)\hat{f}_n(t,\eta)d\eta.
    \end{aligned}
\end{equation*}
Next, we perform dyadic decomposition on the frequency spaces of $f_m$ and $f_n$,
\begin{equation*}
    \begin{aligned}
        \partial_t\hat{f}_l(t,\xi)
        =&\sum_{\substack{c_m+c_n\neq 0\\m\geq n}}A_{lmn}\sum_{k_1,k_2\in\Z}\int_{\R^3}e^{it(c_l|\xi|^2-c_m|\xi-\eta|^2-c_n|\eta|^2)}\hat{f}_{m,k_1}(t,\xi-\eta)\hat{f}_{n,k_2}(t,\eta)d\eta \\
        &+\sum_{\substack{c_m+c_n= 0\\m\geq n}}A_{lmn}\sum_{k_1,k_2\in\Z}\int_{\R^3}e^{it(c_l|\xi|^2-c_m|\xi|^2+2c_m\xi\cdot\eta)}q(\xi-\eta,\eta)\hat{f}_{m,k_1}(t,\xi-\eta)\hat{f}_{n,k_2}(t,\eta)d\eta.
    \end{aligned}
\end{equation*}
\label{key}
By a change of variable with $\xi-\eta=\zeta$, we get
\begin{equation*}
\begin{aligned}
    &\int_{\R^3}e^{it(c_l|\xi|^2-c_m|\xi-\eta|^2-c_n|\eta|^2)}\hat{f}_{m,k_1}(t,\xi-\eta)\hat{f}_{n,k_2}(t,\eta)d\eta\\
    =& \int_{\R^3}e^{it(c_l|\xi|^2-c_m|\zeta|^2-c_n|\xi-\zeta|^2)}\hat{f}_{m,k_1}(t,\zeta)\hat{f}_{n,k_2}(t,\xi-\zeta)d\zeta
\end{aligned}
\end{equation*}
and
\begin{equation*}
\begin{aligned}
    &\int_{\R^3}e^{it(c_l|\xi|^2-c_m|\xi|^2+2c_m\xi\cdot\eta)}q(\xi-\eta,\eta)\hat{f}_{m,k_1}(t,\xi-\eta)\hat{f}_{n,k_2}(t,\eta)d\eta\\
    =& \int_{\R^3}e^{it(c_l|\xi|^2+c_m|\xi|^2-2c_m\xi\cdot\zeta)}q(\zeta,\xi-\eta)\hat{f}_{m,k_1}(t,\zeta)\hat{f}_{n,k_2}(t,\xi-\zeta)d\zeta
\end{aligned}
\end{equation*}
The symmetry observed above allows us to reduce the problem to the cases when $k_1\geq k_2$,
\begin{align*}
    \partial_t\hat{f}_l(t,\xi)
    =&\sum_{\substack{c_m+c_n\neq 0\\m\geq n}}A_{lmn}\bigg(\sum_{k_1\geq k_2}\int_{\R^3}e^{it(c_l|\xi|^2-c_m|\xi-\eta|^2-c_n|\eta|^2)}\hat{f}_{m,k_1}(t,\xi-\eta)\hat{f}_{n,k_2}(t,\eta)d\eta\\
    &\qquad\qquad+\sum_{k_1<k_2}\int_{\R^3}e^{it(c_l|\xi|^2-c_m|\xi-\eta|^2-c_n|\eta|^2)}\hat{f}_{m,k_1}(t,\xi-\eta)\hat{f}_{n,k_2}(t,\eta)d\eta\bigg)\\
    &+\sum_{\substack{c_m+c_n= 0\\m\geq n}}A_{lmn}\bigg(\sum_{k_1\geq k_2}\int_{\R^3}e^{it(c_l|\xi|^2-c_m|\xi|^2+2c_m\xi\cdot\eta)}q(\xi-\eta,\eta)\hat{f}_{m,k_1}(t,\xi-\eta)\hat{f}_{n,k_2}(t,\eta)d\eta\\
    &\qquad\qquad+\sum_{k_1<k_2}\int_{\R^3}e^{it(c_l|\xi|^2-c_m|\xi|^2+2c_m\xi\cdot\eta)}q(\xi-\eta,\eta)\hat{f}_{m,k_1}(t,\xi-\eta)\hat{f}_{n,k_2}(t,\eta)d\eta \bigg)\\
    =&\sum_{c_m+c_n\neq 0}A_{lmn}\sum_{k_1\geq k_2}\int_{\R^3}e^{it(c_l|\xi|^2-c_m|\xi-\eta|^2-c_n|\eta|^2)}\hat{f}_{m,k_1}(t,\xi-\eta)\hat{f}_{n,k_2}(t,\eta)d\eta \\
    &+\sum_{c_m+c_n= 0}A_{lmn}\sum_{k_1\geq k_2}\int_{\R^3}e^{it(c_l|\xi|^2-c_m|\xi|^2+2c_m\xi\cdot\eta)}q(\xi-\eta,\eta)\hat{f}_{m,k_1}(t,\xi-\eta)\hat{f}_{n,k_2}(t,\eta)d\eta .
\end{align*}
Truncating the difference by $\psi_k(\xi)$ and taking
\begin{equation}
    \label{adef}
    a=\max_{\substack{c_m+c_n\neq 0\\c_l\neq c_m}}\{\big|2\log_2|c_l|\big|+3,2-\log_2|c_m/c_n+1|, 10-\log_2|c_m/c_l-1|,3-\log_2(|c_l-c_m|/|c_n|)\},
\end{equation}
we have
\begin{align*}
    \partial_t\hat{f}_{l,k}(t,\xi)
    =&\sum_{c_m+c_n\neq 0}A_{lmn}\sum_{k_1\geq k_2}\int_{\R^3}e^{it(c_l|\xi|^2-c_m|\xi-\eta|^2-c_n|\eta|^2)}\hat{f}_{m,k_1}(t,\xi-\eta)\hat{f}_{n,k_2}(t,\eta)\psi_k(\xi)d\eta \\
    &+\sum_{c_m+c_n= 0}A_{lmn}\sum_{k_1\geq k_2}\int_{\R^3}e^{it(c_l|\xi|^2-c_m|\xi|^2+2c_m\xi\cdot\eta)}q(\xi-\eta,\eta)\hat{f}_{m,k_1}(t,\xi-\eta)\hat{f}_{n,k_2}(t,\eta)d\eta \\
    =&\sum_{\substack{c_m+c_n\neq 0\\m\geq n}}A_{lmn}\sum_{|k_1-k_2|< a}\int_{\R^3}e^{it(c_l|\xi|^2-c_m|\xi-\eta|^2-c_n|\eta|^2)}\hat{f}_{m,k_1}(t,\xi-\eta)\hat{f}_{n,k_2}(t,\eta)\psi_k(\xi)d\eta \\
    &+\sum_{\substack{c_m+c_n= 0\\m\geq n}}A_{lmn}\sum_{|k_1-k_2|< a}\int_{\R^3}e^{it(c_l|\xi|^2-c_m|\xi|^2+2c_m\xi\cdot\eta)}q(\xi-\eta,\eta)\hat{f}_{m,k_1}(t,\xi-\eta)\hat{f}_{n,k_2}(t,\eta)\psi_k(\xi)d\eta \\
    &+\sum_{c_m+c_n\neq 0}A_{lmn}\sum_{k_2 \leq k_1-a}\int_{\R^3}e^{it(c_l|\xi|^2-c_m|\xi-\eta|^2-c_n|\eta|^2)}\hat{f}_{m,k_1}(t,\xi-\eta)\hat{f}_{n,k_2}(t,\eta)\psi_k(\xi)d\eta \\
    &+\sum_{c_m+c_n= 0}A_{lmn}\sum_{k_2\leq k_1-a}\int_{\R^3}e^{it(c_l|\xi|^2-c_m|\xi|^2+2c_m\xi\cdot\eta)}q(\xi-\eta,\eta)\hat{f}_{m,k_1}(t,\xi-\eta)\hat{f}_{n,k_2}(t,\eta)\psi_k(\xi)d\eta \\
    =&\sum_{\substack{c_m+c_n\neq 0\\m\geq n}}A_{lmn}\sum_{\substack{|k_1-k_2|< a\\k\leq k_1+a+2}}\int_{\R^3}e^{it(c_l|\xi|^2-c_m|\xi-\eta|^2-c_n|\eta|^2)}\hat{f}_{m,k_1}(t,\xi-\eta)\hat{f}_{n,k_2}(t,\eta)\psi_k(\xi)d\eta \\
    & + \sum_{\substack{c_m+c_n= 0\\m\geq n}}A_{lmn}\sum_{\substack{|k_1-k_2|< a\\k\leq k_1+a+2}}\int_{\R^3}e^{it(c_l|\xi|^2-c_m|\xi|^2+2c_m\xi\cdot\eta)}q(\xi-\eta,\eta)\hat{f}_{m,k_1}(t,\xi-\eta)\hat{f}_{n,k_2}(t,\eta)\psi_k(\xi)d\eta \\
    &+\sum_{c_m+c_n\neq 0}A_{lmn}\sum_{\substack{k_2+a\leq k_1\\|k-k_1|\leq 3}}\int_{\R^3}e^{it(c_l|\xi|^2-c_m|\xi-\eta|^2-c_n|\eta|^2)}\hat{f}_{m,k_1}(t,\xi-\eta)\hat{f}_{n,k_2}(t,\eta)\psi_k(\xi)d\eta \\
    &+\sum_{c_m+c_n= 0}A_{lmn}\sum_{\substack{k_2+a\leq k_1\\|k-k_1|\leq 3}}\int_{\R^3}e^{it(c_l|\xi|^2-c_m|\xi|^2+2c_m\xi\cdot\eta)}q(\xi-\eta,\eta)\hat{f}_{m,k_1}(t,\xi-\eta)\hat{f}_{n,k_2}(t,\eta)\psi_k(\xi)d\eta,
\end{align*}
where $\chi^1_k=\{(k_1,k_2):|k-k_1|\leq 3,k_2\leq k_1-a\}$,
$\chi^2_k=\{(k_1,k_2):|k_1-k_2|< a,k< k_1-a-2\}$ and $\chi^3_k=\{(k_1,k_2):|k_1-k_2|< a,|k-k_1|\leq a+2\}$.

Thus, we have
\begin{equation}
\label{partialtf}
    \begin{aligned}
    \partial_t\hat{f}_{l,k}(t,\xi)
    =&\sum_{c_m+c_n\neq 0}A_{lmn}\sum_{(k_1,k_2)\in\chi_k^1} F^1_{k,k_1,k_2}+\sum_{\substack{c_m+c_n\neq 0\\m\geq n}}A_{lmn}\sum_{(k_1,k_2)\in\chi_k^2\cup\chi_k^3} F^1_{k,k_1,k_2}\\
    &+\sum_{c_m+c_n= 0}A_{lmn}\sum_{(k_1,k_2)\in\chi^1_k} F^2_{k,k_1,k_2}+\sum_{\substack{c_m+c_n= 0\\m\geq n}}A_{lmn}\sum_{(k_1,k_2)\in\chi^2_k\cup\chi^3_k} F^2_{k,k_1,k_2},
    \end{aligned}
\end{equation}
where
\begin{align}
\label{partialtf1}
    F^1_{k,k_1,k_2}
    =&\int_{\R^3}e^{it(c_l|\xi|^2-c_m|\xi-\eta|^2-c_n|\eta|^2)}\hat{f}_{m,k_1}(t,\xi-\eta)\hat{f}_{n,k_2}(t,\eta)d\eta\psi_k(\xi)
\end{align}
and
\begin{align}
\label{partialtf2}
    F^2_{k,k_1,k_2}
    =&\int_{\R^3}e^{it(c_l|\xi|^2-c_m|\xi|^2+2c_m\xi\cdot\eta)}q(\xi-\eta,\eta)\hat{f}_{m,k_1}(t,\xi-\eta)\hat{f}_{n,k_2}(t,\eta)\psi_k(\xi)d\eta.
\end{align}
By the bilinear estimate $L^4\times L^4\rightarrow L^2$ in Lemma \ref{bilinear} and Lemma \ref{dualitycomp},
\begin{equation}
\label{partialtfbound}
    \begin{aligned}
    \|\eqref{partialtf1}\|_{L^2}
    \leq &\bigg\|\int_{\R^3}e^{ic_lt|\xi|^2}e^{-ic_mt|\xi-\eta|^2}\hat{f}_{m,k_1}(t,\xi-\eta)e^{-ic_nt|\eta|^2}\hat{f}_{n,k_2}(t,\eta)\psi_k(\xi)d\eta\bigg\|_{L^2}\\
    \lesssim & \min\{\|e^{ic_mt\la}f_{m,k_1}\|_{L^4}\|e^{ic_nt\la}f_{n,k_2}\|_{L^4},2^{3\min\{k,k_2\}/2}\|e^{ic_mt\la}f_{m,k_1}\|_{L^2}\|e^{ic_nt\la}f_{n,k_2}\|_{L^2}\}\\
    \lesssim &\min\{\|e^{ic_mt\la}f_{m,k_1}\|_{L^4}\|e^{ic_nt\la}f_{n,k_2}\|_{L^4},2^{3\min\{k,k_2\}/2}\|f_{m,k_1}\|_{L^2}\|f_{n,k_2}\|_{L^2}\}.
    \end{aligned}
\end{equation}
Then, using the bounds in \eqref{l4}, \eqref{l2}, we get for any $t\in[2^{M-1},2^M]$,
\begin{equation}
\label{partialtfbound1.1}
    \begin{aligned}
    &\sum_{(k_1,k_2)\in\chi_k^1\cup\chi_k^3}\|\eqref{partialtf1}\|_{L^2}\\
    &\lesssim \sum_{(k_1,k_2)\in\chi_k^1\cup\chi_k^3}\min\{\|e^{ic_mt\la}f_{m,k_1}\|_{L^4}\|e^{ic_nt\la}f_{n,k_2}\|_{L^4}, 2^{3\min\{k,k_2\}/2}\|f_{m,k_1}\|_{L^2}\|f_{n,k_2}\|_{L^2}\}\\
    &\lesssim \sum_{(k_1,k_2)\in\chi_k^1\cup\chi_k^3}2^{-2k_{1,+}-2k_{2,+}+\gamma k_1+\gamma k_2}\min\{2^{-3M/2-k_1/4-k_2/4}, 2^{3\min\{k,k_2\}/2+k_1/2+k_2/2}\}\e_1^2\\
    &\lesssim 2^{-2k_{+}+\gamma k-M+k/2}\e_1^2.
    \end{aligned}
\end{equation}
Now, we shall show $\sum_{(k_1,k_2)\in\chi_k^2}\|\eqref{partialtf1}\|_{L^2}\lesssim 2^{-m-2k_++\gamma k+k/2}\e_1$. 

Recall that we defined $\phi(\xi,\eta)=c_l|\xi|^2-c_m|\xi-\eta|^2-c_n|\eta|^2.$
When $(k_1,k_2)\in\chi^1_k$, we know $k_2\leq k_1-a$, so $\psi_k(\xi)\psi_{k_1}(\xi-\eta)\psi_{k_2}(\eta)\neq 0$ implies
$$|\nabla_\eta\phi(\xi,\eta)|=|2c_m(\xi-\eta)-2c_n\eta|\geq 2^{k_1+1-2}|c_m|-2^{k_2+1}|c_n|\geq |c_m|2^{k_1-2}\geq 2^{k_1-2-a}$$
and
$$|\nabla_\eta\phi(\xi,\eta)|\leq 2^{k_1+1}|c_m|+2^{k_2+1}|c_n|\leq |c_m|2^{k_1+2}\leq 2^{k_1+2+a},$$
if $c_m\neq-c_n$.
When $(k_1,k_2)\in\chi^2_k$, we have $k\leq k_1-a-2$ and $\psi_k(\xi)\psi_{k_1}(\xi-\eta)\psi_{k_2}(\eta)\neq 0$ implies
$$|\nabla_\eta\phi(\xi,\eta)|\geq 2|c_m+c_n||\xi-\eta|-2|c_n||\xi|\geq |c_n|(2^{k_1+1-a}-2^{k+1})\geq |c_n|2^{k_1-a}\geq 2^{k_1-2a}$$
and
$$|\nabla_\eta\phi(\xi,\eta)|\leq 2|c_m+c_n||\xi-\eta|+2|c_n||\xi|\leq 2^{k_1+1+a}+2^{k+1+a}\leq 2^{k_1+2+a}.$$
Therefore, when $(k_1,k_2)\in\chi^1_k\cup\chi^2_k$
\begin{equation}
\label{xi-2eta}
    \Tilde{\psi}_{k_1}(\nabla_\eta\phi(\xi,\eta))\psi_k(\xi)\psi_{k_1}(\xi-\eta)\psi_{k_2}(\eta)=\psi_k(\xi)\psi_{k_1}(\xi-\eta)\psi_{k_2}(\eta).
\end{equation}
So, we use the identity $\sum_m\frac{\partial_{\eta_m}\phi(\xi,\eta)}{it|\nabla_\eta\phi(\xi,\eta)|^2}\partial_{\eta_m} e^{it\phi(\xi,\eta)}=e^{it\phi(\xi,\eta)}$ to integrate by parts,
\begin{align}
    \eqref{partialtf1}
    =&\int_{\R^3}\partial_{\eta_m} e^{it\phi(\xi,\eta)}\frac{\partial_{\eta_m}\phi(\xi,\eta)}{it|\nabla_\eta \phi(\xi,\eta)|^2}\hat{f}_{m,k_1}(t,\xi-\eta)\hat{f}_{n,k_2}(t,\eta)d\eta\psi_k(\xi)\nonumber\\
    =&-\int_{\R^3}e^{it\phi(\xi,\eta)}\partial_{\eta_m}\frac{\partial_{\eta_m}\phi(\xi,\eta)}{it|\nabla_\eta \phi(\xi,\eta)|^2}\hat{f}_{m,k_1}(t,\xi-\eta)\hat{f}_{n,k_2}(t,\eta)d\eta\psi_k(\xi)\label{l2t.1}\\
    &-\int_{\R^3} e^{it\phi(\xi,\eta)}\frac{\partial_{\eta_m}\phi(\xi,\eta)}{it|\nabla_\eta \phi(\xi,\eta)|^2}\partial_{\eta_m}\hat{f}_{m,k_1}(t,\xi-\eta)\hat{f}_{n,k_2}(t,\eta)d\eta\psi_k(\xi)\label{l2t.2}\\
    &-\int_{\R^3} e^{it\phi(\xi,\eta)}\frac{\partial_{\eta_m}\phi(\xi,\eta)}{it|\nabla_\eta \phi(\xi,\eta)|^2}\hat{f}_{m,k_1}(t,\xi-\eta)\partial_{\eta_m}\hat{f}_{n,k_2}(t,\eta)d\eta\psi_k(\xi)\label{l2t.3}.
\end{align}
Next, using the $L^{4}\times L^4\rightarrow L^2$ bilinear estimate, Lemma \ref{chi+eta2}, Lemma \ref{dualitycomp}, \eqref{l2} and \eqref{l4}, we have
\begin{align*}
    \|\eqref{l2t.1}\|_{L^2}
    \lesssim &2^{-m}\|\F^{-1}\partial_{\eta_m}\frac{\partial_{\eta_m}\phi(\xi,\eta)}{it|\nabla_\eta \phi(\xi,\eta)|^2}\Tilde{\psi}_{k_1}(\nabla_\eta\phi(\xi,\eta))\Tilde{\psi}_{k_2}(\eta)\|_{L^1}\times\\
    &\times\min\{\|e^{it\la}{f}_{m,k_1}\|_{L^4}\|e^{it\la}f_{n,k_2}\|_{L^4},2^{3\min\{k,k_2\}/2}\|e^{it\la}{f}_{m,k_1}\|_{L^2}\|e^{it\la}f_{n,k_2}\|_{L^2}\}\\
    \lesssim &2^{-M-2k_1-2k_{1,+}+\gamma k_1-2k_{2,+}+\gamma k_2}\min\{2^{-3M/2-k_1/4-k_2/4},2^{3\min\{k,k_2\}/2+k_1/2+k_2/2}\}\e_1^2\\
    \lesssim &2^{-2k_{1,+}+\gamma k_1-2k_{2,+}+\gamma k_2}\min\{2^{-5M/2-9k_1/4-k_2/4},2^{-M+3\min\{k,k_2\}/2-3k_1/2+k_2/2}\}\e_1^2\\
    \lesssim &2^{-4k_{1,+}+2\gamma k_1}\min\{2^{-5M/2-5k_1/2},2^{-M+3k/2-k_1}\}\e_1^2,
\end{align*}
since $(k_1,k_2)\in\chi^2_k=\{|k_1-k_2|< a,\,k< k_1-a-2\}$. Then,
using the $L^{18/7}\times L^9\rightarrow L^2$ bilinear estimate, Lemma \ref{op}, Lemma \ref{chi+eta2} and Lemma \ref{dualitycomp}, we get
\begin{align*}
    &\|\eqref{l2t.2}\|_{L^2}+\|\eqref{l2t.3}\|_{L^2}\\
    \lesssim &2^{-M}\|\F^{-1}\frac{\partial_{\eta_m}\phi(\xi,\eta)}{it|\nabla_\eta \phi(\xi,\eta)|^2}\Tilde{\psi}_{k_1}(\nabla_\eta\phi(\xi,\eta))\Tilde{\psi}_{n,k_2}(\eta)\|_{L^1}\times\\
    &\times\big(\min\{\|e^{it\la}\F^{-1}\nabla_\xi\hat{f}_{m,k_1}\|_{L^{18/7}}\|e^{it\la}f_{n,k_2}\|_{L^{9}},2^{3\min\{k,k_2\}/2}\|e^{it\la}\F^{-1}\nabla_\xi\hat{f}_{m,k_1}\|_{L^2}\|e^{it\la}f_{n,k_2}\|_{L^2}\}\\
    &+\min\{\|e^{it\la}{f}_{m,k_1}\|_{L^{9}}\|e^{it\la}\F^{-1}\nabla_\xi\hat{f}_{n,k_2}\|_{L^{18/7}},2^{3\min\{k,k_2\}/2}\|e^{it\la}{f}_{m,k_1}\|_{L^2}\|e^{it\la}\F^{-1}\nabla_\xi\hat{f}_{n,k_2}\|_{L^2}\}\big)\\
    \lesssim &2^{-M-k_1}
    \big(\min\{\|e^{it\la}\F^{-1}\nabla_\xi\hat{f}_{m,k_1}\|_{L^{18/7}}\|e^{it\la}f_{n,k_2}\|_{L^{9}},2^{3\min\{k,k_2\}/2}\|\nabla_\xi\hat{f}_{m,k_1}\|_{L^2}\|f_{n,k_2}\|_{L^2}\}\\
    &+\min\{\|e^{it\la}{f}_{m,k_1}\|_{L^{9}}\|e^{it\la}\F^{-1}\nabla_\xi\hat{f}_{n,k_2}\|_{L^{18/7}},2^{3\min\{k,k_2\}/2}\|{f}_{m,k_1}\|_{L^2}\|\nabla_\xi\hat{f}_{n,k_2}\|_{L^2}\}\big)\\
    \lesssim &2^{-M-k_1}\big(\min\{2^{-3M/2}\|\F^{-1}\nabla_\xi\hat{f}_{m,k_1}\|_{L^{18/11}}\|f_{n,k_2}\|_{L^{9/8}},2^{3\min\{k,k_2\}/2}\|\nabla_\xi\hat{f}_{m,k_1}\|_{L^2}\|f_{n,k_2}\|_{L^2}\}\\
    &+\min\{2^{-3M/2}\|{f}_{m,k_1}\|_{L^{9/8}}\|\F^{-1}\nabla_\xi\hat{f}_{n,k_2}\|_{L^{18/11}},2^{3\min\{k,k_2\}/2}\|{f}_{m,k_1}\|_{L^2}\|\nabla_\xi\hat{f}_{n,k_2}\|_{L^2}\}\big),
\end{align*}
where the $L^p$ norms can be bounded by Lemma \ref{lpnorms}, \eqref{l2} and \eqref{l2first}
\begin{align*}
    \|\eqref{l2t.2}\|_{L^2}+\|\eqref{l2t.3}\|_{L^2}
    \lesssim & 2^{-M-k_1-2k_{1,+}+\gamma k_1-2k_{2,+}+\gamma k_2}\big(\min\{2^{-3M/2-5k_1/6-2k_2/3},2^{3\min\{k,k_2\}/2-k_1/2+k_2/2}\}\\
    &+\min\{2^{-3M/2-2k_1/3-5k_2/6},2^{3\min\{k,k_2\}/2+k_1/2-k_2/2}\}\big)\e_1^2\\
    \lesssim & 2^{-4k_{1,+}+2\gamma k_1}\min\{2^{-5M/2-5k_1/2},2^{-M+3k/2-k_1}\}\e_1^2.
\end{align*}
Thus,
\begin{equation}
\label{partialtfbound1.2}
    \begin{aligned}
    \sum_{(k_1,k_2)\in\chi_k^2}\|\eqref{partialtf1}\|_{L^2}
    &\lesssim \sum_{(k_1,k_2)\in\chi_k^2}2^{-4k_{1,+}+2\gamma k_1}\min\{2^{-5M/2-5k_1/2},2^{-M+3k/2-k_1}\}\e_1^2\\
    &\lesssim \sum_{k_1\geq k}2^{-4k_{1,+}+2\gamma k_1}\min\{2^{-5M/2-5k_1/2},2^{-M+3k/2-k_1}\}\e_1^2\\
    &\lesssim 2^{-M-2k_++\gamma k+k/2}\e_1^2.
    \end{aligned}
\end{equation}
Next, we work on the term $F^2_{k,k_1,k_2}$. By the bilinear estimate $L^4\times L^4\rightarrow L^2$ in Lemma \ref{bilinear} and condition \eqref{q} on the multiplier $q(\xi,\eta)$,
\begin{align*}
    \|\eqref{partialtf2}\|_{L^2}
    =&\bigg\|\int_{\R^3}e^{itc_l|\xi|^2}q(\xi-\eta,\eta)e^{-itc_m|\xi-\eta|^2}\hat{f}_{m,k_1}(t,\xi-\eta)e^{-itc_n|\eta|^2}\hat{f}_{n,k_2}(t,\eta)\psi_k(\xi)d\eta\bigg\|_{L^2}\\
    \lesssim &\|\F^{-1}q(\xi-\eta,\eta)\Tilde{\psi}_k(\xi)\Tilde{\psi}_{k_1}(\xi-\eta)\Tilde{\psi}_{k_2}(\eta)\|_{L^1}\|e^{ic_mt\la}f_{m,k_1}\|_{L^4}\|e^{ic_nt\la}f_{n,k_2}\|_{L^4}\\
    \lesssim &2^{\e k_-}\|e^{ic_mt\la}f_{m,k_1}\|_{L^4}\|e^{ic_nt\la}f_{n,k_2}\|_{L^4}
\end{align*}
Moreover, Lemma \ref{dualitycomp} gives us,
\begin{align*}
    \|\eqref{partialtf2}\|_{L^2}
    \lesssim &2^{3\min\{k,k_2\}/2}\|\F^{-1}q(\xi-\eta,\eta)\Tilde{\psi}_k(\xi)\Tilde{\psi}_{k_1}(\xi-\eta)\Tilde{\psi}_{k_2}(\eta)\|_{L^1}\|e^{ic_mt\la}f_{m,k_1}\|_{L^2}\|e^{ic_nt\la}f_{n,k_2}\|_{L^2}\\
    \lesssim &2^{\e k_-+3\min\{k,k_2\}/2}\|f_{m,k_1}\|_{L^2}\|f_{n,k_2}\|_{L^2}
\end{align*}
Combing the results above and employ the bounds \eqref{l4}, \eqref{l2}, we get for any $t\in[2^{M-1},2^M]$,
\begin{equation}
    \label{partialtfbound2}
    \begin{aligned}
    &\sum_{(k_1,k_2)\in\chi_k^1\cup\chi_k^2\cup\chi_k^3} \|\eqref{partialtf2}\|_{L^2}\\
    \lesssim &\sum_{(k_1,k_2)\in\chi_k^1\cup\chi_k^2\cup\chi_k^3}2^{\e k_-}\min\{\|e^{ic_mt\la}f_{m,k_1}\|_{L^4}\|e^{ic_nt\la}f_{n,k_2}\|_{L^4}, 2^{3\min\{k,k_2\}/2}\|f_{m,k_1}\|_{L^2}\|f_{n,k_2}\|_{L^2}\}\\
    \lesssim &\sum_{(k_1,k_2)\in\chi_k^1\cup\chi_k^2\cup\chi_k^3}2^{\e k_--2k_{1,+}-2k_{2,+}+\gamma k_1+\gamma k_2}\min\{2^{-3M/2-k_1/4-k_2/4}, 2^{3\min\{k,k_2\}/2+k_1/2+k_2/2}\}\e_1^2\\
    \lesssim &2^{-2k_{+}+\gamma k-M+k/2}\e_1^2.
    \end{aligned}
\end{equation}
Thus, \eqref{partialtfbound1.1}, \eqref{partialtfbound1.2}, and \eqref{partialtfbound2} shows
$$\|\partial_t\hat{f}_k\|_{L^2}\lesssim 2^{-2k_{+}+\gamma k-M+k/2}\e_1^2.$$
If we use the $L^2\times L^\infty\rightarrow L^2$ bilinear estimate in Lemma \ref{bilinear},
\begin{align*}
    \|\eqref{partialtf1}\|_{L^2}+\|\eqref{partialtf2}\|_{L^2}
    \lesssim (1+2^{\e k_-})\|e^{ic_mt\la}f_{m,k_1}\|_{L^2}\|e^{ic_nt\la}f_{n,k_2}\|_{L^\infty}
    \lesssim \|f_{m,k_1}\|_{L^2}\|e^{ic_nt\la}f_{n,k_2}\|_{L^\infty}.
\end{align*}
By \eqref{linfinity}, \eqref{l2}, for any $t\in[2^{M-1},2^M]$
\begin{align*}
    \|\partial_t\hat{f}_{l,k}\|_{L^2}\lesssim &\sum_{(k_1,k_2)\in\chi_k^1\cup\chi_k^2\cup\chi_k^3}\|\eqref{partialtf1}\|_{L^2}+\|\eqref{partialtf2}\|_{L^2}\\
    &\lesssim \sum_{(k_1,k_2)\in\chi_k^1\cup\chi_k^2\cup\chi_k^3}\min\{\|f_{m,k_1}\|_{L^2}\|e^{ic_nt\la}f_{n,k_2}\|_{L^\infty}, 2^{3\min\{k,k_2\}/2}\|f_{m,k_1}\|_{L^2}\|f_{n,k_2}\|_{L^2}\}\\
    &\lesssim \sum_{(k_1,k_2)\in\chi_k^1\cup\chi_k^2\cup\chi_k^3}2^{-10k_{1,+}-2k_{2,+}+\gamma k_2}\min\{2^{-3M/2+\gamma M/8+\gamma k_2/4-k_2},2^{3\min\{k,k_2\}/2+k_2/2}\}\e_1^2\\
    &\lesssim 2^{-10k_+}\sum_{(k_1,k_2)\in\chi_k^1\cup\chi_k^2\cup\chi_k^3}2^{-2k_{2,+}+\gamma k_2}\min\{2^{-3M/2+\gamma M/8+\gamma k_2/4-k_2},2^{3\min\{k,k_2\}/2+k_2/2}\}\e_1^2\\
    &\lesssim 2^{-10k_+-M-\gamma M/2}\e_1^2.
\end{align*}
Next, we consider $\|e^{ic_lt\la}\partial_t f_{l,k}(t,\xi)\|_{L^6}=\|\F^{-1}e^{-ic_lt|\xi|^2}\partial_t\hat{f}_{l,k}(t,\xi)\|_{L^6}$. Using the $L^6\times L^\infty\rightarrow L^6$ bilinear estimate, Lemma \ref{op}, Lemma \ref{lpnorms} and \eqref{linfinity},
\begin{align*}
    &\|\F^{-1}e^{-ic_lt|\xi|^2}\eqref{partialtf1}\|_{L^6}+\|\F^{-1}e^{-ic_lt|\xi|^2}\eqref{partialtf2}\|_{L^6}\\
    \lesssim &(1+2^{\e k_-})\|e^{ic_mt\la}f_{m,k_1}\|_{L^6}\|e^{ic_nt\la}f_{n,k_2}\|_{L^\infty}\\
    \lesssim & 2^{-M}\|f_{m,k_1}\|_{L^{6/5}}2^{-2k_{2,+}+\gamma k_2}\min\{2^{-3M/2+\gamma M/8-k_2+\gamma k_2/4},2^{2k_2}\}\e_1\\
    \lesssim & 2^{-2k_{1,+}+\gamma k_1-2k_{2,+}+\gamma k_2}\min\{2^{-5M/2+\gamma M/8-k_2+\gamma k_2/4-k_1/2},2^{-M+2k_2-k_1/2}\}\e_1^2.
\end{align*}
Moreover, by Bernstein's inequality, Lemma \ref{dualitycomp} and \eqref{l2},
\begin{align*}
    \|\F^{-1}e^{-ic_lt|\xi|^2}\eqref{partialtf1}\|_{L^6}+\|\F^{-1}e^{-ic_lt|\xi|^2}\eqref{partialtf2}\|_{L^6}
    \lesssim &2^k\big(\|\F^{-1}e^{-it|\xi|^2}\eqref{partialtf1}\|_{L^2}+\|\F^{-1}e^{-it|\xi|^2}\eqref{partialtf2}\|_{L^2}\big)\\
    \lesssim & (1+2^{\e k_-})2^{k+3k/2}\|e^{ic_mt\la}f_{m,k_1}\|_{L^2}\|e^{ic_nt\la}f_{n,k_2}\|_{L^2}\\
    \lesssim & 2^{-2k_{1,+}+\gamma k_1-2k_{2,+}+\gamma k_2+5k/2+k_1/2+k_2/2}\e_1^2.
\end{align*}
Thus,
\begin{align*}
    &\|e^{ic_lt\la}\partial_t f_{l,k}(t,\xi)\|_{L^6}\\
    \leq &\sum_{(k_1,k_2)\in\chi^1_k\cup\chi^2_k\cup\chi^3_k}\|\F^{-1}e^{-ic_lt|\xi|^2}\eqref{partialtf1}\|_{L^6}+\|\F^{-1}e^{-ic_lt|\xi|^2}\eqref{partialtf2}\|_{L^6}\\
    \leq &\sum_{(k_1,k_2)\in\chi^1_k}2^{-2k_{1,+}+\gamma k_1-2k_{2,+}+\gamma k_2}\min\{2^{-5M/2+\gamma M/8-k_2+\gamma k_2/4-k_1/2},2^{-M+2k_2-k_1/2}\}\e_1^2\\
    &+\sum_{(k_1,k_2)\in\chi^2_k\cup\chi^3_k}2^{-2k_{1,+}+\gamma k_1-2k_{2,+}+\gamma k_2}\min\{2^{-5M/2+\gamma M/8-k_2+\gamma k_2/4-k_1/2},2^{5k/2+k_1/2+k_2/2}\}\e_1^2\\
    \lesssim & 2^{-2k_++\gamma k-2M-k/2}\e_1^2.
\end{align*}
\end{proof}
Based on the setup we used in Lemma \ref{timel2} for $\partial_t\hat{f}_{l,k}$, we can further take one derivative in $\xi$ and find bounds on the mixed derivatives of $\hat{f}_{l,k}$.
\begin{lem}
\label{mixder}
Under the assumption \eqref{b1}, we have for any $t\in[2^{M-1},2^M]$
\begin{align}
\label{l2mixedalphader}
    \|D^\alpha\partial_t\hat{f}_{l,k}(t,\xi)\|_{L^2}\lesssim 2^{-M-2k_++(1-\alpha)\gamma k+k/2-\alpha k}\e_1^2
\end{align}
and
\begin{align}
\label{l2mixedder}
    \|\nabla_\xi\partial_t\hat{f}_{l,k}(t,\xi)\|_{L^2}\lesssim (1+2^{M/2+k})2^{-M-2k_+-k/2}\e_1^2.
\end{align}
\end{lem}
\begin{proof}
We split $\partial_t \hat{f}_k=F_{1,k}+F_{2,k}$ to discuss the problem differently when $(k_1,k_2)\in\chi^1_k\cup\chi^2_k$ and when $(k_1,k_2)\in\chi^3_k$ for $F^1_{k,k_1,k_2}$. Recall the equation of $\partial_t f_k$ in \eqref{partialtf} and define
\begin{align}
\label{Fsetupdef}
    F_{1,k}=\sum_{(k_1,k_2)\in\chi_k^1\cup\chi_k^2}F^1_{k,k_1,k_2}+\sum_{(k_1,k_2)\in\chi_k^1\cup\chi_k^2\cup\chi_k^3}F^2_{k,k_1,k_2},
    \qquad
    F_{2,k}=\sum_{(k_1,k_2)\in\chi_k^3}F^1_{k,k_1,k_2}.
\end{align}
Hence, \eqref{partialtfbound1.1}, \eqref{partialtfbound1.2}, and \eqref{partialtfbound2} imply
\begin{equation}
\label{l2timeder1}
    \|F_{1,k}\|_{L^2}\lesssim 2^{-2k_{+}+\gamma k-M+k/2}\e_1^2
\end{equation}
and
\begin{equation}
\label{l2timeder2}
\begin{aligned}
    \|F_{2,k}\|_{L^2}\leq&\sum_{(k_1,k_2)\in\chi^3_k}\bigg\|\int_{\R^3}e^{it(c_l|\xi|^2-c_m|\xi-\eta|^2-c_n|\eta|^2)}\hat{f}_{m,k_1}(t,\xi-\eta)\hat{f}_{n,k_2}(t,\eta)d\eta\psi_k(\xi)\bigg\|_{L^2}\\
    \leq&\sum_{(k_1,k_2)\in\chi^3_k}2^{-2k_{1,+}-2k_{2,+}+\gamma k_1+\gamma k_2}\min\{2^{-3M/2-k_1/4-k_2/4}, 2^{3\min\{k,k_2\}/2+k_1/2+k_2/2}\}\e_1^2\\
    \lesssim & 2^{-4k_++2\gamma k}\min\{2^{-3M/2-k/2},2^{5k/2}\}\e_1^2.
\end{aligned}
\end{equation}
Next,
\begin{align}
    \partial_{\xi_l}F^1_{k,k_1,k_2}
    =&\int_{\R^3}e^{it\phi(\xi,\eta)}it\partial_{\xi_l}\phi(\xi,\eta)\hat{f}_{m,k_1}(t,\xi-\eta)\hat{f}_{n,k_2}(t,\eta)d\eta\psi_k(\xi)\label{mixed1.1}\\
    & + \int_{\R^3}e^{it\phi(\xi,\eta)}\partial_{\xi_l}\hat{f}_{m,k_1}(t,\xi-\eta)\hat{f}_{n,k_2}(t,\eta)d\eta\psi_k(\xi)\label{mixed1.2}\\
    &+ \int_{\R^3}e^{it\phi(\xi,\eta)}\hat{f}_{m,k_1}(t,\xi-\eta)\hat{f}_{n,k_2}(t,\eta)d\eta\partial_{\xi_l}\psi_k(\xi)\label{mixed1.3},
\end{align}
\begin{align}
    \partial_{\xi_l}F^2_{k,k_1,k_2}
    =&\int_{\R^3}e^{it\phi(\xi,\eta)}it\partial_{\xi_l}\phi(\xi,\eta)q(\xi-\eta,\eta)\hat{f}_{m,k_1}(t,\xi-\eta)\hat{f}_{n,k_2}(t,\eta)d\eta\psi_k(\xi)\label{mixed2.1}\\
    &+\int_{\R^3}e^{it\phi(\xi,\eta)}\partial_{\xi_l}q(\xi-\eta,\eta)\hat{f}_{m,k_1}(t,\xi-\eta)\hat{f}_{n,k_2}(t,\eta)d\eta\psi_k(\xi)\label{mixed2.2}\\
    &+\int_{\R^3}e^{it\phi(\xi,\eta)}q(\xi-\eta,\eta)\partial_{\xi_l}\hat{f}_{m,k_1}(t,\xi-\eta)\hat{f}_{n,k_2}(t,\eta)d\eta\psi_k(\xi)\label{mixed2.3}\\
    &+\int_{\R^3}e^{it\phi(\xi,\eta)}q(\xi-\eta,\eta)\hat{f}_{m,k_1}(t,\xi-\eta)\hat{f}_{n,k_2}(t,\eta)d\eta\partial_{\xi_l}\psi_k(\xi)\label{mixed2.4},
\end{align}
When $c_m+c_n\neq 0$ and $(k_1,k_2)\in\chi^1_k\cup\chi^2_k$, we have \eqref{xi-2eta} and may use the identity 
$$\sum_{m=1}^3\frac{\partial_{\eta_m}\phi(\xi,\eta)}{it|\nabla_\eta\phi(\xi,\eta)|^2}\partial_{\eta_m} e^{it\phi(\xi,\eta)}=e^{it\phi(\xi,\eta)}$$
to integrate by parts,
\begin{align}
    \eqref{mixed1.1}
    =&-\int_{\R^3} e^{it\phi(\xi,\eta)}\partial_{\eta_m}\frac{\partial_{\xi_l}\phi(\xi,\eta)\partial_{\eta_m}\phi(\xi,\eta)}{|\nabla_\eta\phi(\xi,\eta)|^2}\hat{f}_{m,k_1}(t,\xi-\eta)\hat{f}_{n,k_2}(t,\eta)d\eta\psi_k(\xi)\label{mixed1.1.1}\\
    &-\int_{\R^3} e^{it\phi(\xi,\eta)}\frac{\partial_{\xi_l}\phi(\xi,\eta)\partial_{\eta_m}\phi(\xi,\eta)}{|\nabla_\eta\phi(\xi,\eta)|^2}\partial_{\eta_m}\hat{f}_{m,k_1}(t,\xi-\eta)\hat{f}_{n,k_2}(t,\eta)d\eta\psi_k(\xi)\label{mixed1.1.2}\\
    &-\int_{\R^3} e^{it\phi(\xi,\eta)}\frac{\partial_{\xi_l}\phi(\xi,\eta)\partial_{\eta_m}\phi(\xi,\eta)}{|\nabla_\eta\phi(\xi,\eta)|^2}\hat{f}_{m,k_1}(t,\xi-\eta)\partial_{\eta_m}\hat{f}_{n,k_2}(t,\eta)d\eta\psi_k(\xi)\label{mixed1.1.3}.
\end{align}
For the term \eqref{mixed2.1}, $\phi(\xi,\eta)=(c_l+c_n)|\xi|^2-2c_n\xi\cdot\eta$ since $c_m+c_n=0$. In this case, we may perform integration by parts using the identity $\sum_{m=1}^3-\frac{\xi_m}{it2c_n|\xi|^2}\partial_{\eta_m}e^{it\phi(\xi,\eta)}=e^{it\phi(\xi,\eta)}$
\begin{align}
    \eqref{mixed2.1}
    =&\int_{\R^3}e^{it\phi(\xi,\eta)}\frac{\xi_m\partial_{\eta_m}\partial_{\xi_l}\phi(\xi,\eta)}{2c_n|\xi|^2}q(\xi-\eta,\eta)\hat{f}_{m,k_1}(t,\xi-\eta)\hat{f}_{n,k_2}(t,\eta)d\eta\psi_k(\xi)\label{mixed2.1.1}\\
    &+\int_{\R^3}e^{it\phi(\xi,\eta)}\frac{\xi_m\partial_{\xi_l}\phi(\xi,\eta)}{2c_n|\xi|^2}\partial_{\eta_m}q(\xi-\eta,\eta)\hat{f}_{m,k_1}(t,\xi-\eta)\hat{f}_{n,k_2}(t,\eta)d\eta\psi_k(\xi)\label{mixed2.1.2}\\
    &+\int_{\R^3}e^{it\phi(\xi,\eta)}\frac{\xi_m\partial_{\xi_l}\phi(\xi,\eta)}{2c_n|\xi|^2}q(\xi-\eta,\eta)\partial_{\eta_m}\hat{f}_{m,k_1}(t,\xi-\eta)\hat{f}_{n,k_2}(t,\eta)d\eta\psi_k(\xi)\label{mixed2.1.3}\\
    &+\int_{\R^3}e^{it\phi(\xi,\eta)}\frac{\xi_m\partial_{\xi_l}\phi(\xi,\eta)}{2c_n|\xi|^2}q(\xi-\eta,\eta)\hat{f}_{m,k_1}(t,\xi-\eta)\partial_{\eta_m}\hat{f}_{n,k_2}(t,\eta)d\eta\psi_k(\xi)\label{mixed2.1.4},
\end{align}
The $L^{18/7}\times L^9\rightarrow L^2$ bilinear estimate in Lemma \ref{bilinear}, Lemma \ref{chi+eta}, Lemma \ref{chi+eta2}, Lemma \ref{dualitycomp}, and the property of $q(\xi,\eta)$ in \eqref{q} imply
\begin{align*}
    &\|\eqref{mixed1.2}\|_{L^2}+\|\eqref{mixed1.1.2}\|_{L^2}+\|\eqref{mixed2.3}\|_{L^2}+\|\eqref{mixed2.1.3}\|_{L^2}\\
    \lesssim &\big(1+\|\F^{-1}\frac{\partial_{\xi_l}\phi(\xi,\eta)\partial_{\eta_m}\phi(\xi,\eta)}{|\nabla_\eta\phi(\xi,\eta)|^2}\Tilde{\psi}_{k_1}(\nabla_\eta\phi(\xi,\eta))\Tilde{\psi}_{k_2}(\eta)\|_{L^1}
    +\|\F^{-1}q(\xi-\eta,\eta)\Tilde{\psi}_k(\xi)\Tilde{\psi}_{k_1}(\xi-\eta)\Tilde{\psi}_{k_2}(\eta)\|_{L^1}+\\
    &\qquad + \|\F^{-1}\frac{\xi_m\partial_{\xi_l}\phi(\xi,\eta)}{2c_n|\xi|^2}q(\xi-\eta,\eta)\Tilde{\psi}_k(\xi)\Tilde{\psi}_{k_1}(\xi-\eta)\Tilde{\psi}_{k_2}(\eta)\|_{L^1} \big)\times\\
    &\times\min\{\|e^{ic_mt\la}\F^{-1}\nabla_\xi \hat{f}_{m,k_1}\|_{L^{18/7}}\|e^{ic_nt\la}f_{n,k_2}\|_{L^9},\\
    &\qquad\qquad\qquad\qquad  2^{3\min\{k,k_2\}/2}\|e^{ic_mt\la}\F^{-1}\nabla_\xi\hat{f}_{m,k_1}\|_{L^2}\|e^{ic_nt\la}f_{n,k_2}\|_{L^2}\}\\
    \lesssim & (1+2^{\e k_-}+2^{\e k_-+k_2-k})\min\{\|e^{ic_mt\la}\F^{-1}\nabla_\xi \hat{f}_{m,k_1}\|_{L^{18/7}}\|e^{ic_nt\la}f_{n,k_2}\|_{L^9},\\
    &\qquad\qquad\qquad\qquad 2^{3\min\{k,k_2\}/2}\|\nabla_\xi\hat{f}_{m,k_1}\|_{L^2}\|f_{n,k_2}\|_{L^2}\},
\end{align*}
\begin{align*}
    &\|\eqref{mixed1.3}\|_{L^2}+\|\eqref{mixed2.2}\|_{L^2}+\|\eqref{mixed2.4}\|_{L^2}\\
    \lesssim & \big(\|\partial_{\xi_l}\psi_k(\xi)\|_{L^\infty} + \|\F^{-1}\partial_{\xi_l}q(\xi-\eta,\eta)\Tilde{\psi}_k(\xi)\Tilde{\psi}_{k_1}(\xi-\eta)\Tilde{\psi}_{k_2}(\eta)\|_{L^1}+\\
    & + \|\partial_{\xi_l}\psi_k(\xi)\|_{L^\infty}\|\F^{-1}q(\xi-\eta,\eta)\Tilde{\psi}_k(\xi)\Tilde{\psi}_{k_1}(\xi-\eta)\Tilde{\psi}_{k_2}(\eta)\|_{L^1}\big)\times\\
    &\times \min\{\|e^{ic_mt\la}f_{m,k_1}\|_{L^{18/7}}\|e^{ic_nt\la}f_{n,k_2}\|_{L^9},2^{3\min\{k,k_2\}/2}\|e^{ic_mt\la}f_{m,k_1}\|_{L^2}\|e^{ic_nt\la}f_{n,k_2}\|_{L^2}\}\\
    \lesssim & (2^{-k}+2^{\e k_--k})\min\{\|e^{ic_mt\la}f_{m,k_1}\|_{L^{18/7}}\|e^{ic_nt\la}f_{n,k_2}\|_{L^9},2^{3\min\{k,k_2\}/2}\|e^{ic_mt\la}f_{m,k_1}\|_{L^2}\|e^{ic_nt\la}f_{n,k_2}\|_{L^2}\}\\
    \lesssim & 2^{-k}\min\{\|e^{ic_mt\la}f_{m,k_1}\|_{L^{18/7}}\|e^{ic_nt\la}f_{n,k_2}\|_{L^9},2^{3\min\{k,k_2\}/2}\|f_{m,k_1}\|_{L^2}\|f_{n,k_2}\|_{L^2}\},
\end{align*}
and
\begin{align*}
    &\|\eqref{mixed1.1.1}\|_{L^2}+\|\eqref{mixed2.1.1}\|_{L^2}+\|\eqref{mixed2.1.2}\|_{L^2}\\
    \lesssim &\big(\|\F^{-1}\partial_{\eta_m}\frac{\partial_{\xi_l}\phi(\xi,\eta)\partial_{\eta_m}\phi(\xi,\eta)}{|\nabla_\eta\phi(\xi,\eta)|^2}\Tilde{\psi}_{k_1}(\nabla_\eta\phi(\xi,\eta))\Tilde{\psi}_{k_2}(\eta)\|_{L^1}+\\
    &+\|\F^{-1}\frac{\xi_m\partial_{\eta_m}\partial_{\xi_l}\phi(\xi,\eta)}{2c_n|\xi|^2}q(\xi-\eta,\eta)\Tilde{\psi}_{k_1}(\xi-\eta)\Tilde{\psi}_{k_2}(\eta)\Tilde{\psi}_k(\xi)\|_{L^1}+\\
    &+\|\F^{-1}\frac{\xi_m\partial_{\xi_l}\phi(\xi,\eta)}{2c_n|\xi|^2}\partial_{\eta_m}q(\xi-\eta,\eta)\Tilde{\psi}_{k_1}(\xi-\eta)\Tilde{\psi}_{k_2}(\eta)\Tilde{\psi}_k(\xi)\|_{L^1})\times\\
    &\times \min\{\|e^{ic_mt\la}f_{m,k_1}\|_{L^{18/7}}\|e^{ic_nt\la}f_{n,k_2}\|_{L^9},2^{3\min\{k,k_2\}/2}\|e^{ic_mt\la}f_{m,k_1}\|_{L^2}\|e^{ic_nt\la}f_{n,k_2}\|_{L^2}\}\\
    \lesssim &(2^{-k_1}+2^{\e k_--k}+2^{\e k_--k_2})\min\{\|e^{ic_mt\la}f_{m,k_1}\|_{L^{18/7}}\|e^{ic_nt\la}f_{n,k_2}\|_{L^9},2^{3\min\{k,k_2\}/2}\|f_{m,k_1}\|_{L^2}\|f_{n,k_2}\|_{L^2}\}.
\end{align*}
Next, by the $L^{\infty}\times L^2\rightarrow L^2$ bilinear estimate in Lemma \ref{bilinear}, Lemma \ref{chi+eta}, Lemma \ref{chi+eta2}, and Lemma \ref{dualitycomp}, we obtain
\begin{align*}
    &\|\eqref{mixed1.1.3}\|_{L^2}+\|\eqref{mixed2.1.4}\|_{L^2}\\
    \lesssim &\big(\|\F^{-1}\frac{\partial_{\xi_l}\phi(\xi,\eta)\partial_{\eta_m}\phi(\xi,\eta)}{|\nabla_\eta\phi(\xi,\eta)|^2}\Tilde{\psi}_{k_1}(\nabla_\eta\phi(\xi,\eta))\Tilde{\psi}_{k_2}(\eta)\|_{L^1}
    + \\
    &+\|\F^{-1}\frac{\xi_m\partial_{\xi_l}\phi(\xi,\eta)}{2c_n|\xi|^2}q(\xi-\eta,\eta)\Tilde{\psi}_{k_1}(\xi-\eta)\Tilde{\psi}_{k_2}(\eta)\Tilde{\psi}_k(\xi)\|_{L^1}\big)
    \times\\
    &\times\min\{\|e^{ic_mt\la}f_{m,k_1}\|_{L^\infty}\|e^{ic_nt\la}\F^{-1}\nabla_\xi\hat{f}_{n,k_2}\|_{L^2},2^{3\min\{k,k_2\}/2}\|e^{ic_mt\la}f_{m,k_1}\|_{L^2}\|e^{ic_nt\la}\F^{-1}\nabla_\xi\hat{f}_{n,k_2}\|_{L^2}\}\\
    \lesssim & (1+2^{\e k_-+k_2-k})\min\{\|e^{ic_mt\la}f_{m,k_1}\|_{L^\infty}\|e^{ic_nt\la}\F^{-1}\nabla_\xi\hat{f}_{n,k_2}\|_{L^2},2^{3\min\{k,k_2\}/2}\|f_{m,k_1}\|_{L^2}\|\nabla_\xi\hat{f}_{n,k_2}\|_{L^2}\},
\end{align*}
and we can obtain another bound for \eqref{mixed2.1.2},
\begin{align*}
    &\|\eqref{mixed2.1.2}\|_{L^2}\\
    \lesssim &\|\F^{-1}\frac{\xi_m\partial_{\xi_l}\phi(\xi,\eta)}{2c_n|\xi|^2}\partial_{\eta_m}q(\xi-\eta,\eta)\Tilde{\psi}_{k_1}(\xi-\eta)\Tilde{\psi}_{k_2}(\eta)\Tilde{\psi}_k(\xi)\|_{L^1}\times\\
    &\times \min\{\|e^{ic_mt\la}f_{m,k_1}\|_{L^\infty}\|e^{ic_nt\la}f_{n,k_2}\|_{L^2},2^{3\min\{k,k_2\}/2}\|e^{ic_mt\la}f_{m,k_1}\|_{L^2}\|e^{ic_nt\la}f_{n,k_2}\|_{L^2}\}\\
    \lesssim &(2^{\e k_--k}+2^{\e k_--k_2})\min\{\|e^{ic_mt\la}f_{m,k_1}\|_{L^\infty}\|f_{n,k_2}\|_{L^2},2^{3\min\{k,k_2\}/2}\|f_{m,k_1}\|_{L^2}\|f_{n,k_2}\|_{L^2}\}.
\end{align*}
Combining the results above and using Lemma \ref{op}, we have for $t\in[2^{M-1},2^M]$
\begin{align*}
    &\sum_{(k_1,k_2)\in\chi_k^1\cup\chi_k^2}\|\nabla_\xi F^1_{k,k_1,k_2}\|_{L^2}+\sum_{(k_1,k_2)\in\chi_k^1\cup\chi_k^2\cup\chi_k^3}\|\nabla_\xi F^2_{k,k_1,k_2}\|_{L^2}\\
    \lesssim &\sum_{(k_1,k_2)\in\chi_k^1\cup\chi_k^2}\|\eqref{mixed1.2}+\eqref{mixed1.3}+\eqref{mixed1.1.1}+\eqref{mixed1.1.2}+\eqref{mixed1.1.3}\|_{L^2}+\\
    &+\sum_{(k_1,k_2)\in\chi_k^1\cup\chi_k^2\cup\chi_k^3}\|\eqref{mixed2.2}+\eqref{mixed2.3}+\eqref{mixed2.4}+\eqref{mixed2.1.1}+\eqref{mixed2.1.2}+\eqref{mixed2.1.3}+\eqref{mixed2.1.4}\|_{L^2}\\
    \lesssim & \sum_{(k_1,k_2)\in\chi_k^1\cup\chi_k^2\cup\chi_k^3}\min\{ (2^{-k}\|e^{ic_mt\la}f_{m,k_1}\|_{L^{18/7}}+(1+2^{\e k_-+k_2-k})\|e^{ic_mt\la}\F^{-1}\nabla_\xi\hat{f}_{m,k_1}\|_{L^{18/7}})\times\\
    &\qquad\qquad \times\|e^{ic_nt\la}f_{n,k_2}\|_{L^9},2^{3\min\{k,k_2\}/2}(2^{-k}\|f_{m,k_1}\|_{L^2}+(1+2^{\e k_-+k_2-k})\|\nabla_\xi\hat{f}_{m,k_1}\|_{L^2})\|f_{n,k_2}\|_{L^2}\}+\\
    &+\sum_{(k_1,k_2)\in\chi_k^1\cup\chi_k^2\cup\chi_k^3}(2^{k_2-k_1}+1+2^{\e k_-+k_2-k})\min\{\|e^{ic_mt\la}f_{m,k_1}\|_{L^\infty}\|\nabla_\xi\hat{f}_{n,k_2}\|_{L^2}, \\
    &\qquad\qquad\qquad\qquad 2^{3\min\{k,k_2\}/2}\|f_{m,k_1}\|_{L^2}\|\nabla_\xi\hat{f}_{n,k_2}\|_{L^2}\}+\\
    &+\sum_{(k_1,k_2)\in\chi_k^1}2^{\e k_--k_2}\min\{\|e^{ic_mt\la}f_{m,k_1}\|_{L^\infty}\|f_{n,k_2}\|_{L^2}, 2^{3\min\{k,k_2\}/2}\|f_{m,k_1}\|_{L^2}\|f_{n,k_2}\|_{L^2}\}\\
    \lesssim & \sum_{(k_1,k_2)\in\chi_k^1\cup\chi_k^2\cup\chi_k^3}\min\{(2^{-k}t^{-1/3}\|f_{m,k_1}\|_{L^{18/11}}+(1+2^{\e k_-+k_2-k})t^{-1/3}\|\F^{-1}\nabla_\xi\hat{f}_{m,k_1}\|_{L^{18/11}})\times\\
    &\qquad\qquad \times t^{-7/6}\|f_{n,k_2}\|_{L^{9/8}},2^{3\min\{k,k_2\}/2}(2^{-k}\|f_{m,k_1}\|_{L^2}+(1+2^{\e k_-+k_2-k})\|\nabla_\xi\hat{f}_{m,k_1}\|_{L^2})\|f_{n,k_2}\|_{L^2}\}+\\
    &+\sum_{(k_1,k_2)\in\chi_k^1\cup\chi_k^2\cup\chi_k^3}(1+2^{\e k_-+k_2-k})\min\{\|e^{ic_mt\la}f_{m,k_1}\|_{L^\infty}\|\nabla_\xi\hat{f}_{n,k_2}\|_{L^2}, \\
    &\qquad\qquad\qquad\qquad 2^{3\min\{k,k_2\}/2}\|f_{m,k_1}\|_{L^2}\|\nabla_\xi\hat{f}_{n,k_2}\|_{L^2}\}+\\
    &+\sum_{(k_1,k_2)\in\chi_k^1}2^{\e k_--k_2}\min\{\|e^{ic_mt\la}f_{m,k_1}\|_{L^\infty}\|f_{n,k_2}\|_{L^2}, 2^{3k_2/2}\|f_{m,k_1}\|_{L^2}\|f_{n,k_2}\|_{L^2}\}\\
    \lesssim & \sum_{(k_1,k_2)\in\chi_k^1\cup\chi_k^2\cup\chi_k^3}\min\{2^{-3M/2}(2^{-k}\|f_{m,k_1}\|_{L^{18/11}}+(1+2^{\e k_-+k_2-k})\|\F^{-1}\nabla_\xi\hat{f}_{m,k_1}\|_{L^{18/11}})\|f_{n,k_2}\|_{L^{9/8}},\\
    &\qquad\qquad\qquad\qquad 2^{3\min\{k,k_2\}/2}(2^{-k}\|f_{m,k_1}\|_{L^2}+(1+2^{\e k_-+k_2-k})\|\nabla_\xi\hat{f}_{m,k_1}\|_{L^2})\|f_{n,k_2}\|_{L^2}\}+\\
    &+\sum_{(k_1,k_2)\in\chi_k^1\cup\chi_k^2\cup\chi_k^3}(1+2^{\e k_-+k_2-k})\min\{\|e^{ic_mt\la}f_{m,k_1}\|_{L^\infty}\|\nabla_\xi\hat{f}_{n,k_2}\|_{L^2}, \\
    &\qquad\qquad\qquad\qquad 2^{3\min\{k,k_2\}/2}\|f_{m,k_1}\|_{L^2}\|\nabla_\xi\hat{f}_{n,k_2}\|_{L^2}\}+\\
    &+\sum_{(k_1,k_2)\in\chi_k^1}2^{\e k_--k_2}\min\{\|e^{ic_mt\la}f_{m,k_1}\|_{L^\infty}\|f_{n,k_2}\|_{L^2}, 2^{3k_2/2}\|f_{m,k_1}\|_{L^2}\|f_{n,k_2}\|_{L^2}\}\\
    \lesssim & \sum_{(k_1,k_2)\in\chi_k^1\cup\chi_k^2\cup\chi_k^3}2^{-2k_{1,+}+\gamma k_1-2k_{2,+}+\gamma k_2}\min\{2^{-3M/2}(2^{-k+2k_1-11k_1/6}+(1+2^{\e k_-+k_2-k})2^{k_1-11k_1/6})\times \\
    &\qquad\qquad \times 2^{2k_2-8k_2/3}, 2^{3\min\{k,k_2\}/2}(2^{-k+k_1/2}+(1+2^{\e k_-+k_2-k})2^{-k_1/2})2^{k_2/2}\}\e_1^2+\\
    &+\sum_{(k_1,k_2)\in\chi_k^1\cup\chi_k^2\cup\chi_k^3}(1+2^{\e k_-+k_2-k})2^{-2k_{1,+}+\gamma k_1-2k_{2,+}+\gamma k_2}\min\{2^{-3M/2+\gamma M/8-k_1+\gamma k_1/4-k_2/2}, \\
    &\qquad\qquad\qquad\qquad 2^{3\min\{k,k_2\}/2+k_1/2-k_2/2}\}\e_1^2+\\
    &+\sum_{(k_1,k_2)\in\chi_k^1}2^{\e k_--k_2-2k_{1,+}+\gamma k_1-2k_{2,+}+\gamma k_2}\min\{2^{-3M/2+\gamma M/8-k_1+\gamma k_1/4+k_2/2}, 2^{3k_2/2+k_1/2+k_2/2}\}\e_1^2\\
    \lesssim & \sum_{(k_1,k_2)\in\chi_k^1\cup\chi_k^2\cup\chi_k^3}2^{-k-2k_{1,+}+\gamma k_1-2k_{2,+}+\gamma k_2}\min\{2^{-3M/2+k_1/6-2k_2/3}, 2^{3\min\{k,k_2\}/2+k_1/2+k_2/2}\}\e_1^2+\\
    &+\sum_{(k_1,k_2)\in\chi_k^1\cup\chi_k^2\cup\chi_k^3}(1+2^{\e k_-+k_2-k})2^{-2k_{1,+}+\gamma k_1-2k_{2,+}+\gamma k_2}\min\{2^{-3M/2+\gamma M/8-k_1+\gamma k_1/4-k_2/2}, \\
    &\qquad\qquad\qquad\qquad 2^{3\min\{k,k_2\}/2+k_1/2-k_2/2}\}\e_1^2\\
    \lesssim & 2^{-M-2k_+-k/2}\e_1^2,
\end{align*}
where the $L^p$ norms are bounded using \eqref{l2}, \eqref{l2first}, \eqref{lpnormseq2}, \eqref{lpnormseq3}, and \eqref{linfinity}.

Hence, we showed
\begin{equation}
    \label{l2mixedder1}
    \|\nabla_\xi F_{1,k}\|_{L^2}\leq \sum_{(k_1,k_2)\in\chi_k^1\cup\chi_k^2}\|\nabla_\xi F^1_{k,k_1,k_2}\|_{L^2}+\sum_{(k_1,k_2)\in\chi_k^1\cup\chi_k^2\cup\chi_k^3}\|\nabla_\xi F^2_{k,k_1,k_2}\|_{L^2}\lesssim  2^{-M-2k_+-k/2}\e_1^2.
\end{equation}
Using Lemma \ref{dalpha} and \eqref{l2timeder1}, we have
\begin{equation}
\label{l2mixedalphader1}
    \|D^\alpha F_{1,k}\|_{L^2}\lesssim \|F_{1,k}\|_{L^2}^{1-\alpha}\|\nabla F_{1,k}\|_{L^2}^\alpha\lesssim2^{-2k_+ +(\gamma k-M+k/2)(1-\alpha)+(-M-k/2)\alpha}\e_1^2\lesssim 2^{-2k_++(1-\alpha)\gamma k -M+k/2-\alpha k}\e_1^2.
\end{equation}
Next, we need to find a bound on $\|\nabla_\xi F_{2,k}\|_{L^2}$. Immediately by $\eqref{partialtfbound1.1}$, we get
\begin{equation}
\label{mixed1.3est}
    \begin{aligned}
    \|\eqref{mixed1.3}\|_{L^2}
    &\leq \|\nabla\psi_k\|_{L^\infty}\bigg\|\int_{\R^3}e^{it\phi(\xi,\eta)}\hat{f}_{m,k_1}(t,\xi-\eta)\hat{f}_{n,k_2}(t,\eta)d\eta\Tilde{\psi}_k(\xi)\bigg\|_{L^2}\\
    &\lesssim 2^{-k-2k_{1,+}+\gamma k_1-2k_{2,+}+\gamma k_2}\min\{2^{-3M/2-k_1/4-k_2/4},2^{3\min\{k,k_2\}/2+k_1/2+k_2/2}\}\e_1^2.
\end{aligned}
\end{equation}
By the $L^{18/7}\times L^9\rightarrow L^2$ bilinear estimate, Lemma \ref{dualitycomp}, Lemma \ref{op}, Lemma \ref{chi+eta2}, Lemma \ref{lpnorms}, \eqref{l2}, and \eqref{l2first}, for $t\in[2^{M-1},2^M]$
\begin{equation}
    \label{mixed1.1est}
    \begin{aligned}
    \|\eqref{mixed1.1}\|_{L^2}
    \lesssim &2^{M}\|\F^{-1}\partial_{\xi_l}\phi(\xi,\eta)\Tilde{\psi}_k(\xi)\Tilde{\psi}_{k_2}(\eta)\|_{L^1}\min\{\|e^{ic_mt\la}f_{m,k_1}\|_{L^{18/7}}\|e^{ic_nt\la}f_{n,k_2}\|_{L^9},\\
    &\qquad\qquad\qquad\qquad
    2^{3\min\{k,k_2\}/2}\|e^{ic_mt\la}f_{m,k_1}\|_{L^2}\|e^{ic_nt\la}f_{n,k_2}\|_{L^2}\}\\
    \lesssim & 2^{M+k_2}\min\{2^{-3M/2}\|f_{m,k_1}\|_{L^{18/11}}\|f_{n,k_2}\|_{L^{9/8}},2^{3\min\{k,k_2\}/2}\|f_{m,k_1}\|_{L^2}\|f_{n,k_2}\|_{L^2}\}\\
    \lesssim &2^{M+k_2-2k_{1,+}+\gamma k_1-2k_{2,+}+\gamma k_2}\min\{2^{-3M/2+k_1/6-2k_2/3},2^{3\min\{k,k_2\}/2+k_1/2+k_2/2}\}\e_1^2
\end{aligned}
\end{equation}
and
\begin{equation}
\label{mixed1.2est}
    \begin{aligned}
    \|\eqref{mixed1.2}\|_{L^2}
    \lesssim & \min\{\|e^{ic_mt\la}\F^{-1}\nabla_\xi\hat{f}_{m,k_1}\|_{L^{18/7}}\|e^{ic_nt\la}f_{n,k_2}\|_{L^{9}},\\
    &\qquad\qquad\qquad\qquad 2^{3\min\{k,k_2\}/2}\|e^{ic_mt\la}\F^{-1}\nabla_\xi\hat{f}_{m,k_1}\|_{L^2}\|e^{ic_nt\la}f_{n,k_2}\|_{L^2}\}\\
    \lesssim & \min\{2^{-3M/2}\|\F^{-1}\nabla_\xi\hat{f}_{m,k_1}\|_{L^{18/11}}\|f_{n,k_2}\|_{L^{9/8}},2^{3\min\{k,k_2\}/2}\|\F^{-1}\nabla_\xi\hat{f}_{m,k_1}\|_{L^2}\|f_{n,k_2}\|_{L^2}\}\\
    \lesssim &2^{-2k_{1,+}+\gamma k_1-2k_{2,+}+\gamma k_2}\min\{2^{-3M/2-5k_1/6-2k_2/3},2^{3\min\{k,k_2\}/2-k_1/2+k_2/2}\}\e_1^2.
\end{aligned}
\end{equation}
Thus, we showed
\begin{align*}
    \|\nabla_\xi F_{2,k}\|_{L^2}
    \lesssim &\sum_{(k_1,k_2)\in\chi_k^3}\|\nabla_\xi F^1_{k,k_1,k_2}\|_{L^2}\\
    \lesssim & \sum_{(k_1,k_2)\in\chi_k^3}\|\eqref{mixed1.1}+\eqref{mixed1.2}+\eqref{mixed1.3}\|_{L^2}\\
    \lesssim & 2^{M+k-4k_++2\gamma k}\min\{2^{-3M/2-k/2},2^{3k/2+k}\}\e_1^2+2^{-4k_++2\gamma k}\min\{2^{-3M/2-3k/2},2^{3k/2}\}\e_1^2\\
    \lesssim &2^{-4k_++2\gamma k}(1+2^{M+2k})\min\{2^{-3M/2-3k/2},2^{3k/2}\}\e_1^2.
\end{align*}
Recall the estimate in \eqref{l2timeder2}, when $M>-2k$,
\begin{align*}
    \| F_{2,k}\|_{L^2}
    \lesssim 2^{-4k_++2\gamma k-3M/2-k/2}\e_1^2
\end{align*}
and
\begin{align*}
    \|\nabla_\xi F_{2,k}\|_{L^2}
    \lesssim 2^{-4k_++2\gamma k}(2^{-3M/2-3k/2}+2^{-M/2+k/2})\e_1^2\lesssim 2^{-4k_++2\gamma k-M/2+k/2}\e_1^2.
\end{align*}
By Lemma \ref{dalpha}, we get
\begin{align*}
    \|D^\alpha F_{2,k}\|_{L^2}
    &\lesssim \|F_{2,k}\|_{L^2}^{1-\alpha}\|\nabla F_{2,k}\|_{L^2}^\alpha\\
    & \lesssim2^{-4k_++2\gamma k +(-3M/2-k/2)(1-\alpha)+(-M/2+k/2)\alpha}\e_1^2\\
    &\lesssim 2^{-4k_++2\gamma k-3M/2+\alpha M-k/2+\alpha k}\e_1^2\\
    &\lesssim 2^{-4k_++2\gamma k-M - (1/2-\alpha)
    M-k/2+\alpha k}\e_1^2\\
    &\lesssim 2^{-4k_++2\gamma k-M + k-2\alpha k-k/2+\alpha k}\e_1^2\\
    &= 2^{-4k_++2\gamma k -M+k/2-\alpha k}\e_1^2.
\end{align*}
When $M\leq -2k$, 
\begin{align*}
    \| F_{2,k}\|_{L^2}
    \lesssim 2^{-4k_++2\gamma k+5k/2}\e_1^2
\end{align*}
and
\begin{align*}
    \|\nabla_\xi F_{2,k}\|_{L^2}
    \lesssim 2^{-4k_++2\gamma k}(2^{3k/2}+2^{M+7k/2})\e_1^2\lesssim 2^{-4k_++2\gamma k+3k/2}\e_1^2,
\end{align*}
hence
\begin{align*}
    \|D^\alpha F_{2,k}\|_{L^2}
    &\lesssim \|F_{2,k}\|_{L^2}^{1-\alpha}\|\nabla F_{2,k}\|_{L^2}^\alpha\\
    &\lesssim2^{-4k_++2\gamma k +5k/2(1-\alpha)+\alpha3k/2}\e_1^2\\
    &\lesssim 2^{-4k_++2\gamma k +5k/2-\alpha k}\e_1^2\\
    &\lesssim 2^{-4k_++2\gamma k -M+k/2-\alpha k}\e_1^2.
\end{align*}
Therefore,
\begin{equation}
\label{l2mixedalphader2}
    \|D^\alpha F_{2,k}\|_{L^2}\lesssim 2^{-4k_++2\gamma k -M+k/2-\alpha k}\e_1^2\leq 2^{-2k_++(1-\alpha)\gamma k -M+k/2-\alpha k}\e_1^2,
\end{equation}
and \eqref{l2mixedalphader1} combined with \eqref{l2mixedalphader2} proves \eqref{l2mixedalphader}.\\
Moreover, from the computation above, we also have
\begin{equation}
\label{l2mixedder2}
    \|\nabla_\xi F_{2,k}\|_{L^2}\lesssim (1+2^{M/2+k})2^{-2k_+-M-k/2}\e_1^2,
\end{equation}
and \eqref{l2mixedder} is an immediate result of \eqref{l2mixedder1} and \eqref{l2mixedder2}.
\end{proof}
Lastly, using parts of the estimates we obtained in Lemma \ref{mixder}, we may conclude this section with the following result for the $L^3$ norm of the mixed derivatives.
\begin{lem}
\label{mixderl3}
Under the assumption \eqref{b1}, we have for any $t\in[2^{M-1},2^M]$
\begin{align*}
    \|e^{ic_lt\la}\F^{-1}\nabla_\xi\partial_t\hat{f}_{l,k}(t,\xi)\|_{L^3}\lesssim 2^{-M-2k_+}\e_1^2.
\end{align*}
\end{lem}
\begin{proof}
We will keep the definition of $F_{1,k}$ and $F_{2,k}$ same as in the proof of Lemma \ref{mixder}, then it suffices to show
\begin{align*}
    \|\F^{-1}e^{-ic_lt|\xi|^2}\nabla_\xi F_{1,k}\|_{L^3}+\|\F^{-1}e^{-ic_lt|\xi|^2}\nabla_\xi F_{2,k}\|_{L^3}\lesssim 2^{-M-2k_+}\e_1^2.
\end{align*}
Using Bernstein's inequality in Lemma \ref{bernstein} and \eqref{l2mixedder1},
\begin{align*}
    \|\F^{-1}e^{-ic_lt|\xi|^2}\nabla_\xi F_{1,k}\|_{L^3}\lesssim 2^{k/2}\|\F^{-1}e^{-ic_lt|\xi|^2}\nabla_\xi F_{1,k}\|_{L^2}\lesssim 2^{-M-2k_+}\e_1^2.
\end{align*}
Next, from \eqref{Fsetupdef} we observe that
\begin{align*}
    \|\F^{-1}e^{-ic_lt|\xi|^2}\nabla_\xi F_{2,k}\|_{L^3}\leq \sum_{(k_1,k_2)\in\chi^3_k}\|\F^{-1}e^{-ic_lt|\xi|^2}\nabla_\xi F^1_{k,k_1,k_2}\|_{L^3}
\end{align*}
and
\begin{align}
    \F^{-1}e^{-ic_lt|\xi|^2}\partial_{\xi_l}F^1_{k,k_1,k_2}
    =&\F^{-1}\int_{\R^3}e^{-ic_mt|\xi-\eta|^2-ic_nt|\eta|^2}it\partial_{\xi_l}\phi(\xi,\eta)\hat{f}_{m,k_1}(t,\xi-\eta)\hat{f}_{n,k_2}(t,\eta)d\eta\psi_k(\xi)\label{l3mixed1.1}\\
    & + \F^{-1}\int_{\R^3}e^{-ic_mt|\xi-\eta|^2-ic_nt|\eta|^2}\partial_{\xi_l}\hat{f}_{m,k_1}(t,\xi-\eta)\hat{f}_{n,k_2}(t,\eta)d\eta\psi_k(\xi)\label{l3mixed1.2}\\
    &+ \F^{-1}\int_{\R^3}e^{-ic_mt|\xi-\eta|^2-ic_nt|\eta|^2}\hat{f}_{m,k_1}(t,\xi-\eta)\hat{f}_{n,k_2}(t,\eta)d\eta\partial_{\xi_l}\psi_k(\xi)\label{l3mixed1.3}.
\end{align}
By Bernstein's inequality, $\eqref{mixed1.3est}$, and $\eqref{mixed1.2est}$, we have
\begin{align*}
    \|\eqref{l3mixed1.3}\|_{L^3}
    &\lesssim 2^{k/2}\bigg\|\int_{\R^3}e^{it\phi(\xi,\eta)}\hat{f}_{m,k_1}(t,\xi-\eta)\hat{f}_{n,k_2}(t,\eta)d\eta\partial_{\xi_l}\psi_k(\xi)\bigg\|_{L^2}\\
    &= 2^{k/2}\|\eqref{mixed1.3}\|_{L^2}
    \lesssim 2^{-k/2-2k_{1,+}+\gamma k_1-2k_{2,+}+\gamma k_2}\min\{2^{-3M/2-k_1/4-k_2/4},2^{3\min\{k,k_2\}/2+k_1/2+k_2/2}\}\e_1^2
\end{align*}
and
\begin{align*}
    \|\eqref{l3mixed1.2}\|_{L^3}
    \lesssim & 2^{k/2}\bigg\|\int_{\R^3}e^{it\phi(\xi,\eta)}\partial_{\xi_l}\hat{f}_{m,k_1}(t,\xi-\eta)\hat{f}_{n,k_2}(t,\eta)d\eta\psi_k(\xi)\bigg\|_{L^2}\\
    =&2^{k/2}\|\eqref{mixed1.2}\|_{L^2}
    \lesssim 2^{k/2-2k_{1,+}+\gamma k_1-2k_{2,+}+\gamma k_2}\min\{2^{-3M/2-5k_1/6-2k_2/3},2^{3\min\{k,k_2\}/2-k_1/2+k_2/2}\}\e_1^2.
\end{align*}
Lastly, using the $L^6\times L^6\rightarrow L^3$ bilinear estimate, Lemma \ref{op}, Lemma \ref{chi+eta2} and Lemma \ref{lpnorms}, we obtain
\begin{align*}
    \|\eqref{l3mixed1.1}\|_{L^3}
    \lesssim &2^{M}\|\F^{-1}\partial_{\xi_l}\phi(\xi,\eta)\Tilde{\psi}_k(\xi)\Tilde{\psi}_{k_2}(\eta)\|_{L^1}\|e^{ic_mt\la}f_{m,k_1}\|_{L^6}\|e^{ic_nt\la}f_{n,k_2}\|_{L^6}\\
    \lesssim  &2^{M+k-2M}\|f_{m,k_1}\|_{L^{6/5}}\|f_{n,k_2}\|_{L^{6/5}}
    \lesssim 2^{-M+k_2-2k_{1,+}+\gamma k_1-2k_{2,+}+\gamma k_2-k_1/2-k_2/2}\e_1^2.
\end{align*}
Summarizing the bounds above, we have
\begin{align*}
    \|\F^{-1}e^{-ic_lt|\xi|^2}\nabla_\xi F_{2,k}\|_{L^3}
    \lesssim & \sum_{(k_1,k_2)\in\chi_k^3}\|\eqref{l3mixed1.1}+\eqref{l3mixed1.2}+\eqref{l3mixed1.3}\|_{L^3}\\
    \lesssim & 2^{-M-4k_++2\gamma k}\e_1^2+2^{-4k_++2\gamma k}\min\{2^{-3M/2-k},2^{2k}\}\e_1^2\\
    \lesssim &2^{-M-4k_++2\gamma k}\e_1^2\\
    \lesssim &2^{-M-2k_+}\e_1^2.
\end{align*}
\end{proof}

\section{Proof of Theorem \ref{mainthm}}\label{main proof}
Assume the initial condition $\|u_{l0}\|_{H^{10}}+\|e^{-ic_l\Delta}u_{l0}\|_Z\leq \e_0$. The local well-posedness in Proposition \ref{lwp} gives the bootstrap assumption \eqref{b1}.

First, we perform an energy estimate to bound the term $\|u_l(t,x)\|_{H^{10}_x}=\|e^{ic_lt\la}f_l(t,x)\|_{H^{10}_x}=\|f_l(t,x)\|_{H^{10}_x}$. Take any $t\in[1,T]$, by the initial condition \eqref{initial}, the assumption \eqref{b1}, and the estimation in Lemma \ref{timel2}
\begin{align*}
    \|u_l(t,x)\|^2_{H^{10}_x}
    &\lesssim \|u_{l0}\|^2_{H^{10}_x}+\sum_{k\in\Z}\int_1^t\partial_s\|f_{l,k}(s,x)\|^2_{H^{10}_x}ds\\
    &=\e_0^2+\sum_{k\in\Z}\int_1^t\partial_s\|(1+|\xi|^2)^5\hat{f}_{l,k}(s,\xi)\|^2_{L^2_\xi}ds\\
    &=\e_0^2+\sum_{k\in\Z}\int_1^t\partial_s\int_{\R^3}(1+|\xi|^2)^{10}\hat{f}_{l,k}(s,\xi)\bar{\hat{f}}_{l,k}(s,\xi)d\xi ds\\
    &\lesssim\e_0^2+\sum_{k\in\Z}\int_1^t 2^{20k_+}\int_{\R^3}\partial_s \hat{f}_{l,k}(s,\xi)\bar{\hat{f}}_{l,k}(s,\xi)+\hat{f}_{l,k}(s,\xi)\partial_s\bar{\hat{f}}_{l,k}(s,\xi)d\xi ds\\
    &\lesssim \e_0^2+\sum_{k\in\Z}\int_1^t2^{20k_+}\|\hat{f}_{l,k}(s,\xi)\|_{L^2_\xi}\|\partial_s \hat{f}_{l,k}(s,\xi)\|_{L^2_\xi} ds\\
    &\lesssim \e_0^2+\sum_{k\in\Z}\int_1^t2^{20k_+-10k_+}\|\hat{f}_{l,k}(s,\xi)\|_{L^2_\xi}\frac{\e_1^2}{s^{1+\gamma/2}} ds\\
    &\lesssim \e_0^2+\|f_l\|_{L^\infty_t([1,T])H^{10}_x}\int_1^t\frac{\e_1^2}{s^{1+\gamma/2}} ds\\
    &\lesssim \e_0^2+\|f_l\|_{L^\infty_t([1,T])H^{10}_x}\frac{\e_1^2}{\gamma}
    \lesssim \e_0^2+\e_1^3\lesssim \e_0^2.
\end{align*}

Next, we use the result below whose proof is given in Section \ref{prop}.
\begin{prop}
\label{main}
Under assumption \eqref{b1},
\begin{equation}
    \|e^{-ic_lt\la}u_l(t,x)\|_Z=\|f_l(t,x)\|_Z\lesssim \e_0
\end{equation}
for all $l$ and $t\in [1,T]$.
\end{prop}
Hence,
\begin{equation*}
    \sup_{t\in[1,T]}\|u_l(t,x)\|_{H^{10}_x}+\|e^{-ic_lt\la}u_l(t,x)\|_Z\lesssim \e_0.
\end{equation*}
Define 
$$T^+=\sup\{T>1: \{u_l\}\text{ is a solution to }\eqref{pde}\text{ s.t. }\sup_{l}\sup_{t\in[1,T]}\|u_l(t,x)\|_{H^{10}_x}+\|e^{-ic_lt\la}u_l(t,x)\|_{Z}\leq \e_1\}.$$ Suppose $T^+<\infty$, then 
\begin{equation*}
    \sup_l\sup_{t\in[1,T^+]}\|u_l(t,x)\|_{H^{10}_x}+\|e^{-ic_lt\la}u_l(t,x)\|_Z\lesssim \e_0<\e_1,
\end{equation*}
given $\e_0$ sufficiently small. Hence, by local well-posedness, there exists some $T'>T^+$ such that $\{\Tilde{u}_l\}$ solves \eqref{pde} and 
\begin{equation*}
    \sup_l\sup_{t\in[1,T']}\|\Tilde{u}_l(t,x)\|_{H^{10}_x}+\|e^{-ic_lt\la}\Tilde{u}_l(t,x)\|_Z\leq \e_1,
\end{equation*}
which gives a contradiction and shows $T^+=\infty$.

Thus, we can conclude
\begin{equation}
\label{con}
    \sup_{t\in[1,\infty)}\|u_l(t,x)\|_{H^{10}_x}+\|e^{-ic_lt\la}u_l(t,x)\|_Z\lesssim \e_0,
\end{equation}
which proves the existence of a global solution to the IVP \eqref{pde} provided $\e_0$ sufficiently small. The uniqueness follows from the uniqueness for local well-posedness in Proposition \ref{lwp}.

Furthermore, as a result of \eqref{con} above and \eqref{linfinity}, for any integer $M\geq 1$ and $t\in[2^{M-1},2^M]$
\begin{align*}
    \|u_l(t,x)\|_{L^\infty_x}
    &\lesssim\sum_{k\in\Z}\|e^{ic_lt\la}f_{l,k}\|_{L^\infty_x}\\
    &\lesssim\sum_{k\in\Z}\min\{2^{-3M/2+\gamma M/8-2k_++\gamma k+\gamma k/4-k},2^{\gamma k-2k_++2k}\}\e_0\\
    &\lesssim\sum_{k\leq -M/2}2^{\gamma k-2k_++2k}\e_0+\sum_{k> -M/2}2^{-3M/2+\gamma M/8-2k_++\gamma k+\gamma k/4-k}\e_0\\
    &\lesssim 2^{-M(1+\gamma/2)}\e_0\lesssim t^{-(1+\gamma/2)}\e_0,
\end{align*}
which proves \eqref{gse}.\\
Lastly, we show the scattering result. Since we now have $$\sup_{t\in[1,\infty)}\|u_l(t,x)\|_{H^{10}_x}+\|e^{-ic_lt\la}u_l(t,x)\|_{Z}=\sup_{t\in[1,\infty)}\|f_l(t,x)\|_{H^{10}_x}+\|f_l(t,x)\|_{Z}\lesssim \e_0,$$ we know $\|\partial_t\hat{f}_{l,k}(t,x)\|_{L^2_x}\lesssim 2^{-10 k_+} t^{-1-\gamma/2}\e_0^2$ for all $t\geq 1$ from Lemma \ref{timel2}. Then for any $1\leq t_1< t_2$,
\begin{align*}
    \|f_l(t_2,x)-f_l(t_1,x)\|^2_{H^{10}_x}
    &\lesssim \sum_{k\in\Z}\int_{t_1}^{t_2}\partial_s\|f_{l,k}(s,x)\|^2_{H^{10}_x}ds\\
    &\lesssim\sum_{k\in\Z}\int_{t_1}^{t_2} 2^{20k_+}\int_{\R^3}\partial_s \hat{f}_{l,k}(s,\xi)\bar{\hat{f}}_{l,k}(s,\xi)+\hat{f}_{l,k}(s,\xi)\partial_s\bar{\hat{f}}_{l,k}(s,\xi)d\xi ds\\
    &\lesssim \sum_{k\in\Z}\int_{t_1}^{t_2}2^{20k_+}\|\hat{f}_{l,k}(s,x)\|_{L^2_x}\|\partial_s \hat{f}_{l,k}(s,x)\|_{L^2_x} ds\\
    &\lesssim \sum_{k\in\Z}\int_{t_1}^{t_2}2^{20k_+-10k_+}\|\hat{f}_{l,k}(s,x)\|_{L^2}\frac{\e_0^2}{s^{1+\gamma/2}} ds\\
    &\lesssim \|f_l(t,x)\|_{L^\infty_t([1,\infty))H^{10}_x}\int_{t_1}^{t_2}\frac{\e_0^2}{s^{1+\gamma/2}} ds\\
    &\lesssim \frac{\e_0^3}{\gamma}(\frac{1}{t_1^{\gamma/2}}-\frac{1}{t_2^{\gamma/2}}).
\end{align*}
Thus, $\Tilde{f}_n(x)=f_l(n,x)$ forms a Cauchy sequence in $H^{10}(\R^3)$. Hence, there exists some $u^*(x)\in H^{10}(\R^3)$ so that $\|f_l(n,x)-u^*(x)\|_{H^{10}}\rightarrow 0$ as $n\rightarrow\infty$. If $\lfloor t\rfloor$ denote the integer part of $t$, then
$$\|f_l(t,x)-u^*(x)\|_{H^{10}_x}\leq \|f_l(\lfloor t\rfloor,x)-u^*(x)\|_{H^{10}_x}+\|f_l(t,x)-f_l(\lfloor t\rfloor,x)\|_{H^{10}_x}$$
implies $\|u_l(t,x)-e^{ic_lt\la}u^*(x)\|_{H^{10}_x}=\|f_l(t,x)-u^*(x)\|_{H^{10}_x}\rightarrow 0$ as $t\rightarrow\infty$.
\section{Proof of Local Well-Posedness}
\label{plwp}
The following local well-posedness result is essential for us to make the bootstrap assumption \eqref{b1}. The result indicates that if we start at any time $t_0>1$ and have initial data $u_{l0}$ with $e^{-ic_lt_0\Delta}u_{l0}$ bounded in the $H^{10}\cap Z$ space, we can find a unique solution to the partial differential equation that exists on the time interval $[t_0,T]$ where $T$ is only dependent on the starting time $t_0$ and the size of our initial data. The proof uses Duhamel's integral equation and is based on the contraction mapping principle.
\label{lwpproof}
\begin{prop}[local well-posedness]
\label{lwp}
Given $t_0\geq 1$. If $\|u_{l0}\|_{H^{10}}+\|e^{-ic_lt_0\la}u_{l0}\|_Z\leq \e$ for some $\e>0$, then there exists $T=T(t_0,\e)>t_0$ so that the equation \begin{equation}
\label{lwppde}
    \begin{cases}
     \partial_t u_l= ic_l\Delta u_l+\sum_{c_m+c_n\neq 0}A_{lmn} u_mu_n+\sum_{c_m+c_n= 0}A_{lmn} Q(u_m,u_n)\\
     u_l|_{t=t_0}=u_{l0}
    \end{cases},
\end{equation}
has a unique solution $u_l(t,x)$ such that $e^{-ic_lt\la}u_l(t,x)\in\mathcal{C}([t_0,T];H^{10}\cap Z)$.
\end{prop}
\begin{proof}
Define an operator
\begin{align*}
    \Psi(\vec{f})_l(t,x)= &e^{-ic_lt_0\la}u_{l0}+\int_{t_0}^te^{-ic_ls\la}\bigg(\sum_{\substack{c_m+c_n=0\\m\geq n}}A_{lmn}Q(e^{ic_ms\la}f_m,e^{ic_ns\la}f_n)+\\
    &\qquad\qquad\qquad\qquad +\sum_{\substack{c_m+c_n\neq 0\\m\geq n}}A_{lmn}(e^{ic_ms\la}f_m)(e^{ic_ns\la}f_n)\bigg)ds.
\end{align*}
We shall show there exists some $T>t_0$ and $r>0$ such that for
\begin{align*}
    X=\{\vec{f}\in\prod_l\mathcal{C}([t_0,T]; H^{10}\cap Z): \sup_l\sup_{t\in[t_0,T]}\|f_l\|_{H^{10}_x}+\|f_l\|_{Z}\leq r\},
\end{align*}
$\Psi:X\rightarrow X$, and furthermore, $\Psi$ is a contraction map.\\
Take an arbitrary $\vec{f}\in X$ and $t\in[t_0,T]$. Our goal is to establish bounds for $\|\Psi(\vec{f})_l\|_{H^{10}_x}$ and $\|\Psi(\vec{f})_l\|_{Z}$. These bounds will rely on the initial data $u_0$, the radius $r$, and the time interval $[t_0, T]$. To proceed, we need to compute
\begin{align*}
    \|\Psi(\vec{f})_l\|_{H^{10}_x}
    \leq & \|e^{-ic_lt_0\la}u_{l0}\|_{H^{10}}+(T-t_0)\sum_{\substack{c_m+c_n=0\\m\geq n}}A_{lmn}\|e^{-ic_ls\la}Q(e^{ic_ms\la}f_m,e^{ic_ns\la}f_n)\|_{L^\infty_s([t_0,T]) H^{10}_x}+\\
    &\qquad\qquad\qquad\qquad +(T-t_0)\sum_{\substack{c_m+c_n\neq 0\\m\geq n}}A_{lmn}\|(e^{ic_ms\la}f_m)(e^{ic_ns\la}f_n)\|_{L^\infty_s([t_0,T]) H^{10}_x}\\
    \leq & \|u_{l0}\|_{H^{10}}+(T-t_0)\sum_{\substack{c_m+c_n=0\\m\geq n}}A_{lmn}\|(1+|\xi|^2)^5\F Q(e^{ic_ms\la}f_m,e^{ic_ns\la}f_n)\|_{L^\infty_s([t_0,T]) L^2_\xi}+\\
    &\qquad\qquad\qquad\qquad (T-t_0)\sum_{\substack{c_m+c_n\neq 0\\m\geq n}}A_{lmn}\|e^{ic_ms\la}f_m\|_{L^\infty_s([1,T])H^{10}_x}\|e^{ic_ns\la}f_n\|_{L^\infty_s([t_0,T])H^{10}_x}.
\end{align*}
\label{keypage1}
Using the $L^2\times L^\infty\rightarrow L^2$ bilinear estimate in Lemma \ref{bilinear} and Bernstein's inequality in Lemma \ref{bernstein}, we have
\begin{align*}
    &\sum_{\substack{c_m+c_n=0\\m\geq n}}A_{lmn}\|(1+|\xi|^2)^5\F Q(e^{ic_ms\la}f_m,e^{ic_ns\la}f_n)\|_{L^\infty_s([t_0,T]) L^2_\xi}\\
    \leq & \sum_{\substack{c_m+c_n=0\\m\geq n}}A_{lmn}\bigg\|\sum_{k,k_1,k_2\in\Z}2^{10k_+}\psi_k(\xi)\int_{\R^3} q(\xi-\eta,\eta)\times\\
    &\qquad\qquad\qquad\qquad \times e^{-ic_ms|\xi-\eta|^2}\hat{f}_{m,k_1}(s,\xi-\eta)e^{-ic_ns|\eta|^2}\hat{f}_{n,k_2}(s,\eta)d\eta\bigg\|_{L^\infty_s([t_0,T])L^2_\xi}\\
    \leq & \sum_{\substack{c_m+c_n=0\\m\geq n}}A_{lmn}\bigg\|\sum_{k_1,k_2\in\Z}\sum_{k\leq \max\{k_1,k_2\}+1}2^{10k_+}\psi_k(\xi)\int_{\R^3} q(\xi-\eta,\eta)\times\\
    &\qquad\qquad\qquad\qquad \times e^{-ic_ms|\xi-\eta|^2}\hat{f}_{m,k_1}(s,\xi-\eta) e^{-ic_ns|\eta|^2}\hat{f}_{n,k_2}(s,\eta)d\eta\bigg\|_{L^\infty_s([t_0,T])L^2_\xi}\\
    \leq &\sum_{c_m+c_n=0}A_{lmn}\bigg\|\sum_{k_1\in\Z}\sum_{k_2\leq k_1}\sum_{k\leq k_1}2^{10k_++\e k_-}\|e^{ic_ms\la}f_{m,k_1}\|_{L^2_x}\|e^{ic_ns\la}f_{n,k_2}\|_{L^\infty_x}\bigg\|_{L^\infty_s([t_0,T])}\\
    \lesssim &\sum_{c_m+c_n=0}A_{lmn}\bigg\|\sum_{k_1\in\Z}\sum_{k_2\leq k_1}2^{10k_{1,+}}\|e^{ic_ms\la}f_{m,k_1}\|_{L^2_x}2^{3k_2/2}\|e^{ic_ns\la}f_{n,k_2}\|_{L^2_x}\bigg\|_{L^\infty_s([t_0,T])}\\
    \lesssim &\sum_{c_m+c_n=0}A_{lmn}\bigg\|\sum_{k_1\in\Z}\sum_{k_2\leq k_1}2^{10k_{1,+}}\|f_{m,k_1}\|_{L^2_x}2^{10k_{2,+}}\|f_{n,k_2}\|_{L^2_x}\bigg\|_{L^\infty_s([t_0,T])}\\
    \lesssim &\sum_{c_m+c_n=0}A_{lmn}\|f_m\|_{L^\infty_s([t_0,T])H^{10}_x}\|f_n\|_{L^\infty_s([t_0,T])H^{10}_x}.
\end{align*}
Hence, there is some constant $C>0$ independent of $f$ so that
\begin{align*}
    \|\Psi(\vec{f})_l\|_{L^\infty_t([t_0,T])H^{10}_x}
    \lesssim & \|u_{l0}\|_{H^{10}}+\sum_{c_m+c_n=0}A_{lmn}(T-t_0)\|f_m\|_{L^\infty_t([1,T])H^{10}_x}\|f_n\|_{L^\infty_t([t_0,T])H^{10}_x}\\
    \leq & \|u_{l0}\|_{H^{10}}+C(T-t_0)r^2.
\end{align*}
Next, we move on to bounding the $Z$ norm
\begin{align*}
    &\|\Psi(\vec{f})_l\|_Z\\
    \leq & \|e^{-i\la}u_{l0}\|_{Z}+ (T-t_0)\sup_{k\in\Z}2^{-\gamma k+2k_++k/2+\alpha k}\bigg\|D^{1+\alpha}_\xi \F e^{-ic_ls\la}\times\\
    &\times \bigg(\sum_{\substack{c_m+c_n=0\\m\geq n}}A_{lmn}Q(e^{ic_ms\la}f_m,e^{ic_ns\la}f_n)+\sum_{\substack{c_m+c_n\neq 0\\m\geq n}}A_{lmn}(e^{ic_ms\la}f_m)(e^{ic_ns\la}f_n)\bigg)\psi_k(\xi)\bigg\|_{L^\infty_s([t_0,T])L^2_\xi},
\end{align*}
where according to the computation on page \pageref{key} and \eqref{partialtf},
\begin{align*}
    &\F e^{-ic_ls\la} \bigg(\sum_{\substack{c_m+c_n=0\\m\geq n}}A_{lmn}Q(e^{ic_ms\la}f_m,e^{ic_ns\la}f_n)+\sum_{\substack{c_m+c_n\neq 0\\m\geq n}}A_{lmn}(e^{ic_ms\la}f_m)(e^{ic_ns\la}f_n)\bigg)\psi_k(\xi)\\
    =&\sum_{\substack{c_m+c_n=0\\m\geq n}}A_{lmn}\sum_{k_1,k_2\in\Z}\int_{\R^3}e^{is\phi(\xi,\eta)}q(\xi-\eta,\eta)\hat{f}_{m,k_1}(s,\xi-\eta)\hat{f}_{n,k_2}(s,\eta)\psi_k(\xi)d\eta\\
    &+\sum_{\substack{c_m+c_n\neq 0\\m\geq n}}A_{lmn}\sum_{k_1,k_2\in\Z}\int_{\R^3}e^{is\phi(\xi,\eta)}\hat{f}_{m,k_1}(s,\xi-\eta)\hat{f}_{n,k_2}(x,\eta)\psi_k(\xi)d\eta\\
    =&\sum_{\substack{c_m+c_n=0\\m\geq n}}A_{lmn}\sum_{(k_1,k_2)\in\chi^2_k\cup\chi^3_k}\int_{\R^3}e^{is\phi(\xi,\eta)}q(\xi-\eta,\eta)\hat{f}_{m,k_1}(s,\xi-\eta)\hat{f}_{n,k_2}(s,\eta)\psi_k(\xi)d\eta \\
    &+\sum_{\substack{c_m+c_n\neq 0\\m\geq n}}A_{lmn}\sum_{(k_1,k_2)\in\chi^2_k\cup\chi^3_k}\int_{\R^3}e^{is\phi(\xi,\eta)}\hat{f}_{m,k_1}(s,\xi-\eta)\hat{f}_{n,k_2}(s,\eta)\psi_k(\xi)d\eta\\
    &+\sum_{c_m+c_n\neq 0}A_{lmn}\sum_{(k_1,k_2)\in\chi^1_k}\int_{\R^3}e^{is\phi(\xi,\eta)}q(\xi-\eta,\eta)\hat{f}_{m,k_1}(s,\xi-\eta)\hat{f}_{n,k_2}(s,\eta)\psi_k(\xi)d\eta \\
    &+\sum_{c_m+c_n=0}A_{lmn}\sum_{(k_1,k_2)\in\chi^1_k}\int_{\R^3}e^{is\phi(\xi,\eta)}\hat{f}_{m,k_1}(s,\xi-\eta)\hat{f}_{n,k_2}(s,\eta)\psi_k(\xi)d\eta,
\end{align*}
for $\phi(\xi,\eta)=c_l|\xi|^2-c_m|\xi-\eta|^2-c_n|\eta|^2$.
Since $D^{1+\alpha}_\xi=D^\alpha_\xi\nabla_\xi$, we look at the first derivative in $\xi$ for each term above,
\begin{align}
    &\partial_{\xi_l} \int_{\R^3}e^{is\phi(\xi,\eta)}\hat{f}_{m,k_1}(s,\xi-\eta)\hat{f}_{n,k_2}(s,\eta)\psi_k(\xi)d\eta\nonumber\\
    =&\int_{\R^3}e^{is\phi(\xi,\eta)}is\partial_{\xi_l}\phi(\xi,\eta)\hat{f}_{m,k_1}(s,\xi-\eta)\hat{f}_{n,k_2}(s,\eta)\psi_k(\xi)d\eta\label{lwp1.1}\\
    &+\int_{\R^3}e^{is\phi(\xi,\eta)}\partial_{\xi_l}\hat{f}_{m,k_1}(s,\xi-\eta)\hat{f}_{n,k_2}(s,\eta)\psi_k(\xi)d\eta\label{lwp1.2}\\
    &+\int_{\R^3}e^{is\phi(\xi,\eta)}\hat{f}_{m,k_1}(s,\xi-\eta)\hat{f}_{n,k_2}(s,\eta)\partial_{\xi_l}\psi_k(\xi)d\eta\label{lwp1.3}
\end{align}
and
\begin{align}
    &\partial_{\xi_l} \int_{\R^3}e^{is\phi(\xi,\eta)}q(\xi-\eta,\eta)\hat{f}_{m,k_1}(s,\xi-\eta)\hat{f}_{n,k_2}(s,\eta)\psi_k(\xi)d\eta\nonumber\\
    =& \int_{\R^3}e^{is\phi(\xi,\eta)}is\partial_{\xi_l}\phi(\xi,\eta) q(\xi-\eta,\eta)\hat{f}_{m,k_1}(s,\xi-\eta)\hat{f}_{n,k_2}(s,\eta)\psi_k(\xi)d\eta\label{lwp2.1}\\
    &+\int_{\R^3}e^{is\phi(\xi,\eta)}\partial_{\xi_l} q(\xi-\eta,\eta)\hat{f}_{m,k_1}(s,\xi-\eta)\hat{f}_{n,k_2}(s,\eta)\psi_k(\xi)d\eta\label{lwp2.3-1}\\
    &+\int_{\R^3}e^{is\phi(\xi,\eta)} q(\xi-\eta,\eta)\partial_{\xi_l}\hat{f}_{m,k_1}(s,\xi-\eta)\hat{f}_{n,k_2}(s,\eta)\psi_k(\xi)d\eta\label{lwp2.2}\\
    &+\int_{\R^3}e^{is\phi(\xi,\eta)} q(\xi-\eta,\eta)\hat{f}_{m,k_1}(s,\xi-\eta)\hat{f}_{n,k_2}(s,\eta)\partial_{\xi_l}\psi_k(\xi)d\eta\label{lwp2.3}.
\end{align}
Split the terms in the first derivative according to the number of derivatives on $f_{m,k_1}$ and define
$$F^0_{k,k_1,k_2}=\eqref{lwp1.1}+\eqref{lwp1.3}+\eqref{lwp2.1}+\eqref{lwp2.3-1}+\eqref{lwp2.3},$$
$$F^1_{k,k_1,k_2}=\eqref{lwp1.2}+\eqref{lwp2.2}.$$
From condition \eqref{q} on the multiplier $q$, Lemma \ref{chi+eta}, and Lemma \ref{chi+eta2}, we know
\begin{align*}
    \|\F^{-1}\partial_{\xi_l}\phi(\xi,\eta)\Tilde{\psi}_k(\xi)\Tilde{\psi}_{k_2}(\eta)\|_{L^1}
    +\|\F^{-1}\partial_{\xi_l}\phi(\xi,\eta)q(\xi-\eta,\eta)\Tilde{\psi}_k(\xi)\Tilde{\psi}_{k_1}(\xi-\eta)\Tilde{\psi}_{k_2}(\eta)\|_{L^1}
    \lesssim 2^{k}+2^{k_2}\lesssim 2^{k_1}
\end{align*}
and 
\begin{align*}
    \|\F^{-1}\nabla_\xi q(\xi-\eta,\eta)\Tilde{\psi}_k(\xi)\Tilde{\psi}_{k_1}(\xi-\eta)\Tilde{\psi}_{k_2}(\eta)\|_{L^1}
    \lesssim 2^{\e k_--k}.
\end{align*}
Then, by Lemma \ref{dualitycomp},
\begin{equation}
\label{zeroderterm1}
    \begin{aligned}
    \|F^0_{k,k_1,k_2}(s,\xi)\|_{L^2_\xi}
    \lesssim & (T2^{k_1}+2^{\e k_--k}+\|\nabla\psi_k\|_{L^\infty})\|f_{m,k_1}(s,x)\|_{L^2_x}2^{3\min\{k,k_2\}/2}\|f_{n,k_2}(s,x)\|_{L^2_x}\\
    \lesssim & (T2^{k_1}+2^{-k})\|f_{m,k_1}(s,x)\|_{L^2_x}2^{3\min\{k,k_2\}/2}\|f_{n,k_2}(s,x)\|_{L^2_x}
    \end{aligned}
\end{equation}
and
\begin{align}
\label{onederterms1}
&\|F^1_{k,k_1,k_2}(s,\xi)\|_{L^2_\xi}
\lesssim \|\nabla_\xi\hat{f}_{m,k_1}(s,\xi)\|_{L^2_\xi}2^{3\min\{k,k_2\}/2}\|f_{n,k_2}(s,x)\|_{L^2_x}.
\end{align}
In order to bound the $1+\alpha$ derivatives, we take another derivative in $\xi$ and get
\begin{align*}
    \partial_{\xi_m}[\eqref{lwp1.1}+\eqref{lwp1.3}]
    =&\int_{\R^3}e^{is\phi(\xi,\eta)}is\partial_{\xi_m}\partial_{\xi_l}\phi(\xi,\eta)\hat{f}_{m,k_1}(s,\xi-\eta)\hat{f}_{n,k_2}(s,\eta)\psi_k(\xi)d\eta\\
    &-\int_{\R^3}e^{is\phi(\xi,\eta)}s^2\partial_{\xi_l}\phi(\xi,\eta)\partial_{\xi_m}\phi(\xi,\eta)\hat{f}_{m,k_1}(s,\xi-\eta)\hat{f}_{n,k_2}(s,\eta)\psi_k(\xi)d\eta\\
    &+\int_{\R^3}e^{is\phi(\xi,\eta)}is\partial_{\xi_l}\phi(\xi,\eta)\partial_{\xi_m}\hat{f}_{m,k_1}(s,\xi-\eta)\hat{f}_{n,k_2}(s,\eta)\psi_k(\xi)d\eta\\
    &+2\int_{\R^3}e^{is\phi(\xi,\eta)}is\partial_{\xi_l}\phi(\xi,\eta)\hat{f}_{m,k_1}(s,\xi-\eta)\hat{f}_{n,k_2}(s,\eta)\partial_{\xi_m}\psi_k(\xi)d\eta\\
    &+\int_{\R^3}e^{is\phi(\xi,\eta)}\partial_{\xi_m}\hat{f}_{m,k_1}(s,\xi-\eta)\hat{f}_{n,k_2}(s,\eta)\partial_{\xi_l}\psi_k(\xi)d\eta\\
    &+\int_{\R^3}e^{is\phi(\xi,\eta)}\hat{f}_{m,k_1}(s,\xi-\eta)\hat{f}_{n,k_2}(s,\eta)\partial_{\xi_m}\partial_{\xi_l}\psi_k(\xi)d\eta,
\end{align*}
\begin{align*}
    \partial_{\xi_m}\eqref{lwp1.2}
    =&\int_{\R^3}e^{is\phi(\xi,\eta)}is\partial_{\xi_m}\phi(\xi,\eta)\partial_{\xi_l}\hat{f}_{m,k_1}(s,\xi-\eta)\hat{f}_{n,k_2}(s,\eta)\psi_k(\xi)d\eta\\
    &+\int_{\R^3}e^{is\phi(\xi,\eta)}\partial_{\xi_m}\partial_{\xi_l}\hat{f}_{m,k_1}(s,\xi-\eta)\hat{f}_{n,k_2}(s,\eta)\psi_k(\xi)d\eta\\
    &+\int_{\R^3}e^{is\phi(\xi,\eta)}\partial_{\xi_l}\hat{f}_{m,k_1}(s,\xi-\eta)\hat{f}_{n,k_2}(s,\eta)\partial_{\xi_m}\psi_k(\xi)d\eta,
\end{align*}
\begin{align*}
    \partial_{\xi_m}[\eqref{lwp2.1}+\eqref{lwp2.3-1}+\eqref{lwp2.3}]
    =&-\int_{\R^3}e^{is\phi(\xi,\eta)}s^2\partial_{\xi_l}\phi(\xi,\eta)\partial_{\xi_m}\phi(\xi,\eta) q(\xi-\eta,\eta)\hat{f}_{m,k_1}(s,\xi-\eta)\hat{f}_{n,k_2}(s,\eta)\psi_k(\xi)d\eta\\
    &+\int_{\R^3}e^{is\phi(\xi,\eta)}is\partial_{\xi_m}\partial_{\xi_l}\phi(\xi,\eta)q(\xi-\eta,\eta)\hat{f}_{m,k_1}(s,\xi-\eta)\hat{f}_{n,k_2}(s,\eta)\psi_k(\xi)d\eta\\
    &+2\int_{\R^3}e^{is\phi(\xi,\eta)}is\partial_{\xi_l}\phi(\xi,\eta) \partial_{\xi_m}q(\xi-\eta,\eta)\hat{f}_{m,k_1}(s,\xi-\eta)\hat{f}_{n,k_2}(s,\eta)\psi_k(\xi)d\eta\\
    &+\int_{\R^3}e^{is\phi(\xi,\eta)}is\partial_{\xi_l}\phi(\xi,\eta) q(\xi-\eta,\eta)\partial_{\xi_m}\hat{f}_{m,k_1}(s,\xi-\eta)\hat{f}_{n,k_2}(s,\eta)\psi_k(\xi)d\eta\\
    &+2\int_{\R^3}e^{is\phi(\xi,\eta)}is\partial_{\xi_l}\phi(\xi,\eta) q(\xi-\eta,\eta)\hat{f}_{m,k_1}(s,\xi-\eta)\hat{f}_{n,k_2}(s,\eta)\partial_{\xi_m}\psi_k(\xi)d\eta\\
    &+\int_{\R^3}e^{is\phi(\xi,\eta)}\partial_{\xi_m}\partial_{\xi_l} q(\xi-\eta,\eta)\hat{f}_{m,k_1}(s,\xi-\eta)\hat{f}_{n,k_2}(s,\eta)\psi_k(\xi)d\eta\\
    &+\int_{\R^3}e^{is\phi(\xi,\eta)}\partial_{\xi_l} q(\xi-\eta,\eta)\partial_{\xi_m}\hat{f}_{m,k_1}(s,\xi-\eta)\hat{f}_{n,k_2}(s,\eta)\psi_k(\xi)d\eta\\
    &+2\int_{\R^3}e^{is\phi(\xi,\eta)}\partial_{\xi_l} q(\xi-\eta,\eta)\hat{f}_{m,k_1}(s,\xi-\eta)\hat{f}_{n,k_2}(s,\eta)\partial_{\xi_m}\psi_k(\xi)d\eta\\
    &+\int_{\R^3}e^{is\phi(\xi,\eta)} q(\xi-\eta,\eta)\partial_{\xi_m}\hat{f}_{m,k_1}(s,\xi-\eta)\hat{f}_{n,k_2}(s,\eta)\partial_{\xi_l}\psi_k(\xi)d\eta\\
    &+\int_{\R^3}e^{is\phi(\xi,\eta)} q(\xi-\eta,\eta)\hat{f}_{m,k_1}(s,\xi-\eta)\hat{f}_{n,k_2}(s,\eta)\partial_{\xi_l}\partial_{\xi_m}\psi_k(\xi)d\eta,
\end{align*}
and
\begin{align*}
    \partial_{\xi_m}\eqref{lwp2.2}
    =&\int_{\R^3}e^{is\phi(\xi,\eta)}is\partial_{\xi_m}\phi(\xi,\eta) q(\xi-\eta,\eta)\partial_{\xi_l}\hat{f}_{m,k_1}(s,\xi-\eta)\hat{f}_{n,k_2}(s,\eta)\psi_k(\xi)d\eta\\
    &+\int_{\R^3}e^{is\phi(\xi,\eta)}\partial_{\xi_m} q(\xi-\eta,\eta)\partial_{\xi_l}\hat{f}_{m,k_1}(s,\xi-\eta)\hat{f}_{n,k_2}(s,\eta)\psi_k(\xi)d\eta\\
    &+\int_{\R^3}e^{is\phi(\xi,\eta)} q(\xi-\eta,\eta)\partial_{\xi_m}\partial_{\xi_l}\hat{f}_{m,k_1}(s,\xi-\eta)\hat{f}_{n,k_2}(s,\eta)\psi_k(\xi)d\eta\\
    &+\int_{\R^3}e^{is\phi(\xi,\eta)} q(\xi-\eta,\eta)\partial_{\xi_l}\hat{f}_{m,k_1}(s,\xi-\eta)\hat{f}_{n,k_2}(s,\eta)\partial_{\xi_m}\psi_k(\xi)d\eta.
\end{align*}
Using Lemma \ref{chi+eta}, Lemma \ref{chi+eta2}, and \eqref{q}, we have
\begin{align*}
    &\|\F^{-1}\partial_{\xi_l}\phi(\xi,\eta)\partial_{\xi_m}\phi(\xi,\eta)\Tilde{\psi}_k(\xi)\Tilde{\psi}_{k_2}(\eta)\|_{L^1}
    +\\
    &+\|\F^{-1}\partial_{\xi_l}\phi(\xi,\eta)\partial_{\xi_m}\phi(\xi,\eta) q(\xi-\eta,\eta)\Tilde{\psi}_k(\xi)\Tilde{\psi}_{k_1}(\xi-\eta)\Tilde{\psi}_{k_2}(\eta)\|_{L^1}
    \lesssim 2^{2k_1}
\end{align*}
and 
\begin{align*}
    \|\F^{-1}\nabla^2_\xi q(\xi-\eta,\eta)\Tilde{\psi}_k(\xi)\Tilde{\psi}_{k_1}(\xi-\eta)\Tilde{\psi}_{k_2}(\eta)\|_{L^1}
    \lesssim 2^{\e k_--2k}.
\end{align*}
Hence, employing Lemma \ref{dualitycomp}, we obtain
\begin{equation}
    \label{zeroderterms2}
    \begin{aligned}
    \|\nabla_\xi F^0_{k,k_1,k_2}(s,\xi)\|_{L^2_\xi}
    \lesssim & (T+T^22^{2k_1}+T2^{k_1}\|\nabla\psi_k\|_{L^\infty}+T2^{k_1+\e k_--k}+\|\nabla^2\psi_k\|_{L^\infty}+2^{\e k_--2k}+\\
    &+2^{\e k_--k}\|\nabla\psi_k\|_{L^\infty})\|f_{m,k_1}(s,x)\|_{L^2_x}2^{3\min\{k,k_2\}/2}\|f_{n,k_2}(s,x)\|_{L^2_x}+\\
    &+(T2^{k_1}+\|\nabla\psi_k\|_{L^\infty}+2^{\e k_--k})\|\F^{-1}\nabla_\xi\hat{f}_{m,k_1}(s,\xi)\|_{L^2_x}2^{3\min\{k,k_2\}/2}\|f_{n,k_2}(s,x)\|_{L^2_x}\\
    \lesssim & (T+T^22^{2k_1}+T2^{k_1-k}+2^{-2k})\|f_{m,k_1}(s,x)\|_{L^2_x}2^{3\min\{k,k_2\}/2}\|f_{n,k_2}(s,x)\|_{L^2_x}+\\
    &+(T2^{k_1}+2^{-k})\|\F^{-1}\nabla_\xi\hat{f}_{m,k_1}(s,\xi)\|_{L^2_x}2^{3\min\{k,k_2\}/2}\|f_{n,k_2}(s,x)\|_{L^2_x}
    \end{aligned}
\end{equation}
and
\begin{equation}
    \label{onederterms2}
    \begin{aligned}
    \|\nabla_\xi F^1_{k,k_1,k_2}(s,\xi)\|_{L^2_\xi}
    \lesssim &(T2^{k_1}+\|\nabla\psi_k\|_{L^\infty}+2^{\e k_--k})\|\F^{-1}\nabla_\xi\hat{f}_{m,k_1}(s,\xi)\|_{L^2_x}2^{3\min\{k,k_2\}/2}\|f_{n,k_2}(s,x)\|_{L^2_x}+\\
    &+\|\F^{-1}\nabla^2_\xi\hat{f}_{m,k_1}(s,\xi)\|_{L^2_x}2^{3\min\{k,k_2\}/2}\|f_{n,k_2}(s,x)\|_{L^2_x}\\
    \lesssim &(T2^{k_1}+2^{-k})\|\nabla_\xi\hat{f}_{m,k_1}(s,\xi)\|_{L^2_\xi}2^{3\min\{k,k_2\}/2}\|f_{n,k_2}(s,x)\|_{L^2_x}+\\
    &+\|\nabla^2_\xi\hat{f}_{m,k_1}(s,\xi)\|_{L^2_\xi}2^{3\min\{k,k_2\}/2}\|f_{n,k_2}(s,x)\|_{L^2_x}\\
    \lesssim &T2^{k_1}\|\nabla_\xi\hat{f}_{m,k_1}(s,\xi)\|_{L^2_\xi}2^{3\min\{k,k_2\}/2}\|f_{n,k_2}(s,x)\|_{L^2_x}+\\
    &+(1+2^{-k+k_1})\|\nabla^2_\xi\hat{f}_{m,k_1}(s,\xi)\|_{L^2_\xi}2^{3\min\{k,k_2\}/2}\|f_{n,k_2}(s,x)\|_{L^2_x}\\
    \lesssim &(T2^{k_1}+2^{-k+k_1}+1)\|\hat{f}_{m,k_1}(s,\xi)\|_{H^2_\xi}2^{3\min\{k,k_2\}/2}\|f_{n,k_2}(s,x)\|_{L^2_x}.
    \end{aligned}
\end{equation}
since \eqref{addder} implies $\|\nabla_\xi\hat{f}_{m,k_1}\|_{L^2_\xi}\lesssim 2^{k_1}\|\nabla^2_\xi\hat{f}_{m,k_1}\|_{L^2_\xi}$.\\
Using Lemma \ref{sobolev space}, $f\in X$ implies for any $t\in[t_0,T]$
\begin{align}
\label{bound on h1 h1+alpha}
    \|\hat{f}_{l,k}\|_{H^1_\xi}\lesssim 2^{-2k_++\gamma k-k/2}r,\,\|\hat{f}_{l,k}\|_{H^{1+\alpha}_\xi}\lesssim 2^{-2k_++\gamma k-k/2-\alpha k}r,
\end{align}
and \eqref{l21}, \eqref{l22} implies $\|f_{l,k}\|_{L^\infty_s([t_0,T])L^2_x}\lesssim \min\{2^{-2k_++\gamma k+k/2},2^{-10k_+}\}r.$\\
Thus, as a result of Lemma \ref{dalpha},
\begin{align*}
    &\|D^\alpha F^0_{k,k_1,k_2}(s,\xi)\|_{L^2_\xi}
    \leq  \|F^0_{k,k_1,k_2}(s,\xi)\|_{L^2_\xi}^{1-\alpha}\|\nabla_\xi F^0_{k,k_1,k_2}(s,\xi)\|_{L^2_\xi}^\alpha,
\end{align*}
and using \eqref{zeroderterm1} and \eqref{zeroderterms2}, we have
\begin{align*}
    &\sum_{(k_1,k_2)\in\chi^1_k}\|D^\alpha F^0_{k,k_1,k_2}(s,\xi)\|_{L^2_\xi}\\
    \lesssim & \sum_{(k_1,k_2)\in\chi^1_k}(T2^{k_1}+2^{-k}+T+T^22^{2k_1}+T2^{k_1-k}+2^{-2k})\|f_{m,k_1}(s,x)\|^{1-\alpha}_{L^2_x}\times\\
    &\qquad\qquad\qquad\qquad\times \|\F^{-1}\nabla_\xi\hat{f}_{m,k_1}(s,\xi)\|^\alpha_{L^2_x}2^{3k_2/2}\|f_{n,k_2}(s,x)\|_{L^2_x}\\
    \lesssim & \sum_{k_2\leq k}(T2^{k}+2^{-k}+T+T^22^{2k}+2^{-2k})\min\{2^{-2k_++\gamma k+k/2-\alpha k},2^{-6k_+}\}r2^{3k_2/2}\|f_{n,k_2}(s,x)\|_{L^2_x}\\
    \lesssim & \sum_{k_2\leq k}(T2^{2k}+1+T2^k+T^22^{3k}+2^{-k})2^{-2k_++\gamma k-k/2-\alpha
    k}\min\{1,2^{-4k_+-\gamma k-k/2+\alpha k}\}r2^{3k_2/2}\|f_{n,k_2}(s,x)\|_{L^2_x}\\
    \lesssim & 2^{-2k_++\gamma k-k/2-\alpha
    k}(1+T2^k+T2^{2k}+T^22^{3k})\min\{1,2^{-4k_+-\gamma k-k/2+\alpha k}\}r\sum_{k_2\leq k}2^{10k_{2,+}}\|f_{n,k_2}(s,x)\|_{L^2_x}\\
    \lesssim & 2^{-2k_++\gamma k-k/2-\alpha
    k}T^2r\|f_n\|_{L^\infty_s([t_0,T])H^{10}_x}
\end{align*}
and
\begin{align*}
    &\sum_{(k_1,k_2)\in\chi^2_k\cup\chi^3_k}\|D^\alpha F^0_{k,k_1,k_2}(s,\xi)\|_{L^2_\xi}\\
    \lesssim & \sum_{(k_1,k_2)\in\chi^2_k\cup\chi^3_k}(T2^{k_2}+2^{-k}+T+T^22^{2k_2}+T2^{k_2-k}+2^{-2k})\times\\
    &\qquad\qquad\qquad\qquad\times\|f_{m,k_1}(s,x)\|^{1-\alpha}_{L^2_x}\|\F^{-1}\nabla_\xi\hat{f}_{m,k_1}(s,\xi)\|^\alpha_{L^2_x}2^{3k/2}\|f_{n,k_2}(s,x)\|_{L^2_x}\\
    \lesssim & \sum_{k_2\geq k-2a-2}(T2^{k_2}+2^{-k}+T^22^{2k_2}+T2^{k_2-k}+2^{-2k})2^{-2k_{2,+}+\gamma k_2+k_2/2-\alpha k_2}r2^{3k/2}\|f_{n,k_2}(s,x)\|_{L^2_x}\\
    = & 2^{-k/2}\sum_{k_2\geq k-2a-2}(T2^{3k_2/2+2k}+2^{k+k_2/2}+T^22^{5k_2/2+2k}+T2^{3k_2/2+k}+2^{k_2/2})2^{-2k_{2,+}+\gamma k_2-\alpha k_2}r\|f_{n,k_2}(s,x)\|_{L^2_x}\\
    \lesssim & 2^{-2k_++\gamma k-k/2-\alpha
    k}T^2r\sum_{k_2\geq k-2a-2}2^{10k_{2,+}}\|f_{n,k_2}(s,x)\|_{L^2_x}\\
    \lesssim & 2^{-2k_++\gamma k-k/2-\alpha
    k}T^2r\|f_n\|_{L^\infty_s([t_0,T])H^{10}_x}.
\end{align*}
For the estimation of $F^1_{k,k_1,k_2}$, we will use the interpolation result in Lemma \ref{si}. We can define a family of operators on the strip $\{z\in\C:0\leq \Re(z)\leq 1\}$
\begin{align*}
    T_z \hat{g}(s,\xi)
    =D^z_\xi \bigg(&\int_{\R^3}e^{is\phi(\xi,\eta)}\partial_{\xi_l}\hat{g}_{k_1}(s,\xi-\eta)\hat{f}_{n,k_2}(s,\eta)\psi_k(\xi)d\eta\\
    &+\int_{\R^3}e^{is\phi(\xi,\eta)} q(\xi-\eta,\eta)\partial_{\xi_l}\hat{g}_{k_1}(s,\xi-\eta)\hat{f}_{n,k_2}(s,\eta)\psi_k(\xi)d\eta\bigg).
\end{align*}
From \eqref{onederterms1} and \eqref{onederterms2}, we know
\begin{align*}
\|T_{0+iy}\hat{g}(s,\xi)\|_{L^2_\xi}
\lesssim & \|\nabla_\xi\hat{g}_{k_1}(s,\xi)\|_{L^2_\xi}2^{3\min\{k,k_2\}/2}\|f_{n,k_2}(s,x)\|_{L^2_x}\\
\leq & \|\hat{g}(s,\xi)\|_{H^1_\xi}2^{3\min\{k,k_2\}/2}\|f_{n,k_2}(s,x)\|_{L^2_x}
\end{align*}
and
\begin{align*}
    \|T_{1+iy}\hat{g}(s,\xi)\|_{L^2_\xi}
    \lesssim &(T2^{k_1}+2^{-k+k_1}+1)\|\nabla^2_\xi\hat{g}_{k_1}(s,\xi)\|_{L^2_\xi}2^{3\min\{k,k_2\}/2}\|f_{n,k_2}(s,x)\|_{L^2_x}\\
    \leq & (T2^{k_1}+2^{-k+k_1}+1)\|\hat{g}(s,\xi)\|_{H^2_\xi}2^{3\min\{k,k_2\}/2}\|f_{n,k_2}(s,x)\|_{L^2_x}.
\end{align*}
Thus, by the interpolation result in Lemma \ref{si},
\begin{align*}
    \|D^\alpha F^1_{k,k_1,k_2}(s,\xi)\|_{L^2_\xi}
    =& \|T_{\alpha} (\hat{f}\Tilde{\psi}_{k_1})(s,\xi)\|_{L^2_\xi}\\
    \lesssim & (T2^{k_1}+2^{-k+k_1}+1)\|\hat{f}_{m,k_1}(s,\xi)\|_{H^{1+\alpha}_\xi}2^{3\min\{k,k_2\}/2}\|f_{n,k_2}(s,x)\|_{L^2_x},
\end{align*}
and \eqref{bound on h1 h1+alpha} implies for any $s\in[t_0,T]$
\begin{align*}
    &\sum_{(k_1,k_2)\in\chi^2_k\cup\chi^3_k}\|D^\alpha F^1_{k,k_1,k_2}(s,\xi)\|_{L^2_\xi}\\
    \lesssim &\sum_{(k_1,k_2)\in\chi^2_k\cup\chi^3_k}(T2^{k_2}+2^{-k+k_1}+1)2^{-2k_{1,+}+\gamma k_1-k_1/2-\alpha k_1}r2^{3k/2}\|f_{n,k_2}(s,x)\|_{L^2_x}\\
    \lesssim &2^{-2k_++\gamma k-k/2-\alpha k}\sum_{k_2\geq k-2a-2}(T2^{k_2+3k/2}+2^{k_2+k/2})2^{(\gamma-1/2-\alpha)(k_2-k)}r\|f_{n,k_2}(s,x)\|_{L^2_x}\\
    \lesssim &2^{-2k_++\gamma k-k/2-\alpha k}Tr\sum_{k_2\geq k-2a-2}2^{10k_{2,+}}\|f_{n,k_2}(s,x)\|_{L^2_x}\\
    \lesssim & 2^{-2k_++\gamma k-k/2-\alpha k}Tr\|f_n\|_{L^\infty_s([t_0,T])H^{10}_x}.
\end{align*}
For $(k_1,k_2)\in\chi^1_k$, we need a new bound for $\|\nabla F^1_{k,k_1,k_2}\|_{L^2}$. Using the identity $\sum_{j}\frac{\partial_{\eta_j}\phi}{it|\nabla_\eta\phi|^2}\partial_{\eta_j}e^{it\phi(\xi,\eta)}=e^{it\phi(\xi,\eta)}$, we have
\begin{align*}
    \partial_{\xi_m}\eqref{lwp1.2}=-&\int_{\R^3}e^{is\phi(\xi,\eta)}\partial_{\eta_j}\frac{\partial_{\eta_j}\phi(\xi,\eta)\partial_{\xi_m}\phi(\xi,\eta)}{|\nabla_\eta\phi|^2}\partial_{\xi_l}\hat{f}_{m,k_1}(s,\xi-\eta)\hat{f}_{n,k_2}(s,\eta)\psi_k(\xi)d\eta\\
    -&\int_{\R^3}e^{is\phi(\xi,\eta)}\frac{\partial_{\eta_j}\phi(\xi,\eta)\partial_{\xi_m}\phi(\xi,\eta)}{|\nabla_\eta\phi|^2}\partial_{\eta_j}\partial_{\xi_l}\hat{f}_{m,k_1}(s,\xi-\eta)\hat{f}_{n,k_2}(s,\eta)\psi_k(\xi)d\eta\\
    -&\int_{\R^3}e^{is\phi(\xi,\eta)}\frac{\partial_{\eta_j}\phi(\xi,\eta)\partial_{\xi_m}\phi(\xi,\eta)}{|\nabla_\eta\phi|^2}\partial_{\xi_l}\hat{f}_{m,k_1}(s,\xi-\eta)\partial_{\eta_j}\hat{f}_{n,k_2}(s,\eta)\psi_k(\xi)d\eta\\
    &+\int_{\R^3}e^{is\phi(\xi,\eta)}\partial_{\xi_m}\partial_{\xi_l}\hat{f}_{m,k_1}(s,\xi-\eta)\hat{f}_{n,k_2}(s,\eta)\psi_k(\xi)d\eta\\
    &+\int_{\R^3}e^{is\phi(\xi,\eta)}\partial_{\xi_l}\hat{f}_{m,k_1}(s,\xi-\eta)\hat{f}_{n,k_2}(s,\eta)\partial_{\xi_m}\psi_k(\xi)d\eta,
\end{align*}
and
\begin{align*}
    \partial_{\xi_m}\eqref{lwp2.2}
    =&-\int_{\R^3}e^{is\phi(\xi,\eta)}\partial_{\eta_j}\frac{\partial_{\eta_j}\phi(\xi,\eta)\partial_{\xi_m}\phi(\xi,\eta)}{|\nabla_\eta\phi|^2} q(\xi-\eta,\eta)\partial_{\xi_l}\hat{f}_{m,k_1}(s,\xi-\eta)\hat{f}_{n,k_2}(s,\eta)\psi_k(\xi)d\eta\\
    &-\int_{\R^3}e^{is\phi(\xi,\eta)}\frac{\partial_{\eta_j}\phi(\xi,\eta)\partial_{\xi_m}\phi(\xi,\eta)}{|\nabla_\eta\phi|^2} \partial_{\eta_j}q(\xi-\eta,\eta)\partial_{\xi_l}\hat{f}_{m,k_1}(s,\xi-\eta)\hat{f}_{n,k_2}(s,\eta)\psi_k(\xi)d\eta\\
    &-\int_{\R^3}e^{is\phi(\xi,\eta)}\frac{\partial_{\eta_j}\phi(\xi,\eta)\partial_{\xi_m}\phi(\xi,\eta)}{|\nabla_\eta\phi|^2} q(\xi-\eta,\eta)\partial_{\eta_j}\partial_{\xi_l}\hat{f}_{m,k_1}(s,\xi-\eta)\hat{f}_{n,k_2}(s,\eta)\psi_k(\xi)d\eta\\
    &-\int_{\R^3}e^{is\phi(\xi,\eta)}\frac{\partial_{\eta_j}\phi(\xi,\eta)\partial_{\xi_m}\phi(\xi,\eta)}{|\nabla_\eta\phi|^2} q(\xi-\eta,\eta)\partial_{\xi_l}\hat{f}_{m,k_1}(s,\xi-\eta)\partial_{\eta_j}\hat{f}_{n,k_2}(s,\eta)\psi_k(\xi)d\eta\\
    &+\int_{\R^3}e^{is\phi(\xi,\eta)}\partial_{\xi_m} q(\xi-\eta,\eta)\partial_{\xi_l}\hat{f}_{m,k_1}(s,\xi-\eta)\hat{f}_{n,k_2}(s,\eta)\psi_k(\xi)d\eta\\
    &+\int_{\R^3}e^{is\phi(\xi,\eta)} q(\xi-\eta,\eta)\partial_{\xi_m}\partial_{\xi_l}\hat{f}_{m,k_1}(s,\xi-\eta)\hat{f}_{n,k_2}(s,\eta)\psi_k(\xi)d\eta\\
    &+\int_{\R^3}e^{is\phi(\xi,\eta)} q(\xi-\eta,\eta)\partial_{\xi_l}\hat{f}_{m,k_1}(s,\xi-\eta)\hat{f}_{n,k_2}(s,\eta)\partial_{\xi_m}\psi_k(\xi)d\eta.
\end{align*}
Since \eqref{xi-2eta} implies
\begin{align*}
    \|\F^{-1}\partial_{\eta_j}\frac{\partial_{\eta_j}\phi(\xi,\eta)\partial_{\xi_m}\phi(\xi,\eta)}{|\nabla_\eta\phi|^2}\Tilde{\psi}_{k}(\xi)\Tilde{\psi}_{k_1}(\xi-\eta)\Tilde{\psi}_{k_2}(\eta)\|_{L^1}\lesssim 2^{-k_1}
\end{align*}
and
\begin{align*}
    \|\F^{-1}\frac{\partial_{\eta_j}\phi(\xi,\eta)\partial_{\xi_m}\phi(\xi,\eta)}{|\nabla_\eta\phi|^2}\Tilde{\psi}_{k}(\xi)\Tilde{\psi}_{k_1}(\xi-\eta)\Tilde{\psi}_{k_2}(\eta)\|_{L^1}\lesssim 1,
\end{align*}
and $\nabla_\eta\phi(\xi,\eta)=2c_n\xi$ when $c_n+c_m=0$ implies
\begin{align*}
    &\|\F^{-1}\partial_{\eta_j}\frac{\partial_{\eta_j}\phi(\xi,\eta)\partial_{\xi_m}\phi(\xi,\eta)}{|\nabla_\eta\phi|^2}q(\xi-\eta,\eta)\Tilde{\psi}_{k}(\xi)\Tilde{\psi}_{k_1}(\xi-\eta)\Tilde{\psi}_{k_2}(\eta)\|_{L^1}+\\
    &+\|\F^{-1}\frac{\partial_{\eta_j}\phi(\xi,\eta)\partial_{\xi_m}\phi(\xi,\eta)}{|\nabla_\eta\phi|^2}\partial_{\eta_j}q(\xi-\eta,\eta)\Tilde{\psi}_{k}(\xi)\Tilde{\psi}_{k_1}(\xi-\eta)\Tilde{\psi}_{k_2}(\eta)\|_{L^1}\lesssim 2^{\e k_--k_2}
\end{align*}
and
\begin{align*}
    \|\F^{-1}\frac{\partial_{\eta_j}\phi(\xi,\eta)\partial_{\xi_m}\phi(\xi,\eta)}{|\nabla_\eta\phi|^2}q(\xi-\eta,\eta)\Tilde{\psi}_{k}(\xi)\Tilde{\psi}_{k_1}(\xi-\eta)\Tilde{\psi}_{k_2}(\eta)\|_{L^1}\lesssim 2^{\e k_-},
\end{align*}
we have
\begin{align*}
    \|\nabla F^1_{k,k_1,k_2}\|_{L^2_\xi}
    \lesssim &(1+2^{-k_2})\|\hat{f}_{m,k_1}(s,\xi)\|_{H^2_\xi}2^{3\min\{k,k_2\}/2}\|f_{n,k_2}(s,x)\|_{L^2_x}\\
    &+\|\hat{f}_{m,k_1}(s,\xi)\|_{H^2_\xi}2^{3\min\{k,k_2\}/2}\|\nabla_\xi\hat{f}_{n,k_2}(s,\xi)\|_{L^2_{\xi}}\\
    \lesssim &(1+2^{k_2})\|\hat{f}_{m,k_1}(s,\xi)\|_{H^2_\xi}2^{3\min\{k,k_2\}/2}\|\nabla_\xi\hat{f}_{n,k_2}(s,\xi)\|_{L^2_{
    \xi}}.
\end{align*}
Thus, by the interpolation result in Lemma \ref{si},
\begin{align*}
    \|D^\alpha F^1_{k,k_1,k_2}(s,\xi)\|_{L^2_\xi}
    \lesssim & (1+2^{k_2})\|\hat{f}_{m,k_1}(s,\xi)\|_{H^{1+\alpha}_\xi}2^{3\min\{k,k_2\}/2}\|f_{n,k_2}(s,x)\|_{L^2_x}^{1-\alpha}\|\hat{f}_{n,k_2}(s,\xi)\|_{H^1_{
    \xi}}^{\alpha},
\end{align*}
and \eqref{sobolevnorms}, \eqref{bound on h1 h1+alpha} imply that for any $s\in[t_0,T]$
\begin{align*}
    \sum_{(k_1,k_2)\in\chi^1_k}\|D^\alpha F^1_{k,k_1,k_2}(s,\xi)\|_{L^2_\xi}
    \lesssim &\sum_{(k_1,k_2)\in\chi^1_k}(1+2^{k_2})2^{-2k_{1,+}+\gamma k_1-k_1/2-\alpha k_1}r2^{3k_2/2}\|f_{k_2}(s,x)\|^{1/2}_{L^2_x}\|\hat{f}_{n,k_2}(s,\xi)\|_{H^1_{
    \xi}}^{1/2}\\
    \lesssim &2^{-2k_++\gamma k-k/2-\alpha k}r\sum_{k_2\leq k}(1+2^{k_2})2^{3k_2/2}\|f_{n,k_2}(s,x)\|^{1/2}_{L^2_x}2^{-k_{2,+}+\gamma k_2-k_2/4}r^{1/2}\\
    \lesssim &2^{-2k_++\gamma k-k/2-\alpha k}r^{3/2}\sum_{k_2\leq k}2^{5k_{2,+}}\|f_{n,k_2}(s,x)\|^{1/2}_{L^2_x}\\
    \lesssim & 2^{-2k_++\gamma k-k/2-\alpha k}r^{3/2}\|f_n\|_{L^\infty_s([t_0,T])H^{10}_x}^{1/2}.
\end{align*}
\label{keypage2}
Summarizing the results above, we have
\begin{align*}
    &\bigg\|D^{1+\alpha}_\xi \F e^{-ic_ls\la}\big(\sum_{\substack{c_m+c_n=0\\m\geq n}}A_{lmn}Q(e^{ic_ms\la}f_m,e^{ic_ns\la}f_n)+\\
    &+\sum_{\substack{c_m+c_n\neq 0\\m\geq n}}A_{lmn}(e^{ic_ms\la}f_m)(e^{ic_ns\la}f_n)\big)\psi_k(\xi)\bigg\|_{L^\infty_s([t_0,T])L^2_\xi}\\
    \lesssim & 2^{-2k_++\gamma k-k/2-\alpha k}Tr\|f_n\|_{L^\infty_s([t_0,T])H^{10}_x}+2^{-2k_++\gamma k-k/2-\alpha k}r^{3/2}\|f_n\|_{L^\infty_s([t_0,T])H^{10}_x}^{1/2}\\
    \lesssim & 2^{-2k_++\gamma k-k/2-\alpha k}Tr^2.
\end{align*}
Hence, there exists some constant $C'>0$ independent of $f$ so that 
\begin{align*}
    \|\Psi(\vec{f})_l\|_Z
    \leq & \|e^{-ic_l\la}u_{l0}\|_{Z}+ C'(T-t_0)Tr^2.
\end{align*}
Since $\|u_{l0}\|_{H^{10}}+\|e^{-ic_lt_0\la}u_{l0}\|_Z\leq \e$, we can pick $r=2\e$ and $T=T(\e,t_0)>t_0$ so that $\max\{C,C'\}(T-t_0)Tr < 1/12$. Then
\begin{align*}
    \sup_{t\in[1,T]}\|\Psi(\vec{f})_l\|_{H^{10}_x}+\|\Psi(\vec{f})_l\|_Z
    \leq  \e+C(T-t_0)r^2+C'(T-t_0)Tr^2 \leq \e+r/12+r/12\leq r.
\end{align*}
Thus, we proved $\Psi:X\rightarrow X$.\\
Following along the same lines from page \pageref{keypage1} to page \pageref{keypage2}, we have for arbitrary $\vec{f},\vec{g}\in X$
\begin{align*}
    \sum_{\substack{c_m+c_n\neq 0\\m\geq n}}A_{lmn}\|Q(e^{ic_ms\la}f_m,e^{ic_ns\la}g_n)\|_{L^\infty_s([t_0,T]) H^{10}_x}
    \leq &C \|f_m\|_{L^\infty_t([t_0,T])H^{10}_x}\|g_n\|_{L^\infty_t([t_0,T])H^{10}_x}\\
    \leq & C \|\vec{f}\|_X\|\vec{g}\|_X
\end{align*}
and
\begin{align*}
    &\bigg\|D^{1+\alpha}_\xi \F e^{-ic_ls\la}\big(\sum_{\substack{c_m+c_n=0\\m\geq n}}A_{lmn}Q(e^{ic_ms\la}f_m,e^{ic_ns\la}g_n)+\\
    &\qquad\qquad\qquad\qquad+\sum_{\substack{c_m+c_n\neq 0\\m\geq n}}A_{lmn}(e^{ic_ms\la}f_m)(e^{ic_ns\la}g_n)\big)\psi_k(\xi)\bigg\|_{L^\infty_s([t_0,T])L^2_\xi}\\
    \leq & C'2^{-2k_++\gamma k-k/2-\alpha k}T(\|f_m\|_{L^\infty_s([t_0,T])Z}+\|f_m\|_{L^\infty_s([t_0,T])H^{10}_x})(\|g_n\|_{L^\infty_s([t_0,T])Z}+\|g_n\|_{L^\infty_s([t_0,T])H^{10}_x})\\
    \leq & C'2^{-2k_++\gamma k-k/2-\alpha k}T\|\vec{f}\|_X\|\vec{g}\|_X,
\end{align*}
which imply
\begin{align*}
   &\|e^{-ic_ls\la}\big(\sum_{\substack{c_m+c_n=0\\m\geq n}}A_{lmn}Q(e^{ic_ms\la}f_m,e^{ic_ns\la}g_n)+\sum_{\substack{c_m+c_n\neq 0\\m\geq n}}A_{lmn}(e^{ic_ms\la}f_m)(e^{ic_ns\la}g_n)\big)\|_{L^\infty_s([t_0,T])Z}\\
    \leq & C'T\|\vec{f}\|_X\|\vec{g}\|_X.
\end{align*}
Now, we are ready to show $\Psi$ is a contraction map. Suppose $\vec{f},\vec{g}\in X$, then
\begin{align*}
    &\Psi(\vec{f})_l-\Psi(\vec{g})_l\\
    =&\int_{t_0}^te^{-ic_ls\la}\bigg(\sum_{\substack{c_m+c_n=0\\m\geq n}}A_{lmn}(Q(e^{ic_ms\la}f_m,e^{ic_ns\la}f_n)-Q(e^{ic_ms\la}g_m,e^{ic_ns\la}g_n))+\\
    &\qquad\qquad\qquad\qquad +\sum_{\substack{c_m+c_n\neq 0\\m\geq n}}A_{lmn}((e^{ic_ms\la}f_m)(e^{ic_ns\la}f_n)-(e^{ic_ms\la}g_m)(e^{ic_ns\la}g_n))\bigg)ds\\
    =&\int_{t_0}^te^{-ic_ls\la}\bigg(\sum_{\substack{c_m+c_n=0\\m\geq n}}A_{lmn}(Q(e^{ic_ms\la}f_m,e^{ic_ns\la}f_n)-Q(e^{ic_ms\la}g_m,e^{ic_ns\la}f_n)+\\
    &\qquad\qquad\qquad\qquad
    +Q(e^{ic_ms\la}g_m,e^{ic_ns\la}f_n)-Q(e^{ic_ms\la}g_m,e^{ic_ns\la}g_n))+\\
    &\qquad\qquad\qquad +\sum_{\substack{c_m+c_n\neq 0\\m\geq n}}A_{lmn}((e^{ic_ms\la}f_m)(e^{ic_ns\la}f_n)-(e^{ic_ms\la}g_m)(e^{ic_ns\la}f_n)+\\
    &\qquad\qquad\qquad\qquad+(e^{ic_ms\la}g_m)(e^{ic_ns\la}f_n)-(e^{ic_ms\la}g_m)(e^{ic_ns\la}g_n))\bigg)ds\\
    =&\int_{t_0}^te^{-ic_ls\la}\bigg(\sum_{\substack{c_m+c_n=0\\m\geq n}}A_{lmn}(Q(e^{ic_ms\la}(f_m-g_m),e^{ic_ns\la}f_n)+Q(e^{ic_ms\la}g_m,e^{ic_ns\la}(f_n-g_n))+\\
    &\qquad\qquad\qquad +\sum_{\substack{c_m+c_n\neq 0\\m\geq n}}A_{lmn}((e^{ic_ms\la}(f_m-g_m))(e^{ic_ns\la}f_n)+(e^{ic_ms\la}g_m)(e^{ic_ns\la}(f_n-g_n)))\bigg)ds.
\end{align*}
For any $t\in[t_0,T]$,
\begin{align*}
    &\sum_{\substack{c_m+c_n= 0\\m\geq n}}A_{lmn}\|Q(e^{ic_ms\la}(f_m-g_m),e^{ic_ns\la}f_n)+Q(e^{ic_ms\la}g_m,e^{ic_ns\la}(f_n-g_n))\|_{L^\infty_s([t_0,T]) H^{10}_x}+\\
    &+\sum_{\substack{c_m+c_n\neq 0\\m\geq n}}A_{lmn}\|(e^{ic_ms\la}(f_m-g_m))(e^{ic_ns\la}f_n)+(e^{ic_ms\la}g_m)(e^{ic_ns\la}(f_n-g_n)\|_{L^\infty_s([t_0,T]) H^{10}_x}\\
    \leq &C \|\vec{f}-\vec{g}\|_{X}\|\vec{f}\|_{X}+C \|\vec{g}\|_{X}\|\vec{f}-\vec{g}\|_{X}
\end{align*}
and
\begin{align*}
    &\bigg\|e^{-ic_ls\la}\big(\sum_{\substack{c_m+c_n=0\\m\geq n}}A_{lmn}Q(e^{ic_ms\la}(f_m-g_m),e^{ic_ns\la}f_n)+\\
    &\qquad\qquad\qquad\qquad+\sum_{\substack{c_m+c_n\neq 0\\m\geq n}}A_{lmn}(e^{ic_ms\la}(f_m-g_m))(e^{ic_ns\la}f_n)\big)\bigg\|_{L^\infty_s([t_0,T])Z}+\\
    &+\bigg\|e^{-ic_ls\la}\big(\sum_{\substack{c_m+c_n=0\\m\geq n}}A_{lmn}Q(e^{ic_ms\la}g_m,e^{ic_ns\la}(f_n-g_n)+\\
    &\qquad\qquad\qquad\qquad+\sum_{\substack{c_m+c_n\neq 0\\m\geq n}}A_{lmn}(e^{ic_ms\la}g_m)(e^{ic_ns\la}(f_n-g_n))\big)\bigg\|_{L^\infty_s([t_0,T])Z}\\
    \leq & C'T\|\vec{f}-\vec{g}\|_{X}\|\vec{f}\|_{X}
    +C'T\|\vec{g}\|_{X}\|\vec{f}-\vec{g}\|_{X}.
\end{align*}
Since $\vec{f},\vec{g}\in X$, we know $\|\Vec{f}\|_X+\|\vec{g}\|_X\leq 2r$. Thus,
\begin{align*}
    \|\Psi(\vec{f})_l-\Psi(\vec{g})_l\|_{L^\infty_t([t_0,T])H^{10}_x}
    \leq &2(T-t_0)C (\|\vec{f}-\vec{g}\|_{X}\|\vec{f}\|_{X}+\|\vec{g}\|_{X}\|\vec{f}-\vec{g}\|_{X})\\
    \leq &4r(T-t_0)C\|\vec{f}-\vec{g}\|_{X}
\end{align*}
and
\begin{align*}
    \|\Psi(f)-\Psi(g)\|_{L^\infty_t([t_0,T])Z}
    \leq &(T-t_0)(C'T\|\vec{f}-\vec{g}\|_{X}\|\vec{f}\|_{X}
    +C'T\|\vec{g}\|_{X}\|\vec{f}-\vec{g}\|_{X})\\
    \leq & 2rC'T(T-t_0)\|\vec{f}-\vec{g}\|_{X}.
\end{align*}
Since we required $\max\{C,C'\}(T-t_0)Tr < 1/12$, 
\begin{align*}
    \|\Psi(\vec{f})-\Psi(\vec{g})\|_{X}
    \leq & 4r(T-t_0)C\|\vec{f}-\vec{g}\|_{X}+2rC'T(T-t_0)\|\vec{f}-\vec{g}\|_{X}\\
    \leq & (4r(T-t_0)C+2rC'T(T-t_0))\|\vec{f}-\vec{g}\|_{X}
\end{align*}
with $4r(T-t_0)C+2rC'T(T-t_0)<1$. This shows that $\Psi:X\rightarrow X$ is a contraction map. Therefore, by the contraction mapping principle $\Psi$ has a unique fixed point $\vec{f}\in X$ such that
\begin{align*}
    f_l(t,x)= &e^{-ic_lt_0\la}u_{l0}+\int_{t_0}^te^{-ic_ls\la}\bigg(\sum_{\substack{c_m+c_n=0\\m\geq n}}A_{lmn}Q(e^{ic_ms\la}f_m,e^{ic_ns\la}f_n)+\\
    &\qquad\qquad\qquad\qquad +\sum_{\substack{c_m+c_n\neq 0\\m\geq n}}A_{lmn}(e^{ic_ms\la}f_m)(e^{ic_ns\la}f_n)\bigg)ds.
\end{align*}
Let $u_l(t,x)=e^{ic_lt\la}f_l(t,x)$ and observe that
\begin{align*}
    u_l(t,x)=e^{ic_lt\la}f_l(t,x)
    &= e^{ic_l(t-t_0)\la}u_{l0}+\int_{t_0}^te^{ic_l(t-s)\la}\bigg(\sum_{\substack{c_m+c_n=0\\m\geq n}}A_{lmn}Q(e^{ic_ms\la}f_m,e^{ic_ns\la}f_n)+\\
    &\qquad\qquad\qquad\qquad +\sum_{\substack{c_m+c_n\neq 0\\m\geq n}}A_{lmn}(e^{ic_ms\la}f_m)(e^{ic_ns\la}f_n)\bigg)ds,
\end{align*}
solves \eqref{lwppde} according to Duhamel's principle.
\end{proof}

\section{Proof of Proposition \ref{main}}
\label{prop}
In this section, we prove the main proposition stated in Section \ref{main proof}. This is the crucial result that shows the $Z$ norm of the profile $f_l(t,x)=e^{-ic_lt\Delta}u_l(t,x)$ remains bounded independently of the growth of the time variable, which leads to global well-posedness. The main goal of the proof is to deal with the additional powers of $t$ introduced by the $1+\alpha$ derivatives in the $Z$ norm.

Similarly to what we did in the proof of Proposition \ref{lwp} and Lemma \ref{timel2}, we will employ the integral equation and use the Littlewood-Paley operator to further decompose the frequency spaces of profiles $f_m,f_n$ in the nonlinear term.

Recall the definition of $\|\cdot\|_Z$ in \eqref{Znormdef}. We decompose the time interval $[1,T]$ into dyadic intervals and observe 
\begin{equation*}
\begin{aligned}
    \|f_l(t,x)\|_Z
    &\lesssim\|u_{l0}\|_Z+\big\|\int_{1}^{t} \partial_s f_l(s,x)ds\big\|_Z\\
    &\lesssim \e_0
    +\sum_{1\leq M\leq \log T}\sup_{2^{M-1}\leq t_1<t_2\leq 2^M}\|f_l(t_2,x)-f_l(t_1,x)\|_{Z}\\
    &\lesssim \e_0
    +\sup_{k\in\Z}\sum_{1\leq M\leq \log T}\sup_{2^{M-1}\leq t_1<t_2\leq 2^M}2^{-\gamma k+2k_+ + k/2+\alpha k}\|D^{1+\alpha}_\xi\hat{f}_{l,k}(t_2,\xi)-D^{1+\alpha}_\xi\hat{f}_{l,k}(t_1,\xi)\|_{L^2_\xi}.
\end{aligned}
\end{equation*}
To prove 
$$\|f_l(t,x)\|_Z\lesssim \e_0,$$
for all $t\in[1,T]$, it suffices to show for all $k\in\Z$,
\begin{equation*}
    \sum_{1\leq M\leq \log T}\sup_{2^{M-1}\leq t_1<t_2\leq 2^M}2^{-\gamma k+2k_++k/2+\alpha k}\|D^{1+\alpha}_\xi\hat{f}_{l,k}(t_2,\xi)-D^{1+\alpha}_\xi\hat{f}_{l,k}(t_1,\xi)\|_{L^2_\xi}\lesssim \e_1^2,
\end{equation*}
since $\e_1^2=\e_0^{10/6}\leq \e_0$.

For $2^{M-1}\leq t_1<t_2\leq 2^M$, let
\begin{equation*}
    G^{M,1}_{k,k_1,k_2}= 
    \int_{t_1}^{t_2}\int_{\R^3}e^{it\phi(\xi,\eta)}\hat{f}_{m,k_1}(t,\xi-\eta)\hat{f}_{n,k_2}(t,\eta)\psi_k(\xi)d\eta dt
\end{equation*}
and
\begin{equation*}
    G^{M,2}_{k,k_1,k_2}= 
    \int_{t_1}^{t_2}\int_{\R^3}e^{it\phi(\xi,\eta)}q(\xi-\eta,\eta)\hat{f}_{m,k_1}(t,\xi-\eta)\hat{f}_{n,k_2}(t,\eta)\psi_k(\xi)d\eta dt.
\end{equation*}
Then we have
\begin{equation}
    \begin{aligned}
        [\hat{f}_l(t_2,\xi)-\hat{f}_l(t_1,\xi)]\psi_k(\xi)
        =&\sum_{\substack{c_m+c_n\neq 0\\m\geq n}}A_{lmn}\sum_{(k_1,k_2)\in\chi^2_k\cup\chi^3_k} G^{M,1}_{k,k_1,k_2}
        +\sum_{c_m+c_n\neq 0}A_{lmn}\sum_{(k_1,k_2)\in\chi^1_k} G^{M,1}_{k,k_1,k_2}\\
        &+\sum_{\substack{c_m+c_n= 0\\m\geq n}}A_{lmn}\sum_{(k_1,k_2)\in\chi^2_k\cup\chi^3_k} G^{M,2}_{k,k_1,k_2}
        +\sum_{c_m+c_n= 0}A_{lmn}\sum_{(k_1,k_2)\in\chi^1_k} G^{M,2} _{k,k_1,k_2}.
    \end{aligned}
\end{equation}
Taking one derivative on $G^{M,1}_{k,k_1,k_2}$ and $G^{M,2}_{k,k_1,k_2}$ results in the terms $I^{M,i}_{k,k_1,k_2}$ and $J^{M,i}_{k,k_1,k_2}$, 
\begin{equation}
\label{i0}
\begin{aligned}
    I^{M,0}_{k,k_1,k_2}=&\int_{t_1}^{t_2}\int_{\R^3}e^{it\phi(\xi,\eta)}\hat{f}_{m,k_1}(t,\xi-\eta)\hat{f}_{n,k_2}(t,\eta)\partial_{\xi_l}\psi_k(\xi)d\eta dt,
\end{aligned}
\end{equation}
\begin{equation}
\label{i1}
    I^{M,1}_{k,k_1,k_2}=\int_{t_1}^{t_2}\int_{\R^3}e^{it\phi(\xi,\eta)}\partial_{\xi_l}\hat{f}_{m,k_1}(t,\xi-\eta)\hat{f}_{n,k_2}(t,\eta)\psi_k(\xi)d\eta dt,
\end{equation}
\begin{equation}
\label{i2}
    I^{M,2}_{k,k_1,k_2}=\int_{t_1}^{t_2}\int_{\R^3}e^{it\phi(\xi,\eta)}it\partial_{\xi_l}\phi(\xi,\eta) \hat{f}_{m,k_1}(t,\xi-\eta)\hat{f}_{n,k_2}(t,\eta)\psi_k(\xi)d\eta dt,
\end{equation}
\begin{equation}
\label{j0}
\begin{aligned}
    J^{M,0}_{k,k_1,k_2}
    =&\int_{t_1}^{t_2}\int_{\R^3}e^{it\phi(\xi,\eta)}q(\eta,\xi-\eta)\hat{f}_{m,k_1}(t,\xi-\eta)\hat{f}_{n,k_2}(t,\eta)\partial_{\xi_l}\psi_k(\xi)d\eta dt\\
    &+\int_{t_1}^{t_2}\int_{\R^3}e^{it\phi(\xi,\eta)}\partial_{\xi_l}q(\eta,\xi-\eta)\hat{f}_{m,k_1}(t,\xi-\eta)\hat{f}_{n,k_2}(t,\eta)\psi_k(\xi)d\eta dt,
\end{aligned}
\end{equation}
\begin{equation}
\label{j1}
\begin{aligned}
    J^{M,1}_{k,k_1,k_2}
    &=\int_{t_1}^{t_2}\int_{\R^3}e^{it\phi(\xi,\eta)}q(\eta,\xi-\eta)\partial_{\xi_l}\hat{f}_{m,k_1}(t,\xi-\eta)\hat{f}_{n,k_2}(t,\eta)\psi_k(\xi)d\eta dt,
\end{aligned}
\end{equation}
\begin{equation}
\label{j2}
\begin{aligned}
    J^{M,2}_{k,k_1,k_2}
    &=\int_{t_1}^{t_2}\int_{\R^3}e^{it\phi(\xi,\eta)}it\partial_{\xi_l}\phi(\xi,\eta) q(\eta,\xi-\eta)\hat{f}_{m,k_1}(t,\xi-\eta)\hat{f}_{n,k_2}(t,\eta)\psi_k(\xi)d\eta dt.
\end{aligned}
\end{equation} 
Since 
\begin{align*}
    &D^{1+\alpha}_\xi[\hat{f}_{l,k}(t_2,\xi)-\hat{f}_{l,k}(t_1,\xi)]\\
    =&D^{\alpha}_\xi\nabla_\xi[\hat{f}_{l,k}(t_2,\xi)-\hat{f}_{l,k}(t_1,\xi)]\\
    =&\sum_{c_m+c_n\neq 0}A_{lmn}\sum_{(k_1,k_2)\in\chi_k^1} \sum_{i=0,1,2}D^\alpha_\xi I^{M,i}_{k,k_1,k_2}+\sum_{c_m+c_n= 0}A_{lmn}\sum_{(k_1,k_2)\in\chi_k^1}\sum_{i=0,1,2} D^\alpha_\xi J^{M,i}_{k,k_1,k_2}\\
    &+\sum_{\substack{c_m+c_n\neq 0\\m\geq n}}A_{lmn}\sum_{(k_1,k_2)\in\chi_k^2\cup\chi^3_k}\sum_{i=0,1,2} D^\alpha_\xi I^{M,i}_{k,k_1,k_2}+\sum_{\substack{c_m+c_n= 0\\m\geq n}}A_{lmn}\sum_{(k_1,k_2)\in\chi_k^2\cup\chi^3_k}\sum_{i=0,1,2} D^\alpha_\xi J^{M,i}_{k,k_1,k_2},
\end{align*}
the problem is now reduced to estimating the $\alpha$ derivative of $I^{M,i}_{k,k_1,k_2}$ and $J^{M,i}_{k,k_1,k_2}$, where $i=0,1,2$.

We shall use the interpolation results developed in Section \ref{section for fractional derivatives} to tackle the estimation for the $\alpha$ derivative of the terms above. We will discuss the terms $D^{\alpha}I^{M,i}_{k,k_1,k_2}$ and $D^{\alpha}J^{M,i}_{k,k_1,k_2}$ for $i=0,1,2$ in Sections \ref{0}, \ref{1}, \ref{2} respectively.

\subsection{$D^{\alpha}I^{M,0}_{k,k_1,k_2}$ and $D^{\alpha}J^{M,0}_{k,k_1,k_2}$}
\label{0}
In this section, we will show
\begin{equation}
\label{im0wts}
    \begin{aligned}
        &\sum_{1\leq m\leq \log T}\sup_{2^{M-1}\leq t_1<t_2\leq 2^M} 2^{-\gamma k+2k_++k/2+\alpha k}\sum_{c_m+c_n\neq 0}A_{lmn}\sum_{(k_1,k_2)\in\chi^1_k\cup\chi^2_k\cup\chi^3_k}\|D^{\alpha}_\xi I^{M,0}_{k,k_1,k_2}\|_{L^2}+\\
        &+\sum_{1\leq m\leq \log T}\sup_{2^{M-1}\leq t_1<t_2\leq 2^M} 2^{-\gamma k+2k_++k/2+\alpha k}\sum_{c_m+c_n= 0}A_{lmn}\sum_{(k_1,k_2)\in\chi^1_k\cup\chi^2_k\cup\chi^3_k}\|D^{\alpha}_\xi J^{M,0}_{k,k_1,k_2}\|_{L^2}\lesssim\e_1^2
    \end{aligned}
\end{equation}
In order to gain decay in $t$, we want to use the following identity
$$e^{it\phi(\xi,\eta)}=\sum_{n=1}^3\frac{\partial_{\eta_n}\phi(\xi,\eta)}{it|\nabla_\eta\phi(\xi,\eta)|^2}\partial_{\eta_n}e^{it\phi(\xi,\eta)}$$
for integration by parts. Bounding the right hand side requires the norm of $\nabla_\eta\phi(\xi,\eta)=2c_m(\xi-\eta)-2c_n\eta$ to have a lower bound. However, we can only achieve a lower bound when $(k_1,k_2)\in\chi^1_k\cup\chi^2_k$ for a general pair of $c_m$ and $c_n$. When $c_m+c_n=0$, we observe $\nabla_\eta\phi(\xi,\eta)=2c_m\xi$, whose norm is bounded from below due to the localization. 

Thus, we will discuss the bounds in two different cases. The bound for $\|D^\alpha I^{M,0}_{k,k_1,k_2}\|_{L^2_\xi}$ when $(k_1,k_2)\in\chi^3_k$ will be dealt with in Lemma \ref{im0-2}, while the lemma below covers all the other terms.
\begin{lem}
\label{im0}
Given $t_1,t_2\in[2^{M-1},2^M]$ and $\sup_{t\in[1,T]}\|f_l(t,x)\|_{Z}\leq \e_1$, we have
\begin{align*}
    \|D^{\alpha}I^{M,0}_{k,k_1,k_2}\|_{L^2}
    \lesssim &(1+2^{-M-k-k_1})2^{-k-2k_{2,+}+\gamma k_2-2k_{1,+}+\gamma k_1-\alpha k_1-k_2/2}\min\{2^{-M/2}
    ,2^{M+3\min\{k,k_2\}/2+3k_1/2}\}\e_1^2,
\end{align*}
when $c_m+c_n\neq 0$ and $(k_1,k_2)\in\chi_k^1\cup\chi_k^2$, and
\begin{align*}
    \|D^{\alpha}J^{M,0}_{k,k_1,k_2}\|_{L^2}
    \lesssim & 2^{\e k_--\alpha k-k-2k_{2,+}+\gamma k_2-2k_{1,+}+\gamma k_1-k_2/2}\min\{2^{-M/2}
    ,2^{M+3\min\{k,k_2\}/2 +3k_1/2}\}\e_1^2,
\end{align*}
when $c_m+c_n = 0$ and $(k_1,k_2)\in\chi^1_k\cup\chi^2_k\cup\chi^3_k$.
\end{lem}
\begin{proof}
Consider the following families of operators on $\{z\in\C:0\leq \Re(z)\leq 1\}$,
\begin{align*}
    T_{z} \hat{g} =D^z\int_{t_1}^{t_2}\int_{\R^3}e^{it\phi(\xi,\eta)}\hat{g}_{k_1}(t,\xi-\eta)\hat{f}_{n,k_2}(t,\eta)\partial_{\xi_l}\psi_k(\xi)d\eta dt,
\end{align*}
\begin{align*}
    T^1_{z} \hat{g} =D^z\int_{t_1}^{t_2}\int_{\R^3}e^{it\phi(\xi,\eta)}q(\xi-\eta,\eta)\hat{g}_{k_1}(t,\xi-\eta)\hat{f}_{n,k_2}(t,\eta)\partial_{\xi_l}\psi_k(\xi)d\eta dt,
\end{align*}
and
\begin{align*}
    T^2_{z} \hat{g} =D^z\int_{t_1}^{t_2}\int_{\R^3}e^{it\phi(\xi,\eta)}\partial_{\xi_l}q(\xi-\eta,\eta)\hat{g}_{k_1}(t,\xi-\eta)\hat{f}_{n,k_2}(t,\eta)\psi_k(\xi)d\eta dt.
\end{align*}
Since we want to estimate the $\alpha$ derivative, we may employ the interpolation result in Lemma \ref{si}. The first operator $T_z$ is used to estimate $I_{k,k_1,k_2}^{M,0}$, while $T_z^1$ and $T^2_z$ are defined to handle $J^{M,0}_{k,k_1,k_2}$.
In this proof, we will show $T_{0+iy}:L^{\infty}_t([2^{M-1},2^M])L^2_\xi\rightarrow L^2_\xi$, $T_{1+iy}:L^{\infty}_t([2^{M-1},2^M])H^1_\xi\rightarrow L^2_\xi$, $T^1_{0+iy},T^2_{0+iy}:L^{\infty}_t([2^{M-1},2^M])H^1_\xi\rightarrow L^2_\xi$, $T^1_{1+iy},T^2_{1+iy}:L^{\infty}_t([2^{M-1},2^M])H^2_\xi\rightarrow L^2_\xi$ are bounded. Then performing interpolation will give the desired estimates on the $\alpha$ derivative in the lemma.

We start with obtaining bounds for the operators when $\Re(z) = 0$. By the bilinear estimate $L^4\times L^4\rightarrow L^2$ in Lemma \ref{bilinear}, Lemma \ref{dualitycomp}, \eqref{l4}, \eqref{l2}, and
\eqref{l2first}, we have
\begin{align*}
    &\|T_{0+iy} \hat{g}\|_{L^2}\\
    =&\bigg\|\int_{t_1}^{t_2}\int_{\R^3}e^{it\phi(\xi,\eta)}\hat{g}_{k_1}(t,\xi-\eta)\hat{f}_{n,k_2}(t,\eta)\partial_{\xi_l}\psi_k(\xi)d\eta dt\bigg\|_{L^2}\\
    \lesssim & \|\nabla_\xi\psi_k\|_{L^\infty}\bigg\|\F^{-1}\int_{t_1}^{t_2}\int_{\R^3}e^{it\phi(\xi,\eta)}\hat{g}_{k_1}(t,\xi-\eta)\hat{f}_{n,k_2}(t,\eta)\Tilde{\psi}_k(\xi)d\eta dt\bigg\|_{L^2}\\
    \lesssim & 2^{M-k}\sup_{t\in[2^{M-1},2^M]}\min\{\|e^{ic_mt\la}g_{k_1}\|_{L^4}\|e^{ic_nt\la}f_{n,k_2}\|_{L^4},2^{3\min\{k,k_2\}/2}\|e^{ic_mt\la}g_{k_1}\|_{L^2}\|e^{ic_nt\la}f_{n,k_2}\|_{L^2}\}\\
    \lesssim & 2^{M-k}\sup_{t\in[2^{M-1},2^M]}\min\{2^{-3M/2+k_1/4+k_2/4}\|\nabla_\xi\hat{f}_{n,k_2}\|_{L^2}\|\nabla_\xi\hat{g}_{k_1}\|_{L^2},2^{3\min\{k,k_2\}/2}\|g_{k_1}\|_{L^2}\|f_{n,k_2}\|_{L^2}\}\\
    \lesssim & 2^{-k-2k_{2,+}+\gamma k_2}\min\{2^{-M/2+k_1/4-k_2/4}\e_1\|\hat{g}_{k_1}\|_{L^\infty_t([2^{M-1},2^M])H^1_\xi}, 2^{M+3\min\{k,k_2\}/2+k_2/2}\e_1\|\hat{g}_{k_1}\|_{L^\infty_t([2^{M-1},2^M])L^2_\xi}\}.
\end{align*}
Using the same method as above with \eqref{q}, we obtain
\begin{align*}
    &\|T^1_{0+iy} \hat{g}\|_{L^2}+\|T^2_{0+iy} \hat{g}\|_{L^2}\\
    =&\bigg\|\int_{t_1}^{t_2}\int_{\R^3}e^{it\phi(\xi,\eta)}q(\xi-\eta,\eta)\hat{g}_{k_1}(t,\xi-\eta)\hat{f}_{n,k_2}(t,\eta)\partial_{\xi_l}\psi_k(\xi)d\eta dt\bigg\|_{L^2}\\
    &+\bigg\|\int_{t_1}^{t_2}\int_{\R^3}e^{it\phi(\xi,\eta)}\partial_{\xi_l}q(\xi-\eta,\eta)\hat{g}_{k_1}(t,\xi-\eta)\hat{f}_{n,k_2}(t,\eta)\psi_k(\xi)d\eta dt\bigg\|_{L^2}\\
    \lesssim & 2^{M}\big(\|\nabla_\xi\psi_k\|_{L^\infty_\xi}\|\F^{-1}q(\xi-\eta,\eta)\Tilde{\psi}_k(\xi)\Tilde{\psi}_{k_1}(\xi-\eta)\Tilde{\psi}_{k_2}(\eta)\|_{L^1}+\\
    &\qquad +\|\F^{-1}\nabla_\xi q(\xi-\eta,\eta)\Tilde{\psi}_k(\xi)\Tilde{\psi}_{k_1}(\xi-\eta)\Tilde{\psi}_{k_2}(\eta)\|_{L^1}\big)\times\\
    &\times \sup_{t\in[2^{M-1},2^M]}\min\{\|e^{ic_mt\la}g_{k_1}\|_{L^4_x}\|e^{ic_nt\la}f_{n,k_2}\|_{L^4_x},2^{3\min\{k,k_2\}/2}\|e^{ic_mt\la}g_{k_1}\|_{L^2}\|e^{ic_nt\la}f_{n,k_2}\|_{L^2}\}\\
    \lesssim & 2^{\e k_--k-2k_{2,+}+\gamma k_2}\min\{2^{-M/2+k_1/4-k_2/4}\e_1\|\hat{g}_{k_1}\|_{L^\infty_t([2^{M-1},2^M])H^1_\xi}, \\
    &\qquad\qquad\qquad\qquad 2^{M+3\min\{k,k_2\}/2+k_2/2}\e_1\|\hat{g}_{k_1}\|_{L^\infty_t([2^{M-1},2^M])L^2_\xi}\}.
\end{align*}
Next, we look at the operators when $\Re(z) = 1$,
\begin{align}
    \|T_{1+iy} \hat{g}\|_{L^2}
    \leq &\bigg\|\partial_{\xi_m}\int_{t_1}^{t_2}\int_{\R^3}e^{it\phi(\xi,\eta)}\hat{g}_{k_1}(t,\xi-\eta)\hat{f}_{n,k_2}(t,\eta)\partial_{\xi_l}\psi_k(\xi)d\eta dt\bigg\|_{L^2}\nonumber\\
    \leq &\bigg\|\int_{t_1}^{t_2}\int_{\R^3}e^{it\phi(\xi,\eta)}\partial_{\xi_m}\hat{g}_{k_1}(t,\xi-\eta)\hat{f}_{n,k_2}(t,\eta)\partial_{\xi_l}\psi_k(\xi)d\eta dt\bigg\|_{L^2}\label{im0.1}\\
    &+ \bigg\|\int_{t_1}^{t_2}\int_{\R^3}e^{it\phi(\xi,\eta)}\hat{g}_{k_1}(t,\xi-\eta)\hat{f}_{n,k_2}(t,\eta)\partial_{\xi_m}\partial_{\xi_l}\psi_k(\xi)d\eta dt\bigg\|_{L^2}\label{im0.12}\\
    & + \bigg\|\int_{t_1}^{t_2}\int_{\R^3}e^{it\phi(\xi,\eta)}it\partial_{\xi_m}\phi(\xi,\eta)\hat{g}_{k_1}(t,\xi-\eta)\hat{f}_{n,k_2}(t,\eta)\partial_{\xi_l}\psi_k(\xi)d\eta dt\bigg\|_{L^2}\label{im0.2},
\end{align}
\begin{align}
    \|T^1_{1+iy} \hat{g}\|_{L^2}
    \leq &\bigg\|\partial_{\xi_m}\int_{t_1}^{t_2}\int_{\R^3}e^{it\phi(\xi,\eta)}q(\xi-\eta,\eta)\hat{g}_{k_1}(t,\xi-\eta)\hat{f}_{n,k_2}(t,\eta)\partial_{\xi_l}\psi_k(\xi)d\eta dt\bigg\|_{L^2}\nonumber\\
    \leq &\bigg\|\int_{t_1}^{t_2}\int_{\R^3}e^{it\phi(\xi,\eta)}q(\xi-\eta,\eta)\partial_{\xi_m}\hat{g}_{k_1}(t,\xi-\eta)\hat{f}_{n,k_2}(t,\eta)\partial_{\xi_l}\psi_k(\xi)d\eta dt\bigg\|_{L^2}\label{km0.1}\\
    &+ \bigg\|\int_{t_1}^{t_2}\int_{\R^3}e^{it\phi(\xi,\eta)}\partial_{\xi_m} q(\xi-\eta,\eta)\hat{g}_{k_1}(t,\xi-\eta)\hat{f}_{n,k_2}(t,\eta)\partial_{\xi_l}\psi_k(\xi)d\eta dt\bigg\|_{L^2}\label{km0.11}\\
    &+ \bigg\|\int_{t_1}^{t_2}\int_{\R^3}e^{it\phi(\xi,\eta)}q(\xi-\eta,\eta)\hat{g}_{k_1}(t,\xi-\eta)\hat{f}_{n,k_2}(t,\eta)\partial_{\xi_m}\partial_{\xi_l}\psi_k(\xi)d\eta dt\bigg\|_{L^2}\label{km0.12}\\
    & + \bigg\|\int_{t_1}^{t_2}\int_{\R^3}e^{it\phi(\xi,\eta)}it\partial_{\xi_m}\phi(\xi,\eta) q(\xi-\eta,\eta)\hat{g}_{k_1}(t,\xi-\eta)\hat{f}_{n,k_2}(t,\eta)\partial_{\xi_l}\psi_k(\xi)d\eta dt\bigg\|_{L^2}\label{km0.2},
\end{align}
and
\begin{align}
    \|T^2_{1+iy} \hat{g}\|_{L^2}
    \leq &\bigg\|\partial_{\xi_m}\int_{t_1}^{t_2}\int_{\R^3}e^{it\phi(\xi,\eta)}\partial_{\xi_l}q(\xi-\eta,\eta)\hat{g}_{k_1}(t,\xi-\eta)\hat{f}_{n,k_2}(t,\eta)\psi_k(\xi)d\eta dt\bigg\|_{L^2}\nonumber\\
    \leq &\bigg\|\int_{t_1}^{t_2}\int_{\R^3}e^{it\phi(\xi,\eta)}\partial_{\xi_l}q(\xi-\eta,\eta)\partial_{\xi_m}\hat{g}_{k_1}(t,\xi-\eta)\hat{f}_{n,k_2}(t,\eta)\psi_k(\xi)d\eta dt\bigg\|_{L^2}\label{km0.1-2}\\
    &+ \bigg\|\int_{t_1}^{t_2}\int_{\R^3}e^{it\phi(\xi,\eta)}\partial_{\xi_m} \partial_{\xi_l}q(\xi-\eta,\eta)\hat{g}_{k_1}(t,\xi-\eta)\hat{f}_{n,k_2}(t,\eta)\psi_k(\xi)d\eta dt\bigg\|_{L^2}\label{km0.11-2}\\
    &+ \bigg\|\int_{t_1}^{t_2}\int_{\R^3}e^{it\phi(\xi,\eta)}\partial_{\xi_l}q(\xi-\eta,\eta)\hat{g}_{k_1}(t,\xi-\eta)\hat{f}_{n,k_2}(t,\eta)\partial_{\xi_m}\psi_k(\xi)d\eta dt\bigg\|_{L^2}\label{km0.12-2}\\
    & + \bigg\|\int_{t_1}^{t_2}\int_{\R^3}e^{it\phi(\xi,\eta)}it\partial_{\xi_m}\phi(\xi,\eta) \partial_{\xi_l}q(\xi-\eta,\eta)\hat{g}_{k_1}(t,\xi-\eta)\hat{f}_{n,k_2}(t,\eta)\psi_k(\xi)d\eta dt\bigg\|_{L^2}\label{km0.2-2}.
\end{align}
Since $(k_1,k_2)\in\chi^1_k\cup\chi^2_k$, we have \eqref{xi-2eta} for $c_m+c_n\neq 0$ and are able to perform an integration by parts on $\eqref{im0.12}$ and $\eqref{im0.2}$ using the identity $e^{it\phi(\xi,\eta)}=\sum_{n=1}^3\frac{\partial_{\eta_n}\phi(\xi,\eta)}{it|\nabla_\eta\phi(\xi,\eta)|^2}\partial_{\eta_n}e^{it\phi(\xi,\eta)}$,
\begin{align}
    &\int_{t_1}^{t_2}\int_{\R^3}e^{it\phi(\xi,\eta)} \hat{g}_{k_1}(t,\xi-\eta)\hat{f}_{n,k_2}(t,\eta)\partial_{\xi_m}\partial_{\xi_l}\psi_k(\xi)d\eta dt\nonumber\\
    = & \int_{t_1}^{t_2}\int_{\R^3}\partial_{\eta_n}e^{it\phi(\xi,\eta)}\frac{\partial_{\eta_n}\phi(\xi,\eta)}{it|\nabla_\eta\phi(\xi,\eta)|^2}\hat{g}_{k_1}(t,\xi-\eta)\hat{f}_{n,k_2}(t,\eta)\partial_{\xi_m}\partial_{\xi_l}\psi_k(\xi)d\eta dt\nonumber\\
    = & -\int_{t_1}^{t_2}\int_{\R^3} e^{it\phi(\xi,\eta)}\partial_{\eta_n}\frac{\partial_{\eta_n}\phi(\xi,\eta)}{it|\nabla_\eta\phi(\xi,\eta)|^2}\hat{g}_{k_1}(t,\xi-\eta)\hat{f}_{n,k_2}(t,\eta)\partial_{\xi_m}\partial_{\xi_l}\psi_k(\xi)d\eta dt\label{im0.3-2}\\
    & -\int_{t_1}^{t_2}\int_{\R^3} e^{it\phi(\xi,\eta)}\frac{\partial_{\eta_n}\phi(\xi,\eta)}{it|\nabla_\eta\phi(\xi,\eta)|^2}\partial_{\eta_n}\hat{g}_{k_1}(t,\xi-\eta)\hat{f}_{n,k_2}(t,\eta)\partial_{\xi_m}\partial_{\xi_l}\psi_k(\xi)d\eta dt\label{im0.4-2}\\
    & -\int_{t_1}^{t_2}\int_{\R^3} e^{it\phi(\xi,\eta)}\frac{\partial_{\eta_n}\phi(\xi,\eta)}{it|\nabla_\eta\phi(\xi,\eta)|^2}\hat{g}_{k_1}(t,\xi-\eta)\partial_{\eta_n}\hat{f}_{n,k_2}(t,\eta)\partial_{\xi_m}\partial_{\xi_l}\psi_k(\xi)d\eta dt\label{im0.5-2}
\end{align}
and
\begin{align}
    &\int_{t_1}^{t_2}\int_{\R^3}e^{it\phi(\xi,\eta)}it\partial_{\xi_m}\phi(\xi,\eta) \hat{g}_{k_1}(t,\xi-\eta)\hat{f}_{n,k_2}(t,\eta)\partial_{\xi_l}\psi_k(\xi)d\eta dt\nonumber\\
    = & \int_{t_1}^{t_2}\int_{\R^3}\partial_{\eta_n}e^{it\phi(\xi,\eta)}\frac{\partial_{\xi_m}\phi(\xi,\eta)\partial_{\eta_n}\phi(\xi,\eta)}{|\nabla_\eta\phi(\xi,\eta)|^2}\hat{g}_{k_1}(t,\xi-\eta)\hat{f}_{n,k_2}(t,\eta)\partial_{\xi_l}\psi_k(\xi)d\eta dt\nonumber\\
    = & -\int_{t_1}^{t_2}\int_{\R^3} e^{it\phi(\xi,\eta)}\partial_{\eta_n}\frac{\partial_{\xi_m}\phi(\xi,\eta)\partial_{\eta_n}\phi(\xi,\eta)}{|\nabla_\eta\phi(\xi,\eta)|^2}\hat{g}_{k_1}(t,\xi-\eta)\hat{f}_{n,k_2}(t,\eta)\partial_{\xi_l}\psi_k(\xi)d\eta dt\label{im0.3}\\
    & -\int_{t_1}^{t_2}\int_{\R^3} e^{it\phi(\xi,\eta)}\frac{\partial_{\xi_m}\phi(\xi,\eta)\partial_{\eta_n}\phi(\xi,\eta)}{|\nabla_\eta\phi(\xi,\eta)|^2}\partial_{\eta_n}\hat{g}_{k_1}(t,\xi-\eta)\hat{f}_{n,k_2}(t,\eta)\partial_{\xi_l}\psi_k(\xi)d\eta dt\label{im0.4}\\
    & -\int_{t_1}^{t_2}\int_{\R^3} e^{it\phi(\xi,\eta)}\frac{\partial_{\xi_m}\phi(\xi,\eta)\partial_{\eta_n}\phi(\xi,\eta)}{|\nabla_\eta\phi(\xi,\eta)|^2}\hat{g}_{k_1}(t,\xi-\eta)\partial_{\eta_n}\hat{f}_{n,k_2}(t,\eta)\partial_{\xi_l}\psi_k(\xi)d\eta dt\label{im0.5}.
\end{align}
For $c_m+c_n = 0$, recall that $\phi(\xi,\eta)=(c_l+c_n)|\xi|^2-2c_n\xi\cdot\eta$. We thus integrate by parts using the identity $e^{it\phi(\xi,\eta)}=-\sum_{j=1}^3\frac{\xi_j}{it2c_n|\xi|^2}\partial_{\eta_j}e^{it\phi(\xi,\eta)}$ on the terms $\eqref{km0.2}$ and $\eqref{km0.2-2}$,
\begin{align}
    &\int_{t_1}^{t_2}\int_{\R^3}e^{it\phi(\xi,\eta)}it\partial_{\xi_m}\phi(\xi,\eta) q(\xi-\eta,\eta)\hat{g}_{k_1}(t,\xi-\eta)\hat{f}_{n,k_2}(t,\eta)\partial_{\xi_l}\psi_k(\xi)d\eta dt\nonumber\\
    = & \int_{t_1}^{t_2}\int_{\R^3}\partial_{\eta_j}e^{it\phi(\xi,\eta)}\frac{\xi_j\partial_{\xi_m}\phi(\xi,\eta)}{2c_n|\xi|^2}q(\xi-\eta,\eta)\hat{g}_{k_1}(t,\xi-\eta)\hat{f}_{n,k_2}(t,\eta)\partial_{\xi_l}\psi_k(\xi)d\eta dt\nonumber\\
    = & -\int_{t_1}^{t_2}\int_{\R^3} e^{it\phi(\xi,\eta)}\partial_{\eta_j}\big(\frac{\xi_j\partial_{\xi_m}\phi(\xi,\eta)}{2c_n|\xi|^2}q(\xi-\eta,\eta)\big)\hat{g}_{k_1}(t,\xi-\eta)\hat{f}_{n,k_2}(t,\eta)\partial_{\xi_l}\psi_k(\xi)d\eta dt\label{km0.3}\\
    & -\int_{t_1}^{t_2}\int_{\R^3} e^{it\phi(\xi,\eta)}\frac{\xi_j\partial_{\xi_m}\phi(\xi,\eta)}{2c_n|\xi|^2}q(\xi-\eta,\eta)\partial_{\eta_j}\hat{g}_{k_1}(t,\xi-\eta)\hat{f}_{n,k_2}(t,\eta)\partial_{\xi_l}\psi_k(\xi)d\eta dt\label{km0.4}\\
    & -\int_{t_1}^{t_2}\int_{\R^3} e^{it\phi(\xi,\eta)}\frac{\xi_j\partial_{\xi_m}\phi(\xi,\eta)}{2c_n|\xi|^2}q(\xi-\eta,\eta)\hat{g}_{k_1}(t,\xi-\eta)\partial_{\eta_j}\hat{f}_{n,k_2}(t,\eta)\partial_{\xi_l}\psi_k(\xi)d\eta dt\label{km0.5}
\end{align}
and
\begin{align}       
    &\int_{t_1}^{t_2}\int_{\R^3}e^{it\phi(\xi,\eta)}it\partial_{\xi_m}\phi(\xi,\eta) \partial_{\xi_l}q(\xi-\eta,\eta)\hat{g}_{k_1}(t,\xi-\eta)\hat{f}_{n,k_2}(t,\eta)\psi_k(\xi)d\eta dt\nonumber\\
    = & \int_{t_1}^{t_2}\int_{\R^3}\partial_{\eta_j}e^{it\phi(\xi,\eta)}\frac{\xi_j\partial_{\xi_m}\phi(\xi,\eta)}{2c_n|\xi|^2}\partial_{\xi_l}q(\xi-\eta,\eta)\hat{g}_{k_1}(t,\xi-\eta)\hat{f}_{n,k_2}(t,\eta)\psi_k(\xi)d\eta dt\nonumber\\
    = & -\int_{t_1}^{t_2}\int_{\R^3} e^{it\phi(\xi,\eta)}\partial_{\eta_j}\big(\frac{\xi_j\partial_{\xi_m}\phi(\xi,\eta)}{2c_n|\xi|^2}\partial_{\xi_l}q(\xi-\eta,\eta)\big)\hat{g}_{k_1}(t,\xi-\eta)\hat{f}_{n,k_2}(t,\eta)\psi_k(\xi)d\eta dt\label{km0.3-2}\\
    & -\int_{t_1}^{t_2}\int_{\R^3} e^{it\phi(\xi,\eta)}\frac{\xi_j\partial_{\xi_m}\phi(\xi,\eta)}{2c_n|\xi|^2}\partial_{\xi_l}q(\xi-\eta,\eta)\partial_{\eta_j}\hat{g}_{k_1}(t,\xi-\eta)\hat{f}_{n,k_2}(t,\eta)\psi_k(\xi)d\eta dt\label{km0.4-2}\\
    & -\int_{t_1}^{t_2}\int_{\R^3} e^{it\phi(\xi,\eta)}\frac{\xi_j\partial_{\xi_m}\phi(\xi,\eta)}{2c_n|\xi|^2}\partial_{\xi_l}q(\xi-\eta,\eta)\hat{g}_{k_1}(t,\xi-\eta)\partial_{\eta_j}\hat{f}_{n,k_2}(t,\eta)\psi_k(\xi)d\eta dt\label{km0.5-2},
\end{align}
By the bilinear estimate $L^4\times L^4\rightarrow L^2$, Lemma \ref{dualitycomp}, Lemma \ref{chi+eta2}, and \eqref{l2first}
\begin{align*}
    &\eqref{im0.1}+\|\eqref{im0.4-2}\|_{L^2}+\|\eqref{im0.4}\|_{L^2}\\
    \lesssim & \big[2^{M}\|\nabla_\xi\psi_k\|_{L^\infty}(1+\|\F^{-1}\frac{\partial_{\xi_m}\phi(\xi,\eta)\partial_{\eta_n}\phi(\xi,\eta)}{|\nabla_\eta\phi(\xi,\eta)|^2}\Tilde{\psi}_{k_1}(\nabla_\eta\phi(\xi,\eta))\Tilde{\psi}_{k_2}(\eta)\|_{L^1})+\\
    &\qquad +\|\nabla^2_\xi\psi_k\|_{L^\infty}\|\F^{-1}\frac{\partial_{\eta_n}\phi(\xi,\eta)}{|\nabla_\eta\phi(\xi,\eta)|^2}\Tilde{\psi}_{k_1}(\nabla_\eta\phi(\xi,\eta))\Tilde{\psi}_{k_2}(\eta)\|_{L^1}\big]\times\\
    &\times \sup_{t\in[2^{M-1},2^M]}\min\{\|e^{ic_mt\la}\F^{-1}\nabla_\xi\hat{g}_{k_1}\|_{L^4}\|e^{ic_nt\la}f_{n,k_2}\|_{L^4},\\
    &\qquad\qquad\qquad\qquad 2^{3\min\{k,k_2\}/2}\|e^{ic_mt\la}\F^{-1}\nabla_\xi\hat{g}_{k_1}\|_{L^2}\|e^{ic_nt\la}f_{n,k_2}\|_{L^2}\}\\
    \lesssim & [2^{M-k}+2^{-2k-k_1}]\sup_{t\in[2^{M-1},2^M]}\min\{\|e^{ic_mt\la}\F^{-1}\nabla_\xi\hat{g}_{k_1}\|_{L^4}\|e^{ic_nt\la}f_{n,k_2}\|_{L^4},\\
    &\qquad\qquad\qquad\qquad 2^{3\min\{k,k_2\}/2}\|\nabla_\xi\hat{g}_{k_1}\|_{L^2}\|f_{n,k_2}\|_{L^2}\}\\
    \lesssim & 2^{M-k}(1+2^{-M-k-k_1})\sup_{t\in[2^{M-1},2^M]}\min\{2^{-3M/2+k_1/4+k_2/4}\|\nabla^2_\xi\hat{g}_{k_1}\|_{L^2}\|\nabla_\xi\hat{f}_{n,k_2}\|_{L^2},\\
    &\qquad\qquad\qquad\qquad 2^{3\min\{k,k_2\}/2}\|\nabla_\xi\hat{g}_{k_1}\|_{L^2}\|f_{n,k_2}\|_{L^2}\}\\
    \lesssim & (1+2^{-M-k-k_1})2^{-k-2k_{2,+}+\gamma k_2}\min\{2^{-M/2+k_1/4-k_2/4}\e_1\|\hat{g}_{k_1}\|_{L^\infty_t([2^{M-1},2^M])H^2_\xi},\\
    &\qquad\qquad\qquad\qquad 2^{M+3\min\{k,k_2\}/2+k_2/2}\e_1\|\hat{g}_{k_1}\|_{L^\infty_t([2^{M-1},2^M])H^1_\xi}\},
\end{align*}
and using the condition \eqref{q} on the multiplier $q$,
\begin{align*}
    &\eqref{km0.1}+\eqref{km0.1-2}+\|\eqref{km0.4}\|_{L^2}+\|\eqref{km0.4-2}\|_{L^2}\\
    \lesssim & 2^{M}\big(\|\nabla_\xi\psi_k\|_{L^\infty}\|\F^{-1}q(\xi-\eta,\eta)\Tilde{\psi}_k(\xi)\Tilde{\psi}_{k_1}(\xi-\eta)\Tilde{\psi}_{k_2}(\eta)\|_{L^1}
    +\|\F^{-1}\nabla_\xi q(\xi-\eta,\eta)\Tilde{\psi}_k(\xi)\Tilde{\psi}_{k_1}(\xi-\eta)\Tilde{\psi}_{k_2}(\eta)\|_{L^1}+\\
    & \qquad + \|\nabla_\xi\psi_k\|_{L^\infty}\|\F^{-1}\frac{\xi_j\partial_{\xi_m}\phi(\xi,\eta)}{2c_n|\xi|^2}q(\xi-\eta,\eta)\Tilde{\psi}_k(\xi)\Tilde{\psi}_{k_1}(\xi-\eta)\Tilde{\psi}_{k_2}(\eta)\|_{L^1}+\\
    &\qquad + \|\F^{-1}\frac{\xi_j\partial_{\xi_m}\phi(\xi,\eta)}{2c_n|\xi|^2}\nabla_\xi q(\xi-\eta,\eta)\Tilde{\psi}_k(\xi)\Tilde{\psi}_{k_1}(\xi-\eta)\Tilde{\psi}_{k_2}(\eta)\|_{L^1}\big)\times\\
    &\times\sup_{t\in[2^{M-1},2^M]}\min\{
    \|e^{ic_mt\la}\F^{-1}\nabla_\xi\hat{g}_{k_1}\|_{L^4_x}\|e^{ic_nt\la}f_{n,k_2}\|_{L^4_x},\\
    &\qquad\qquad\qquad\qquad\qquad 2^{3\min\{k,k_2\}/2}\|e^{ic_mt\la}\F^{-1}\nabla_\xi\hat{g}_{k_1}\|_{L^2_x}\|e^{ic_nt\la}f_{n,k_2}\|_{L^2_x}\}\\
    \lesssim & 2^{M+\e k_-+k_1-2k}\sup_{t\in[2^{M-1},2^M]}\min\{2^{-3M/2+k_1/4+k_2/4}\|\nabla^2_\xi\hat{g}_{k_1}\|_{L^2}\|\nabla_\xi\hat{f}_{n,k_2}\|_{L^2},\\
    &\qquad\qquad\qquad\qquad 2^{3\min\{k,k_2\}/2}\|\nabla_\xi\hat{g}_{k_1}\|_{L^2}\|f_{n,k_2}\|_{L^2}\}\\
    \leq & 2^{\e k_-+k_1-2k-2k_{2,+}+\gamma k_2}\min\{2^{-M/2+k_1/4-k_2/4}\e_1\|\hat{g}_{k_1}\|_{L^\infty_t([2^{M-1},2^M])H^2_\xi}, \\
    &\qquad\qquad\qquad\qquad 2^{M+3\min\{k,k_2\}/2+k_2/2}\e_1\|\hat{g}_{k_1}\|_{L^\infty_t([2^{M-1},2^M])H^1_\xi}\}.
\end{align*}
Then the bilinear estimate $L^2\times L^\infty\rightarrow L^2$, Lemma \ref{op}, Lemma \ref{dualitycomp}, \eqref{l2}, \eqref{addder} and \eqref{lpnormseq1} imply
\begin{align*}
    \|\eqref{im0.3-2}\|_{L^2}+\|\eqref{im0.3}\|_{L^2}
    \lesssim &\big(2^{M}\|\nabla_\xi\psi_k\|_{L^\infty}\|\F^{-1}\partial_{\eta_n}\frac{\partial_{\xi_m}\phi(\xi,\eta)\partial_{\eta_n}\phi(\xi,\eta)}{|\nabla_\eta\phi(\xi,\eta)|^2}\Tilde{\psi}_{k_1}(\nabla_\eta\phi(\xi,\eta))\Tilde{\psi}_{k_2}(\eta)\|_{L^1}+\\
    &+\|\nabla^2_\xi\psi_k\|_{L^\infty}\|\F^{-1}\partial_{\eta_n}\frac{\partial_{\eta_n}\phi(\xi,\eta)}{|\nabla_\eta\phi(\xi,\eta)|^2}\Tilde{\psi}_{k_1}(\nabla_\eta\phi(\xi,\eta))\Tilde{\psi}_{k_2}(\eta)\|_{L^1}\big)\times\\
    &\times \sup_{t\in[2^{M-1},2^M]}\min\{\|e^{ic_mt\la}g_{k_1}\|_{L^\infty}\|e^{ic_nt\la}f_{n,k_2}\|_{L^2},\\
    &\qquad\qquad\qquad\qquad 2^{3\min\{k,k_2\}/2}\|e^{ic_mt\la}g_{k_1}\|_{L^2}\|e^{ic_nt\la}f_{n,k_2}\|_{L^2}\}\\
    \lesssim & (2^{M-k-k_1}+2^{-2k-2k_1}) \sup_{t\in[2^{M-1},2^M]}\min\{2^{-3M/2}\|g_{k_1}\|_{L^1}\|f_{n,k_2}\|_{L^2},\\
    &\qquad\qquad\qquad\qquad 2^{3\min\{k,k_2\}/2}\|g_{k_1}\|_{L^2}\|f_{n,k_2}\|_{L^2}\}\\
    \lesssim & 2^{M-k-k_1}(1+2^{-M-k-k_1}) \sup_{t\in[2^{M-1},2^M]}\min\{2^{-3M/2+k_1/2}\|\nabla^2_\xi\hat{g}_{k_1}\|_{L^2}\|f_{k_2}\|_{L^2},\\
    &\qquad\qquad\qquad\qquad 2^{3\min\{k,k_2\}/2+k_1}\|\nabla_\xi\hat{g}_{k_1}\|_{L^2}\|f_{k_2}\|_{L^2}\}\\
    \lesssim & (1+2^{-M-k-k_1})2^{-k-2k_{2,+}+\gamma k_2} \min\{2^{-M/2-k_1/2+k_2/2}\e_1\|\hat{g}_{k_1}\|_{L^\infty_t([2^{M-1},2^M])H^2_\xi}, \\
    &\qquad\qquad\qquad\qquad 2^{M+3\min\{k,k_2\}/2+k_2/2}\e_1\|\hat{g}_{k_1}\|_{L^\infty_t([2^{M-1},2^M])H^1_\xi}\}
\end{align*}
and
\begin{align*}
    \|\eqref{im0.5}\|_{L^2}+\|\eqref{im0.5-2}\|_{L^2}
    \lesssim &\big(2^{M}\|\nabla_\xi\psi_k\|_{L^\infty}\|\F^{-1}\frac{\partial_{\xi_m}\phi(\xi,\eta)\partial_{\eta_n}\phi(\xi,\eta)}{|\nabla_\eta\phi(\xi,\eta)|^2} \Tilde{\psi}_{k_1}(\nabla_\eta\phi(\xi,\eta))\Tilde{\psi}_{k_2}(\eta)\|_{L^1}+\\
    &+\|\nabla^2_\xi\psi_k\|_{L^\infty}\|\F^{-1}\frac{\partial_{\eta_n}\phi(\xi,\eta)}{|\nabla_\eta\phi(\xi,\eta)|^2} \Tilde{\psi}_{k_1}(\nabla_\eta\phi(\xi,\eta))\Tilde{\psi}_{k_2}(\eta)\|_{L^1}\big)\times\\
    &\times\sup_{t\in[2^{M-1},2^M]}\min\{\|e^{ic_mt\la}g_{k_1}\|_{L^\infty}\|e^{ic_nt\la}\F^{-1}\nabla_\xi\hat{f}_{n,k_2}\|_{L^2},\\
    &\qquad\qquad\qquad\qquad 2^{3\min\{k,k_2\}/2}\|e^{ic_mt\la}g_{k_1}\|_{L^2}\|e^{ic_nt\la}\F^{-1}\nabla_\xi\hat{f}_{n,k_2}\|_{L^2}\}\\
    \lesssim &(2^{M-k}+2^{-2k-k_1})\sup_{t\in[2^{M-1},2^M]}\min\{\|e^{ic_mt\la}g_{k_1}\|_{L^\infty}\|\nabla_\xi\hat{f}_{n,k_2}\|_{L^2},\\
    &\qquad\qquad\qquad\qquad 2^{3\min\{k,k_2\}/2}\|g_{k_1}\|_{L^2}\|\nabla_\xi\hat{f}_{n,k_2}\|_{L^2}\}\\
    \lesssim &(1+2^{-M-k-k_1})2^{M-k}\sup_{t\in[2^{M-1},2^M]}\min\{2^{-3M/2+k_1/2}\|\nabla^2_\xi\hat{g}_{k_1}\|_{L^2}\|\nabla_\xi\hat{f}_{n,k_2}\|_{L^2}, \\
    &\qquad\qquad\qquad\qquad 2^{3\min\{k,k_2\}/2+k_1}\|\nabla_\xi\hat{g}_{k_1}\|_{L^2}\|\nabla_\xi\hat{f}_{n,k_2}\|_{L^2}\}\\
    \lesssim &(1+2^{-M-k-k_1})2^{M-k-2k_{2,+}+\gamma k_2}\sup_{t\in[2^{M-1},2^M]}\min\{2^{-3M/2+k_1/2-k_2/2}\e_1\|\nabla^2_\xi\hat{g}_{k_1}\|_{L^2}, \\
    &\qquad\qquad\qquad\qquad 2^{3\min\{k,k_2\}/2+k_1-k_2/2}\e_1\|\nabla_\xi\hat{g}_{k_1}\|_{L^2}\}\\
    \lesssim &(1+2^{-M-k-k_1})2^{-k-2k_{2,+}+\gamma k_2}\min\{2^{-M/2+k_1/2-k_2/2}\e_1\|\hat{g}_{k_1}\|_{L^\infty_t([2^{M-1},2^M])H^2_\xi}, \\
    &\qquad\qquad\qquad\qquad 2^{M+3\min\{k,k_2\}/2+k_1-k_2/2}\e_1\|\hat{g}_{k_1}\|_{L^\infty_t([2^{M-1},2^M])H^1_\xi}\}.
\end{align*}
Repeating the same estimates for the $c_m+c_n = 0$ case, we have 
\begin{align*}
    &\eqref{km0.11}+\eqref{km0.11-2}+\eqref{km0.12}+\eqref{km0.12-2}+\|\eqref{km0.3}\|_{L^2}+\|\eqref{km0.3-2}\|_{L^2}\\
    \lesssim &2^{M}\big(\|\nabla_\xi\psi_k\|_{L^\infty}\|\F^{-1}\nabla_\xi q(\xi-\eta,\eta)\Tilde{\psi}_k(\xi)\Tilde{\psi}_{k_1}(\xi-\eta)\Tilde{\psi}_{k_2}(\eta)\|_{L^1}\\
    & + \|\F^{-1}\nabla^2_\xi q(\xi-\eta,\eta)\Tilde{\psi}_k(\xi)\Tilde{\psi}_{k_1}(\xi-\eta)\Tilde{\psi}_{k_2}(\eta)\|_{L^1}\\
    & + \|\nabla^2_\xi\psi_k\|_{L^\infty}\|\F^{-1} q(\xi-\eta,\eta)\Tilde{\psi}_k(\xi)\Tilde{\psi}_{k_1}(\xi-\eta)\Tilde{\psi}_{k_2}(\eta)\|_{L^1}\\
    & + \|\nabla_\xi\psi_k\|_{L^\infty}\|\F^{-1} \frac{\xi_j\partial_{\eta_j}\partial_{\xi_m}\phi(\xi,\eta)}{2c_n|\xi|^2}q(\xi-\eta,\eta)\Tilde{\psi}_k(\xi)\Tilde{\psi}_{k_1}(\xi-\eta)\Tilde{\psi}_{k_2}(\eta)\|_{L^1}\\
    & + \|\nabla_\xi\psi_k\|_{L^\infty}\|\F^{-1} \frac{\xi_j\partial_{\xi_m}\phi(\xi,\eta)}{2c_n|\xi|^2}\nabla_\eta q(\xi-\eta,\eta)\Tilde{\psi}_k(\xi)\Tilde{\psi}_{k_1}(\xi-\eta)\Tilde{\psi}_{k_2}(\eta)\|_{L^1}\\
    & + \|\F^{-1} \frac{\xi_j\partial_{\eta_j}\partial_{\xi_m}\phi(\xi,\eta)}{2c_n|\xi|^2}\nabla_\xi q(\xi-\eta,\eta)\Tilde{\psi}_k(\xi)\Tilde{\psi}_{k_1}(\xi-\eta)\Tilde{\psi}_{k_2}(\eta)\|_{L^1}\\
    & + \|\F^{-1} \frac{\xi_j\partial_{\xi_m}\phi(\xi,\eta)}{2c_n|\xi|^2}\nabla_\eta\nabla_\xi q(\xi-\eta,\eta)\Tilde{\psi}_k(\xi)\Tilde{\psi}_{k_1}(\xi-\eta)\Tilde{\psi}_{k_2}(\eta)\|_{L^1}\big)\times\\
    &\times\sup_{t\in[2^{M-1},2^M]}\min\{\|e^{ic_mt\la}g_{k_1}\|_{L^\infty_x}\|\hat{f}_{n,k_2}\|_{L^2_\xi},2^{3\min\{k,k_2\}/2}\|g_{k_1}\|_{L^2_x}\|\hat{f}_{n,k_2}\|_{L^2_\xi}\}\\
    \lesssim & (2^{-2k}+2^{-2k+k_1-k_2})2^{\e k_--2k_{2,+}+\gamma k_2}\min\{2^{-M/2+k_1/2+k_2/2}\e_1\|\hat{g}_{k_1}\|_{L^\infty_t([2^{M-1},2^M])H^2_\xi}, \\
    &\qquad\qquad\qquad\qquad 2^{M+3\min\{k,k_2\}/2+k_2/2+k_1}\e_1\|\hat{g}_{k_1}\|_{L^\infty_t([2^{M-1},2^M])H^1_\xi}\}
\end{align*}
and
\begin{align*}
    &\|\eqref{km0.5}\|_{L^2}+\|\eqref{km0.5-2}\|_{L^2}\\
    \lesssim &2^{M}\big(\|\nabla_\xi\psi_k\|_{L^\infty}\|\F^{-1}\frac{\xi_j\partial_{\xi_m}\phi(\xi,\eta)}{2c_n|\xi|^2} q(\xi-\eta,\eta)\Tilde{\psi}_k(\xi)\Tilde{\psi}_{k_1}(\xi-\eta)\Tilde{\psi}_{k_2}(\eta)\|_{L^1}\\
    &+\|\F^{-1}\frac{\xi_j\partial_{\xi_m}\phi(\xi,\eta)}{2c_n|\xi|^2} \nabla_\xi q(\xi-\eta,\eta)\Tilde{\psi}_k(\xi)\Tilde{\psi}_{k_1}(\xi-\eta)\Tilde{\psi}_{k_2}(\eta)\|_{L^1}\big)\times\\
    &\times \sup_{t\in[2^{M-1},2^M]}\min\{\|e^{ic_mt\la}g_{k_1}\|_{L^\infty_x}\|\nabla_\xi\hat{f}_{n,k_2}\|_{L^2_\xi},2^{3\min\{k,k_2\}/2}\|g_{k_1}\|_{L^2_x}\|\nabla_\xi\hat{f}_{n,k_2}\|_{L^2_\xi}\}\\
    \lesssim & 2^{\e k_-+k_1-2k-2k_{2,+}+\gamma k_2}\e_1\min\{2^{-M/2+k_1/2-k_2/2}\e_1\|\hat{g}_{k_1}\|_{L^\infty_t([2^{M-1},2^M])H^2_\xi}, \\
    &\qquad\qquad\qquad\qquad 2^{M+3\min\{k,k_2\}/2+k_1-k_2/2}\e_1\|\hat{g}_{k_1}\|_{L^\infty_t([2^{M-1},2^M])H^1_\xi}\}.
\end{align*}
Hence, for $(k_1,k_2)\in
\chi^1_k\cup\chi^2_k$,
\begin{align*}
    \|T_{1+iy}\hat{g}\|_{L^2}
    \lesssim (1+2^{-M-k-k_1})2^{-k-2k_{2,+}+\gamma k_2}\min\{2^{-M/2+k_1/2-k_2/2}\e_1\|\hat{g}_{k_1}\|_{L^\infty_t([2^{M-1},2^M])H^2_\xi},\\ 2^{M+3\min\{k,k_2\}/2+k_1-k_2/2}\e_1\|\hat{g}_{k_1}\|_{L^\infty_t([2^{M-1},2^M])H^1_\xi}\}
\end{align*}
and
\begin{align*}
    \|T_{0+iy} \hat{g}\|_{L^2}
    \lesssim  2^{-k-2k_{2,+}+\gamma k_2}\min\{2^{-M/2+k_1/4-k_2/4}\e_1\|\hat{g}_{k_1}\|_{L^\infty_t([2^{M-1},2^M])H^1_\xi},\\
    2^{M+3\min\{k,k_2\}/2+k_2/2}\e_1\|\hat{g}_{k_1}\|_{L^\infty_t([2^{M-1},2^M])L^2_\xi}\}.
\end{align*}
For $(k_1,k_2)\in\chi^1_k\cup\chi^2_k\cup\chi^3_k$,
\begin{align*}
    \|T^1_{1+iy}\hat{g}\|_{L^2}+\|T^2_{1+iy}\hat{g}\|_{L^2}
    \lesssim 2^{\e k_-+k_1-2k-2k_{2,+}+\gamma k_2}\min\{2^{-M/2+k_1/2-k_2/2}\e_1\|\hat{g}_{k_1}\|_{L^\infty_t([2^{M-1},2^M])H^2_\xi},\\ 2^{M+3\min\{k,k_2\}/2+k_1-k_2/2}\e_1\|\hat{g}_{k_1}\|_{L^\infty_t([2^{M-1},2^M])H^1_\xi}\}
\end{align*}
and
\begin{align*}
    \|T^1_{0+iy} \hat{g}\|_{L^2}+\|T^2_{0+iy} \hat{g}\|_{L^2}
    \lesssim  2^{\e k_--k-2k_{2,+}+\gamma k_2}\min\{2^{-M/2+k_1/4-k_2/4}\e_1\|\hat{g}_{k_1}\|_{L^\infty_t([2^{M-1},2^M])H^1_\xi},\\
    2^{M+3\min\{k,k_2\}/2+k_2/2}\e_1\|\hat{g}_{k_1}\|_{L^\infty_t([2^{M-1},2^M])L^2_\xi}\}.
\end{align*}
Now, we use the variation of the Stein's interpolation theorem in Lemma \ref{si}, and get for any $\alpha\in[0,1]$,
\begin{align*}
    \|T_{\alpha}\hat{g}\|_{L^2}\lesssim (1+2^{-M-k-k_1})2^{-k-2k_{2,+}+\gamma k_2}\min\{2^{-M/2+k_1/2-k_2/2}\e_1\|\hat{g}\|_{L^\infty_t([2^{M-1},2^M])H^{1+\alpha}_\xi},\\
    2^{M+3\min\{k,k_2\}/2+k_1-k_2/2}\e_1\|\hat{g}\|_{L^\infty_t([2^{M-1},2^M])H^\alpha_\xi}\}
\end{align*}
and
\begin{align*}
    \|T^1_{\alpha}\hat{g}\|_{L^2}+\|T^2_{\alpha}\hat{g}\|_{L^2}\lesssim 2^{\e k_-+\alpha(k_1-k)-k-2k_{2,+}+\gamma k_2}\min\{2^{-M/2+k_1/2-k_2/2}\e_1\|\hat{g}\|_{L^\infty_t([2^{M-1},2^M])H^{1+\alpha}_\xi},\\
    2^{M+3\min\{k,k_2\}/2+k_1-k_2/2}\e_1\|\hat{g}\|_{L^\infty_t([2^{M-1},2^M])H^\alpha_\xi}\}.
\end{align*}
Recall the definition of $I^{M,0}_{k,k_1,k_2}$ and $J^{M,0}_{k,k_1,k_2}$ in \eqref{i0}, \eqref{j0} and use the results of the interpolation with $g=\F^{-1}\hat{f}_m\Tilde{\psi}_{k_1}$. We obtain
\begin{align*}
    &\|D^{\alpha}I^{M,0}_{k,k_1,k_2}\|_{L^2}\\
    \lesssim & (1+2^{-M-k-k_1})2^{-k-2k_{2,+}+\gamma k_2}\min\{2^{-M/2+k_1/2-k_2/2}\e_1\|\hat{f}_{m,k_1}\|_{L^\infty_t([2^{M-1},2^M])H^{1+\alpha}_\xi}
    ,\\
    &\qquad\qquad\qquad\qquad2^{M+3\min\{k,k_2\}/2 +k_1-k_2/2}\e_1\|\hat{f}_{m,k_1}\|_{L^\infty_t([2^{M-1},2^M])H^\alpha_\xi}\}\\
    \lesssim & (1+2^{-M-k-k_1})2^{-k-2k_{2,+}+\gamma k_2-2k_{1,+}+\gamma k_1}\min\{2^{-M/2+k_1/2-k_2/2-k_1/2-\alpha k_1}
    ,\\
    &\qquad\qquad\qquad\qquad 2^{M+3\min\{k,k_2\}/2+k_1-k_2/2+k_1/2-\alpha k_1}\}\e_1^2\\
    \lesssim & (1+2^{-M-k-k_1})2^{-k-2k_{2,+}+\gamma k_2-2k_{1,+}+\gamma k_1-\alpha k_1-k_2/2}\min\{2^{-M/2}
    ,2^{M+3\min\{k,k_2\}/2+3k_1/2}\}\e_1^2
\end{align*}
and
\begin{align*}
    &\|D^{\alpha}J^{M,0}_{k,k_1,k_2}\|_{L^2}\\
    \lesssim & 2^{\e k_-+\alpha(k_1-k)-k-2k_{2,+}+\gamma k_2}\min\{2^{-M/2+k_1/2-k_2/2}\e_1\|\hat{f}_{m,k_1}\|_{L^\infty_t([2^{M-1},2^M])H^{1+\alpha}_\xi}
    ,\\
    &\qquad\qquad\qquad\qquad2^{M+3\min\{k,k_2\}/2 +k_1-k_2/2}\e_1\|\hat{f}_{m,k_1}\|_{L^\infty_t([2^{M-1},2^M])H^\alpha_\xi}\}\\
    \lesssim & 2^{\e k_-+\alpha(k_1-k)-k-2k_{2,+}+\gamma k_2-2k_{1,+}+\gamma k_1}\min\{2^{-M/2-k_2/2-\alpha k_1}
    ,2^{M+3\min\{k,k_2\}/2 -k_2/2+3k_1/2-\alpha k_1}\}\e_1^2\\
    \lesssim & 2^{\e k_--\alpha k-k-2k_{2,+}+\gamma k_2-2k_{1,+}+\gamma k_1-k_2/2}\min\{2^{-M/2}
    ,2^{M+3\min\{k,k_2\}/2 +3k_1/2}\}\e_1^2,
\end{align*}
where the bounds on the Sobolev norms are given in \eqref{sobolevnorms}.
\end{proof}
Thus, as a result of Lemma \ref{im0},
\begin{align}
\label{im0wts1}
    &\sum_{1\leq m\leq \log T}\sup_{2^{M-1}\leq t_1<t_2\leq 2^M} 2^{-\gamma k+2k_++k/2+\alpha k}\sum_{c_m+c_n\neq 0}A_{lmn}\sum_{(k_1,k_2)\in\chi^1_k\cup\chi^2_k}\|D^{\alpha}_\xi I^{M,0}_{k,k_1,k_2}\|_{L^2}+\nonumber\\
    &+\sum_{1\leq m\leq \log T}\sup_{2^{M-1}\leq t_1<t_2\leq 2^M} 2^{-\gamma k+2k_++k/2+\alpha k}\sum_{c_m+c_n= 0}A_{lmn}\sum_{(k_1,k_2)\in\chi^1_k\cup\chi^2_k\cup\chi^3_k}\|D^{\alpha}_\xi J^{M,0}_{k,k_1,k_2}\|_{L^2}\nonumber\\
    \lesssim &\sum_{1\leq M\leq \log T}\sum_{(k_1,k_2)\in\chi^1_k\cup\chi^2_k} (1+2^{-M-k-k_1})2^{-\gamma k+2k_+-k/2+\alpha (k-k_1)-2k_{2,+}+\gamma k_2-2k_{1,+}+\gamma k_1-k_2/2}\times\nonumber\\
    &\qquad\qquad\qquad\qquad\qquad\qquad\times\min\{2^{-M/2}
    ,2^{M+3\min\{k,k_2\}/2 +3k_1/2}\}\e_1^2+\nonumber\\
    &+\sum_{1\leq M\leq \log T}\sum_{(k_1,k_2)\in\chi^1_k\cup\chi^2_k\cup\chi^3_k} 2^{\e k_--\gamma k+2k_+-k/2-2k_{2,+}+\gamma k_2-2k_{1,+}+\gamma k_1-k_2/2}\times\nonumber\\
    &\qquad\qquad\qquad\qquad\qquad\qquad\times\min\{2^{-M/2}
    ,2^{M+3\min\{k,k_2\}/2 +3k_1/2}\}\e_1^2\nonumber\\
    \lesssim &\sum_{1\leq M\leq \log T}\bigg(\sum_{k_2\leq k-2a-2} (1+2^{-M-2k})2^{-k/2-2k_{2,+}+\gamma k_2-k_2/2}\min\{2^{-M/2}
    ,2^{M+3k_2/2 +3k/2}\}\e_1^2+\nonumber\\
    &+\sum_{k\leq k_2+2a+2} (1+2^{-M-k-k_2})2^{-\gamma k+2k_+-k/2+\alpha (k-k_2)-4k_{2,+}+2\gamma k_2-k_2/2}\min\{2^{-M/2}
    ,2^{M+3k/2 +3k_2/2}\}\e_1^2+\nonumber\\
    &+\sum_{k\leq k_2+2a+2} 2^{\e k_--\gamma k+2k_+-k/2-4k_{2,+}+2\gamma k_2-k_2/2}\min\{2^{-M/2}
    ,2^{M+3k/2 +3k_2/2}\}\e_1^2\bigg)\nonumber\\
    \lesssim &\e_1^2.
\end{align}
We are left to show
\begin{align*}
    &\sum_{1\leq M\leq \log T}\sup_{2^{M-1}\leq t_1<t_2\leq 2^M} 2^{-\gamma k+2k_++k/2+\alpha k}\sum_{c_m+c_n\neq 0}A_{lmn}\sum_{(k_1,k_2)\in\chi^3_k}\|D^{\alpha}_\xi I^{M,0}_{k,k_1,k_2}\|_{L^2}
    \lesssim  \e_1^2.
\end{align*}
For $(k_1,k_2)\in\chi_k^3=\{|k_1-k_2|\leq a, |k-k_1|\leq a+2\}$, we observe 
$$|\nabla_\eta\phi(\xi,\eta)|=|2c_m\xi-2(c_m+c_n)\eta|=0$$ 
if $c_m\xi = (c_m+c_n)\eta$. Hence, in the proof above, integration by parts method using the identity $e^{it\phi(\xi,\eta)}=\sum_{n=1}^3\frac{\partial_{\eta_n}\phi(\xi,\eta)}{it|\nabla_\eta\phi(\xi,\eta)|^2}\partial_{\eta_n}e^{it\phi(\xi,\eta)}$ cannot be applied to estimate $\eqref{im0.12}$ and $\eqref{im0.2}$. Instead, we will take advantage of the fact that $\chi^3_k$ is a finite set with uniformly bounded cardinality for all $k$. In Lemma \ref{im0-2}, we estimate the $\alpha$ derivative using the interpolation result in Lemma \ref{dalpha}.
\begin{lem}
\label{im0-2}
Given $t_1,t_2\in[2^{M-1},2^M]$, $(k_1,k_2)\in\chi^3_k$, and $\sup_{t\in[1,T]}\|f_l\|_{Z}\leq \e_1$, we have
\begin{align*}
    &\|D^\alpha I^{M,0}_{k,k_1,k_2}\|_{L^2}
    \lesssim  2^{-4k_++2\gamma k+M+3k/2-\alpha k}\e_1^2
\end{align*}
when $M\leq -2k$, and
\begin{align*}
    &\|D^\alpha I^{M,0}_{k,k_1,k_2}\|_{L^2}
    \lesssim  2^{-4k_++2\gamma k-M/2-3k/2+\alpha M+\alpha k}\e_1^2
\end{align*}
when $M>-2k$.
\end{lem}
\begin{proof}
Since Lemma \ref{dalpha} implies 
$$\|D^\alpha I^{M,0}_{k,k_1,k_2}\|_{L^2}\leq \|I^{M,0}_{k,k_1,k_2}\|_{L^2}^{1-\alpha}\|\nabla I^{M,0}_{k,k_1,k_2}\|_{L^2}^\alpha,$$ we shall find bounds for $\|I^{M,0}_{k,k_1,k_2}\|_{L^2}$ and $\|\nabla I^{M,0}_{k,k_1,k_2}\|_{L^2}$ when $(k_1,k_2)\in\chi^3_k$.\\
First, using the bilinear estimate $L^4\times L^4\rightarrow L^2$ in Lemma \ref{bilinear}, together with Lemma \ref{dualitycomp}, \eqref{l4} and \eqref{l2}, we get
\begin{align*}
    \|I^{M,0}_{k,k_1,k_2}\|_{L^2}
    =&\bigg\|\int_{t_1}^{t_2}\int_{\R^3}e^{it\phi(\xi,\eta)}\hat{f}_{m,k_1}(t,\xi-\eta)\hat{f}_{n,k_2}(t,\eta)\partial_{\xi_l}\psi_k(\xi)d\eta dt\bigg\|_{L^2}\\
    \lesssim & \|\nabla_\xi\psi_k\|_{L^\infty_\xi}\bigg\|\F^{-1}\int_{t_1}^{t_2}\int_{\R^3}e^{it\phi(\xi,\eta)}\hat{f}_{m,k_1}(t,\xi-\eta)\hat{f}_{n,k_2}(t,\eta)\Tilde{\psi}_k(\xi)d\eta dt\bigg\|_{L^2_x}\\
    \lesssim & 2^{M-k}\min\{\|e^{ic_mt\la}f_{m,k_1}\|_{L^\infty_t([2^{M-1},2^M])L^4_x}\|e^{ic_nt\la}f_{n,k_2}\|_{L^\infty_t([2^{M-1},2^M])L^4_x},\\
    &\qquad\qquad\qquad\qquad 2^{3\min\{k,k_2\}/2}\|e^{ic_mt\la}f_{m,k_1}\|_{L^\infty_t([2^{M-1},2^M])L^2_x}\|e^{ic_nt\la}f_{n,k_2}\|_{L^\infty_t([2^{M-1},2^M])L^2_x}\}\\
    \lesssim & 2^{M-k-2k_{1,+}+\gamma k_1-2k_{2,+}+\gamma k_2}\min\{2^{-3M/2-k_1/4-k_2/4},2^{3\min\{k,k_2\}/2+k_1/2+k_2/2}\}\e_1^2\\
    \lesssim &2^{-4k_++2\gamma k}\min\{2^{-M/2-3k/2},2^{M+3k/2}\}\e_1^2,
\end{align*}
given that $k,k_1,k_2$ are about the same size when $(k_1,k_2)\in\chi^3_k$.\\
Next, we look at the first-order derivative,
\begin{align}
    \partial_{\xi_m} I^{M,0}_{k,k_1,k_2}
    = &\partial_{\xi_m}\int_{t_1}^{t_2}\int_{\R^3}e^{it\phi(\xi,\eta)}\hat{f}_{m,k_1}(t,\xi-\eta)\hat{f}_{n,k_2}(t,\eta)\partial_{\xi_l}\psi_k(\xi)d\eta dt\nonumber\\
    = &\int_{t_1}^{t_2}\int_{\R^3}e^{it\phi(\xi,\eta)}\partial_{\xi_m}\hat{f}_{m,k_1}(t,\xi-\eta)\hat{f}_{n,k_2}(t,\eta)\partial_{\xi_l}\psi_k(\xi)d\eta dt\label{im0-2.1}\\
    &+ \int_{t_1}^{t_2}\int_{\R^3}e^{it\phi(\xi,\eta)}\hat{f}_{m,k_1}(t,\xi-\eta)\hat{f}_{n,k_2}(t,\eta)\partial_{\xi_m}\partial_{\xi_l}\psi_k(\xi)d\eta dt\label{im0-2.12}\\
    & + \int_{t_1}^{t_2}\int_{\R^3}e^{it\phi(\xi,\eta)}it\partial_{\xi_m}\phi(\xi,\eta) \hat{f}_{m,k_1}(t,\xi-\eta)\hat{f}_{n,k_2}(t,\eta)\partial_{\xi_l}\psi_k(\xi)d\eta dt\label{im0-2.2}.
\end{align}
By the bilinear estimate $L^{18/7}\times L^9\rightarrow L^2$, Lemma \ref{op}, Lemma \ref{dualitycomp}, Lemma \ref{lpnorms}, \eqref{l2} and \eqref{l2first}
\begin{align*}
    &\|\nabla I^{M,0}_{k,k_1,k_2}\|_{L^2}\\
    \leq & \|\eqref{im0-2.1}\|_{L^2}+\|\eqref{im0-2.12}\|_{L^2}+\|\eqref{im0-2.2}\|_{L^2}\\
    \lesssim & 2^{M}\sup_{t\in[2^{M-1},2^M]}\|\nabla_\xi\psi_k\|_{L^\infty_\xi}\min\{\|e^{ic_mt\la}\F^{-1}\nabla_\xi\hat{f}_{m,k_1}\|_{L^{18/7}_x}\|e^{ic_nt\la}f_{n,k_2}\|_{L^9_x},\\
    &\qquad\qquad\qquad\qquad 2^{3\min\{k,k_2\}/2}\|e^{ic_mt\la}\F^{-1}\nabla_\xi\hat{f}_{m,k_1}\|_{L^2_x}\|e^{ic_nt\la}f_{n,k_2}\|_{L^2_x}\}+\\
    &+\big(2^{M}\|\nabla^2_\xi\psi_k\|_{L^\infty_\xi}+2^{2m}\|\nabla_\xi\psi_k\|_{L^\infty_\xi}\|\F^{-1}\partial_{\xi_m}\phi(\xi,\eta)\Tilde{\psi}_{k}(\xi)\Tilde{\psi}_{k_2}(\eta)\|_{L^1}\big)\times\\
    &\times \sup_{t\in[2^{M-1},2^M]}\min\{\|e^{ic_mt\la}f_{m,k_1}\|_{L^\infty_t([2^{M-1},2^M])L^{18/7}_x}\|e^{ic_nt\la}f_{n,k_2}\|_{L^9_x},\\
    &\qquad\qquad\qquad\qquad 2^{3\min\{k,k_2\}/2}\|e^{ic_mt\la}f_{m,k_1}\|_{L^2_x}\|e^{ic_nt\la}f_{n,k_2}\|_{L^2_x}\}\\
    \lesssim & 2^{M-k}\sup_{t\in[2^{M-1},2^M]}\min\{2^{-3M/2}\|\F^{-1}\nabla_\xi\hat{f}_{m,k_1}\|_{L^{18/11}}\|f_{n,k_2}\|_{L^{9/8}},2^{3\min\{k,k_2\}/2}\|\nabla_\xi\hat{f}_{m,k_1}\|_{L^2}\|f_{n,k_2}\|_{L^2}\}\\
    &+(2^{M-2k}+2^{2M-k+k_2})\sup_{t\in[2^{M-1},2^M]}\min\{2^{-3M/2}\|f_{m,k_1}\|_{L^{18/11}_x}\|f_{n,k_2}\|_{L^{9/8}_x},2^{3\min\{k,k_2\}/2}\|f_{m,k_1}\|_{L^2_x}\|f_{n,k_2}\|_{L^2_x}\}\\
    \lesssim & 2^{M-k-2k_{1,+}+\gamma k_1-2k_{2,+}+\gamma k_2}\min\{2^{-3M/2-5k_1/6-2k_2/3},2^{3\min\{k,k_2\}/2-k_1/2+k_2/2}\}\e_1^2\\
    &+(1+2^{M+k+k_2})2^{M-2k-2k_{1,+}+\gamma k_1-2k_{2,+}+\gamma k_2}\min\{2^{-3M/2+k_1/6-2k_2/3},2^{3\min\{k,k_2\}/2+k_1/2+k_2/2}\}\e_1^2\\
    \lesssim & (1+2^{M+2k})2^{-4k_{+}+2\gamma k}\min\{2^{-M/2-5k/2},2^{M+k/2}\}\e_1^2.
\end{align*}
Thus, we have when $M\leq -2k$,
\begin{align*}
    \|I^{M,0}_{k,k_1,k_2}\|_{L^2}
    \lesssim 2^{-4k_++2\gamma k+M+3k/2}\e_1^2
\end{align*}
and
\begin{align*}
    \|\nabla I^{M,0}_{k,k_1,k_2}\|_{L^2}
    \lesssim 2^{-4k_{+}+2\gamma k+M+k/2}\e_1^2,
\end{align*}
so Lemma \ref{dalpha} implies
\begin{align*}
    \|D^{\alpha}I^{M,0}_{k,k_1,k_2}\|_{L^2}
    \lesssim 2^{-4k_++2\gamma k+M+3k/2-\alpha k}\e_1^2.
\end{align*}
In the other case when $M>-2k$,
\begin{align*}
    \|I^{M,0}_{k,k_1,k_2}\|_{L^2}
    \lesssim 2^{-4k_++2\gamma k-M/2-3k/2}\e_1^2
\end{align*}
and
\begin{align*}
    \|\nabla I^{M,0}_{k,k_1,k_2}\|_{L^2}
    \lesssim 2^{-4k_{+}+2\gamma k+M/2-k/2}\e_1^2.
\end{align*}
Then, applying Lemma \ref{dalpha} gives
\begin{align*}
    \|D^{\alpha}I^{M,0}_{k,k_1,k_2}\|_{L^2}
    \lesssim 2^{-4k_++2\gamma k-M/2-3k/2+\alpha M+\alpha k}\e_1^2.
\end{align*}
\end{proof}
Therefore, Lemma \ref{im0-2} implies for any $k\in\Z$
\begin{align*}
    &\sum_{1\leq M\leq \log T}\sup_{2^{M-1}\leq t_1\leq t_2\leq 2^M}2^{-\gamma k+2k_++k/2+\alpha k}\sum_{c_m+c_n\neq 0}A_{lmn}\sum_{(k_1,k_2)\in\chi^3_k}\|D^\alpha I^{M,0}_{k,k_1,k_2}\|_{L^2}\\
    \lesssim  &\sum_{1\leq M\leq -2k}2^{-2k_{+}+\gamma k+M+2k}\e_1^2+\sum_{-2k< M\leq \log T}2^{-2k_{+}+\gamma k-M/2-k+\alpha M+2\alpha k}\e_1^2\\
    \lesssim &\e_1^2,
\end{align*}
since $\alpha<1/2$. This result combined with \eqref{im0wts1} gives \eqref{im0wts}.

\subsection{$D^{\alpha}I^{M,1}_{k,k_1,k_2}$,$D^{\alpha}J^{M,1}_{k,k_1,k_2}$}
\label{1}
In this section, we will show
\begin{equation}
\label{im1wts}
    \begin{aligned}
        &\sum_{1\leq m\leq \log T}\sup_{2^{M-1}\leq t_1<t_2\leq 2^M} 2^{-\gamma k+2k_++k/2+\alpha k}\sum_{c_m+c_n\neq 0}A_{lmn}\sum_{(k_1,k_2)\in\chi^1_k\cup\chi^2_k\cup\chi^3_k}\|D^{\alpha}_\xi I^{M,1}_{k,k_1,k_2}\|_{L^2}+\\
        &+\sum_{1\leq m\leq \log T}\sup_{2^{M-1}\leq t_1<t_2\leq 2^M} 2^{-\gamma k+2k_++k/2+\alpha k}\sum_{c_m+c_n= 0}A_{lmn}\sum_{(k_1,k_2)\in\chi^1_k\cup\chi^2_k\cup\chi^3_k}\|D^{\alpha}_\xi J^{M,1}_{k,k_1,k_2}\|_{L^2}\lesssim\e_1^2
    \end{aligned}
\end{equation}
When $(k_1,k_2)\in\chi^1_k$, i.e. $|\xi|\sim|\xi-\eta|\geq 2^{a-2}|\eta|$, we want to integrate by parts in the time variable using the identity
$$e^{it\phi(\xi,\eta)}=\frac{1}{i\phi(\xi,\eta)}\partial_{t}e^{it\phi(\xi,\eta)}.$$
Recall that we defined $\phi(\xi,\eta)=c_l|\xi|^2-c_m|\xi-\eta|^2-c_n|\eta|^2$. There is no guaranteed lower bound on $|\phi(\xi,\eta)|$ when $c_l=c_m$. Hence, in Lemma \ref{jm1}, we only look at $(k_1,k_2)\in\chi^1_k$ excluding $c_l=c_m$. The special case of $(k_1,k_2)\in\chi^1_k$ and $c_l=c_m$ will be tackled later in Lemma \ref{sp1}.

Now we show the lower bound on $|\phi(\xi,\eta)|$ that we will use when $(k_1,k_2)\in\chi^1_k$ and $c_l\neq c_m$. Recall the definition of $a$ in \eqref{adef}, we have
\begin{equation}
    \label{lbt}
    \begin{aligned}
    |\phi(\xi,\eta)|
    &= \big|(c_l-c_m)|\xi-\eta|^2+2c_l\xi\cdot\eta-(c_l+c_n)|\eta|^2\big|\\
    &= |c_l-c_m|\big||\xi-\eta|^2+\frac{2c_l}{c_l-c_m}\xi\cdot\eta-\frac{c_l+c_n}{c_l-c_m}|\eta|^2\big|\\
    &\geq |c_l-c_m|\big(2^{2k_1-4}-\frac{2}{|1-c_m/c_l|}2^{k+k_2}-\frac{1}{|1-c_m/c_l|}2^{2k_2}-\frac{|c_n|}{|c_l-c_m|}2^{2k_2}\big)\\
    &\geq |c_l-c_m|\big(2^{2k_1-4}-\frac{1}{|1-c_m/c_l|}2^{2k_1+4-a}-\frac{1}{|1-c_m/c_l|}2^{2k_1-2a}-\frac{|c_n|}{|c_l-c_m|}2^{2k_1-2a}\big)\\
    &\geq |c_l-c_m|\big(2^{2k_1-4}-2^{2k_1-6}-2^{2k_1-10}-2^{2k_1-6}\big)\\
    &\geq |c_l-c_m|\big(2^{2k_1-5}-2^{2k_1-10}\big)
    \gtrsim 2^{2k_1}.
\end{aligned}
\end{equation}
In the proof for Lemma \ref{jm1}, we will employ the bound above with Lemma \ref{dalpha} to obtain the desired estimates.
\begin{lem}
\label{jm1}
Suppose $t_1,t_2\in[2^{M-1},2^M]$ and $\sup_{t\in[1,T]}\|f_l\|_{Z}\leq \e_1$. For $(k_1,k_2)\in\chi^1_k$ and $c_l\neq c_m$, we have
\begin{align*}
    &\|D^{\alpha}_\xi I^{M,1}_{k,k_1,k_2}\|_{L^2}+\|D^{\alpha}_\xi J^{M,1}_{k,k_1,k_2}\|_{L^2}\\
    \lesssim & 2^{-2k_{+}+\gamma k-\gamma k_--2k_{2,+}+\gamma k_2}\min\{2^{-M/2-k-\alpha k-k_2/2},2^{M+2k_2-k/2-\alpha k_2}\}\e_1^2,
\end{align*}
if $M+2k\leq 0$, and
\begin{align*}
    &\|D^{\alpha}_\xi I^{M,1}_{k,k_1,k_2}\|_{L^2}+\|D^{\alpha}_\xi J^{M,1}_{k,k_1,k_2}\|_{L^2}\\
    \lesssim & 2^{-2k_++\gamma k-\gamma k_--2k_{2,+}+\gamma k_2}\min\{2^{-M/2+\gamma M/4+\gamma k/2-2k+\alpha k-k_2/2},2^{M-k+3k_2/2+\alpha k}\}\e_1^2,
\end{align*}
if $M+2k> 0$.
\end{lem}
\begin{proof}
Recall the definition of $I^{M,1}_{k,k_1,k_2}$ and $J^{M,1}_{k,k_1,k_2}$ in \eqref{i1} and \eqref{j1}. Define
\begin{align*}
    F^1(\xi)= I^{M,1}_{k,k_1,k_2}(\xi)=\int_{t_1}^{t_2}\int_{\R^3}e^{it\phi(\xi,\eta)}\partial_{\xi_l}\hat{f}_{m,k_1}(t,\xi-\eta)\hat{f}_{n,k_2}(t,\eta)\psi_k(\xi)d\eta dt
\end{align*}
and
\begin{align*}
    F^2(\xi)= J^{M,1}_{k,k_1,k_2}(\xi)=\int_{t_1}^{t_2}\int_{\R^3}e^{it\phi(\xi,\eta)}q(\eta,\xi-\eta)\partial_{\xi_l}\hat{f}_{m,k_1}(t,\xi-\eta)\hat{f}_{n,k_2}(t,\eta)\psi_k(\xi)d\eta dt.
\end{align*}
By Lemma \ref{dalpha}, we get
$$\|D^{\alpha}_\xi I^{M,1}_{k,k_1,k_2}\|_{L^2}=\|D^\alpha F^1\|_{L^2}\leq \|F^1\|_{L^2}^{1-\alpha}\|\nabla_\xi F^1\|_{L^2}^\alpha$$
and 
$$\|D^{\alpha}_\xi J^{M,1}_{k,k_1,k_2}\|_{L^2}=\|D^\alpha F^2\|_{L^2}\leq \|F^2\|_{L^2}^{1-\alpha}\|\nabla_\xi F^2\|_{L^2}^\alpha.$$
We shall find $L^2$ bounds on $F^1$, $F^2$ and their derivatives.

First, by the bilinear estimate $L^3\times L^6\rightarrow L^2$ in Lemma \ref{bilinear}, Lemma \ref{dualitycomp}, and \eqref{q},
\begin{align*}
    &\|F^1\|_{L^2}+\|F^2\|_{L^2}\\
    \lesssim & 2^{M}(1+\|\F^{-1}q(\xi-\eta,\eta)\Tilde{\psi}_k(\xi)\Tilde{\psi}_{k_1}(\xi-\eta)\Tilde{\psi}_{k_2}(\eta)\|_{L^1})\times\\
    &\times\sup_{t\in[2^{M-1},2^M]}\min\{\|e^{ic_mt\la}\F^{-1}\nabla\hat{f}_{m,k_1}\|_{L^3_x}\|e^{ic_nt\la}{f}_{n,k_2}\|_{L^6_x},
    2^{3k_2/2}\|e^{ic_mt\la}\F^{-1}\nabla_\xi f_{m,k_1}\|_{L^2_x}\|e^{ic_nt\la}f_{n,k_2}\|_{L^2_x}\}\\
    \lesssim & 2^{M}\sup_{t\in[2^{M-1},2^M]}\min\{\|e^{ic_mt\la}\F^{-1}\nabla\hat{f}_{m,k_1}\|_{L^3_x}\|e^{ic_nt\la}{f}_{n,k_2}\|_{L^6_x},2^{3k_2/2}\|\nabla_\xi \hat{f}_{m,k_1}\|_{L^2_\xi}\|f_{n,k_2}\|_{L^2_x}\}.
\end{align*}
Using Lemma \ref{lpnorms}, \eqref{l2}, \eqref{l2first}, and \eqref{l3nablaf}, we obtain
\begin{equation}
\label{jm1bound1}
\begin{aligned}
    \|F^1\|_{L^2}+\|F^2\|_{L^2}
    \lesssim & 2^{M-2k_{1,+}+\gamma k_1-2k_{2,+}+\gamma k_2}\min\{2^{-3M/2+\gamma M/8-k_1+\gamma k_1/4-k_2/2},2^{3k_2/2-k_1/2+k_2/2}\}\e_1^2\\
    \lesssim & 2^{-2k_{+}+\gamma k-2k_{2,+}+\gamma k_2}\min\{2^{-M/2+\gamma M/8-k+\gamma k/4-k_2/2},2^{M+2k_2-k/2}\}\e_1^2.
\end{aligned}
\end{equation}
However, this bound on $\|F^1+F^2\|_{L^2}$ is not small enough when $M+2k > 0$. In this case,
we integrate by parts using the identity $e^{it\phi(\xi,\eta)}=\frac{1}{i\phi(\xi,\eta)}\partial_{t}e^{it\phi(\xi,\eta)}$ to obtain another $L^2$ bound,
\begin{align}
    F^1(\xi)+F^2(\xi)
    = &\int_{t_1}^{t_2}\int_{\R^3}\partial_te^{it\phi(\xi,\eta)}\frac{1}{i\phi(\xi,\eta)}(1+q(\eta,\xi-\eta))\partial_{\xi_l}\hat{f}_{m,k_1}(t,\xi-\eta)\hat{f}_{n,k_2}(t,\eta)\psi_k(\xi)d\eta dt\nonumber\\
    = & \int_{\R^3}e^{it_2\phi(\xi,\eta)}\frac{1}{i\phi(\xi,\eta)}(1+q(\eta,\xi-\eta))\partial_{\xi_l}\hat{f}_{m,k_1}(t_2,\xi-\eta)\hat{f}_{n,k_2}(t_2,\eta)\psi_k(\xi)d\eta \label{jm1.b1}\\
    & - \int_{\R^3}e^{it_1\phi(\xi,\eta)}\frac{1}{i\phi(\xi,\eta)}(1+q(\eta,\xi-\eta))\partial_{\xi_l}\hat{f}_{m,k_1}(t_1,\xi-\eta)\hat{f}_{n,k_2}(t_1,\eta)\psi_k(\xi)d\eta\label{jm1.b2}\\
    & - \int_{t_1}^{t_2}\int_{\R^3}e^{it\phi(\xi,\eta)}\frac{1}{i\phi(\xi,\eta)}(1+q(\eta,\xi-\eta))\partial_t\partial_{\xi_l}\hat{f}_{m,k_1}(t,\xi-\eta)\hat{f}_{n,k_2}(t,\eta)\psi_k(\xi)d\eta dt\label{jm1.4}\\
    & - \int_{t_1}^{t_2}\int_{\R^3}e^{it\phi(\xi,\eta)}\frac{1}{i\phi(\xi,\eta)}(1+q(\eta,\xi-\eta))\partial_{\xi_l}\hat{f}_{m,k_1}(t,\xi-\eta)\partial_t\hat{f}_{n,k_2}(t,\eta)\psi_k(\xi)d\eta dt\label{jm1.5}.
\end{align}
Since $(k_1,k_2)\in\chi^1_k$ and $c_l\neq c_m$, $$\big|\phi(\xi,\eta)\big|=\big|c_l|\xi|^2-c_m|\xi-\eta|^2-c_n|\eta|^2\big|\gtrsim 2^{2k_1},$$ ideally, we shall gain a factor of 
$$ t^{-1}|\phi(\xi,\eta)|^{-1}\sim 2^{-M-2k_1}$$
through integration by parts. But given that there are only $1+\alpha$ derivatives of $\hat{f}_k$ in the $Z$ norm for us to work with, we cannot quite achieve $2^{-M-2k_1}$. This is shown in the estimations below.\\
By the $L^3\times L^6\rightarrow L^2$ bilinear estimate, Lemma \ref{dualitycomp}, Lemma \ref{chi,eta}, Lemma \ref{op}, Lemma \ref{lpnorms}, along with \eqref{l2}, \eqref{l2first}, and \eqref{l3nablaf},
\begin{align*}
    &\|\eqref{jm1.b1}\|_{L^2}+\|\eqref{jm1.b2}\|_{L^2}\\
    \lesssim & \|\F^{-1}\frac{1}{\phi(\xi,\eta)}(1+q(\eta,\xi-\eta))\Tilde{\psi}_k(\xi)\Tilde{\psi}_{k_1}(\xi-\eta)\Tilde{\psi}_{k_2}(\eta)\|_{L^1}\times\\
    &\times\sup_{t\in[2^{M-1},2^M]}\min\{\|e^{ic_mt\la}\F^{-1}\nabla_\xi\hat{f}_{m,k_1}\|_{L^3_x}\|e^{ic_nt\la}f_{n,k_2}\|_{L^6_x}, 2^{3k_2/2}\|e^{ic_mt\la}\F^{-1}\nabla_\xi\hat{f}_{m,k_1}\|_{L^2_x}\|e^{ic_nt\la}f_{n,k_2}\|_{L^2_x}\}\\
    \lesssim & 2^{-2k}\sup_{t\in[2^{M-1},2^M]}\min\{\|e^{ic_mt\la}\F^{-1}\nabla_\xi\hat{f}_{m,k_1}\|_{L^3_x}\|e^{ic_nt\la}f_{n,k_2}\|_{L^6_x},2^{3k_2/2}\|\nabla_\xi\hat{f}_{m,k_1}\|_{L^2_\xi}\|f_{n,k_2}\|_{L^2_x}\}\\
    \lesssim & 2^{-2k-2k_{1,+}+\gamma k_1-2k_{2,+}+\gamma k_2}\min\{2^{-M/2+\gamma M/8-k_1+\gamma k_1/4-M-k_2/2},2^{3k_2/2-k_1/2+k_2/2}\}\e_1^2\\
    \lesssim & 2^{-2k-2k_++\gamma k-2k_{2,+}+\gamma k_2}\min\{2^{-3M/2+\gamma M/8-k+\gamma k/4-k_2/2},2^{2k_2-k/2}\}\e_1^2
\end{align*}
and
\begin{align*}
    \|\eqref{jm1.4}\|_{L^2}+\|\eqref{jm1.5}\|_{L^2}
    \lesssim & 2^M\|\F^{-1}\frac{1}{\phi(\xi,\eta)}(1+q(\eta,\xi-\eta))\Tilde{\psi}_k(\xi)\Tilde{\psi}_{k_1}(\xi-\eta)\Tilde{\psi}_{k_2}(\eta)\|_{L^1}\times\\
    &\times\sup_{t\in[2^{M-1},2^M]}\big( \min\{\|e^{ic_mt\la}\F^{-1}\partial_t\nabla_\xi\hat{f}_{m,k_1}\|_{L^3_x}\|e^{ic_nt\la}f_{n,k_2}\|_{L^6_x},\\
    &\qquad\qquad\qquad\qquad\qquad 2^{3k_2/2}\|e^{ic_mt\la}\F^{-1}\partial_t\nabla_\xi\hat{f}_{m,k_1}\|_{L^2_x}\|e^{ic_nt\la}f_{n,k_2}\|_{L^2_x}\}+\\
    &\qquad\qquad\qquad\qquad + \min\{\|e^{ic_mt\la}\F^{-1}\nabla_\xi\hat{f}_{m,k_1}\|_{L^3_x}\|e^{ic_nt\la}\partial_tf_{n,k_2}\|_{L^6_x},\\
    &\qquad\qquad\qquad\qquad\qquad 2^{3k_2/2}\|e^{ic_mt\la}\F^{-1}\nabla_\xi\hat{f}_{m,k_1}\|_{L^2_x}\|e^{ic_nt\la}\partial_tf_{n,k_2}\|_{L^2_x}\}\big).
\end{align*}
Then, using the bounds on the time derivatives in Lemma \ref{timel2}, Lemma \ref{mixder}, and Lemma \ref{mixderl3}, we have
\begin{align*}
    &\|\eqref{jm1.4}\|_{L^2}+\|\eqref{jm1.5}\|_{L^2}\\
    \lesssim & 2^{M-2k-2k_{1,+}+\gamma k_1-2k_{2,+}+\gamma k_2}\big( \min\{2^{-M-\gamma k_1-M-k_2/2},(1+2^{M/2+k_1})2^{3k_2/2-M-k_1/2-\gamma k_1+k_2/2}\}+\\
    &\qquad + \min\{2^{-M/2+\gamma M/8-k_1+\gamma k_1/4-2M-k_2/2},2^{3k_2/2-k_1/2-M+k_2/2}\}\big)\e_1^2\\
    \lesssim & (1+2^{M/2+k})2^{-2k-2k_++\gamma k-2k_{2,+}+\gamma k_2}\min\{2^{-3M/2-\gamma k_--k-k_2/2},2^{-k/2-\gamma k_-+2k_2}\}\e_1^2.
\end{align*}
Thus, we conclude
\begin{align}
\label{jm1bound2}
    \|F^1\|_{L^2}+\|F^2\|_{L^2}\lesssim (1+2^{M/2+k})2^{-2k-2k_++\gamma k-\gamma k_--2k_{2,+}+\gamma k_2}\min\{2^{-3M/2-k-k_2/2},2^{-k/2+2k_2}\}\e_1^2.
\end{align}
Now, we deal with the terms $\|\nabla_\xi F^1\|_{L^2}$, $\|\nabla_\xi F^2\|_{L^2}$. Since we will not be able to bound $\|\nabla^2_\xi\hat{f}_{m,k_1}\|_{L^2}$, before taking the derivative in $\xi$, we perform a change of variable $\zeta=\xi-\eta$ on $F^1+F^2$. As a consequence, we will have the two derivatives evenly distributed between $\hat{f}_{m,k_1}$ and $\hat{f}_{n,k_2}$,
\begin{align*}
    F^1+F^2&=\int_{t_1}^{t_2}\int_{\R^3}e^{it\phi(\xi,\eta)}(1+q(\eta,\xi-\eta))\partial_{\xi_l}\hat{f}_{m,k_1}(t,\xi-\eta)\hat{f}_{n,k_2}(t,\eta)\psi_k(\xi)d\eta dt\\
    &=\int_{t_1}^{t_2}\int_{\R^3}e^{it\phi(\xi,\xi-\zeta)}(1+q(\xi-\zeta,\zeta))\partial_{\zeta_l}\hat{f}_{m,k_1}(t,\zeta)\hat{f}_{n,k_2}(t,\xi-\zeta)\psi_k(\xi)d\zeta dt.
\end{align*}
Hence,
\begin{align}
    &\partial_{\xi_m} F^1+\partial_{\xi_m} F^2\\
    = &\int_{t_1}^{t_2}\int_{\R^3}e^{it\phi(\xi,\xi-\eta)}(1+q(\xi-\eta,\eta))\partial_{\eta_l}\hat{f}_{m,k_1}(t,\eta)\partial_{\xi_m}\hat{f}_{n,k_2}(t,\xi-\eta)\psi_k(\xi)d\eta dt\label{jm1.1}\\
    &+ \int_{t_1}^{t_2}\int_{\R^3}e^{it\phi(\xi,\xi-\eta)}\partial_{\xi_m}q(\xi-\eta,\eta)\partial_{\eta_l}\hat{f}_{m,k_1}(t,\eta)\hat{f}_{n,k_2}(t,\xi-\eta)\psi_k(\xi)d\eta dt\label{jm1.11}\\
    &+ \int_{t_1}^{t_2}\int_{\R^3}e^{it\phi(\xi,\xi-\eta)}(1+q(\xi-\eta,\eta))\partial_{\eta_l}\hat{f}_{m,k_1}(t,\eta)\hat{f}_{n,k_2}(t,\xi-\eta)\partial_{\xi_m}\psi_k(\xi)d\eta dt\label{jm1.12}\\
    & + \int_{t_1}^{t_2}\int_{\R^3}e^{it\phi(\xi,\xi-\eta)}it\partial_{\xi_m}\phi(\xi,\xi-\eta)(1+q(\xi-\eta,\eta))\partial_{\eta_l}\hat{f}_{m,k_1}(t,\eta)\hat{f}_{n,k_2}(t,\xi-\eta)\psi_k(\xi)d\eta dt\label{jm1.2}.
\end{align}
For the term $\eqref{jm1.2}$, after integrating by parts using the identity $e^{it\phi(\xi,\xi-\eta)}=\frac{1}{i\phi(\xi,\xi-\eta)}\partial_{t}e^{it\phi(\xi,\xi-\eta)}$, we get
\begin{align}
    \eqref{jm1.2}
    = & \int_{t_1}^{t_2}\int_{\R^3}\partial_t e^{it\phi(\xi,\xi-\eta)}\frac{t\partial_{\xi_m}\phi(\xi,\xi-\eta)}{\phi(\xi,\xi-\eta)}(1+q(\xi-\eta,\eta))\partial_{\eta_l}\hat{f}_{m,k_1}(t,\eta)\hat{f}_{n,k_2}(t,\xi-\eta)\psi_k(\xi)d\eta dt\nonumber\\
    = & \int_{\R^3} e^{it_2\phi(\xi,\xi-\eta)}\frac{t_2\partial_{\xi_m}\phi(\xi,\xi-\eta)}{\phi(\xi,\xi-\eta)}(1+q(\xi-\eta,\eta))\partial_{\eta_l}\hat{f}_{m,k_1}(t_2,\eta)\hat{f}_{n,k_2}(t_2,\xi-\eta)\psi_k(\xi)d\eta\label{jm1.b1-2}\\
    & - \int_{\R^3} e^{it_1\phi(\xi,\xi-\eta)}\frac{t_1\partial_{\xi_m}\phi(\xi,\xi-\eta)}{\phi(\xi,\xi-\eta)}(1+q(\xi-\eta,\eta))\partial_{\eta_l}\hat{f}_{m,k_1}(t_1,\eta)\hat{f}_{n,k_2}(t_1,\xi-\eta)\psi_k(\xi)d\eta\label{jm1.b2-2}\\
    & - \int_{t_1}^{t_2}\int_{\R^3} e^{it\phi(\xi,\xi-\eta)}\frac{\partial_{\xi_m}\phi(\xi,\xi-\eta)}{\phi(\xi,\xi-\eta)}(1+q(\xi-\eta,\eta))\partial_{\eta_l}\hat{f}_{m,k_1}(t,\eta)\hat{f}_{n,k_2}(t,\xi-\eta)\psi_k(\xi)d\eta dt\label{jm1.3-2}\\
    & - \int_{t_1}^{t_2}\int_{\R^3} e^{it\phi(\xi,\xi-\eta)}\frac{t\partial_{\xi_m}\phi(\xi,\xi-\eta)}{\phi(\xi,\xi-\eta)}(1+q(\xi-\eta,\eta))\partial_t\partial_{\eta_l}\hat{f}_{m,k_1}(t,\eta)\hat{f}_{n,k_2}(t,\xi-\eta)\psi_k(\xi)d\eta dt\label{jm1.4-2}\\
    & - \int_{t_1}^{t_2}\int_{\R^3} e^{it\phi(\xi,\xi-\eta)}\frac{t\partial_{\xi_m}\phi(\xi,\xi-\eta)}{\phi(\xi,\xi-\eta)}(1+q(\xi-\eta,\eta))\partial_{\eta_l}\hat{f}_{m,k_1}(t,\eta)\partial_t\hat{f}_{n,k_2}(t,\xi-\eta)\psi_k(\xi)d\eta dt\label{jm1.5-2}.
\end{align}
Using the $L^3\times L^6\rightarrow L^2$ bilinear estimate and Lemma \ref{dualitycomp}, we have
\begin{align*}
    \|\eqref{jm1.1}\|_{L^2}
    \lesssim & 2^M\sup_{t\in[2^{M-1},2^M]}\|\F^{-1}(1+q(\xi-\eta,\eta))\Tilde{\psi}_k(\xi)\Tilde{\psi}_{k_1}(\eta)\Tilde{\psi}_{k_2}(\xi-\eta)\|_{L^1}\times\\
    &\times \min\{\|e^{ic_mt\la}\F^{-1}\nabla_\xi\hat{f}_{m,k_1}\|_{L^3_x}\|e^{ic_nt\la}\F^{-1}\nabla_\xi\hat{f}_{n,k_2}\|_{L^6_x},\\
    &\qquad\qquad\qquad\qquad 2^{3k_2/2}\|e^{ic_mt\la}\F^{-1}\nabla_\xi\hat{f}_{m,k_1}\|_{L^2_x}\|e^{ic_nt\la}\F^{-1}\nabla_\xi\hat{f}_{n,k_2}\|_{L^2_x}\}\\
    \lesssim & 2^{M}\sup_{t\in[2^{M-1},2^M]}\min\{\|e^{ic_mt\la}\F^{-1}\nabla_\xi\hat{f}_{m,k_1}\|_{L^3_x}\|e^{ic_nt\la}\F^{-1}\nabla_\xi\hat{f}_{n,k_2}\|_{L^6_x}, \\
    &\qquad\qquad\qquad\qquad2^{3k_2/2}\|\nabla_\xi\hat{f}_{m,k_1}\|_{L^2_\xi}\|\nabla_\xi\hat{f}_{n,k_2}\|_{L^2_\xi}\}.
\end{align*}
Then, by Bernstein's inequality, \eqref{l2}, \eqref{l2first}, and \eqref{l3nablaf},
\begin{align*}
    \|\eqref{jm1.1}\|_{L^2}
    \lesssim & 2^{M}\sup_{t\in[2^{M-1},2^M]}\min\{\|e^{ic_mt\la}\F^{-1}\nabla_\xi\hat{f}_{m,k_1}\|_{L^3_x}2^{k_2/2}\|e^{ic_nt\la}\F^{-1}\nabla_\xi\hat{f}_{n,k_2}\|_{L^3_x}, \\
    &\qquad\qquad\qquad\qquad 2^{3k_2/2}\|\nabla_\xi\hat{f}_{m,k_1}\|_{L^2_\xi}\|\nabla_\xi\hat{f}_{n,k_2}\|_{L^2_\xi}\}\\
    \lesssim & 2^{M-2k_{1,+}+\gamma k_1-2k_{2,+}+\gamma k_2}\min\{2^{-M+\gamma M/4+\gamma k_1/4-k_1+\gamma k_2/4-k_2/2}, 2^{3k_2/2-k_1/2-k_2/2}\}\e_1^2\\
    \lesssim & 2^{-2k_++\gamma k-2k_{2,+}+\gamma k_2}\min\{2^{\gamma M/4+\gamma k/4-k+\gamma k_2/4-k_2/2}, 2^{M-k/2+k_2}\}\e_1^2.
\end{align*}
Next, the bilinear estimate $L^3\times L^6\rightarrow L^2$, together with Lemma \ref{chi,eta}, Lemma \ref{dualitycomp}, Lemma \ref{lpnorms}, \eqref{l2}, and \eqref{l3nablaf} implies
\begin{align*}
    &\|\eqref{jm1.11}+\eqref{jm1.12}+\eqref{jm1.b1-2}+\eqref{jm1.b2-2}+\eqref{jm1.3-2}\|_{L^2}\\
    \lesssim &2^{M}\big(\|\F^{-1}\nabla_\xi q(\xi-\eta,\eta)\Tilde{\psi}_k(\xi)\Tilde{\psi}_{k_1}(\eta)\Tilde{\psi}_{k_2}(\xi-\eta)\|_{L^1}\\
    & + \|\F^{-1} (1+q(\xi-\eta,\eta))\Tilde{\psi}_k(\xi)\Tilde{\psi}_{k_1}(\eta)\Tilde{\psi}_{k_2}(\xi-\eta)\|_{L^1}\|\nabla_\xi\psi_k\|_{L^\infty}\\
    & +\|\F^{-1} \frac{\partial_{\xi_m}\phi(\xi,\xi-\eta)}{\phi(\xi,\xi-\eta)}(1+ q(\xi-\eta,\eta))\Tilde{\psi}_k(\xi)\Tilde{\psi}_{k_1}(\eta)\Tilde{\psi}_{k_2}(\xi-\eta)\|_{L^1}\big)\times\\
    &\times\sup_{t\in[2^{M-1},2^M]}\min\{\|e^{ic_mt\la}\F^{-1}\nabla\hat{f}_{m,k_1}\|_{L^3_x}\|e^{ic_nt\la}{f}_{n,k_2}\|_{L^6_x},2^{3k_2/2}\|e^{ic_mt\la}\F^{-1}\nabla\hat{f}_{m,k_1}\|_{L^2_x}\|e^{ic_nt\la}{f}_{n,k_2}\|_{L^2_x}\}\\
    \lesssim & 2^{M-k}\sup_{t\in[2^{M-1},2^M]}\min\{\|e^{ic_mt\la}\F^{-1}\nabla\hat{f}_{m,k_1}\|_{L^3_x}\|e^{ic_nt\la}f_{n,k_2}\|_{L^6_x},2^{3k_2/2}\|\F^{-1}\nabla\hat{f}_{m,k_1}\|_{L^2_x}\|{f}_{n,k_2}\|_{L^2_x}\}\\
    \lesssim & 2^{-k-2k_{1,+}-2k_{2,+}+\gamma k_1+\gamma k_2}\min\{2^{-M/2+\gamma M/8-k_1+\gamma k_1/4-k_2/2},2^{M+3k_2/2-k_1/2+k_2/2}\}\e_1^2\\
    \lesssim & 2^{-2k_{+}-2k_{2,+}+\gamma k+\gamma k_2}\min\{2^{-M/2+\gamma M/8-2k+\gamma k/4-k_2/2},2^{M+2k_2-3k/2}\}\e_1^2,
\end{align*}
\begin{align*}
    &\|\eqref{jm1.4-2}+\eqref{jm1.5-2}\|_{L^2}\\
    \lesssim &2^{2M}\|\F^{-1} \frac{\partial_{\xi_m}\phi(\xi,\xi-\eta)}{\phi(\xi,\xi-\eta)} (1+q(\xi-\eta,\eta))\Tilde{\psi}_k(\xi)\Tilde{\psi}_{k_1}(\eta)\Tilde{\psi}_{k_2}(\xi-\eta)\|_{L^1}\times\\
    &\qquad\times\sup_{t\in[2^{M-1},2^M]}\big(\min\{\|e^{ic_mt\la}\F^{-1}\partial_t\nabla\hat{f}_{m,k_1}\|_{L^3_x}\|e^{ic_nt\la}{f}_{n,k_2}\|_{L^6_x},\\
    &\qquad\qquad\qquad\qquad\qquad 2^{3k_2/2}\|e^{ic_mt\la}\F^{-1}\partial_t\nabla\hat{f}_{m,k_1}\|_{L^2_x}\|e^{ic_nt\la}{f}_{n,k_2}\|_{L^2_x}\}+\\
    &\qquad\qquad\qquad\qquad+\min\{\|e^{ic_mt\la}\F^{-1}\nabla\hat{f}_{m,k_1}\|_{L^3_x}\|e^{ic_nt\la}\partial_t{f}_{n,k_2}\|_{L^6_x},\\
    &\qquad\qquad\qquad\qquad\qquad 2^{3k_2/2}\|e^{ic_mt\la}\F^{-1}\nabla\hat{f}_{m,k_1}\|_{L^2_x}\|e^{ic_nt\la}\partial_t{f}_{n,k_2}\|_{L^2_x}\}\big)\\
    \lesssim & 2^{2M-k}\sup_{t\in[2^{M-1},2^M]}\big(\min\{\|e^{ic_mt\la}\F^{-1}\partial_t\nabla\hat{f}_{m,k_1}\|_{L^3_x}\|e^{ic_nt\la}{f}_{n,k_2}\|_{L^6_x},2^{3k_2/2}\|\partial_t\nabla\hat{f}_{m,k_1}\|_{L^2_\xi}\|{f}_{n,k_2}\|_{L^2_x}\}+\\
    &\qquad\qquad\qquad\qquad+\min\{\|e^{ic_mt\la}\F^{-1}\nabla\hat{f}_{m,k_1}\|_{L^3_x}\|e^{ic_nt\la}\partial_t{f}_{n,k_2}\|_{L^6_x},2^{3k_2/2}\|\nabla\hat{f}_{m,k_1}\|_{L^2_\xi}\|\partial_t{f}_{n,k_2}\|_{L^2_x}\}\big),
\end{align*}
and by Lemma \ref{timel2}, Lemma \ref{mixder}, Lemma \ref{mixderl3}, \eqref{l2}, \eqref{l2first}, and \eqref{l3nablaf},
\begin{align*}
    &\|\eqref{jm1.4-2}+\eqref{jm1.5-2}\|_{L^2}\\
    \lesssim & 2^{2M-k-2k_{1,+}+\gamma k_1-2k_{2,+}+\gamma k_2}\big(\min\{2^{-M-\gamma k_1-M-k_2/2},(1+2^{M/2+k_1})2^{3k_2/2-M-k_1/2-\gamma k_1+k_2/2}\}\e_1^2\\
    &+\min\{2^{-M/2+\gamma M/8-k_1+\gamma k_1/4-2M-k_2/2},2^{3k_2/2-k_1/2-M+k_2/2}\}\big)\e_1^2\\
    \lesssim & 2^{-2k_++\gamma k-2k_{2,+}+\gamma k_2}\big(\min\{2^{-k-\gamma k-k_2/2},(1+2^{M/2+k})2^{M-3k/2-\gamma k+2k_2}\}\e_1^2\\
    &+\min\{2^{-M/2+\gamma M/8-2k+\gamma k/4-k_2/2},2^{M-3k/2+2k_2}\}\big)\e_1^2\\
    \lesssim & (1+2^{M/2+k})2^{-2k_++\gamma k-\gamma k_--2k_{2,+}+\gamma k_2}\min\{2^{-M/2-2k-k_2/2},2^{M-3k/2+2k_2}\}\e_1^2.
\end{align*}
Thus, we have
\begin{align*}
    &\|\nabla_\xi F^1\|_{L^2}+\|\nabla_\xi F^2\|_{L^2}\\
    \lesssim &(1+2^{M/2+k})2^{-2k_++\gamma k-\gamma k_--2k_{2,+}+\gamma k_2}\min\{(1+2^{\gamma M/4+\gamma k/2})2^{-M/2-2k-k_2/2},2^{M-k/2+k_2}\}\e_1^2.
\end{align*}
When $M+2k\leq 0$, using Lemma \ref{dalpha} and \eqref{jm1bound1}, we obtain
\begin{align*}
    \|D^\alpha_\xi F^1\|_{L^2}+\|D^\alpha_\xi F^2\|_{L^2}
    &\leq \|F^1\|_{L^2}^{1-\alpha}\|\nabla_\xi F^1\|_{L^2}^\alpha+\|F^2\|_{L^2}^{1-\alpha}\|\nabla_\xi F^2\|_{L^2}^\alpha\\
    &\lesssim 2^{-2k_{+}+\gamma k-\gamma k_--2k_{2,+}+\gamma k_2}\min\{2^{-M/2-k-\alpha k-k_2/2},2^{M+2k_2-k/2-\alpha k_2}\}\e_1^2.
\end{align*}
And when $M+2k> 0$, we use Lemma \ref{dalpha} and \eqref{jm1bound2} to get
\begin{align*}
    &\|D^\alpha_\xi F^1\|_{L^2}+\|D^\alpha_\xi F^2\|_{L^2}\\
    \lesssim & 2^{-2k_++\gamma k-\gamma k_--2k_{2,+}+\gamma k_2}\min\{2^{-M+\alpha M+\gamma M/4+\gamma k/2-2k+\alpha k-k_2/2},2^{M/2+\alpha M-3k/2+2k_2+\alpha (2k-k_2)}\}\e_1^2\\
    \lesssim & 2^{-2k_++\gamma k-\gamma k_--2k_{2,+}+\gamma k_2}\min\{2^{-M/2+\gamma M/4+\gamma k/2-2k+\alpha k-k_2/2},2^{M-k+3k_2/2+\alpha k}\}\e_1^2.
\end{align*}
The last line uses the fact that $\alpha (k-k_2)\leq k/2-k_2/2$, since $k>k_2$ for $(k_1,k_2)\in\chi^1_k$ and $\alpha<1/2$.
\end{proof}
Since the denominator in the identity
$$e^{it\phi(\xi,\eta)}=\sum_{n=1}^3\frac{\partial_{\eta_n}\phi(\xi,\eta)}{it|\nabla_{\eta}\phi(\xi,\eta)|^2}\partial_{\eta_n}e^{it\phi(\xi,\eta)}$$
has a lower bound, $|\nabla_\eta\phi(\xi,\eta)|\gtrsim 2^{k_1}$, for all $(k_1,k_2)\in\chi^1_k\cup\chi^2_k$, we will employ it to handle the general case when $(k_1,k_2)\in\chi^2_k$ and the exception case in Lemma \ref{jm1} when $(k_1,k_2)\in\chi^1_k$ and $c_l=c_m$.
\begin{lem}
\label{sp1}
Suppose $t_1,t_2\in[2^{M-1},2^M]$ and $\sup_{t\in[1,T]}\|f_l\|_{Z}\leq \e_1$. If either $(k_1,k_2)\in\chi^2_k$ or $(k_1,k_2)\in\chi^1_k$ and $c_l=c_m$, then
\begin{align*}
    \|D^\alpha I^{M,1}_{k,k_1,k_2}\|_{L^2}
    \lesssim & 2^{-2k_{1,+}+\gamma k_1-2k_{2,+}+\gamma k_2-\alpha k-k_1/2}\min\{2^{-M/4-k_2/2},2^{M+3\min\{k,k_2\}/2+k_2/2}\}\e_1^2.
\end{align*}
If either $(k_1,k_2)\in\chi^2_k\cup\chi^3_k$ or $(k_1,k_2)\in\chi^1_k$ and $c_l=c_m$, then
\begin{align*}
    \|D^\alpha J^{M,1}_{k,k_1,k_2}\|_{L^2}
    \lesssim & 2^{-2k_{1,+}+\gamma k_1-2k_{2,+}+\gamma k_2-\alpha k-k_1/2}\min\{2^{-M/4-k_2/2},2^{M+3\min\{k,k_2\}/2+k_2/2}\}\e_1^2.
\end{align*}
\end{lem}
\begin{proof}
We will use the interpolation result in Lemma \ref{si} to find bounds on the fractional derivatives. For $I^{M,1}_{k,k_1,k_2}$ and $J^{M,1}_{k,k_1,k_2}$, we consider the families of operators on $\{z\in\C:0\leq \Re(z)\leq 1\}$,
\begin{align*}
    T_{z} \hat{g} =D^z\int_{t_1}^{t_2}\int_{\R^3}e^{it\phi(\xi,\eta)}\hat{g}_{k_1}(t,\xi-\eta)\hat{f}_{n,k_2}(t,\eta)\psi_k(\xi)d\eta dt
\end{align*}
and
\begin{align*}
    T^q_{z} \hat{g} =D^z\int_{t_1}^{t_2}\int_{\R^3}e^{it\phi(\xi,\eta)}q(\xi-\eta,\eta)\hat{g}_{k_1}(t,\xi-\eta)\hat{f}_{n,k_2}(t,\eta)\psi_k(\xi)d\eta dt,
\end{align*}
respectively. We shall show for each $y\in\R$, 
$$T_{0+iy},T^q_{0+iy}:L^{\infty}_t([2^{M-1},2^M])L^2_\xi\rightarrow L^2_\xi$$ and $$T_{1+iy},T^q_{1+iy}:L^{\infty}_t([2^{M-1},2^M])H^1_\xi\rightarrow L^2_\xi$$ are bounded operators.\\
By Lemma \ref{strichartz}, Lemma \ref{dualitycomp}, Lemma \ref{sobolev space}, and \eqref{l2}, we have
\begin{align*}
    &\|T_{0+iy} \hat{g}\|_{L^2}+\|T^q_{0+iy} \hat{g}\|_{L^2}\\
    =&\bigg\|\int_{t_1}^{t_2}\int_{\R^3}e^{it\phi(\xi,\eta)}\hat{g}_{k_1}(t,\xi-\eta)\hat{f}_{n,k_2}(t,\eta)\psi_k(\xi)d\eta dt\bigg\|_{L^2}\\
    &+\bigg\|\int_{t_1}^{t_2}\int_{\R^3}e^{it\phi(\xi,\eta)}q(\xi-\eta,\eta)\hat{g}_{k_1}(t,\xi-\eta)\hat{f}_{n,k_2}(t,\eta)\psi_k(\xi)d\eta dt\bigg\|_{L^2}\\
    \lesssim &(1+2^{\e k_-})\min\{ 2^{-M/4}\|\hat{g}_{k_1}\|_{L^\infty_t([2^{M-1},2^M])L^2_\xi}\|\hat{f}_{k_2}\|_{L^\infty_t([2^{M-1},2^M])H^1_\xi},\\
    &\qquad\qquad\qquad\qquad 2^{M+3\min\{k,k_2\}/2}\|e^{ic_mt\la}g_{k_1}\|_{L^\infty_t([2^{M-1},2^M])L^2_x}\|e^{ic_nt\la}f_{n,k_2}\|_{L^\infty_t([2^{M-1},2^M])L^2_x}\}\\
    \lesssim & \min\{2^{-M/4}\|\hat{f}_{n,k_2}\|_{L^\infty_t([2^{M-1},2^M])H^1_\xi},2^{M+3\min\{k,k_2\}/2}\|f_{n,k_2}\|_{L^\infty_t([2^{M-1},2^M])L^2_x}\}\|\hat{g}_{k_1}\|_{L^\infty_t([2^{M-1},2^M])L^2_\xi}\\
    \lesssim & 2^{-2k_{2,+}+\gamma k_2}\min\{2^{-M/4-k_2/2},2^{M+3\min\{k,k_2\}/2+k_2/2}\}\e_1\|\hat{g}\|_{L^\infty_t([2^{M-1},2^M])L^2_\xi}.
\end{align*}
Next, we observe
\begin{align}
    \|T_{1+iy} \hat{g}\|_{L^2}
    \leq &\bigg\|\partial_{\xi_l}\int_{t_1}^{t_2}\int_{\R^3}e^{it\phi(\xi,\eta)}\hat{g}_{k_1}(t,\xi-\eta)\hat{f}_{n,k_2}(t,\eta)\psi_k(\xi)d\eta dt\bigg\|_{L^2}\nonumber\\
    \leq &\bigg\|\int_{t_1}^{t_2}\int_{\R^3}e^{it\phi(\xi,\eta)}\partial_{\xi_l}\hat{g}_{k_1}(t,\xi-\eta)\hat{f}_{n,k_2}(t,\eta)\psi_k(\xi)d\eta dt\bigg\|_{L^2}\label{stein1.1}\\
    &+ \bigg\|\int_{t_1}^{t_2}\int_{\R^3}e^{it\phi(\xi,\eta)}\hat{g}_{k_1}(t,\xi-\eta)\hat{f}_{n,k_2}(t,\eta)\partial_{\xi_l}\psi_k(\xi)d\eta dt\bigg\|_{L^2}\label{stein1.12}\\
    & + \bigg\|\int_{t_1}^{t_2}\int_{\R^3}e^{it\phi(\xi,\eta)}it\partial_{\xi_l}\phi(\xi,\eta) \hat{g}_{k_1}(t,\xi-\eta)\hat{f}_{k_2}(t,\eta)\psi_k(\xi)d\eta dt\bigg\|_{L^2}\label{stein1.2}
\end{align}
and
\begin{align}
    \|T^q_{1+iy} \hat{g}\|_{L^2}
    \leq &\bigg\|\partial_{\xi_l}\int_{t_1}^{t_2}\int_{\R^3}e^{it\phi(\xi,\eta)}q(\xi-\eta,\eta)\hat{g}_{k_1}(t,\xi-\eta)\hat{f}_{n,k_2}(t,\eta)\psi_k(\xi)d\eta dt\bigg\|_{L^2}\nonumber\\
    \leq &\bigg\|\int_{t_1}^{t_2}\int_{\R^3}e^{it\phi(\xi,\eta)}q(\xi-\eta,\eta)\partial_{\xi_l}\hat{g}_{k_1}(t,\xi-\eta)\hat{f}_{n,k_2}(t,\eta)\psi_k(\xi)d\eta dt\bigg\|_{L^2}\label{kstein1.1}\\
    &+ \bigg\|\int_{t_1}^{t_2}\int_{\R^3}e^{it\phi(\xi,\eta)}\partial_{\xi_l} q(\xi-\eta,\eta)\hat{g}_{k_1}(t,\xi-\eta)\hat{f}_{n,k_2}(t,\eta)\psi_k(\xi)d\eta dt\bigg\|_{L^2}\label{kstein1.11}\\
    &+ \bigg\|\int_{t_1}^{t_2}\int_{\R^3}e^{it\phi(\xi,\eta)}q(\xi-\eta,\eta)\hat{g}_{k_1}(t,\xi-\eta)\hat{f}_{n,k_2}(t,\eta)\partial_{\xi_l}\psi_k(\xi)d\eta dt\bigg\|_{L^2}\label{kstein1.12}\\
    & + \bigg\|\int_{t_1}^{t_2}\int_{\R^3}e^{it\phi(\xi,\eta)}it\partial_{\xi_l}\phi(\xi,\eta) q(\xi-\eta,\eta)\hat{g}_{k_1}(t,\xi-\eta)\hat{f}_{n,k_2}(t,\eta)\psi_k(\xi)d\eta dt\bigg\|_{L^2}\label{kstein1.2}.
\end{align}
Since we have \eqref{xi-2eta} when $(k_1,k_2)\in\chi^1_k\cup\chi^2_k$, we can perform an integration by parts using the identity $e^{it\phi(\xi,\eta)}=\sum_{n=1}^3\frac{\partial_{\eta_n}\phi(\xi,\eta)}{it|\nabla_{\eta}\phi(\xi,\eta)|^2}\partial_{\eta_n}e^{it\phi(\xi,\eta)}$ on the term $\eqref{stein1.2}$, 
\begin{align}
    &\int_{t_1}^{t_2}\int_{\R^3}e^{it\phi(\xi,\eta)}it\partial_{\xi_l}\phi(\xi,\eta) \hat{g}_{k_1}(t,\xi-\eta)\hat{f}_{n,k_2}(t,\eta)\psi_k(\xi)d\eta dt\nonumber\\
    = & \int_{t_1}^{t_2}\int_{\R^3}\partial_{\eta_n}e^{it\phi(\xi,\eta)}\frac{\partial_{\xi_l}\phi(\xi,\eta)\partial_{\eta_n}\phi(\xi,\eta)}{|\nabla_\eta\phi(\xi,\eta)|^2}\hat{g}_{k_1}(t,\xi-\eta)\hat{f}_{n,k_2}(t,\eta)\psi_k(\xi)d\eta dt\nonumber\\
    = & -\int_{t_1}^{t_2}\int_{\R^3}e^{it2\eta\cdot(\xi-\eta)}\partial_{\eta_n}\frac{\partial_{\xi_l}\phi(\xi,\eta)\partial_{\eta_n}\phi(\xi,\eta)}{|\nabla_\eta\phi(\xi,\eta)|^2}\hat{g}_{k_1}(t,\xi-\eta)\hat{f}_{n,k_2}(t,\eta)\psi_k(\xi)d\eta dt\label{stein1.3}\\
    & -\int_{t_1}^{t_2}\int_{\R^3}e^{it2\eta\cdot(\xi-\eta)}\frac{\partial_{\xi_l}\phi(\xi,\eta)\partial_{\eta_n}\phi(\xi,\eta)}{|\nabla_\eta\phi(\xi,\eta)|^2}\partial_{\eta_n}\hat{g}_{k_1}(t,\xi-\eta)\hat{f}_{n,k_2}(t,\eta)\psi_k(\xi)d\eta dt\label{stein1.4}\\
    & -\int_{t_1}^{t_2}\int_{\R^3}e^{it2\eta\cdot(\xi-\eta)}\frac{\partial_{\xi_l}\phi(\xi,\eta)\partial_{\eta_n}\phi(\xi,\eta)}{|\nabla_\eta\phi(\xi,\eta)|^2}\hat{g}_{k_1}(t,\xi-\eta)\partial_{\eta_n}\hat{f}_{n,k_2}(t,\eta)\psi_k(\xi)d\eta dt\label{stein1.5}.
\end{align}
Furthermore, for $c_m+c_n=0$, we get $\phi(\xi,\eta)=(c_l-c_m)|\xi|^2+2c_m\xi\cdot\eta$ and $\nabla_\eta\phi(\xi,\eta) = 2c_m\xi$. We may perform an integration by parts using the identity $e^{it\phi(\xi,\eta)}=\sum_{n=1}^3\frac{\xi_n}{it2c_m|\xi|^2}\partial_{\eta_n}e^{it\phi(\xi,\eta)}$ on the term $\eqref{kstein1.2}$ regardless of $k_1$ and $k_2$,
\begin{align}
    &\int_{t_1}^{t_2}\int_{\R^3}e^{it\phi(\xi,\eta)}it\partial_{\xi_l}\phi(\xi,\eta) q(\xi-\eta,\eta)\hat{g}_{k_1}(t,\xi-\eta)\hat{f}_{n,k_2}(t,\eta)\psi_k(\xi)d\eta dt\nonumber\\
    = & \int_{t_1}^{t_2}\int_{\R^3}\partial_{\eta_n}e^{it\phi(\xi,\eta)}\frac{\xi_n\partial_{\xi_l}\phi(\xi,\eta)}{c_m|\xi|^2}q(\xi-\eta,\eta)\hat{g}_{k_1}(t,\xi-\eta)\hat{f}_{n,k_2}(t,\eta)\psi_k(\xi)d\eta dt\nonumber\\
    = & -\int_{t_1}^{t_2}\int_{\R^3} e^{it\phi(\xi,\eta)}\partial_{\eta_n}\big(\frac{\xi_n\partial_{\xi_l}\phi(\xi,\eta)}{c_m|\xi|^2}q(\xi-\eta,\eta)\big)\hat{g}_{k_1}(t,\xi-\eta)\hat{f}_{n,k_2}(t,\eta)\psi_k(\xi)d\eta dt\label{kstein1.3}\\
    & -\int_{t_1}^{t_2}\int_{\R^3} e^{it\phi(\xi,\eta)}\frac{\xi_n\partial_{\xi_l}\phi(\xi,\eta)}{c_m|\xi|^2}q(\xi-\eta,\eta)\partial_{\eta_n}\hat{g}_{k_1}(t,\xi-\eta)\hat{f}_{n,k_2}(t,\eta)\psi_k(\xi)d\eta dt\label{kstein1.4}\\
    & -\int_{t_1}^{t_2}\int_{\R^3} e^{it\phi(\xi,\eta)}\frac{\xi_n\partial_{\xi_l}\phi(\xi,\eta)}{c_m|\xi|^2}q(\xi-\eta,\eta)\hat{g}_{k_1}(t,\xi-\eta)\partial_{\eta_n}\hat{f}_{n,k_2}(t,\eta)\psi_k(\xi)d\eta dt\label{kstein1.5}.
\end{align}
By Lemma \ref{strichartz}, Lemma \ref{dualitycomp}, Lemma \ref{sobolev space}, Lemma \ref{chi+eta2}, \eqref{l2}, and \eqref{addder}, we know
\begin{align*}
    &\eqref{stein1.1}+\|\eqref{stein1.4}\|_{L^2}\\
    \lesssim & \big(1+\|\F^{-1}\frac{\partial_{\xi_l}\phi(\xi,\eta)\partial_{\eta_n}\phi(\xi,\eta)}{|\nabla_\eta\phi(\xi,\eta)|^2}\Tilde{\psi}_{k_1}(\nabla_\eta\phi(\xi,\eta))\Tilde{\psi}_{k_2}(\eta)\|_{L^1}\big)\times\\
    &\times\min\{2^{-M/4}\|\nabla_\xi\hat{g}_{k_1}\|_{L^\infty_t([2^{M-1},2^M])L^2_\xi}\|\hat{f}_{k_2}\|_{L^\infty_t([2^{M-1},2^M])H^1_\xi},\\
    &\qquad\qquad\qquad\qquad 2^{M+3\min\{k,k_2\}/2}\|e^{ic_mt\la}\F^{-1}\nabla_\xi\hat{g}_{k_1}\|_{L^\infty_t([2^{M-1},2^M])L^2_x}\|e^{ic_nt\la}f_{n,k_2}\|_{L^\infty_t([2^{M-1},2^M])L^2_x}\}\\
    \lesssim & \min\{2^{-M/4}\|\hat{f}_{n,k_2}\|_{L^\infty_t([2^{M-1},2^M])H^1_\xi},2^{M+3\min\{k,k_2\}/2}\|f_{n,k_2}\|_{L^\infty_t([2^{M-1},2^M])L^2_x}\}\|\nabla_\xi\hat{g}_{k_1}\|_{L^\infty_t([2^{M-1},2^M])L^2_\xi}\\
    \leq & 2^{-2k_{2,+}+\gamma k_2}\min\{2^{-M/4-k_2/2},2^{M+3\min\{k,k_2\}/2+k_2/2}\}\e_1\|\hat{g}\|_{L^\infty_t([2^{M-1},2^M])H^1_\xi}
\end{align*}
and
\begin{align*}
    &\eqref{stein1.12}+\|\eqref{stein1.3}\|_{L^2}\\
    \lesssim &\big(\|\nabla_\xi\psi_k\|_{L^\infty}+\|\F^{-1}\partial_{\eta_n}\frac{\partial_{\xi_l}\phi(\xi,\eta)\partial_{\eta_n}\phi(\xi,\eta)}{|\nabla_\eta\phi(\xi,\eta)|^2}\Tilde{\psi}_{k_1}(\nabla_\eta\phi(\xi,\eta))\Tilde{\psi}_{k_2}(\eta)\|_{L^1}\big)\times\\
    &\times \min\{2^{-M/4}\|\hat{g}_{k_1}\|_{L^\infty_t([2^{M-1},2^M])H^1_\xi}\|\hat{f}_{n,k_2}\|_{L^\infty_t([2^{M-1},2^M])L^2_\xi},\\
    &\qquad\qquad\qquad\qquad 2^{M+3\min\{k,k_2\}/2}\|e^{ic_mt\la}g_{k_1}\|_{L^\infty_t([2^{M-1},2^M])L^2_x}\|e^{ic_nt\la}f_{n,k_2}\|_{L^\infty_t([2^{M-1},2^M])L^2_x}\}\\
    \lesssim & (2^{-k}+2^{-k_1})\min\{2^{-M/4}\|\hat{g}_{k_1}\|_{L^\infty_t([2^{M-1},2^M])H^1_\xi}\|\hat{f}_{n,k_2}\|_{L^\infty_t([2^{M-1},2^M])L^2_\xi},\\
    &\qquad\qquad\qquad\qquad 2^{M+3\min\{k,k_2\}/2}\|\hat{g}_{k_1}\|_{L^\infty_t([2^{M-1},2^M])L^2_\xi}\|f_{n,k_2}\|_{L^\infty_t([2^{M-1},2^M])L^2_x}\}\\
    \lesssim & 2^{-k}\min\{2^{-M/4}\|\hat{f}_{n,k_2}\|_{L^\infty_t([2^{M-1},2^M])L^2_\xi},2^{M+3\min\{k,k_2\}/2+k_1}\|f_{n,k_2}\|_{L^\infty_t([2^{M-1},2^M])L^2_x}\}\|\hat{g}_{k_1}\|_{L^\infty_t([2^{M-1},2^M])H^1_\xi}\\
    \lesssim & 2^{-k-2k_{2,+}+\gamma k_2}\min\{2^{-M/4+k_2/2},2^{M+3\min\{k,k_2\}/2+k_1+k_2/2}\}\e_1\|\hat{g}\|_{L^\infty_t([2^{M-1},2^M])H^1_\xi}.
\end{align*}
When $(k_1,k_2)\in\chi^1_k$, we have $c_l=c_m$ and $\phi(\xi,\eta)=2c_m\xi\cdot\eta-(c_m+c_n)|\eta|^2$, so
\begin{align*}
    &\|\eqref{stein1.5}\|_{L^2}\\
    \lesssim & \|\F^{-1}\frac{\partial_{\xi_l}\phi(\xi,\eta)\partial_{\eta_n}\phi(\xi,\eta)}{|\nabla_\eta\phi(\xi,\eta)|^2}\Tilde{\psi}_{k_1}(\nabla_\eta\phi(\xi,\eta))\Tilde{\psi}_{k_2}(\eta)\|_{L^1}\times\\
    &\times \min\{2^{-M/4}\|\hat{g}_{k_1}\|_{L^\infty_t([2^{M-1},2^M])H^1_\xi}\|\nabla_\xi\hat{f}_{n,k_2}\|_{L^\infty_t([2^{M-1},2^M])L^2_\xi},\\
    &\qquad\qquad\qquad\qquad 2^{m+3\min\{k,k_2\}/2}\|e^{ic_mt\la}g_{k_1}\|_{L^\infty_t([2^{M-1},2^M])L^2_x}\|e^{ic_nt\la}\F^{-1}\nabla_\xi\hat{f}_{n,k_2}\|_{L^\infty_t([2^{M-1},2^M])L^2_x}\}\\
    \lesssim & 2^{k_2-k_1}\min\{2^{-M/4}\|\hat{g}_{k_1}\|_{L^\infty_t([2^{M-1},2^M])H^1_\xi}\|\nabla_\xi\hat{f}_{n,k_2}\|_{L^\infty_t([2^{M-1},2^M])L^2_\xi},\\
    &\qquad\qquad\qquad\qquad 2^{M+3\min\{k,k_2\}/2}\|\hat{g}_{k_1}\|_{L^\infty_t([2^{M-1},2^M])L^2_\xi}\|\nabla_\xi\hat{f}_{n,k_2}\|_{L^\infty_t([2^{M-1},2^M])L^2_\xi}\}\\
    \lesssim & 2^{k_2-k_1}\min\{2^{-M/4}\|\nabla_\xi\hat{f}_{n,k_2}\|_{L^\infty_t([2^{M-1},2^M])L^2_\xi},\\
    &\qquad\qquad\qquad\qquad 2^{M+3\min\{k,k_2\}/2+k_1}\|\nabla_\xi\hat{f}_{n,k_2}\|_{L^\infty_t([2^{M-1},2^M])L^2_\xi}\}\|\hat{g}_{k_1}\|_{L^\infty_t([2^{M-1},2^M])H^1_\xi}\\
    \lesssim & 2^{-2k_{2,+}+\gamma k_2}\min\{2^{-M/4-k_1+k_2/2},2^{M+3\min\{k,k_2\}/2+k_2/2}\}\e_1\|\hat{g}\|_{L^\infty_t([2^{M-1},2^M])H^1_\xi}.
\end{align*}
Repeat the same estimate using \eqref{q}. Since either $(k_1,k_2)\in\chi^2_k$ or $\phi(\xi,\eta)=2c_m\xi\cdot\eta-(c_m+c_n)|\eta|^2$, we get
\begin{align*}
    &\eqref{kstein1.1}+\|\eqref{kstein1.4}\|_{L^2}\\
    \lesssim & \big(\|\F^{-1}q(\xi-\eta,\eta)\Tilde{\psi}_{k_1}(\xi-\eta)\Tilde{\psi}_{k_2}(\eta)\Tilde{\psi}_k(\xi)\|_{L^1}+\|\F^{-1}\frac{\xi_n\partial_{\xi_l}\phi(\xi,\eta)}{c_m|\xi|^2}q(\xi-\eta,\eta)\Tilde{\psi}_{k_1}(\xi-\eta)\Tilde{\psi}_{k_2}(\eta)\Tilde{\psi}_{k}(\xi)\|_{L^1}\big)\times\\
    &\times\min\{2^{-M/4}\|\nabla_\xi\hat{g}_{k_1}\|_{L^\infty_t([2^{M-1},2^M])L^2_\xi}\|\hat{f}_{k_2}\|_{L^\infty_t([2^{M-1},2^M])H^1_\xi},\\
    &\qquad\qquad\qquad\qquad 2^{M+3\min\{k,k_2\}/2}\|e^{ic_mt\la}\F^{-1}\nabla_\xi\hat{g}_{k_1}\|_{L^\infty_t([2^{M-1},2^M])L^2_x}\|e^{ic_nt\la}f_{n,k_2}\|_{L^\infty_t([2^{M-1},2^M])L^2_x}\}\\
    \lesssim &(2^{\e k_-}+2^{\e k_-+k_2-k}) 2^{-2k_{2,+}+\gamma k_2}\min\{2^{-M/4-k_2/2},2^{M+3\min\{k,k_2\}/2+k_2/2}\}\e_1\|\hat{g}\|_{L^\infty_t([2^{M-1},2^M])H^1_\xi},
\end{align*}
\begin{align*}
    &\eqref{kstein1.11}+\eqref{kstein1.12}+\|\eqref{kstein1.3}\|_{L^2}\\
    \lesssim &\big(\|\F^{-1}\nabla_\xi q(\xi-\eta,\eta)\Tilde{\psi}_{k_1}(\xi-\eta)\Tilde{\psi}_{k_2}(\eta)\Tilde{\psi}_{k}(\xi)\|_{L^1}
    +\|\nabla_\xi\psi_k\|_{L^\infty}\|\F^{-1}q(\xi-\eta,\eta)\Tilde{\psi}_{k_1}(\xi-\eta)\Tilde{\psi}_{k_2}(\eta)\Tilde{\psi}_{k}(\xi)\|_{L^1}+\\
    &+\|\F^{-1}\partial_{\eta_n}\big(\frac{\xi_n\partial_{\xi_l}\phi(\xi,\eta)}{c_m|\xi|^2}q(\xi-\eta,\eta)\big)\Tilde{\psi}_{k_1}(\xi-\eta)\Tilde{\psi}_{k_2}(\eta)\Tilde{\psi}_k(\xi)\|_{L^1}\big)\times\\
    &\times \min\{2^{-M/4}\|\hat{g}_{k_1}\|_{L^\infty_t([2^{M-1},2^M])H^1_\xi}\|\hat{f}_{n,k_2}\|_{L^\infty_t([2^{M-1},2^M])L^2_\xi},\\
    &\qquad\qquad\qquad\qquad 2^{M+3\min\{k,k_2\}/2}\|e^{ic_mt\la}g_{k_1}\|_{L^\infty_t([2^{M-1},2^M])L^2_x}\|e^{ic_nt\la}f_{n,k_2}\|_{L^\infty_t([2^{M-1},2^M])L^2_x}\}\\
    \lesssim & 2^{\e k_--k-2k_{2,+}+\gamma k_2}\min\{2^{-M/4+k_2/2},2^{M+3\min\{k,k_2\}/2+k_1+k_2/2}\}\e_1\|\hat{g}\|_{L^\infty_t([2^{M-1},2^M])H^1_\xi},
\end{align*}
and
\begin{align*}
    &\|\eqref{kstein1.5}\|_{L^2}\\
    \lesssim & \|\F^{-1}\frac{\xi_n\partial_{\xi_l}\phi(\xi,\eta)}{c_m|\xi|^2}\Tilde{\psi}_{k_1}(\xi-\eta)\Tilde{\psi}_{k_2}(\eta)\Tilde{\psi}_k(\xi)\|_{L^1}\times\\
    &\times \min\{2^{-M/4}\|\hat{g}_{k_1}\|_{L^\infty_t([2^{M-1},2^M])H^1_\xi}\|\nabla_\xi\hat{f}_{n,k_2}\|_{L^\infty_t([2^{M-1},2^M])L^2_\xi},\\
    &\qquad\qquad\qquad\qquad 2^{m+3\min\{k,k_2\}/2}\|e^{ic_mt\la}g_{k_1}\|_{L^\infty_t([2^{M-1},2^M])L^2_x}\|e^{ic_nt\la}\F^{-1}\nabla_\xi\hat{f}_{n,k_2}\|_{L^\infty_t([2^{M-1},2^M])L^2_x}\}\\
    \lesssim & 2^{\e k_--k-2k_{2,+}+\gamma k_2}\min\{2^{-M/4+k_2/2},2^{M+3\min\{k,k_2\}/2+k_1+k_2/2}\}\e_1\|\hat{g}\|_{L^\infty_t([2^{M-1},2^M])H^1_\xi}.
\end{align*}
Thus, we can conclude
\begin{align*}
    \|T_{0+iy}\hat{g}\|_{L^2}\lesssim  2^{-2k_{2,+}+\gamma k_2}\min\{2^{-M/4-k_2/2},2^{M+3\min\{k,k_2\}/2+k_2/2}\}\e_1\|\hat{g}\|_{L^\infty_t([2^{M-1},2^M])L^2_\xi},
\end{align*}
\begin{align*}
    \|T_{1+iy}\hat{g}\|_{L^2}
    \lesssim  &2^{-2k_{2,+}+\gamma k_2}\min\{2^{-M/4-k_2/2},2^{M+3\min\{k,k_2\}/2+k_2/2}\}\e_1\|\hat{g}\|_{L^\infty_t([2^{M-1},2^M])H^1_\xi}\\
    &+2^{-2k_{2,+}+\gamma k_2}\min\{2^{-M/4-k+k_2/2},2^{M+3\min\{k,k_2\}/2+k_1-k+k_2/2}\}\e_1\|\hat{g}\|_{L^\infty_t([2^{M-1},2^M])H^1_\xi}\\
    &+2^{-2k_{2,+}+\gamma k_2}\min\{2^{-M/4-k_1+k_2/2},2^{M+3\min\{k,k_2\}/2+k_2/2}\}\e_1\|\hat{g}\|_{L^\infty_t([2^{M-1},2^M])H^1_\xi}\\
    \lesssim &2^{-2k_{2,+}+\gamma k_2+k_1-k}\min\{2^{-M/4-k_2/2},2^{M+3\min\{k,k_2\}/2+k_2/2}\}\e_1\|\hat{g}\|_{L^\infty_t([2^{M-1},2^M])H^1_\xi},
\end{align*}
and
\begin{align*}
    \|T^q_{0+iy}\hat{g}\|_{L^2}\lesssim  2^{\e k_--2k_{2,+}+\gamma k_2}\min\{2^{-M/4-k_2/2},2^{M+3\min\{k,k_2\}/2+k_2/2}\}\e_1\|\hat{g}\|_{L^\infty_t([2^{M-1},2^M])L^2_\xi},
\end{align*}
\begin{align*}
    \|T^q_{1+iy}\hat{g}\|_{L^2}
    \lesssim &2^{\e k_--2k_{2,+}+\gamma k_2+k_1-k}\min\{2^{-M/4-k_2/2},2^{M+3\min\{k,k_2\}/2+k_2/2}\}\e_1\|\hat{g}\|_{L^\infty_t([2^{M-1},2^M])H^1_\xi}.
\end{align*}
By the variation of the Stein's interpolation theorem in Lemma \ref{si}, for any $\alpha\in(0,1)$,
\begin{align*}
    \|T_{\alpha}\hat{g}\|_{L^2}+\|T^q_{\alpha}\hat{g}\|_{L^2}\lesssim 2^{-2k_{2,+}+\gamma k_2+\alpha (k_1-k)}\min\{2^{-m/4-k_2/2},2^{m+3\min\{k,k_2\}/2+k_2/2}\}\e_1\|\hat{g}\|_{L^\infty_t([2^{M-1},2^M])H^\alpha_\xi}.
\end{align*}
Taking $g=\F^{-1}\nabla_\xi\hat{f}_m(t,\xi)\Tilde{\psi}_{k_1}(\xi)$ and using \eqref{sobolevnorms} finish the proof.
\end{proof}

Now, the only case left is $\|D^{\alpha}_\xi I^{M,1}_{k,k_1,k_2}\|_{L^2}$ for $(k_1,k_2)\in\chi^3_k$. We will start with proving a slightly more general result on the bilinear integrals when $|\xi|\sim|\eta|\sim|\xi-\eta|$. The proof will use Lemma \ref{si} to handle the fractional derivatives.
\begin{lem}
\label{chi3*}
Suppose $m_d:\R^3\times\R^3\rightarrow\R$ is homogeneous of degree $d$ and smooth on $\R^3\times\R^3\setminus\{(0,0)\}$. Given $(k_1,k_2)\in\chi^3_k$ and $\sup_{t\in[1,T]}\|f_l\|_{Z}\leq \e_1$, we have 
\begin{equation}
\label{chi3*1}
    \begin{aligned}
    &\bigg\|D^\alpha\int_{t_1}^{t_2}\int_{\R^3}e^{it\phi(\xi,\eta)}m_d(\xi,\eta)\hat{g}_{k_1}(t,\xi-\eta)\hat{f}_{n,k_2}(t,\eta)\psi_k(\xi)d\eta dt\bigg\|_{L^2}\\
    \lesssim & (1+2^{\alpha(M+2k)})2^{dk-2k_{+}+\gamma k}\min\{2^{-M/2-k+(1-\alpha)(\gamma M/8+\gamma k/4)},2^{M+2k}\}\e_1\|\hat{g}_{k_1}\|_{L^\infty_t([2^{M-1},2^M])H^\alpha_\xi}
    \end{aligned}
\end{equation}
and
\begin{equation}
\label{chi3*2}
    \begin{aligned}
    &\bigg\|D^\alpha\int_{t_1}^{t_2}\int_{\R^3}e^{it\phi(\xi,\eta)}m_d(\xi,\eta)\hat{f}_{m,k_1}(t,\xi-\eta)\hat{g}_{k_2}(t,\eta)\psi_k(\xi)d\eta dt\bigg\|_{L^2}\\
    \lesssim & (1+2^{\alpha(M+2k)})2^{dk-2k_{+}+\gamma k}\min\{2^{-M/2-k+(1-\alpha)(\gamma M/8+\gamma k/4)},2^{M+2k}\}\e_1\|\hat{g}_{k_2}\|_{L^\infty_t([2^{M-1},2^M])H^\alpha_\xi},
    \end{aligned}
\end{equation}
for any $t_1,t_2\in[2^{M-1},2^M]$.
\end{lem}
\begin{proof}
Consider the following two families of operators on $\{z\in\C:0\leq \Re(z)\leq 1\}$, 
\begin{align*}
    T^1_{z} \hat{g} =D^z\int_{t_1}^{t_2}\int_{\R^3}e^{it\phi(\xi,\eta)}m_d(\xi,\eta)\hat{g}_{k_1}(t,\xi-\eta)\hat{f}_{n,k_2}(t,\eta)\psi_k(\xi)d\eta dt
\end{align*}
and
\begin{align*}
    T^2_{z} \hat{g} =D^z\int_{t_1}^{t_2}\int_{\R^3}e^{it\phi(\xi,\eta)}m_d(\xi,\eta)\hat{f}_{m,k_1}(t,\xi-\eta)\hat{g}_{k_2}(t,\eta)\psi_k(\xi)d\eta dt,
\end{align*}
which correspond to \eqref{chi3*1} and \eqref{chi3*2}.\\
We start with bounding $T^1_{0+iy}$ and $T^2_{0+iy}$.
Using the bilinear estimate $L^2\times L^\infty\rightarrow L^2$, Lemma \ref{chi,eta}, Lemma \ref{dualitycomp}, \eqref{l2}, and \eqref{linfinity}, we get
\begin{align*}
    &\|T^1_{0+iy} \hat{g}\|_{L^2}\\
    =&\bigg\|\int_{t_1}^{t_2}\int_{\R^3}e^{it\phi(\xi,\eta)}m_d(\xi,\eta)\hat{g}_{k_1}(t,\xi-\eta)\hat{f}_{n,k_2}(t,\eta)\psi_k(\xi)d\eta dt\bigg\|_{L^2}\\
    \lesssim &  2^{M}\|\F^{-1}m_d(\xi,\eta)\Tilde{\psi}_k(\xi)\Tilde{\psi}_{k_1}(\xi-\eta)\Tilde{\psi}_{k_2}(\eta)\|_{L^1}\sup_{t\in[2^{M-1},2^M]}\min\{ \|e^{ic_mt\la}g_{k_1}\|_{L^2_x}\|e^{ic_nt\la}f_{n,k_2}\|_{L^\infty_x},\\
    &\qquad\qquad\qquad\qquad 2^{3\min\{k,k_2\}/2}\|e^{ic_mt\la}g_{k_1}\|_{L^2_x}\|e^{ic_nt\la}f_{n,k_2}\|_{L^2_x}\}\\
    \lesssim & 2^{M+dk}\sup_{t\in[2^{M-1},2^M]}\min\{ \|e^{ic_nt\la}f_{n,k_2}\|_{L^\infty_x},2^{3\min\{k,k_2\}/2}\|f_{n,k_2}\|_{L^2_x}\}\|g_{k_1}\|_{L^2_x}\\
    \lesssim & 2^{M+dk-2k_{2,+}+\gamma k_2}\min\{2^{-3M/2+\gamma M/8-k_2+\gamma k_2/4},2^{3\min\{k,k_2\}/2+k_2/2}\}\e_1\|g_{k_1}\|_{L^\infty_t([2^{M-1},2^M])L^2_x}\\
    \lesssim & 2^{dk-2k_{+}+\gamma k}\min\{2^{-M/2+\gamma M/8-k+\gamma k/4},2^{M+2k}\}\e_1\|g\|_{L^\infty_t([2^{M-1},2^M])L^2_x},
\end{align*}
since $\chi^3_k=\{|k_1-k_2|\leq a, |k-k_1|\leq a+2\}$. Similarly, we also have
\begin{align*}
    &\|T^2_{0+iy} \hat{g}\|_{L^2}\\
    =&\bigg\|\int_{t_1}^{t_2}\int_{\R^3}e^{it\phi(\xi,\eta)}m_d(\xi,\eta)\hat{f}_{m,k_1}(t,\xi-\eta)\hat{g}_{k_2}(t,\eta)\psi_k(\xi)d\eta dt\bigg\|_{L^2}\\
    \lesssim &  2^{M}\|\F^{-1}m_d(\xi,\eta)\Tilde{\psi}_k(\xi)\Tilde{\psi}_{k_1}(\xi-\eta)\Tilde{\psi}_{k_2}(\eta)\|_{L^1}
    \sup_{t\in[2^{M-1},2^M]}\min\{ \|e^{ic_mt\la}f_{m,k_1}\|_{L^2_x}\|e^{ic_nt\la}g_{k_2}\|_{L^\infty_x},\\
    &\qquad\qquad\qquad\qquad 2^{3\min\{k,k_2\}/2}\|e^{ic_mt\la}f_{m,k_1}\|_{L^2_x}\|e^{ic_nt\la}g_{k_2}\|_{L^2_x}\}\\
    \lesssim & 2^{M+dk}\sup_{t\in[2^{M-1},2^M]}\min\{ \|e^{ic_mt\la}f_{m,k_1}\|_{L^\infty_x},2^{3\min\{k,k_2\}/2}\|f_{m,k_1}\|_{L^2_x}\}\|g_{k_2}\|_{L^2_x}\\
    \lesssim & 2^{M+dk-2k_{1,+}+\gamma k_1}\min\{2^{-3M/2+\gamma M/8-k_1+\gamma k_1/4},2^{3\min\{k,k_2\}/2+k_1/2}\}\e_1\|g_{k_2}\|_{L^\infty_t([2^{M-1},2^M])L^2_x}\\
    \lesssim & 2^{dk-2k_{+}+\gamma k}\min\{2^{-M/2+\gamma M/8-k+\gamma k/4},2^{M+2k}\}\e_1\|g\|_{L^\infty_t([2^{M-1},2^M])L^2_x}.
\end{align*}
Next, we observe
\begin{align}
    \|T^1_{1+iy} \hat{g}\|_{L^2}
    = &\bigg\|\partial_{\xi_m}\int_{t_1}^{t_2}\int_{\R^3}e^{it\phi(\xi,\eta)}m_d(\xi,\eta)\hat{g}_{k_1}(t,\xi-\eta)\hat{f}_{n,k_2}(t,\eta)\psi_k(\xi)d\eta dt \bigg\|_{L^2}\nonumber\\
    \leq &\bigg\|\int_{t_1}^{t_2}\int_{\R^3}e^{it\phi(\xi,\eta)}m_d(\xi,\eta)\partial_{\xi_m}\hat{g}_{k_1}(t,\xi-\eta)\hat{f}_{n,k_2}(t,\eta)\psi_k(\xi)d\eta dt\bigg\|_{L^2}\label{chi3*.1}\\
    &+ \bigg\|\int_{t_1}^{t_2}\int_{\R^3}e^{it\phi(\xi,\eta)}m_d(\xi,\eta)\hat{g}_{k_1}(t,\xi-\eta)\hat{f}_{n,k_2}(t,\eta)\partial_{\xi_m}\psi_k(\xi)d\eta dt\bigg\|_{L^2}\label{chi3*.12}\\
    & + \bigg\|\int_{t_1}^{t_2}\int_{\R^3}e^{it\phi(\xi,\eta)}it\partial_{\xi_m}\phi(\xi,\eta)m_d(\xi,\eta)\hat{g}_{k_1}(t,\xi-\eta)\hat{f}_{n,k_2}(t,\eta)\psi_k(\xi)d\eta dt\bigg\|_{L^2}\label{chi3*.2}
\end{align}
and
\begin{align}
    \|T^2_{1+iy} \hat{g}\|_{L^2}
    = &\bigg\|\partial_{\xi_m}\int_{t_1}^{t_2}\int_{\R^3}e^{it\phi(\xi,\eta)}m_d(\xi,\eta)\hat{f}_{m,k_1}(t,\xi-\eta)\hat{g}_{k_2}(t,\eta)\psi_k(\xi)d\eta dt\bigg\|_{L^2}\nonumber\\
    \leq &\bigg\|\int_{t_1}^{t_2}\int_{\R^3}e^{it\phi(\xi,\eta)}m_d(\xi,\eta)\partial_{\xi_m}\hat{f}_{m,k_1}(t,\xi-\eta)\hat{g}_{k_2}(t,\eta)\psi_k(\xi)d\eta dt\bigg\|_{L^2}\label{chi3*.1-2}\\
    &+ \bigg\|\int_{t_1}^{t_2}\int_{\R^3}e^{it\phi(\xi,\eta)}m_d(\xi,\eta)\hat{f}_{m,k_1}(t,\xi-\eta)\hat{g}_{k_2}(t,\eta)\partial_{\xi_m}\psi_k(\xi)d\eta dt\bigg\|_{L^2}\label{chi3*.12-2}\\
    & + \bigg\|\int_{t_1}^{t_2}\int_{\R^3}e^{it\phi(\xi,\eta)}it\partial_{\xi_m}\phi(\xi,\eta)m_d(\xi,\eta)\hat{f}_{m,k_1}(t,\xi-\eta)\hat{g}_{k_2}(t,\eta)\psi_k(\xi)d\eta dt\bigg\|_{L^2}\label{chi3*.2-2}.
\end{align}
By the $L^2\times L^\infty\rightarrow L^2$ bilinear estimate, Lemma \ref{chi,eta}, Lemma \ref{dualitycomp}, \eqref{l2}, \eqref{addder}, \eqref{l2first}, and \eqref{linfinity2}, we get
\begin{align*}
    &\eqref{chi3*.1} + \eqref{chi3*.1-2}\\
    \lesssim 
    &2^{M}\|\F^{-1}m_d(\xi,\eta)\Tilde{\psi}_{k_1}(\xi-\eta)\Tilde{\psi}_{k_2}(\eta)\Tilde{\psi}_{k}(\xi)\|_{L^1} \sup_{t\in[2^{M-1},2^M]}\big(\min\{\|e^{ic_mt\la}\F^{-1}\nabla_\xi\hat{g}_{k_1}\|_{L^2_x}\|e^{ic_nt\la}f_{n,k_2}\|_{L^\infty_x},\\
    &\qquad\qquad\qquad\qquad 2^{3\min\{k,k_2\}/2}\|e^{ic_mt\la}\F^{-1}\nabla_\xi\hat{g}_{k_1}\|_{L^2_x}\|e^{ic_nt\la}f_{n,k_2}\|_{L^2_x}\}+\\
    &+\min\{\|e^{ic_mt\la}\F^{-1}\nabla_\xi\hat{f}_{m,k_1}\|_{L^2_x}\|e^{ic_nt\la}g_{k_2}\|_{L^\infty_x},2^{3\min\{k,k_2\}/2}\|e^{ic_mt\la}\F^{-1}\nabla_\xi\hat{f}_{m,k_1}\|_{L^2_x}\|e^{ic_nt\la}g_{k_2}\|_{L^2_x}\}\big)\\
    \lesssim & 2^{M+dk}\sup_{t\in[2^{M-1},2^M]} \big(\min\{\|\nabla_\xi\hat{g}_{k_1}\|_{L^2_\xi}2^{-M/2+3k_2/2}\|\nabla_\xi\hat{f}_{n,k_2}\|_{L^2_\xi},2^{3\min\{k,k_2\}/2}\|\nabla_\xi\hat{g}_{k_1}\|_{L^2_\xi}\|f_{n,k_2}\|_{L^2_x}+\\
    &+\min\{\|\nabla_\xi\hat{f}_{k_1}\|_{L^2_\xi}2^{-M/2+3k_2/2}\|\nabla_\xi\hat{g}_{k_2}\|_{L^2_\xi},2^{3\min\{k,k_2\}/2+k_2}\|\nabla_\xi\hat{f}_{m,k_1}\|_{L^2_\xi}\|\nabla_\xi\hat{g}_{k_2}\|_{L^2_\xi}\}\big)\\
    \lesssim & 2^{M+dk}\sup_{t\in[2^{M-1},2^M]} \big(\min\{2^{-M/2+3k_2/2}\|\nabla_\xi\hat{f}_{n,k_2}\|_{L^2_\xi},2^{3\min\{k,k_2\}/2}\|f_{n,k_2}\|_{L^2_x}\}\|\nabla_\xi\hat{g}_{k_1}\|_{L^2_\xi}+\\
    &+\min\{\|\nabla_\xi\hat{f}_{m,k_1}\|_{L^2_\xi}2^{-M/2+3k_2/2},2^{3\min\{k,k_2\}/2+k_2}\|\nabla_\xi\hat{f}_{m,k_1}\|_{L^2_\xi}\}\|\nabla_\xi\hat{g}_{k_2}\|_{L^2_\xi}\big)\\
    \lesssim & 2^{M+dk} \sup_{t\in[2^{M-1},2^M]}\big(2^{-2k_{2,+}+\gamma k_2}\min\{2^{-M/2+k_2},2^{3\min\{k,k_2\}/2+k_2/2}\}\e_1\|\nabla_\xi\hat{g}_{k_1}\|_{L^2_\xi}+\\
    &+2^{-2k_{1,+}+\gamma k_1}\min\{2^{-M/2+3k_2/2-k_1/2},2^{3\min\{k,k_2\}/2+k_2-k_1/2}\}\e_1\|\nabla_\xi\hat{g}_{k_2}\|_{L^2_\xi}\big)\\
    \lesssim & 2^{dk-2k_{+}+\gamma k}\min\{2^{M/2+k},2^{M+2k}\}\e_1\|\hat{g}\|_{L^\infty_t([2^{m-2},2^M])H^1_\xi},
\end{align*}
and the $L^4\times L^4\rightarrow L^2$ bilinear estimate, Lemma \ref{dualitycomp}, \eqref{addder}, and \eqref{l4} imply
\begin{align*}
    &\eqref{chi3*.12}+\eqref{chi3*.2}+\eqref{chi3*.12-2}+\eqref{chi3*.2-2}\\
    \lesssim &2^{M}\big(\|\nabla_\xi\psi_k\|_{L^\infty}\|\F^{-1}m_d(\xi,\eta)\Tilde{\psi}_{k}(\xi)\Tilde{\psi}_{k_1}(\xi-\eta)\Tilde{\psi}_{k_2}(\eta)\|_{L^1}+\\
    &\qquad\qquad\qquad\qquad+2^{M}\|\F^{-1}\partial_{\xi_m}\phi(\xi,\eta)m_d(\xi,\eta)\Tilde{\psi}_{k}(\xi)\Tilde{\psi}_{k_1}(\xi-\eta)\Tilde{\psi}_{k_2}(\eta)\|_{L^1}\big)\times\\
    &\times \sup_{t\in[2^{M-1},2^M]}\big(\min\{\|e^{ic_mt\la}f_{m,k_1}\|_{L^4_x}\|e^{ic_nt\la}g_{k_2}\|_{L^4_x},2^{3\min\{k,k_2\}/2}\|e^{ic_mt\la}f_{m,k_1}\|_{L^2_x}\|e^{ic_nt\la}g_{k_2}\|_{L^2_x}\}+\\
    &\qquad\qquad\qquad\qquad+\min\{\|e^{ic_mt\la}g_{k_1}\|_{L^4_x}\|e^{ic_nt\la}f_{n,k_2}\|_{L^4_x},2^{3\min\{k,k_2\}/2}\|e^{ic_mt\la}g_{k_1}\|_{L^2_x}\|e^{ic_nt\la}f_{n,k_2}\|_{L^2_x}\}\big)\\
    \lesssim & (2^{M-k+dk}+2^{2M+dk+k})\times\\
    &\times\sup_{t\in[2^{M-1},2^M]}\big(\min\{\|e^{ic_mt\la}f_{m,k_1}\|_{L^4_x}\|e^{ic_nt\la}g_{k_2}\|_{L^4_x},2^{3\min\{k,k_2\}/2+k_2}\|f_{m,k_1}\|_{L^2_x}\|\nabla_\xi\hat{g}_{k_2}\|_{L^2_\xi}\}+\\
    &\qquad\qquad\qquad\qquad+\min\{\|e^{ic_mt\la}g_{k_1}\|_{L^4_x}\|e^{ic_nt\la}f_{n,k_2}\|_{L^4_x},2^{3\min\{k,k_2\}/2+k_1}\|\nabla_\xi\hat{g}_{k_1}\|_{L^2_\xi}\|f_{n,k_2}\|_{L^2_x}\}\big)\\
    \lesssim & (2^{M-k+dk}+2^{2M+dk+k})\times\\
    &\times \sup_{t\in[2^{M-1},2^M]}\big(2^{-2k_{1,+}+\gamma k_1}\min\{2^{-3M/2-k_1/4+k_2/4},2^{3\min\{k,k_2\}/2+k_2+k_1/2}\}\e_1\|\nabla_\xi\hat{g}_{k_2}\|_{L^2_\xi}+\\
    &\qquad\qquad\qquad\qquad+2^{-2k_{2,+}+\gamma k_2}\min\{2^{-3M/2+k_1/4-k_2/2},2^{3\min\{k,k_2\}/2+k_1+k_2/2}\}\e_1\|\nabla_\xi\hat{g}_{k_1}\|_{L^2_\xi}\big)\\
    \lesssim & (1+2^{M+2k})2^{dk-2k_++\gamma k}\min\{2^{-M/2-k},2^{M+2k}\}\e_1\|\hat{g}\|_{L^\infty_t([2^{M-1},2^M])H^1_\xi}.
\end{align*}
When $M\leq -2k$, we have $2^{M+2k}\leq 1$. Thus,
\begin{align*}
    \|T^1_{0+iy}\hat{g}\|_{L^2}+\|T^2_{0+iy}\hat{g}\|_{L^2}\lesssim  2^{dk-2k_{+}+\gamma k}\min\{2^{-M/2+\gamma M/8-k+\gamma k/4},2^{M+2k}\}\e_1\|\hat{g}\|_{L^\infty_t([2^{M-1},2^M])L^2_\xi}
\end{align*}
and
\begin{align*}
    \|T^1_{1+iy}\hat{g}\|_{L^2}+\|T^2_{1+iy}\hat{g}\|_{L^2}
    \lesssim  &2^{dk-2k_++\gamma k}\min\{2^{-M/2-k},2^{M+2k}\}\e_1\|\hat{g}\|_{L^\infty_t([2^{M-1},2^M])H^1_\xi}.
\end{align*}
By the interpolation result in Lemma \ref{si}, for any $\alpha\in(0,1)$, we get
\begin{align*}
    \|T^1_{\alpha}\hat{g}\|_{L^2}+\|T^2_{\alpha}\hat{g}\|_{L^2}\lesssim 2^{dk-2k_{+}+\gamma k}\min\{2^{-M/2-k+(1-\alpha)(\gamma M/8+\gamma k/4)},2^{M+k}\}\e_1\|\hat{g}\|_{L^\infty_t([2^{M-1},2^M])H^\alpha_\xi}.
\end{align*}
In the other case when $M>-2k$, we have
\begin{align*}
    \|T^1_{0+iy}\hat{g}\|_{L^2}+\|T^2_{0+iy}\hat{g}\|_{L^2}\lesssim  2^{dk-2k_{+}+\gamma k}\min\{2^{-M/2+\gamma M/8-k+\gamma k/4},2^{M+2k}\}\e_1\|\hat{g}\|_{L^\infty_t([2^{M-1},2^M])L^2_\xi}
\end{align*}
and
\begin{align*}
    \|T^1_{1+iy}\hat{g}\|_{L^2}+\|T^2_{1+iy}\hat{g}\|_{L^2}
    \lesssim  &2^{M+2k+dk-2k_++\gamma k}\min\{2^{-M/2-k},2^{M+2k}\}\e_1\|\hat{g}\|_{L^\infty_t([2^{M-1},2^M])H^1_\xi}.
\end{align*}
Lemma \ref{si} implies that for any $\alpha\in(0,1)$,
\begin{align*}
    \|T^1_{\alpha}\hat{g}\|_{L^2}+\|T^2_{\alpha}\hat{g}\|_{L^2}\lesssim 2^{\alpha(M+2k)+dk-2k_{+}+\gamma k}\min\{2^{-M/2-k+(1-\alpha)(\gamma M/8+\gamma k/4)},2^{M+2k}\}\e_1\|\hat{g}\|_{L^\infty_t([2^{M-1},2^M])H^\alpha_\xi}.
\end{align*}
Therefore,
\begin{align*}
    &\bigg\|D^\alpha\int_{t_1}^{t_2}\int_{\R^3}e^{it\phi(\xi,\eta)}m_d(\xi,\eta)\hat{g}_{k_1}(t,\xi-\eta)\hat{f}_{n,k_2}(t,\eta)\psi_k(\xi)d\eta dt\bigg\|_{L^2}
    =\|T^1_{\alpha}(\hat{g}\Tilde{\psi}_{k_1})\|_{L^2}\\
    \lesssim & (1+2^{\alpha(M+2k)})2^{dk-2k_{+}+\gamma k}\min\{2^{-M/2-k+(1-\alpha)(\gamma M/8+\gamma k/4)},2^{M+2k}\}\e_1\|\hat{g}_{k_1}\|_{L^\infty_t([2^{M-1},2^M])H^\alpha_\xi}
\end{align*}
and
\begin{align*}
    &\bigg\|D^\alpha\int_{t_1}^{t_2}\int_{\R^3}e^{it\phi(\xi,\eta)}m_d(\xi,\eta)\hat{f}_{m,k_1}(t,\xi-\eta)\hat{g}_{k_2}(t,\eta)\psi_k(\xi)d\eta dt\bigg\|_{L^2}
    =\|T^2_{\alpha}(\hat{g}\Tilde{\psi}_{k_2})\|_{L^2}\\
    \lesssim & (1+2^{\alpha(M+2k)})2^{dk-2k_{+}+\gamma k}\min\{2^{-M/2-k+(1-\alpha)(\gamma M/8+\gamma k/4)},2^{M+2k}\}\e_1\|\hat{g}_{k_2}\|_{L^\infty_t([2^{M-1},2^M])H^\alpha_\xi}.
\end{align*}
\end{proof}
Observe $\|D^\alpha I^{M,1}_{k,k_1,k_2}\|_{L^2}$ for $(k_1,k_2)\in\chi^3_k$ is special case of \eqref{chi3*1} in Lemma \ref{chi3*} with $\hat{g}=\nabla_\xi\hat{f}_m$, $d=0$. Employing \eqref{sobolevnorms}, we have
\begin{align*}
    &\|D^\alpha I^{M,1}_{k,k_1,k_2}\|_{L^2}\\
    \lesssim &(1+2^{\alpha(M+2k)})2^{-2k_{+}+\gamma k}\min\{2^{-M/2-k+(1-\alpha)(\gamma M/8+\gamma k/4)},2^{M+2k}\}\e_1\|\hat{f}_{k_1}\|_{L^\infty_t([2^{M-1},2^M])H^{1+\alpha}_\xi}\\
    \lesssim &(1+2^{\alpha(M+2k)})2^{-2k_{+}+\gamma k-2k_{1,+}+\gamma k_1-k_1/2-\alpha k_1}\min\{2^{-M/2-k+(1-\alpha)(\gamma M/8+\gamma k/4)},2^{M+2k}\}\e_1^2.
\end{align*}
Hence, 
\begin{align*}
    &\sum_{1\leq M\leq \log T}\sup_{2^{M-1}\leq t_1\leq t_2\leq 2^M}2^{2k_+-\gamma k+k/2+\alpha k}\sum_{c_m+c_n\neq 0}A_{lmn}\sum_{(k_1,k_2)\in\chi^3_k}\|D^\alpha I^{M,1}_{k,k_1,k_2}\|_{L^2}\\
    \lesssim &\sum_{1\leq M\leq \log T}(1+2^{\alpha(M+2k)})2^{-2k_{+}+\gamma k}\min\{2^{-M/2-k+(1-\alpha)(\gamma M/8+\gamma k/4)},2^{M+2k}\}\e_1^2\\
    \lesssim &\sum_{M\leq -2k}2^{\gamma k+M+2k}\e_1^2+ \sum_{-2k<M\leq \log T}2^{\alpha(M+2k)-2k_{+}+\gamma k-M/2-k+(1-\alpha)(\gamma M/8+\gamma k/4)}\e_1^2\\
    \lesssim & \e_1^2,
\end{align*}
given $\alpha+(1-\alpha)\gamma/8<1/2$, i.e. $\alpha<1/2-\gamma/(16-2\gamma)$.\\
Combining this with the result in Lemma \ref{jm1} and Lemma \ref{sp1}, we have
\begin{align*}
    &\sum_{1\leq M\leq \log T}\sup_{2^{M-1}\leq t_1\leq t_2\leq 2^M}2^{2k_+-\gamma k+k/2+\alpha k}\bigg(\sum_{c_m+c_n\neq 0}A_{lmn}\sum_{(k_1,k_2)\in\chi^3_k}\|D^\alpha I^{M,1}_{k,k_1,k_2}\|_{L^2}\\
    &+\big(\sum_{\substack{c_m+c_n\neq 0\\c_m\neq c_l}}+\sum_{\substack{c_m+c_n\neq 0\\c_m=c_l}}\big)A_{lmn}\sum_{(k_1,k_2)\in\chi^1_k}\|D^\alpha I^{M,1}_{k,k_1,k_2}\|_{L^2}
    +\sum_{c_m+c_n\neq 0}A_{lmn}\sum_{(k_1,k_2)\in\chi^2_k}\|D^\alpha I^{M,1}_{k,k_1,k_2}\|_{L^2}\\
    &+\big(\sum_{\substack{c_m+c_n= 0\\c_m\neq c_l}}+\sum_{\substack{c_m+c_n= 0\\c_m=c_l}}\big)A_{lmn}\sum_{(k_1,k_2)\in\chi^1_k}\|D^\alpha J^{M,1}_{k,k_1,k_2}\|_{L^2}
    +\sum_{c_m+c_n= 0}A_{lmn}\sum_{(k_1,k_2)\in\chi^2_k\cup\chi^3_k}\|D^\alpha J^{M,1}_{k,k_1,k_2}\|_{L^2}\\
    \lesssim & \e_1^2+\sum_{1\leq M\leq \log T}\sum_{(k_1,k_2)\in\chi_k^1\cup\chi_k^2\cup\chi_k^3} 2^{-\gamma k+2k_+-2k_{1,+}+\gamma k_1-2k_{2,+}+\gamma k_2+k/2-k_1/2}\times\\
    &\qquad\qquad\qquad\qquad\qquad\qquad\times\min\{2^{-M/4 -k_2/2},2^{M+3\min\{k,k_2\}/2 +k_2/2}\}\e_1^2\\
    &+\sum_{1\leq M\leq \min\{\log T,-2k\}}\sum_{k_2\leq k} 2^{-\gamma k_--2k_{2,+}+\gamma k_2}\min\{2^{-M/2-k/2-k_2/2},2^{M+2k_2+\alpha k-\alpha k_2}\}\e_1^2\\
    &+\sum_{-2k \leq M\leq \log T}\sum_{k_2\leq k} 2^{-\gamma k_--2k_{2,+}+\gamma k_2+2\alpha k}\min\{2^{-M/2+\gamma M/4+\gamma k/2-3k/2-k_2/2},2^{M-k/2+3k_2/2}\}\e_1^2\\
    \lesssim &\e_1^2,
\end{align*}
which proves \eqref{im1wts}.

\subsection{$D^{\alpha}I^{M,2}_{k,k_1,k_2}$ and $D^{\alpha}J^{M,2}_{k,k_1,k_2}$}
\label{2}
In the current section, we examine the remaining terms, \eqref{j1} and \eqref{j2}. 

For $D^{\alpha}I^{M,2}_{k,k_1,k_2}$, we handle the cases where $(k_1,k_2)\in\chi^1_k$ and $c_m\neq c_l$ in Section \ref{2-2}. In Section \ref{2-0}, we further investigate the cases where $(k_1,k_2)\in\chi^1_k$, but with the coefficients $c_m=c_l$, or the cases where $(k_1,k_2)\in\chi^2_k$. Lastly, the cases where $(k_1,k_2)\in\chi^3_k$ are addressed in Section \ref{2-3}.

As for $D^{\alpha}J^{M,2}_{k,k_1,k_2}$, we look at the cases where $(k_1,k_2)\in\chi^1_k$ and $c_m\neq c_l$ in Section \ref{2-2}, while all the remaining cases, $(k_1,k_2)\in\chi^1_k$ and $c_m=c_l$ or $(k_1,k_2)\in\chi^2_k\cup\chi^3_k$, are estimated in Section \ref{2-1}.
\subsubsection{$D^{\alpha}I^{M,2}_{k,k_1,k_2}$ where $(k_1,k_2)\in\chi^2_k$}
\label{2-0}
Recall the lower bound $|\nabla_\eta\phi(\xi,\eta)|\gtrsim 2^{k_1}$ established in $\eqref{xi-2eta}$, when $(k_1,k_2)\in\chi^1_k\cup\chi^2_k$. We will frequently employ the following identity throughout this section,
$$\sum_{m}\frac{\partial_{\eta_m}\phi(\xi,\eta)}{it|\nabla_\eta\phi(\xi,\eta)|^2}\partial_{\eta_m} e^{it\phi(\xi,\eta)}=e^{it\phi(\xi,\eta)}.$$ 

To estimate the $\alpha$ derivative, we will utilize the interpolation result from Lemma \ref{si}. However, instead of immediately defining a single family of operators for $I^{M,2}_{k,k_1,k_2}$ to apply the lemma to, we will take a preliminary step. We will use the aforementioned identity to expand and split $D^{\alpha}I^{M,2}_{k,k_1,k_2}$ in the following way,
\begin{align}
    D^{\alpha}I^{M,2}_{k,k_1,k_2}
    =&D^\alpha\int_{t_1}^{t_2}\int_{\R^3}e^{it\phi(\xi,\eta)}it\partial_{\xi_l}\phi(\xi,\eta)\hat{f}_{m,k_1}(t,\xi-\eta)\hat{f}_{n,k_2}(t,\eta)\psi_k(\xi)d\eta dt\nonumber\\
    =&-D^\alpha\int_{t_1}^{t_2}\int_{\R^3}e^{it\phi(\xi,\eta)}\partial_{\eta_m}\frac{\partial_{\xi_l}\phi(\xi,\eta)\partial_{\eta_m}\phi(\xi,\eta)}{|\nabla_\eta\phi(\xi,\eta)|^2}\hat{f}_{m,k_1}(t,\xi-\eta)\hat{f}_{n,k_2}(t,\eta)\psi_k(\xi)d\eta dt\label{im2.1}\\
    &-D^\alpha\int_{t_1}^{t_2}\int_{\R^3}e^{it\phi(\xi,\eta)}\frac{\partial_{\xi_l}\phi(\xi,\eta)\partial_{\eta_m}\phi(\xi,\eta)}{|\nabla_\eta\phi(\xi,\eta)|^2}\partial_{\eta_m}\hat{f}_{m,k_1}(t,\xi-\eta)\hat{f}_{n,k_2}(t,\eta)\psi_k(\xi)d\eta dt\label{im2.3}\\
    &-D^\alpha\int_{t_1}^{t_2}\int_{\R^3}e^{it\phi(\xi,\eta)}\frac{\partial_{\xi_l}\phi(\xi,\eta)\partial_{\eta_m}\phi(\xi,\eta)}{|\nabla_\eta\phi(\xi,\eta)|^2}\hat{f}_{m,k_1}(t,\xi-\eta)\partial_{\eta_m}\hat{f}_{n,k_2}(t,\eta)\psi_k(\xi)d\eta dt\label{im2.4}.
\end{align}
This procedure adds flexibility to our estimation by allowing us to deal with each term separately using Lemma \ref{si}.
\begin{lem}
\label{im2-1}
Suppose $t_1,t_2\in[2^{M-1},2^M]$ and $\sup_{t\in[1,T]}\|f_l\|_{Z}\leq \e_1$. If either $(k_1,k_2)\in\chi^2_k$ or $(k_1,k_2)\in\chi^1_k$ and $c_l=c_m$, then
\begin{align*}
    &\|\eqref{im2.1}\|_{L^2}
    \lesssim  2^{-k_1/2-2k_{1,+}+\gamma k_1-2k_{2,+}+\gamma k_2-\alpha k}\min\{2^{-M/4-k_1+k_2/2},2^{M+3\min\{k,k_2\}/2+k_2/2}\}\e_1^2.
\end{align*}
\end{lem}
\begin{proof}
Consider the following family of operators on $\{z\in\C:0\leq \Re(z)\leq 1\}$,
\begin{align*}
    T_{z} \hat{g} =D^z\int_{t_1}^{t_2}\int_{\R^3}e^{it\phi(\xi,\eta)}\partial_{\eta_m}\frac{\partial_{\xi_l}\phi(\xi,\eta)\partial_{\eta_m}\phi(\xi,\eta)}{|\nabla_\eta\phi(\xi,\eta)|^2}\hat{g}_{k_1}(t,\xi-\eta)\hat{f}_{n,k_2}(t,\eta)\psi_k(\xi)d\eta dt.
\end{align*}
First, we want to show the operator $T_{0+iy}: L^\infty_t([2^{M-1},2^M])H^1_\xi\rightarrow L^2_\xi$ is bounded for all $y\in\R$.\\
The Strichartz's estimate in Lemma \ref{strichartz}, together with Lemma \ref{chi+eta2}, Lemma \ref{dualitycomp}, \eqref{sobolevnorms}, \eqref{l2}, and \eqref{addder}, provides
\begin{align*}
    \|T_{0+iy} \hat{g}\|_{L^2}
    =&\bigg\|\int_{t_1}^{t_2}\int_{\R^3}e^{it\phi(\xi,\eta)}\partial_{\eta_m}\frac{\partial_{\xi_l}\phi(\xi,\eta)\partial_{\eta_m}\phi(\xi,\eta)}{|\nabla_\eta\phi(\xi,\eta)|^2}\hat{g}_{k_1}(t,\xi-\eta)\hat{f}_{n,k_2}(t,\eta)\psi_k(\xi)d\eta dt\bigg\|_{L^2}\\
    \lesssim & \|\F^{-1}\partial_{\eta_m}\frac{\partial_{\xi_l}\phi(\xi,\eta)\partial_{\eta_m}\phi(\xi,\eta)}{|\nabla_\eta\phi(\xi,\eta)|^2}\Tilde{\psi}_{k_1}(\nabla_\eta\phi(\xi,\eta))\Tilde{\psi}_{k_2}(\eta)\|_{L^1}\times\\
    &\times \min\{2^{-M/4}\|\hat{g}_{k_1}\|_{L^\infty_t([2^{M-1},2^M])H^1_\xi}\|\hat{f}_{n,k_2}\|_{L^\infty_t([2^{M-1},2^M])L^2_\xi},\\
    &\qquad\qquad\qquad\qquad 2^{M+3\min\{k,k_2\}/2}\|e^{ic_mt\la}g_{k_1}\|_{L^\infty_t([2^{M-1},2^M])L^2_x}\|e^{ic_nt\la}f_{n,k_2}\|_{L^\infty_t([2^{M-1},2^M])L^2_x}\}\\
    \lesssim & 2^{-k_1}\min\{2^{-M/4}\|\hat{g}_{k_1}\|_{L^\infty_t([2^{M-1},2^M])H^1_\xi}\|\hat{f}_{n,k_2}\|_{L^\infty_t([2^{M-1},2^M])L^2_\xi},\\
    &\qquad\qquad\qquad\qquad 2^{M+3\min\{k,k_2\}/2+k_1}\|\nabla_\xi \hat{g}_{k_1}\|_{L^\infty_t([2^{M-1},2^M])L^2_\xi}\|\hat{f}_{n,k_2}\|_{L^\infty_t([2^{M-1},2^M])L^2_\xi}\}\\
    \lesssim & 2^{-k_1-2k_{2,+}+\gamma k_2}\min\{2^{-M/4+k_2/2},2^{M+3\min\{k,k_2\}/2+k_1+k_2/2}\}\e_1\|\hat{g}_{k_1}\|_{L^\infty_t([2^{M-1},2^M])H^1_\xi}.
\end{align*}
Then, we need to prove for all $y\in\R$, the operator $T_{1+iy}:L^\infty_t([2^{M-1},2^M])H^2_\xi\rightarrow L^2_\xi$ is bounded.\\
Observe that
\begin{align}
    \|T_{1+iy} \hat{g}\|_{L^2}
    \leq &\bigg\|\partial_{\xi_n}\int_{t_1}^{t_2}\int_{\R^3}e^{it\phi(\xi,\eta)}\partial_{\eta_m}\frac{\partial_{\xi_l}\phi(\xi,\eta)\partial_{\eta_m}\phi(\xi,\eta)}{|\nabla_\eta\phi(\xi,\eta)|^2}\hat{g}_{k_1}(t,\xi-\eta)\hat{f}_{n,k_2}(t,\eta)\psi_k(\xi)d\eta dt\bigg\|_{L^2}\nonumber\\
    \leq &\bigg\|\int_{t_1}^{t_2}\int_{\R^3}e^{it\phi(\xi,\eta)}\partial_{\eta_m}\frac{\partial_{\xi_l}\phi(\xi,\eta)\partial_{\eta_m}\phi(\xi,\eta)}{|\nabla_\eta\phi(\xi,\eta)|^2}\partial_{\xi_n}\hat{g}_{k_1}(t,\xi-\eta)\hat{f}_{n,k_2}(t,\eta)\psi_k(\xi)d\eta dt\bigg\|_{L^2}\label{im2-1.1}\\
    &+ \bigg\|\int_{t_1}^{t_2}\int_{\R^3}e^{it\phi(\xi,\eta)}\partial_{\xi_n}\partial_{\eta_m}\frac{\partial_{\xi_l}\phi(\xi,\eta)\partial_{\eta_m}\phi(\xi,\eta)}{|\nabla_\eta\phi(\xi,\eta)|^2}\hat{g}_{k_1}(t,\xi-\eta)\hat{f}_{n,k_2}(t,\eta)\psi_k(\xi)d\eta dt\bigg\|_{L^2}\label{im2-1.11}\\
    &+ \bigg\|\int_{t_1}^{t_2}\int_{\R^3}e^{it\phi(\xi,\eta)}\partial_{\eta_m}\frac{\partial_{\xi_l}\phi(\xi,\eta)\partial_{\eta_m}\phi(\xi,\eta)}{|\nabla_\eta\phi(\xi,\eta)|^2}\hat{g}_{k_1}(t,\xi-\eta)\hat{f}_{n,k_2}(t,\eta)\partial_{\xi_n}\psi_k(\xi)d\eta dt\bigg\|_{L^2}\label{im2-1.12}\\
    & + \bigg\|\int_{t_1}^{t_2}\int_{\R^3}e^{it\phi(\xi,\eta)}it\partial_{\xi_n}\phi(\xi,\eta)\partial_{\eta_m}\frac{\partial_{\xi_l}\phi(\xi,\eta)\partial_{\eta_m}\phi(\xi,\eta)}{|\nabla_\eta\phi(\xi,\eta)|^2}\hat{g}_{k_1}(t,\xi-\eta)\hat{f}_{n,k_2}(t,\eta)\psi_k(\xi)d\eta dt\bigg\|_{L^2}\label{im2-1.2},
\end{align}
and perform integration by parts using the identity $e^{it\phi(\xi,\eta)}=\sum_{j=1}^3\frac{\partial_{\eta_j}\phi(\xi,\eta)}{it|\nabla_\eta\phi(\xi,\eta)|^2}\partial_{\eta_j}e^{it\phi(\xi,\eta)}$ on the term $\eqref{im2-1.2}$ ,
\begin{align}
    &\int_{t_1}^{t_2}\int_{\R^3}e^{it\phi(\xi,\eta)}it\partial_{\xi_n}\phi(\xi,\eta)\partial_{\eta_m}\frac{\partial_{\xi_l}\phi(\xi,\eta)\partial_{\eta_m}\phi(\xi,\eta)}{|\nabla_\eta\phi(\xi,\eta)|^2}\hat{g}_{k_1}(t,\xi-\eta)\hat{f}_{n,k_2}(t,\eta)\psi_k(\xi)d\eta dt\nonumber\\
    = & -\int_{t_1}^{t_2}\int_{\R^3}e^{it\phi(\xi,\eta)}\partial_{\eta_j}\big(\frac{\partial_{\xi_n}\phi(\xi,\eta)\partial_{\eta_j}\phi(\xi,\eta)}{|\nabla_\eta\phi(\xi,\eta)|^2}\partial_{\eta_m}\frac{\partial_{\xi_l}\phi(\xi,\eta)\partial_{\eta_m}\phi(\xi,\eta)}{|\nabla_\eta\phi(\xi,\eta)|^2}\big)\hat{g}_{k_1}(t,\xi-\eta)\hat{f}_{n,k_2}(t,\eta)\psi_k(\xi)d\eta dt\label{im2-1.3}\\
    & -\int_{t_1}^{t_2}\int_{\R^3}e^{it\phi(\xi,\eta)}\frac{\partial_{\xi_n}\phi(\xi,\eta)\partial_{\eta_j}\phi(\xi,\eta)}{|\nabla_\eta\phi(\xi,\eta)|^2}\partial_{\eta_m}\frac{\partial_{\xi_l}\phi(\xi,\eta)\partial_{\eta_m}\phi(\xi,\eta)}{|\nabla_\eta\phi(\xi,\eta)|^2}\partial_{\eta_j}\hat{g}_{k_1}(t,\xi-\eta)\hat{f}_{n,k_2}(t,\eta)\psi_k(\xi)d\eta dt\label{im2-1.4}\\
    & -\int_{t_1}^{t_2}\int_{\R^3}e^{it\phi(\xi,\eta)}\frac{\partial_{\xi_n}\phi(\xi,\eta)\partial_{\eta_j}\phi(\xi,\eta)}{|\nabla_\eta\phi(\xi,\eta)|^2}\partial_{\eta_m}\frac{\partial_{\xi_l}\phi(\xi,\eta)\partial_{\eta_m}\phi(\xi,\eta)}{|\nabla_\eta\phi(\xi,\eta)|^2}\hat{g}_{k_1}(t,\xi-\eta)\partial_{\eta_j}\hat{f}_{n,k_2}(t,\eta)\psi_k(\xi)d\eta dt\label{im2-1.5}.
\end{align}
Using the Strichartz's estimate in Lemma \ref{strichartz}, Lemma \ref{dualitycomp}, Lemma \ref{chi+eta2}, as well as estimations \eqref{l2}, \eqref{l2first}, and \eqref{addder}, we have
\begin{align*}
    &\eqref{im2-1.1}+\|\eqref{im2-1.4}\|_{L^2}\\
    \lesssim & \big(\|\F^{-1}\partial_{\eta_m}\frac{\partial_{\xi_l}\phi(\xi,\eta)\partial_{\eta_m}\phi(\xi,\eta)}{|\nabla_\eta\phi(\xi,\eta)|^2}\Tilde{\psi}_{k_1}(\nabla_\eta\phi(\xi,\eta))\Tilde{\psi}_{k_2}(\eta)\|_{L^1}+\\
    &+\|\F^{-1}\frac{\partial_{\xi_n}\phi(\xi,\eta)\partial_{\eta_j}\phi(\xi,\eta)}{|\nabla_\eta\phi(\xi,\eta)|^2}\partial_{\eta_m}\frac{\partial_{\xi_l}\phi(\xi,\eta)\partial_{\eta_m}\phi(\xi,\eta)}{|\nabla_\eta\phi(\xi,\eta)|^2}\Tilde{\psi}_{k_1}(\nabla_\eta\phi(\xi,\eta))\Tilde{\psi}_{k_2}(\eta)\|_{L^1}\big)\times\\
    &\times \min\{2^{-M/4}\|\nabla_\xi\hat{g}_{k_1}\|_{L^\infty_t([2^{M-1},2^M])H^1_\xi}\|\hat{f}_{n,k_2}\|_{L^\infty_t([2^{M-1},2^M])L^2_\xi},\\
    &\qquad\qquad\qquad\qquad 2^{M+3\min\{k,k_2\}/2}\|e^{ic_mt\la}\F^{-1}\nabla_\xi\hat{g}_{k_1}\|_{L^\infty_t([2^{M-1},2^M])L^2_x}\|e^{ic_nt\la}f_{n,k_2}\|_{L^\infty_t([2^{M-1},2^M])L^2_x}\}\\
    \lesssim & 2^{-k_1}\min\{2^{-M/4}\|\nabla_\xi\hat{g}_{k_1}\|_{L^\infty_t([2^{M-1},2^M])H^1_\xi}\|\hat{f}_{n,k_2}\|_{L^\infty_t([2^{M-1},2^M])L^2_\xi},\\
    &\qquad\qquad\qquad\qquad 2^{M+3\min\{k,k_2\}/2+k_1}\|\nabla^2_\xi\hat{g}_{k_1}\|_{L^\infty_t([2^{M-1},2^M])L^2_\xi}\|f_{n,k_2}\|_{L^\infty_t([2^{M-1},2^M])L^2_x}\}\\
    \lesssim & 2^{-k_1}\min\{2^{-M/4}\|\hat{f}_{n,k_2}\|_{L^\infty_t([2^{M-1},2^M])L^2_\xi},\\
    &\qquad\qquad\qquad\qquad 2^{M+3\min\{k,k_2\}/2+k_1}\|f_{n,k_2}\|_{L^\infty_t([2^{M-1},2^M])L^2_x}\}\|\nabla_\xi\hat{g}_{k_1}\|_{L^\infty_t([2^{M-1},2^M])H^1_\xi}\\
    \lesssim & 2^{-k_1-2k_{2,+}+\gamma k_2}\min\{2^{-M/4+k_2/2},2^{M+3\min\{k,k_2\}/2+k_1+k_2/2}\}\e_1\|\hat{g}_{k_1}\|_{L^\infty_t([2^{M-1},2^M])H^2_\xi}
\end{align*}
and
\begin{align*}
    &\eqref{im2-1.11}+\eqref{im2-1.12}+\|\eqref{im2-1.3}\|_{L^2}\\
    \lesssim &\big(\|\F^{-1}\partial_{\xi_n}\partial_{\eta_m}\frac{\partial_{\xi_l}\phi(\xi,\eta)\partial_{\eta_m}\phi(\xi,\eta)}{|\nabla_\eta\phi(\xi,\eta)|^2}\Tilde{\psi}_{k_1}(\nabla_\eta\phi(\xi,\eta))\Tilde{\psi}_{k_2}(\eta)\|_{L^1}+\\
    &+ \|\nabla\psi_k\|_{L^\infty}\|\F^{-1}\partial_{\eta_m}\frac{\partial_{\xi_l}\phi(\xi,\eta)\partial_{\eta_m}\phi(\xi,\eta)}{|\nabla_\eta\phi(\xi,\eta)|^2}\Tilde{\psi}_{k_1}(\nabla_\eta\phi(\xi,\eta))\Tilde{\psi}_{k_2}(\eta)\|_{L^1}\\
    & + \|\F^{-1}\partial_{\eta_j}\big(\frac{\partial_{\xi_n}\phi(\xi,\eta)\partial_{\eta_j}\phi(\xi,\eta)}{|\nabla_\eta\phi(\xi,\eta)|^2}\partial_{\eta_m}\frac{\partial_{\xi_l}\phi(\xi,\eta)\partial_{\eta_m}\phi(\xi,\eta)}{|\nabla_\eta\phi(\xi,\eta)|^2}\big)\Tilde{\psi}_{k_1}(\nabla_\eta\phi(\xi,\eta))\Tilde{\psi}_{k_2}(\eta)\|_{L^1}\big)\times\\
    &\times \min\{2^{-M/4}\|\hat{g}_{k_1}\|_{L^\infty_t([2^{M-1},2^M])H^1_\xi}\|\hat{f}_{n,k_2}\|_{L^\infty_t([2^{M-1},2^M])L^2_\xi},\\
    &\qquad\qquad\qquad\qquad 2^{M+3\min\{k,k_2\}/2}\|e^{ic_mt\la}g_{k_1}\|_{L^\infty_t([2^{M-1},2^M])L^2_x}\|e^{ic_nt\la}f_{n,k_2}\|_{L^\infty_t([2^{M-1},2^M])L^2_x}\}\\
    \lesssim &2^{-2k_1}\min\{2^{-M/4+k_1}\|\hat{g}_{k_1}\|_{L^\infty_t([2^{M-1},2^M])H^2_\xi}\|\hat{f}_{n,k_2}\|_{L^\infty_t([2^{M-1},2^M])L^2_\xi},\\
    &\qquad\qquad\qquad\qquad 2^{M+3\min\{k,k_2\}/2+2k_1}\|\nabla^2_\xi\hat{g}_{k_1}\|_{L^\infty_t([2^{M-1},2^M])L^2_\xi}\|f_{n,k_2}\|_{L^\infty_t([2^{M-1},2^M])L^2_x}\}\\
    \lesssim &2^{-k-k_1}\min\{2^{-M/4+k_1}\|\hat{f}_{n,k_2}\|_{L^\infty_t([2^{M-1},2^M])L^2_\xi},\\
    &\qquad\qquad\qquad\qquad 2^{M+3\min\{k,k_2\}/2+2k_1}\|f_{n,k_2}\|_{L^\infty_t([2^{M-1},2^M])L^2_x}\}\|\hat{g}_{k_1}\|_{L^\infty_t([2^{M-1},2^M])H^2_\xi}\\
    \lesssim &2^{-k-k_1-2k_{2,+}+\gamma k_2}\min\{2^{-M/4+k_1+k_2/2},2^{M+3\min\{k,k_2\}/2+2k_1+k_2/2}\}\e_1\|\hat{g}_{k_1}\|_{L^\infty_t([2^{M-1},2^M])H^2_\xi}.
\end{align*}
In addition, since we either have $(k_1,k_2)\in\chi^2_k$ or $(k_1,k_2)\in\chi^1_k$ and $c_l=c_m$,
\begin{align*}
    &\|\eqref{im2-1.5}\|_{L^2}\\
    \lesssim &\|\F^{-1}\frac{\partial_{\xi_n}\phi(\xi,\eta)\partial_{\eta_j}\phi(\xi,\eta)}{|\nabla_\eta\phi(\xi,\eta)|^2}\partial_{\eta_m}\frac{\partial_{\xi_l}\phi(\xi,\eta)\partial_{\eta_m}\phi(\xi,\eta)}{|\nabla_\eta\phi(\xi,\eta)|^2}\Tilde{\psi}_{k_1}(\nabla_\eta\phi(\xi,\eta))\Tilde{\psi}_{k_2}(\eta)\|_{L^1}\times\\
    &\times \min\{2^{-M/4}\|\hat{g}_{k_1}\|_{L^\infty_t([2^{M-1},2^M])H^1_\xi}\|\nabla_\xi\hat{f}_{k_2}\|_{L^\infty_t([2^{M-1},2^M])L^2_\xi},\\
    &\qquad\qquad\qquad\qquad 2^{M+3\min\{k,k_2\}/2}\|e^{ic_mt\la}g_{k_1}\|_{L^\infty_t([2^{M-1},2^M])L^2_x}\|e^{ic_nt\la}\F^{-1}\nabla_\xi\hat{f}_{n,k_2}\|_{L^\infty_t([2^{M-1},2^M])L^2_x}\}\\
    \lesssim &2^{k_2-2k_1}\min\{2^{-M/4+k_1}\|\hat{g}_{k_1}\|_{L^\infty_t([2^{M-1},2^M])H^2_\xi}\|\nabla_\xi\hat{f}_{n,k_2}\|_{L^\infty_t([2^{M-1},2^M])L^2_\xi},\\
    &\qquad\qquad\qquad\qquad 2^{M+3\min\{k,k_2\}/2+2k_1}\|\nabla^2_\xi g_{k_1}\|_{L^\infty_t([2^{M-1},2^M])L^2_\xi}\|\nabla_\xi\hat{f}_{n,k_2}\|_{L^\infty_t([2^{M-1},2^M])L^2_\xi}\}\\
    \lesssim &2^{k_2-2k_1}\min\{2^{-M/4+k_1}\|\nabla_\xi\hat{f}_{n,k_2}\|_{L^\infty_t([2^{M-1},2^M])L^2_\xi},\\
    &\qquad\qquad\qquad\qquad 2^{M+3\min\{k,k_2\}/2+2k_1}\|\nabla_\xi\hat{f}_{n,k_2}\|_{L^\infty_t([2^{M-1},2^M])L^2_\xi}\}\|\hat{g}_{k_1}\|_{L^\infty_t([2^{M-1},2^M])H^2_\xi}\\
    \lesssim &2^{k_2-2k_1-2k_{2,+}+\gamma k_2}\min\{2^{-M/4+k_1-k_2/2},2^{M+3\min\{k,k_2\}/2+2k_1-k_2/2}\}\e_1\|\hat{g}_{k_1}\|_{L^\infty_t([2^{M-1},2^M])H^2_\xi}.
\end{align*}
Thus, we found
\begin{align*}
    \|T_{0+iy} \hat{g}\|_{L^2}
    \lesssim & 2^{-k_1-2k_{2,+}+\gamma k_2}\min\{2^{-M/4+k_2/2},2^{M+3\min\{k,k_2\}/2+k_1+k_2/2}\}\e_1\|\hat{g}\|_{L^\infty_t([2^{M-1},2^M])H^1_\xi}
\end{align*} 
and
\begin{align*}
    \|T_{1+iy} \hat{g}\|_{L^2}
    \lesssim &2^{-k_1-2k_{2,+}+\gamma k_2}\min\{2^{-M/4+k_1-k+k_2/2},2^{M+3\min\{k,k_2\}/2+2k_1-k+k_2/2}\}\e_1\|\hat{g}\|_{L^\infty_t([2^{M-1},2^M])H^2_\xi}.
\end{align*}
By Lemma \ref{si}, for any $\alpha\in(0,1)$
\begin{align*}
    \|T_{\alpha} \hat{g}\|_{L^2}
    \lesssim & 2^{-k_1-2k_{2,+}+\gamma k_2+\alpha(k_1-k)}\min\{2^{-M/4+k_2/2},2^{M+3\min\{k,k_2\}/2+k_1+k_2/2}\}\e_1\|\hat{g}\|_{L^\infty_t([2^{M-1},2^M])H^{1+\alpha}_\xi}.
\end{align*} 
As we plug in $\F^{-1}\hat{f}_m\Tilde{\psi}_{k_1}$ for $g$ in the result above, we obtain
\begin{align*}
    \|\eqref{im2.1}\|_{L^2}
    \lesssim & 2^{-k_1-2k_{2,+}+\gamma k_2+\alpha(k_1-k)}\min\{2^{-M/4+k_2/2},2^{m+3\min\{k,k_2\}/2+k_1+k_2/2}\}\e_1\|\hat{f}_{k_1}\|_{L^\infty_t([2^{M-1},2^M])H^{1+\alpha}_\xi}\\
    \lesssim & 2^{-3k_1/2-\alpha k_1-2k_{1,+}+\gamma k_1-2k_{2,+}+\gamma k_2+\alpha(k_1-k)}\min\{2^{-M/4+k_2/2},2^{M+3\min\{k,k_2\}/2+k_1+k_2/2}\}\e_1^2\\
    \lesssim & 2^{-k_1/2-2k_{1,+}+\gamma k_1-2k_{2,+}+\gamma k_2-\alpha k}\min\{2^{-M/4-k_1+k_2/2},2^{M+3\min\{k,k_2\}/2+k_2/2}\}\e_1^2,
\end{align*}
using \eqref{sobolevnorms}.
\end{proof}
Next, we turn to the estimation for the second term.
\begin{lem}
\label{im2-3}
Suppose $t_1,t_2\in[2^{M-1},2^M]$ and $\sup_{t\in[1,T]}\|f_l\|_{Z}\leq \e_1$. If either $(k_1,k_2)\in\chi^2_k$ or $(k_1,k_2)\in\chi^1_k$ and $c_m=c_l$, then
\begin{align*}
    &\|\eqref{im2.3}\|_{L^2}
    \lesssim  2^{-k_1/2-2k_{1,+}+\gamma k_1-2k_{2,+}+\gamma k_2-\alpha k}\min\{2^{-M/4-k_1+k_2/2},2^{M+3\min\{k,k_2\}/2-k_1+3k_2/2}\}\e_1^2.
\end{align*}
\end{lem}
\begin{proof}
Consider the following family of operators on $\{z\in\C:0\leq \Re(z)\leq 1\}$,
\begin{align*}
    T_{z} \hat{g} =D^z\int_{t_1}^{t_2}\int_{\R^3}e^{it\phi(\xi,\eta)}\frac{\partial_{\xi_l}\phi(\xi,\eta)\partial_{\eta_m}\phi(\xi,\eta)}{|\nabla_\eta\phi(\xi,\eta)|^2}\partial_{\eta_m}\hat{g}_{k_1}(t,\xi-\eta)\hat{f}_{n,k_2}(t,\eta)\psi_k(\xi)d\eta dt.
\end{align*}
We want to show that for any $y\in\R$, the operators $T_{0+iy}:L^{\infty}_t([2^{M-1},2^M])H^1_\xi\rightarrow L^2_\xi$ and $T_{1+iy}:L^{\infty}_t([2^{M-1},2^M])H^2_\xi\rightarrow L^2_\xi$ are bounded.\\
Employing the Strichartz's estimate in Lemma \ref{strichartz}, along with Lemma \ref{chi+eta2}, Lemma \ref{dualitycomp}, and using \eqref{sobolevnorms}, \eqref{l2}, and \eqref{addder}, we get
\begin{align*}
    &\|T_{0+iy} \hat{g}\|_{L^2}\\
    =&\bigg\|\int_{t_1}^{t_2}\int_{\R^3}e^{it\phi(\xi,\eta)}\frac{\partial_{\xi_l}\phi(\xi,\eta)\partial_{\eta_m}\phi(\xi,\eta)}{|\nabla_\eta\phi(\xi,\eta)|^2}\partial_{\eta_m}\hat{g}_{k_1}(t,\xi-\eta)\hat{f}_{n,k_2}(t,\eta)\psi_k(\xi)d\eta dt\bigg\|_{L^2}\\
    \lesssim & \|\F^{-1}\frac{\partial_{\xi_l}\phi(\xi,\eta)\partial_{\eta_m}\phi(\xi,\eta)}{|\nabla_\eta\phi(\xi,\eta)|^2}\Tilde{\psi}_{k_1}(\nabla_\eta\phi(\xi,\eta))\Tilde{\psi}_{k_2}(\eta)\|_{L^1}\times\\
    &\times \min\{2^{-M/4}\|\nabla_\xi\hat{g}_{k_1}\|_{L^\infty_t([2^{M-1},2^M])L^2_\xi}\|\hat{f}_{n,k_2}\|_{L^\infty_t([2^{M-1},2^M])H^1_\xi},\\
    &\qquad\qquad\qquad\qquad 2^{M+3\min\{k,k_2\}/2}\|e^{ic_mt\la}\F^{-1}\nabla_\xi\hat{g}_{k_1}\|_{L^\infty_t([2^{M-1},2^M])L^2_x}\|e^{ic_nt\la}f_{n,k_2}\|_{L^\infty_t([2^{M-1},2^M])L^2_x}\}\\
    \lesssim & 2^{k_2-k_1}\min\{2^{-M/4}\|\hat{f}_{n,k_2}\|_{L^\infty_t([2^{M-1},2^M])H^1_\xi},2^{M+3\min\{k,k_2\}/2}\|\hat{f}_{n,k_2}\|_{L^\infty_t([2^{M-1},2^M])L^2_\xi}\}\|\nabla_\xi \hat{g}_{k_1}\|_{L^\infty_t([2^{M-1},2^M])L^2_\xi}\\
    \lesssim & 2^{k_2-k_1-2k_{2,+}+\gamma k_2}\min\{2^{-M/4-k_2/2},2^{M+3\min\{k,k_2\}/2+k_2/2}\}\e_1\|\hat{g}_{k_1}\|_{L^\infty_t([2^{M-1},2^M])H^1_\xi}.
\end{align*}
Next, we have
\begin{align}
    \|T_{1+iy} \hat{g}\|_{L^2}
    \leq &\bigg\|\partial_{\xi_n}\int_{t_1}^{t_2}\int_{\R^3}e^{it\phi(\xi,\eta)}\frac{\partial_{\xi_l}\phi(\xi,\eta)\partial_{\eta_m}\phi(\xi,\eta)}{|\nabla_\eta\phi(\xi,\eta)|^2}\partial_{\eta_m}\hat{g}_{k_1}(t,\xi-\eta)\hat{f}_{n,k_2}(t,\eta)\psi_k(\xi)d\eta dt\bigg\|_{L^2}\nonumber\\
    \leq &\bigg\|\int_{t_1}^{t_2}\int_{\R^3}e^{it\phi(\xi,\eta)}\frac{\partial_{\xi_l}\phi(\xi,\eta)\partial_{\eta_m}\phi(\xi,\eta)}{|\nabla_\eta\phi(\xi,\eta)|^2}\partial_{\xi_n}\partial_{\eta_m}\hat{g}_{k_1}(t,\xi-\eta)\hat{f}_{n,k_2}(t,\eta)\psi_k(\xi)d\eta dt\bigg\|_{L^2}\label{im2-3.1}\\
    &+ \bigg\|\int_{t_1}^{t_2}\int_{\R^3}e^{it\phi(\xi,\eta)}\partial_{\xi_n}\frac{\partial_{\xi_l}\phi(\xi,\eta)\partial_{\eta_m}\phi(\xi,\eta)}{|\nabla_\eta\phi(\xi,\eta)|^2}\partial_{\eta_m}\hat{g}_{k_1}(t,\xi-\eta)\hat{f}_{n,k_2}(t,\eta)\psi_k(\xi)d\eta dt\bigg\|_{L^2}\label{im2-3.11}\\
    &+ \bigg\|\int_{t_1}^{t_2}\int_{\R^3}e^{it\phi(\xi,\eta)}\frac{\partial_{\xi_l}\phi(\xi,\eta)\partial_{\eta_m}\phi(\xi,\eta)}{|\nabla_\eta\phi(\xi,\eta)|^2}\partial_{\eta_m}\hat{g}_{k_1}(t,\xi-\eta)\hat{f}_{n,k_2}(t,\eta)\partial_{\xi_n}\psi_k(\xi)d\eta dt\bigg\|_{L^2}\label{im2-3.12}\\
    & + \bigg\|\int_{t_1}^{t_2}\int_{\R^3}e^{it\phi(\xi,\eta)}it\partial_{\xi_n}\phi(\xi,\eta)\frac{\partial_{\xi_l}\phi(\xi,\eta)\partial_{\eta_m}\phi(\xi,\eta)}{|\nabla_\eta\phi(\xi,\eta)|^2}\partial_{\eta_m}\hat{g}_{k_1}(t,\xi-\eta)\hat{f}_{n,k_2}(t,\eta)\psi_k(\xi)d\eta dt\bigg\|_{L^2}\label{im2-3.2}.
\end{align}
Integration by parts using the identity $e^{it\phi(\xi,\eta)}=\sum_{j=1}^3\frac{\partial_{\eta_j}\phi(\xi,\eta)}{it|\nabla_\eta\phi(\xi,\eta)|^2}\partial_{\eta_j}e^{it\phi(\xi,\eta)}$ on the term $\eqref{im2-3.2}$ gives
\begin{align}
    &\int_{t_1}^{t_2}\int_{\R^3}e^{it\phi(\xi,\eta)}it\partial_{\xi_n}\phi(\xi,\eta)\frac{\partial_{\xi_l}\phi(\xi,\eta)\partial_{\eta_m}\phi(\xi,\eta)}{|\nabla_\eta\phi(\xi,\eta)|^2}\partial_{\eta_m}\hat{g}_{k_1}(t,\xi-\eta)\hat{f}_{n,k_2}(t,\eta)\psi_k(\xi)d\eta dt\nonumber\\
    = & -\int_{t_1}^{t_2}\int_{\R^3}e^{it\phi(\xi,\eta)}\partial_{\eta_j}\big(\frac{\partial_{\eta_j}\phi(\xi,\eta)\partial_{\xi_n}\phi(\xi,\eta)}{|\nabla_\eta\phi(\xi,\eta)|^2}\frac{\partial_{\xi_l}\phi(\xi,\eta)\partial_{\eta_m}\phi(\xi,\eta)}{|\nabla_\eta\phi(\xi,\eta)|^2}\big)\partial_{\eta_m}\hat{g}_{k_1}(t,\xi-\eta)\hat{f}_{n,k_2}(t,\eta)\psi_k(\xi)d\eta dt\label{im2-3.3}\\
    &-\int_{t_1}^{t_2}\int_{\R^3}e^{it\phi(\xi,\eta)}\frac{\partial_{\eta_j}\phi(\xi,\eta)\partial_{\xi_n}\phi(\xi,\eta)}{|\nabla_\eta\phi(\xi,\eta)|^2}\frac{\partial_{\xi_l}\phi(\xi,\eta)\partial_{\eta_m}\phi(\xi,\eta)}{|\nabla_\eta\phi(\xi,\eta)|^2}\partial_{\eta_j}\partial_{\eta_m}\hat{g}_{k_1}(t,\xi-\eta)\hat{f}_{n,k_2}(t,\eta)\psi_k(\xi)d\eta dt\label{im2-3.4}\\
    & -\int_{t_1}^{t_2}\int_{\R^3}e^{it\phi(\xi,\eta)}\frac{\partial_{\eta_j}\phi(\xi,\eta)\partial_{\xi_n}\phi(\xi,\eta)}{|\nabla_\eta\phi(\xi,\eta)|^2}\frac{\partial_{\xi_l}\phi(\xi,\eta)\partial_{\eta_m}\phi(\xi,\eta)}{|\nabla_\eta\phi(\xi,\eta)|^2}\partial_{\eta_m}\hat{g}_{k_1}(t,\xi-\eta)\partial_{\eta_j}\hat{f}_{n,k_2}(t,\eta)\psi_k(\xi)d\eta dt\label{im2-3.5}.
\end{align}
Using the Strichartz's estimate in Lemma \ref{strichartz}, together with Lemma \ref{dualitycomp}, Lemma \ref{chi+eta2}, \eqref{sobolevnorms}, \eqref{l2}, \eqref{l2first}, and \eqref{addder}, we obtain
\begin{align*}
    &\eqref{im2-3.1}+\|\eqref{im2-3.4}\|_{L^2}\\
    \lesssim & \big(\|\F^{-1}\frac{\partial_{\xi_l}\phi(\xi,\eta)\partial_{\eta_m}\phi(\xi,\eta)}{|\nabla_\eta\phi(\xi,\eta)|^2}\Tilde{\psi}_{k_1}(\nabla_\eta\phi(\xi,\eta))\Tilde{\psi}_{k_2}(\eta)\|_{L^1}+\\
    &+\|\F^{-1}\frac{\partial_{\eta_j}\phi(\xi,\eta)\partial_{\xi_n}\phi(\xi,\eta)}{|\nabla_\eta\phi(\xi,\eta)|^2}\frac{\partial_{\xi_l}\phi(\xi,\eta)\partial_{\eta_m}\phi(\xi,\eta)}{|\nabla_\eta\phi(\xi,\eta)|^2}\Tilde{\psi}_{k_1}(\nabla_\eta\phi(\xi,\eta))\Tilde{\psi}_{k_2}(\eta)\|_{L^1}\big)\times\\
    &\times \min\{2^{-M/4}\|\nabla^2_\xi\hat{g}_{k_1}\|_{L^\infty_t([2^{M-1},2^M])L^2_\xi}\|\hat{f}_{n,k_2}\|_{L^\infty_t([2^{M-1},2^M])H^1_\xi},\\
    &\qquad\qquad\qquad\qquad 2^{M+3\min\{k,k_2\}/2}\|e^{ic_mt\la}\F^{-1}\nabla^2_\xi\hat{g}_{k_1}\|_{L^\infty_t([2^{M-1},2^M])L^2_x}\|e^{ic_nt\la}f_{n,k_2}\|_{L^\infty_t([2^{M-1},2^M])L^2_x}\}\\
    \lesssim & (2^{k_2-k_1}+2^{2k_2-2k_1})\min\{2^{-M/4}\|\hat{f}_{n,k_2}\|_{L^\infty_t([2^{M-1},2^M])H^1_\xi},\\
    &\qquad\qquad\qquad\qquad 2^{M+3\min\{k,k_2\}/2}\|f_{n,k_2}\|_{L^\infty_t([2^{M-1},2^M])L^2_x}\}\|\hat{g}_{k_1}\|_{L^\infty_t([2^{M-1},2^M])H^2_\xi}\\
    \lesssim & 2^{k_2-k_1-2k_{2,+}+\gamma k_2}\min\{2^{-M/4-k_2/2},2^{M+3\min\{k,k_2\}/2+k_2/2}\}\e_1\|\hat{g}_{k_1}\|_{L^\infty_t([2^{M-1},2^M])H^2_\xi},
\end{align*}
\begin{align*}
    &\eqref{im2-3.11}+\eqref{im2-3.12}+\|\eqref{im2-3.3}\|_{L^2}\\
    \lesssim &\big(\|\F^{-1}\partial_{\xi_n}\frac{\partial_{\xi_l}\phi(\xi,\eta)\partial_{\eta_m}\phi(\xi,\eta)}{|\nabla_\eta\phi(\xi,\eta)|^2}\Tilde{\psi}_{k_1}(\nabla_\eta\phi(\xi,\eta))\Tilde{\psi}_{k_2}(\eta)\|_{L^1}+\\
    &+ \|\nabla\psi_k\|_{L^\infty}\|\F^{-1}\frac{\partial_{\xi_l}\phi(\xi,\eta)\partial_{\eta_m}\phi(\xi,\eta)}{|\nabla_\eta\phi(\xi,\eta)|^2}\Tilde{\psi}_{k_1}(\nabla_\eta\phi(\xi,\eta))\Tilde{\psi}_{k_2}(\eta)\|_{L^1}\\
    & + \|\F^{-1}\partial_{\eta_j}\big(\frac{\partial_{\eta_j}\phi(\xi,\eta)\partial_{\xi_n}\phi(\xi,\eta)}{|\nabla_\eta\phi(\xi,\eta)|^2}\frac{\partial_{\xi_l}\phi(\xi,\eta)\partial_{\eta_m}\phi(\xi,\eta)}{|\nabla_\eta\phi(\xi,\eta)|^2}\big)\Tilde{\psi}_{k_1}(\nabla_\eta\phi(\xi,\eta))\Tilde{\psi}_{k_2}(\eta)\|_{L^1}\big)\times\\
    &\times \min\{2^{-M/4}\|\nabla_\xi\hat{g}_{k_1}\|_{L^\infty_t([2^{M-1},2^M])H^1_\xi}\|\hat{f}_{n,k_2}\|_{L^\infty_t([2^{M-1},2^M])L^2_\xi},\\
    &\qquad\qquad\qquad\qquad 2^{M+3\min\{k,k_2\}/2}\|e^{ic_mt\la}\F^{-1}\nabla_\xi\hat{g}_{k_1}\|_{L^\infty_t([2^{M-1},2^M])L^2_x}\|e^{ic_nt\la}f_{n,k_2}\|_{L^\infty_t([2^{M-1},2^M])L^2_x}\}\\
    \lesssim &(2^{k_2-2k_1}+2^{-k+k_2-k_1}+2^{k_2-2k_1})\min\{2^{-M/4}\|\hat{f}_{n,k_2}\|_{L^\infty_t([2^{M-1},2^M])L^2_\xi},\\
    &\qquad\qquad\qquad\qquad 2^{M+3\min\{k,k_2\}/2+k_1}\|f_{n,k_2}\|_{L^\infty_t([2^{M-1},2^M])L^2_x}\}\|\hat{g}_{k_1}\|_{L^\infty_t([2^{M-1},2^M])H^2_\xi}\\
    \lesssim &2^{-k+k_2-k_1-2k_{2,+}+\gamma k_2}\min\{2^{-M/4+k_2/2},2^{M+3\min\{k,k_2\}/2+k_2/2+k_1}\}\e_1\|\hat{g}_{k_1}\|_{L^\infty_t([2^{M-1},2^M])H^2_\xi},
\end{align*}
and
\begin{align*}
    &\|\eqref{im2-3.5}\|_{L^2}\\
    \lesssim &\|\F^{-1}\frac{\partial_{\eta_j}\phi(\xi,\eta)\partial_{\xi_n}\phi(\xi,\eta)}{|\nabla_\eta\phi(\xi,\eta)|^2}\frac{\partial_{\xi_l}\phi(\xi,\eta)\partial_{\eta_m}\phi(\xi,\eta)}{|\nabla_\eta\phi(\xi,\eta)|^2}\Tilde{\psi}_{k_1}(\nabla_\eta\phi(\xi,\eta))\Tilde{\psi}_{k_2}(\eta)\|_{L^1}\times\\
    &\times \min\{2^{-M/4}\|\nabla_\xi\hat{g}_{k_1}\|_{L^\infty_t([2^{M-1},2^M])H^1_\xi}\|\nabla_\xi\hat{f}_{n,k_2}\|_{L^\infty_t([2^{M-1},2^M])L^2_\xi},\\
    &\qquad\qquad\qquad 2^{M+3\min\{k,k_2\}/2}\|e^{ic_mt\la}\F^{-1}\nabla_\xi\hat{g}_{k_1}\|_{L^\infty_t([2^{M-1},2^M])L^2_x}\|e^{ic_nt\la}\F^{-1}\nabla_\xi\hat{f}_{n,k_2}\|_{L^\infty_t([2^{M-1},2^M])L^2_x}\}\\
    \lesssim &2^{2k_2-2k_1}\min\{2^{-M/4}\|\nabla_\xi\hat{f}_{n,k_2}\|_{L^\infty_t([2^{M-1},2^M])L^2_\xi},\\
    &\qquad\qquad\qquad\qquad 2^{M+3\min\{k,k_2\}/2+k_1}\|\nabla_\xi\hat{f}_{n,k_2}\|_{L^\infty_t([2^{M-1},2^M])L^2_\xi}\}\|\hat{g}_{k_1}\|_{L^\infty_t([2^{M-1},2^M])H^2_\xi}\\
    \lesssim &2^{2k_2-2k_1-2k_{2,+}+\gamma k_2}\min\{2^{-M/4-k_2/2},2^{M+3\min\{k,k_2\}/2+k_1-k_2/2}\}\|\hat{g}_{k_1}\|_{L^\infty_t([2^{M-1},2^M])H^2_\xi}.
\end{align*}
Thus, combining the estimates above, we found
\begin{align*}
    \|T_{0+iy} \hat{g}\|_{L^2}
    \lesssim &2^{-2k_{2,+}+\gamma k_2}\min\{2^{-M/4-k_1+k_2/2},2^{M+3\min\{k,k_2\}/2-k_1+3k_2/2}\}\e_1\|\hat{g}\|_{L^\infty_t([2^{M-1},2^M])H^1_\xi}
\end{align*} 
and
\begin{align*}
    \|T_{1+iy} \hat{g}\|_{L^2}
    \lesssim  &2^{-2k_{2,+}+\gamma k_2}\min\{2^{-M/4-k_1+k_2/2},2^{M+3\min\{k,k_2\}/2-k_1+3k_2/2}\}\e_1\|\hat{g}_{k_1}\|_{L^\infty_t([2^{M-1},2^M])H^2_\xi}\\
    &+ 2^{-2k_{2,+}+\gamma k_2}\min\{2^{-M/4-k_1-k+3k_2/2},2^{M+3\min\{k,k_2\}/2-k+3k_2/2}\}\e_1\|\hat{g}\|_{L^\infty_t([2^{M-1},2^M])H^2_\xi}\\
    &+2^{-2k_{2,+}+\gamma k_2}\min\{2^{-M/4-2k_1+3k_2/2},2^{M+3\min\{k,k_2\}/2-k_1+3k_2/2}\}\|\hat{g}_{k_1}\|_{L^\infty_t([2^{M-1},2^M])H^2_\xi}\\
    \lesssim & 2^{-2k_{2,+}+\gamma k_2}\min\{2^{-M/4-k+k_2/2},2^{M+3\min\{k,k_2\}/2-k+3k_2/2}\}\e_1\|\hat{g}\|_{L^\infty_t([2^{M-1},2^M])H^2_\xi}.
\end{align*}
By the interpolation result in Lemma \ref{si}, we have
\begin{align*}
    \|T_{\alpha} \hat{g}\|_{L^2}
    \lesssim & 2^{-2k_{1,+}+\gamma k_1+\alpha(k_1-k)}\min\{2^{-M/4-k_1+k_2/2},2^{M+3\min\{k,k_2\}/2-k_1+3k_2/2}\}\e_1\|\hat{g}\|_{L^\infty_t([2^{M-1},2^M])H^{1+\alpha}_\xi},
\end{align*} 
for any $\alpha\in(0,1)$.\\
When we take $\F^{-1}\hat{f}_m\Tilde{\psi}_{k_1}$ for $g$ and use \eqref{sobolevnorms}, we obtain
\begin{align*}
    &\|\eqref{im2.3}\|_{L^2}\\
    \lesssim & 2^{-2k_{1,+}+\gamma k_1+\alpha(k_1-k)}\min\{2^{-M/4-k_1+k_2/2},2^{M+3\min\{k,k_2\}/2-k_1+3k_2/2}\}\e_1\|\hat{f}_{m,k_1}\|_{L^\infty_t([2^{M-1},2^M])H^{1+\alpha}_\xi}\\
    \lesssim & 2^{-k_1/2-\alpha k_1-2k_{1,+}+\gamma k_1-2k_{2,+}+\gamma k_2+\alpha(k_1-k)}\min\{2^{-M/4-k_1+k_2/2},2^{M+3\min\{k,k_2\}/2-k_1+3k_2/2}\}\e_1^2\\
    \lesssim & 2^{-k_1/2-2k_{1,+}+\gamma k_1-2k_{2,+}+\gamma k_2-\alpha k}\min\{2^{-M/4-k_1+k_2/2},2^{M+3\min\{k,k_2\}/2-k_1+3k_2/2}\}\e_1^2
\end{align*}
as desired.
\end{proof}
Now, we turn to the final term in the splitting of $I^{M,2}_{k,k_1,k_2}$.
\begin{lem}
\label{im2-4}
Suppose $t_1,t_2\in[2^{M-1},2^M]$ and $\sup_{t\in[1,T]}\|f_l\|_{Z}\leq \e_1$. If either $(k_1,k_2)\in\chi^2_k$ or $(k_1,k_2)\in\chi^1_k$ and $c_m=c_l$, then
\begin{align*}
    &\|\eqref{im2.4}\|_{L^2}
    \lesssim  2^{-\alpha k_2-2k_{1,+}+\gamma k_1-2k_{2,+}+\gamma k_2+\alpha(k_1-k)}\min\{2^{-M/4+k_2/2-3k_1/2},2^{M+3\min\{k,k_2\}/2-k_1/2+k_2/2}\}\e_1^2.
\end{align*}
\end{lem}
\begin{proof}
We start with defining the following family of operators on $\{z\in\C:0\leq \Re(z)\leq 1\}$,
\begin{align*}
    T_{z} \hat{g} =D^z\int_{t_1}^{t_2}\int_{\R^3}e^{it\phi(\xi,\eta)}\frac{\partial_{\xi_l}\phi(\xi,\eta)\partial_{\eta_m}\phi(\xi,\eta)}{|\nabla_\eta\phi(\xi,\eta)|^2}\hat{f}_{m,k_1}(t,\xi-\eta)\partial_{\eta_m}\hat{g}_{k_2}(t,\eta)\psi_k(\xi)d\eta dt,
\end{align*}
which we want to show $T_{0+iy}:L^{\infty}_t([2^{M-1},2^M])H^1_\xi\rightarrow L^2_\xi$ and $T_{1+iy}:L^{\infty}_t([2^{M-1},2^M])H^2_\xi\rightarrow L^2_\xi$ are bounded for all $y\in\R$.\\
Based on the Strichartz's estimate in Lemma \ref{chi+eta2}, as well as Lemma \ref{dualitycomp}, Lemma \ref{strichartz}, \eqref{addder}, \eqref{l2}, and \eqref{sobolevnorms}, we can get
\begin{align*}
    &\|T_{0+iy} \hat{g}\|_{L^2}\\
    =&\bigg\|\int_{t_1}^{t_2}\int_{\R^3}e^{it\phi(\xi,\eta)}\frac{\partial_{\xi_l}\phi(\xi,\eta)\partial_{\eta_m}\phi(\xi,\eta)}{|\nabla_\eta\phi(\xi,\eta)|^2}\hat{f}_{m,k_1}(t,\xi-\eta)\partial_{\eta_m}\hat{g}_{k_2}(t,\eta)\psi_k(\xi)d\eta dt\bigg\|_{L^2}\\
    \lesssim & \|\F^{-1}\frac{\partial_{\xi_l}\phi(\xi,\eta)\partial_{\eta_m}\phi(\xi,\eta)}{|\nabla_\eta\phi(\xi,\eta)|^2}\Tilde{\psi}_{k_1}(\nabla_\eta\phi(\xi,\eta))\Tilde{\psi}_{k_2}(\eta)\|_{L^1}\times\\
    &\times \min\{2^{-M/4}\|\hat{f}_{m,k_1}\|_{L^\infty_t([2^{M-1},2^M])H^1_\xi}\|\nabla_\xi\hat{g}_{k_2}\|_{L^\infty_t([2^{M-1},2^M])L^2_\xi},\\
    &\qquad\qquad\qquad\qquad 2^{M+3\min\{k,k_2\}/2}\|e^{ic_mt\la}f_{m,k_1}\|_{L^\infty_t([2^{M-1},2^M])L^2_x}\|e^{ic_nt\la}\F^{-1}\nabla_\xi\hat{g}_{k_2}\|_{L^\infty_t([2^{M-1},2^M])L^2_x}\}\\
    \lesssim & 2^{k_2-k_1}\min\{2^{-M/4}\|\hat{f}_{m,k_1}\|_{L^\infty_t([2^{M-1},2^M])H^1_\xi},\\
    &\qquad\qquad\qquad\qquad 2^{M+3\min\{k,k_2\}/2}\|\hat{f}_{m,k_1}\|_{L^\infty_t([2^{M-1},2^M])L^2_\xi}\}\|\nabla_\xi \hat{g}_{k_2}\|_{L^\infty_t([2^{M-1},2^M])L^2_\xi}\\
    \lesssim & 2^{k_2-k_1-2k_{1,+}+\gamma k_1}\min\{2^{-M/4-k_1/2},2^{M+3\min\{k,k_2\}/2+k_1/2}\}\e_1\|\hat{g}_{k_2}\|_{L^\infty_t([2^{M-1},2^M])H^1_\xi}.
\end{align*}
Next, we look at the operator $T_{1+iy}$,
\begin{align}
    \|T_{1+iy} \hat{g}\|_{L^2}
    \leq &\bigg\|\partial_{\xi_n}\int_{t_1}^{t_2}\int_{\R^3}e^{it\phi(\xi,\eta)}\frac{\partial_{\xi_l}\phi(\xi,\eta)\partial_{\eta_m}\phi(\xi,\eta)}{|\nabla_\eta\phi(\xi,\eta)|^2}\hat{f}_{m,k_1}(t,\xi-\eta)\partial_{\eta_m}\hat{g}_{k_2}(t,\eta)\psi_k(\xi)d\eta dt\bigg\|_{L^2}\nonumber\\
    \leq &\bigg\|\int_{t_1}^{t_2}\int_{\R^3}e^{it\phi(\xi,\eta)}\frac{\partial_{\xi_l}\phi(\xi,\eta)\partial_{\eta_m}\phi(\xi,\eta)}{|\nabla_\eta\phi(\xi,\eta)|^2}\partial_{\xi_n}\hat{f}_{m,k_1}(t,\xi-\eta)\partial_{\eta_m}\hat{g}_{k_2}(t,\eta)\psi_k(\xi)d\eta dt\bigg\|_{L^2}\label{im2-1.1-1}\\
    &+ \bigg\|\int_{t_1}^{t_2}\int_{\R^3}e^{it\phi(\xi,\eta)}\partial_{\xi_n}\frac{\partial_{\xi_l}\phi(\xi,\eta)\partial_{\eta_m}\phi(\xi,\eta)}{|\nabla_\eta\phi(\xi,\eta)|^2}\hat{f}_{m,k_1}(t,\xi-\eta)\partial_{\eta_m}\hat{g}_{k_2}(t,\eta)\psi_k(\xi)d\eta dt\bigg\|_{L^2}\label{im2-1.11-1}\\
    &+ \bigg\|\int_{t_1}^{t_2}\int_{\R^3}e^{it\phi(\xi,\eta)}\frac{\partial_{\xi_l}\phi(\xi,\eta)\partial_{\eta_m}\phi(\xi,\eta)}{|\nabla_\eta\phi(\xi,\eta)|^2}\hat{f}_{m,k_1}(t,\xi-\eta)\partial_{\eta_m}\hat{g}_{k_2}(t,\eta)\partial_{\xi_n}\psi_k(\xi)d\eta dt\bigg\|_{L^2}\label{im2-1.12-1}\\
    & + \bigg\|\int_{t_1}^{t_2}\int_{\R^3}e^{it\phi(\xi,\eta)}it\partial_{\xi_n}\phi(\xi,\eta)\frac{\partial_{\xi_l}\phi(\xi,\eta)\partial_{\eta_m}\phi(\xi,\eta)}{|\nabla_\eta\phi(\xi,\eta)|^2}\hat{f}_{m,k_1}(t,\xi-\eta)\partial_{\eta_m}\hat{g}_{k_2}(t,\eta)\psi_k(\xi)d\eta dt\bigg\|_{L^2}\label{im2-1.2-1},
\end{align}
and perform an integration by parts on $\eqref{im2-1.2-1}$ using the identity $e^{it\phi(\xi,\eta)}=\sum_{j=1}^3\frac{\partial_{\eta_j}\phi(\xi,\eta)}{it|\nabla_\eta\phi(\xi,\eta)|^2}\partial_{\eta_j}e^{it\phi(\xi,\eta)}$,
\begin{align}
    &\int_{t_1}^{t_2}\int_{\R^3}e^{it\phi(\xi,\eta)}it\partial_{\xi_n}\phi(\xi,\eta)\frac{\partial_{\xi_l}\phi(\xi,\eta)\partial_{\eta_m}\phi(\xi,\eta)}{|\nabla_\eta\phi(\xi,\eta)|^2}\hat{f}_{m,k_1}(t,\xi-\eta)\partial_{\eta_m}\hat{g}_{k_2}(t,\eta)\psi_k(\xi)d\eta dt\nonumber\\
    = & -\int_{t_1}^{t_2}\int_{\R^3}e^{it\phi(\xi,\eta)}\partial_{\eta_j}\big(\frac{\partial_{\xi_n}\phi(\xi,\eta)\partial_{\eta_j}\phi(\xi,\eta)}{|\nabla_\eta\phi(\xi,\eta)|^2}\frac{\partial_{\xi_l}\phi(\xi,\eta)\partial_{\eta_m}\phi(\xi,\eta)}{|\nabla_\eta\phi(\xi,\eta)|^2}\big)\hat{f}_{m,k_1}(t,\xi-\eta)\partial_{\eta_m}\hat{g}_{k_2}(t,\eta)\psi_k(\xi)d\eta dt\label{im2-1.3-1}\\
    & -\int_{t_1}^{t_2}\int_{\R^3}e^{it\phi(\xi,\eta)}\frac{\partial_{\xi_n}\phi(\xi,\eta)\partial_{\eta_j}\phi(\xi,\eta)}{|\nabla_\eta\phi(\xi,\eta)|^2}\frac{\partial_{\xi_l}\phi(\xi,\eta)\partial_{\eta_m}\phi(\xi,\eta)}{|\nabla_\eta\phi(\xi,\eta)|^2}\partial_{\eta_j}\hat{f}_{m,k_1}(t,\xi-\eta)\partial_{\eta_m}\hat{g}_{k_2}(t,\eta)\psi_k(\xi)d\eta dt\label{im2-1.4-1}\\
    & -\int_{t_1}^{t_2}\int_{\R^3}e^{it\phi(\xi,\eta)}\frac{\partial_{\xi_n}\phi(\xi,\eta)\partial_{\eta_j}\phi(\xi,\eta)}{|\nabla_\eta\phi(\xi,\eta)|^2}\frac{\partial_{\xi_l}\phi(\xi,\eta)\partial_{\eta_m}\phi(\xi,\eta)}{|\nabla_\eta\phi(\xi,\eta)|^2}\hat{f}_{m,k_1}(t,\xi-\eta)\partial_{\eta_j}\partial_{\eta_m}\hat{g}_{k_2}(t,\eta)\psi_k(\xi)d\eta dt\label{im2-1.5-1}.
\end{align}
Using the Strichartz's estimate in Lemma \ref{strichartz}, along with Lemma \ref{dualitycomp}, Lemma \ref{chi+eta2}, \eqref{sobolevnorms}, \eqref{l2}, \eqref{l2first}, and \eqref{addder}, we get
\begin{align*}
    &\eqref{im2-1.1-1}+\|\eqref{im2-1.4-1}\|_{L^2}\\
    \lesssim & \big(\|\F^{-1}\frac{\partial_{\xi_l}\phi(\xi,\eta)\partial_{\eta_m}\phi(\xi,\eta)}{|\nabla_\eta\phi(\xi,\eta)|^2}\Tilde{\psi}_{k_1}(\nabla_\eta\phi(\xi,\eta))\Tilde{\psi}_{k_2}(\eta)\|_{L^1}+\\
    &+\|\F^{-1}\frac{\partial_{\xi_n}\phi(\xi,\eta)\partial_{\eta_j}\phi(\xi,\eta)}{|\nabla_\eta\phi(\xi,\eta)|^2}\frac{\partial_{\xi_l}\phi(\xi,\eta)\partial_{\eta_m}\phi(\xi,\eta)}{|\nabla_\eta\phi(\xi,\eta)|^2}\Tilde{\psi}_{k_1}(\nabla_\eta\phi(\xi,\eta))\Tilde{\psi}_{k_2}(\eta)\|_{L^1}\big)\times\\
    &\times \min\{2^{-M/4}\|\nabla_\xi\hat{f}_{m,k_1}\|_{L^\infty_t([2^{M-1},2^M])L^2_\xi}\|\nabla_\xi\hat{g}_{k_2}\|_{L^\infty_t([2^{M-1},2^M])H^1_\xi},\\
    &\qquad\qquad\qquad2^{M+3\min\{k,k_2\}/2}\|e^{ic_mt\la}\F^{-1}\nabla_\xi\hat{f}_{m,k_1}\|_{L^\infty_t([2^{M-1},2^M])L^2_x}\|e^{ic_nt\la}\F^{-1}\nabla_\xi\hat{g}_{k_2}\|_{L^\infty_t([2^{M-1},2^M])L^2_x}\}\\
    \lesssim & (2^{k_2-k_1}+2^{2k_2-2k_1})\min\{2^{-M/4}\|\nabla_\xi\hat{f}_{m,k_1}\|_{L^\infty_t([2^{M-1},2^M])L^2_\xi},\\
    &\qquad\qquad\qquad\qquad 2^{M+3\min\{k,k_2\}/2+k_2}\|\nabla_\xi f_{m,k_1}\|_{L^\infty_t([2^{M-1},2^M])L^2_\xi}\}\|\hat{g}_{k_2}\|_{L^\infty_t([2^{M-1},2^M])H^2_\xi}\\
    \lesssim & 2^{k_2-k_1-2k_{1,+}+\gamma k_1}\min\{2^{-M/4-k_1/2},2^{M+3\min\{k,k_2\}/2+k_2-k_1/2}\}\e_1\|\hat{g}_{k_2}\|_{L^\infty_t([2^{M-1},2^M])H^2_\xi},
\end{align*}
\begin{align*}
    &\eqref{im2-1.11-1}+\eqref{im2-1.12-1}+\|\eqref{im2-1.3-1}\|_{L^2}\\
    \lesssim &\big(\|\F^{-1}\partial_{\xi_n}\frac{\partial_{\xi_l}\phi(\xi,\eta)\partial_{\eta_m}\phi(\xi,\eta)}{|\nabla_\eta\phi(\xi,\eta)|^2}\Tilde{\psi}_{k_1}(\nabla_\eta\phi(\xi,\eta))\Tilde{\psi}_{k_2}(\eta)\|_{L^1}
    + \\
    & + \|\nabla\psi_k\|_{L^\infty}\|\F^{-1}\frac{\partial_{\xi_l}\phi(\xi,\eta)\partial_{\eta_m}\phi(\xi,\eta)}{|\nabla_\eta\phi(\xi,\eta)|^2}\Tilde{\psi}_{k_1}(\nabla_\eta\phi(\xi,\eta))\Tilde{\psi}_{k_2}(\eta)\|_{L^1}+\\
    & + \|\F^{-1}\partial_{\eta_j}\big(\frac{\partial_{\xi_n}\phi(\xi,\eta)\partial_{\eta_j}\phi(\xi,\eta)}{|\nabla_\eta\phi(\xi,\eta)|^2}\frac{\partial_{\xi_l}\phi(\xi,\eta)\partial_{\eta_m}\phi(\xi,\eta)}{|\nabla_\eta\phi(\xi,\eta)|^2}\big)\Tilde{\psi}_{k_1}(\nabla_\eta\phi(\xi,\eta))\Tilde{\psi}_{k_2}(\eta)\|_{L^1}\big)\times\\
    &\times \min\{2^{-M/4}\|\hat{f}_{m,k_1}\|_{L^\infty_t([2^{M-1},2^M])L^2_\xi}\|\nabla_\xi\hat{g}_{k_2}\|_{L^\infty_t([2^{M-1},2^M])H^1_\xi},\\
    &\qquad\qquad\qquad\qquad 2^{M+3\min\{k,k_2\}/2}\|e^{ic_mt\la}f_{m,k_1}\|_{L^\infty_t([2^{M-1},2^M])L^2_x}\|e^{ic_nt\la}\F^{-1}\nabla_\xi\hat{g}_{k_2}\|_{L^\infty_t([2^{M-1},2^M])L^2_x}\}\\
    \lesssim &(2^{k_2-2k_1}+2^{-k+k_2-k_1}+2^{k_2-2k_1})\min\{2^{-M/4}\|\hat{f}_{m,k_1}\|_{L^\infty_t([2^{M-1},2^M])L^2_\xi},\\
    &\qquad\qquad\qquad\qquad 2^{M+3\min\{k,k_2\}/2+k_2}\|f_{m,k_1}\|_{L^\infty_t([2^{M-1},2^M])L^2_x}\}\|\hat{g}_{k_2}\|_{L^\infty_t([2^{M-1},2^M])H^2_\xi}\\
    \lesssim &2^{-k+k_2-k_1-2k_{1,+}+\gamma k_1}\min\{2^{-M/4+k_1/2},2^{M+3\min\{k,k_2\}/2+k_1/2+k_2}\}\e_1\|\hat{g}_{k_2}\|_{L^\infty_t([2^{M-1},2^M])H^2_\xi},
\end{align*}
and
\begin{align*}
    &\|\eqref{im2-1.5-1}\|_{L^2}\\
    \lesssim &\|\F^{-1}\frac{\partial_{\xi_n}\phi(\xi,\eta)\partial_{\eta_j}\phi(\xi,\eta)}{|\nabla_\eta\phi(\xi,\eta)|^2}\frac{\partial_{\xi_l}\phi(\xi,\eta)\partial_{\eta_m}\phi(\xi,\eta)}{|\nabla_\eta\phi(\xi,\eta)|^2}\Tilde{\psi}_{k_1}(\nabla_\eta\phi(\xi,\eta))\Tilde{\psi}_{k_2}(\eta)\|_{L^1}\times\\
    &\times \min\{2^{-M/4}\|\hat{f}_{m,k_1}\|_{L^\infty_t([2^{M-1},2^M])H^1_\xi}\|\nabla_\xi^2\hat{g}_{k_2}\|_{L^\infty_t([2^{M-1},2^M])L^2_\xi},\\
    &\qquad\qquad\qquad\qquad 2^{M+3\min\{k,k_2\}/2}\|e^{ic_mt\la}f_{m,k_1}\|_{L^\infty_t([2^{M-1},2^M])L^2_x}\|e^{ic_nt\la}\F^{-1}\nabla^2_\xi\hat{g}_{k_2}\|_{L^\infty_t([2^{M-1},2^M])L^2_x}\}\\
    \lesssim &2^{2k_2-2k_1}\min\{2^{-M/4}\|\hat{f}_{m,k_1}\|_{L^\infty_t([2^{M-1},2^M])H^1_\xi},\\
    &\qquad\qquad\qquad\qquad 2^{M+3\min\{k,k_2\}/2}\|f_{m,k_1}\|_{L^\infty_t([2^{M-1},2^M])L^2_x}\}\|\hat{g}_{k_2}\|_{L^\infty_t([2^{M-1},2^M])H^2_\xi}\\
    \lesssim &2^{2k_2-2k_1-2k_{1,+}+\gamma k_1}\min\{2^{-M/4-k_1/2},2^{M+3\min\{k,k_2\}/2+k_1/2}\}\|\hat{g}_{k_2}\|_{L^\infty_t([2^{M-1},2^M])H^2_\xi}.
\end{align*}
Summarizing the estimates above, we have
\begin{align*}
    \|T_{0+iy} \hat{g}\|_{L^2}
    \lesssim & 2^{-2k_{1,+}+\gamma k_1}\min\{2^{-M/4+k_2-3k_1/2},2^{M+3\min\{k,k_2\}/2+k_2-k_1/2}\}\e_1\|\hat{g}\|_{L^\infty_t([2^{M-1},2^M])H^1_\xi}
\end{align*} 
and
\begin{align*}
    \|T_{1+iy} \hat{g}\|_{L^2}
    \lesssim &2^{-2k_{1,+}+\gamma k_1}\min\{2^{-M/4-k_1/2-k+k_2},2^{M+3\min\{k,k_2\}/2-k_1/2-k+2k_2}\}\e_1\|\hat{g}_{k_2}\|_{L^\infty_t([2^{M-1},2^M])H^2_\xi}.
\end{align*}
By the interpolation result in Lemma \ref{si}, we obtain
\begin{align*}
    \|T_{\alpha} \hat{g}\|_{L^2}
    \lesssim & 2^{-2k_{1,+}+\gamma k_1+\alpha(k_1-k)}\min\{2^{-M/4+k_2-3k_1/2},2^{M+3\min\{k,k_2\}/2+k_2-k_1/2}\}\e_1\|\hat{g}\|_{L^\infty_t([2^{M-1},2^M])H^{1+\alpha}_\xi},
\end{align*} 
for any $\alpha\in(0,1)$.\\
Thus, taking $\F^{-1}\hat{f}_n\Tilde{\psi}_{k_2}$ for $g$ gives
\begin{align*}
    \|\eqref{im2.4}\|_{L^2}
    \lesssim & 2^{-2k_{1,+}+\gamma k_1+\alpha(k_1-k)}\min\{2^{-M/4+k_2-3k_1/2},2^{M+3\min\{k,k_2\}/2+k_2-k_1/2}\}\e_1\|\hat{f}_{k_2}\|_{L^\infty_t([2^{M-1},2^M])H^{1+\alpha}_\xi}\\
    \lesssim & 2^{-k_2/2-\alpha k_2-2k_{1,+}+\gamma k_1-2k_{2,+}+\gamma k_2+\alpha(k_1-k)}\min\{2^{-M/4+k_2-3k_1/2},2^{M+3\min\{k,k_2\}/2-k_1/2+k_2}\}\e_1^2\\
    \lesssim & 2^{-\alpha k_2-2k_{1,+}+\gamma k_1-2k_{2,+}+\gamma k_2+\alpha(k_1-k)}\min\{2^{-M/4+k_2/2-3k_1/2},2^{M+3\min\{k,k_2\}/2-k_1/2+k_2/2}\}\e_1^2,
\end{align*}
where the bound on the Sobolev norm is provided in \eqref{sobolevnorms}.
\end{proof}
Combining the results in Lemma \ref{im2-1} and Lemma \ref{im2-3}, we compute
\begin{align*}
    &\sum_{1\leq M\leq\log T}\sum_{(k_1,k_2)\in\chi^1_k\cup\chi^2_k}\sup_{2^{M-1}\leq t_1\leq t_2\leq 2^M}2^{2k_+-\gamma k+k/2+\alpha k}\big(\|\eqref{im2.1}\|_{L^2}+\|\eqref{im2.3}\|_{L^2}\big)\\
    \lesssim & \sum_{1\leq M\leq\log T}\sum_{(k_1,k_2)\in\chi^1_k\cup\chi^2_k} 2^{2k_+-\gamma k+k/2-k_1/2-2k_{1,+}+\gamma k_1-2k_{2,+}+\gamma k_2}\times\\
    &\qquad\qquad\qquad\qquad \min\{2^{-M/4-k_1+k_2/2},2^{M+3\min\{k,k_2\}/2+k_2/2}\}\e_1^2\\
    \lesssim &\e_1^2.
\end{align*}
Additionally, Lemma \ref{im2-4} gives us
\begin{align*}
    &\sum_{1\leq M\leq\log T}\sum_{(k_1,k_2)\in\chi^1_k\cup\chi^2_k}\sup_{2^{M-1}\leq t_1\leq t_2\leq 2^M}2^{2k_+-\gamma k+k/2+\alpha k}\|\eqref{im2.4}\|_{L^2}\\
    \lesssim & \sum_{1\leq M\leq\log T}\sum_{(k_1,k_2)\in\chi^1_k\cup\chi^2_k} 2^{2k_+-\gamma k+k/2-2k_{1,+}+\gamma k_1-2k_{2,+}+\gamma k_2+\alpha(k_1-k_2)}\times\\
    &\qquad\qquad\qquad\qquad\times\min\{2^{-M/4+k_2/2-3k_1/2},2^{M+3\min\{k,k_2\}/2-k_1/2+k_2/2}\}\e_1^2\\
    \lesssim &\e_1^2.
\end{align*}
Therefore, we showed that
\begin{equation}
    \label{im2inchi2}
    \begin{aligned}
    \sum_{1\leq M\leq\log T}\sup_{2^{M-1}\leq t_1\leq t_2\leq 2^M}2^{2k_+-\gamma k+k/2+\alpha k}\bigg(\sum_{\substack{c_m+c_n\neq 0\\c_l=c_m}}A_{lmn}\sum_{(k_1,k_2)\in\chi^1_k}\|D^\alpha I^{M,2}_{k,k_1,k_2}\|_{L^2}+\\
    +\sum_{\substack{c_m+c_n\neq 0}}A_{lmn}\sum_{(k_1,k_2)\in\chi^2_k}\|D^\alpha I^{M,2}_{k,k_1,k_2}\|_{L^2}\bigg)
    \lesssim \e_1^2.
    \end{aligned}
\end{equation}

\subsubsection{$D^{\alpha}J^{M,2}_{k,k_1,k_2}$ where $(k_1,k_2)\in\chi^2_k\cup\chi^3_k$}
\label{2-1}

For the term $J^{M,2}_{k,k_1,k_2}$, recall that we restricted our consideration to only the cases when $c_m+c_n=0$. Consequently, we have the following identity
$$-\sum_{m}\frac{\xi_m}{it2c_n|\xi|^2}\partial_{\eta_m} e^{it\phi(\xi,\eta)}=e^{it\phi(\xi,\eta)},$$
which will serve as a key tool in this section, enabling us to perform integration by parts to achieve decay in the time variable. Using the above identity and the definition provided in \eqref{j2}, we can expand $D^{\alpha}J^{M,2}_{k,k_1,k_2}$ as follows,
\begin{align}
    &D^\alpha\int_{t_1}^{t_2}\int_{\R^3}e^{it\phi(\xi,\eta)}it\partial_{\xi_l}\phi(\xi,\eta)q(\xi-\eta,\eta)\hat{f}_{m,k_1}(t,\xi-\eta)\hat{f}_{n,k_2}(t,\eta)\psi_k(\xi)d\eta dt\nonumber\\
    =&-D^\alpha\int_{t_1}^{t_2}\int_{\R^3}e^{it\phi(\xi,\eta)}\partial_{\eta_m}\frac{\xi_m\partial_{\xi_l}\phi(\xi,\eta)}{2c_n|\xi|^2}q(\xi-\eta,\eta)\hat{f}_{m,k_1}(t,\xi-\eta)\hat{f}_{n,k_2}(t,\eta)\psi_k(\xi)d\eta dt\label{km2.1}\\
    &-D^\alpha\int_{t_1}^{t_2}\int_{\R^3}e^{it\phi(\xi,\eta)}\frac{\xi_m\partial_{\xi_l}\phi(\xi,\eta)}{2c_n|\xi|^2}\partial_{\eta_m}q(\xi-\eta,\eta)\hat{f}_{m,k_1}(t,\xi-\eta)\hat{f}_{n,k_2}(t,\eta)\psi_k(\xi)d\eta dt\label{km2.2}\\
    &-D^\alpha\int_{t_1}^{t_2}\int_{\R^3}e^{it\phi(\xi,\eta)}\frac{\xi_m\partial_{\xi_l}\phi(\xi,\eta)}{2c_n|\xi|^2}q(\xi-\eta,\eta)\partial_{\eta_m}\hat{f}_{m,k_1}(t,\xi-\eta)\hat{f}_{n,k_2}(t,\eta)\psi_k(\xi)d\eta dt\label{km2.3}\\
    &-D^\alpha\int_{t_1}^{t_2}\int_{\R^3}e^{it\phi(\xi,\eta)}\frac{\xi_m\partial_{\xi_l}\phi(\xi,\eta)}{2c_n|\xi|^2}q(\xi-\eta,\eta)\hat{f}_{m,k_1}(t,\xi-\eta)\partial_{\eta_m}\hat{f}_{n,k_2}(t,\eta)\psi_k(\xi)d\eta dt\label{km2.4}.
\end{align}
The terms we derived from the expansion will be estimated separately in Lemma \ref{km2-1}, Lemma \ref{km2-2}, and Lemma \ref{km2-3}, employing the interpolation result in Lemma \ref{si}.
\begin{lem}
\label{km2-1}
Suppose $t_1,t_2\in[2^{M-1},2^M]$ and $\sup_{t\in[1,T]}\|f\|_{Z}\leq \e_1$. If either $(k_1,k_2)\in\chi^2_k\cup\chi_k^3$ or $(k_1,k_2)\in\chi^1_k$ and $c_l=c_m$, then
\begin{align*}
    &\|\eqref{km2.1}\|_{L^2}
    +\|\eqref{km2.2}\|_{L^2}\\
    \lesssim & 2^{\e k_--k-2k_{1,+}-2k_{2,+}+\gamma k_1+\gamma k_2-\alpha k}\min\{2^{-M/2-k_2/4-k_1/4}
    ,2^{M+3\min\{k,k_2\}/2 +k_2/2+k_1/2}\}\e_1^2.
\end{align*}
\end{lem}
\begin{proof}
First, we define two families of operators on $\{z\in\C:0\leq \Re(z)\leq 1\}$,
\begin{align*}
    T^1_{z} \hat{g} =D^z\int_{t_1}^{t_2}\int_{\R^3}e^{it\phi(\xi,\eta)}\partial_{\eta_m}\frac{\xi_m\partial_{\xi_l}\phi(\xi,\eta)}{|\xi|^2}q(\xi-\eta,\eta)\hat{g}_{k_1}(t,\xi-\eta)\hat{f}_{n,k_2}(t,\eta)\psi_k(\xi)d\eta dt
\end{align*}
and
\begin{align*}
    T^2_{z} \hat{g} =D^z\int_{t_1}^{t_2}\int_{\R^3}e^{it\phi(\xi,\eta)}\frac{\xi_m\partial_{\xi_l}\phi(\xi,\eta)}{|\xi|^2}\partial_{\eta_m}q(\xi-\eta,\eta)\hat{g}_{k_1}(t,\xi-\eta)\hat{f}_{n,k_2}(t,\eta)\psi_k(\xi)d\eta dt,
\end{align*}
for \eqref{km2.1} and \eqref{km2.2} respectively.
We want to find bounds on the operators
$$T^1_{0+iy},T^2_{0+iy}:L^{\infty}_t([2^{M-1},2^M])L^2_\xi\rightarrow L^2_\xi$$
and 
$$T^1_{1+iy},T^2_{1+iy}:L^{\infty}_t([2^{M-1},2^M])H^1_\xi\rightarrow L^2_\xi$$
for all $y\in\R$.\\
Beginning with $T^1_{0+iy}$ and $T^2_{0+iy}$, we can employ the bilinear estimate $L^4\times L^4\rightarrow L^2$, together with \eqref{l4} and \eqref{q} to get
\begin{align*}
    &\|T^1_{0+iy} \hat{g}\|_{L^2}+\|T^2_{0+iy} \hat{g}\|_{L^2}\\
    =&\bigg\|\int_{t_1}^{t_2}\int_{\R^3}e^{it\phi(\xi,\eta)}\partial_{\eta_m}\frac{\xi_m\partial_{\xi_l}\phi(\xi,\eta)}{|\xi|^2}q(\xi-\eta,\eta)\hat{g}_{k_1}(t,\xi-\eta)\hat{f}_{n,k_2}(t,\eta)\psi_k(\xi)d\eta dt\bigg\|_{L^2}\\
    &+\bigg\|\int_{t_1}^{t_2}\int_{\R^3}e^{it\phi(\xi,\eta)}\frac{\xi_m\partial_{\xi_l}\phi(\xi,\eta)}{|\xi|^2}\partial_{\eta_m}q(\xi-\eta,\eta)\hat{g}_{k_1}(t,\xi-\eta)\hat{f}_{n,k_2}(t,\eta)\psi_k(\xi)d\eta dt\bigg\|_{L^2}\\
    \lesssim & 2^{M}\sup_{t\in[2^{M-1},2^M]}\|\F^{-1}\partial_{\eta_m}\frac{\xi_m\partial_{\xi_l}\phi(\xi,\eta)}{|\xi|^2}q(\xi-\eta,\eta)\Tilde{\psi}_k(\xi)\Tilde{\psi}_{k_1}(\xi-\eta)\Tilde{\psi}_{k_2}(\eta)\|_{L^1}\|e^{ic_mt\la}g_{k_1}\|_{L^4_x}\|e^{ic_nt\la}f_{n,k_2}\|_{L^4_x}\\
    &+2^{M}\sup_{t\in[2^{M-1},2^M]}\|\F^{-1}\frac{\xi_m\partial_{\xi_l}\phi(\xi,\eta)}{|\xi|^2}\partial_{\eta_m}q(\xi-\eta,\eta)\Tilde{\psi}_k(\xi)\Tilde{\psi}_{k_1}(\xi-\eta)\Tilde{\psi}_{k_2}(\eta)\|_{L^1}\|e^{ic_mt\la}g_{k_1}\|_{L^4_x}\|e^{ic_nt\la}f_{n,k_2}\|_{L^4_x}\\
    \lesssim & 2^{\e k_--k-M/2+k_2/4+k_1/4}\sup_{t\in[2^{M-1},2^M]}\|\nabla_\xi\hat{f}_{n,k_2}\|_{L^2_\xi}\|\nabla_\xi\hat{g}_{k_1}\|_{L^2_\xi}\\
    \lesssim & 2^{\e k_--k-M/2-2k_{2,+}+\gamma k_2-k_2/4+k_1/4}\e_1\|\hat{g}\|_{L^\infty_t([2^{M-1},2^M])H^1_\xi}.
\end{align*}
Combining this result with Lemma \ref{dualitycomp} and \eqref{l2} leads to
\begin{align*}
    \|T^1_{0+iy} \hat{g}\|_{L^2}+\|T^2_{0+iy} \hat{g}\|_{L^2}
    \lesssim &2^{\e k_--k+M+3\min\{k,k_2\}/2}\sup_{t\in[2^{M-1},2^M]}\|f_{n,k_2}\|_{L^2_x}\|\hat{g}_{k_1}\|_{L^2_\xi}\\
    \lesssim &2^{\e k_--k+M+3\min\{k,k_2\}/2-2k_{2,+}+\gamma k_2+k_2/2}\e_1\|\hat{g}\|_{L^\infty_t([2^{M-1},2^M])L^2_\xi}.
\end{align*}
Next, for $T^1_{1+iy}$ and $T^2_{1+iy}$, we compute
\begin{align}
    \|T^1_{1+iy} \hat{g}\|_{L^2}
    \leq &\bigg\|\partial_{\xi_n}\int_{t_1}^{t_2}\int_{\R^3}e^{it\phi(\xi,\eta)}\partial_{\eta_m}\frac{\xi_m\partial_{\xi_l}\phi(\xi,\eta)}{|\xi|^2}q(\xi-\eta,\eta)\hat{g}_{k_1}(t,\xi-\eta)\hat{f}_{n,k_2}(t,\eta)\psi_k(\xi)d\eta dt\bigg\|_{L^2}\nonumber\\
    \leq &\bigg\|\int_{t_1}^{t_2}\int_{\R^3}e^{it\phi(\xi,\eta)}\partial_{\eta_m}\frac{\xi_m\partial_{\xi_l}\phi(\xi,\eta)}{|\xi|^2}q(\xi-\eta,\eta)\partial_{\xi_n}\hat{g}_{k_1}(t,\xi-\eta)\hat{f}_{n,k_2}(t,\eta)\psi_k(\xi)d\eta dt\bigg\|_{L^2}\label{km2-1.1}\\
    &+ \bigg\|\int_{t_1}^{t_2}\int_{\R^3}e^{it\phi(\xi,\eta)}\partial_{\xi_n}\big(\partial_{\eta_m}\frac{\xi_m\partial_{\xi_l}\phi(\xi,\eta)}{|\xi|^2}q(\xi-\eta,\eta)\big)\hat{g}_{k_1}(t,\xi-\eta)\hat{f}_{n,k_2}(t,\eta)\psi_k(\xi)d\eta dt\bigg\|_{L^2}\label{km2-1.11}\\
    &+ \bigg\|\int_{t_1}^{t_2}\int_{\R^3}e^{it\phi(\xi,\eta)}\partial_{\eta_m}\frac{\xi_m\partial_{\xi_l}\phi(\xi,\eta)}{|\xi|^2}q(\xi-\eta,\eta)\hat{g}_{k_1}(t,\xi-\eta)\hat{f}_{n,k_2}(t,\eta)\partial_{\xi_n}\psi_k(\xi)d\eta dt\bigg\|_{L^2}\label{km2-1.12}\\
    & + \bigg\|\int_{t_1}^{t_2}\int_{\R^3}e^{it\phi(\xi,\eta)}it\partial_{\xi_n}\phi(\xi,\eta)\partial_{\eta_m}\frac{\xi_m\partial_{\xi_l}\phi(\xi,\eta)}{|\xi|^2}q(\xi-\eta,\eta)\hat{g}_{k_1}(t,\xi-\eta)\hat{f}_{n,k_2}(t,\eta)\psi_k(\xi)d\eta dt\bigg\|_{L^2}\label{km2-1.2}
\end{align}
and
\begin{align}
    \|T^2_{1+iy} \hat{g}\|_{L^2}
    \leq &\bigg\|\partial_{\xi_n}\int_{t_1}^{t_2}\int_{\R^3}e^{it\phi(\xi,\eta)}\frac{\xi_m\partial_{\xi_l}\phi(\xi,\eta)}{|\xi|^2}\partial_{\eta_m}q(\xi-\eta,\eta)\hat{g}_{k_1}(t,\xi-\eta)\hat{f}_{n,k_2}(t,\eta)\psi_k(\xi)d\eta dt\bigg\|_{L^2}\nonumber\\
    \leq &\bigg\|\int_{t_1}^{t_2}\int_{\R^3}e^{it\phi(\xi,\eta)}\frac{\xi_m\partial_{\xi_l}\phi(\xi,\eta)}{|\xi|^2}\partial_{\eta_m}q(\xi-\eta,\eta)\partial_{\xi_n}\hat{g}_{k_1}(t,\xi-\eta)\hat{f}_{n,k_2}(t,\eta)\psi_k(\xi)d\eta dt\bigg\|_{L^2}\label{km2-1.1-2}\\
    &+ \bigg\|\int_{t_1}^{t_2}\int_{\R^3}e^{it\phi(\xi,\eta)}\partial_{\xi_n}\big(\frac{\xi_m\partial_{\xi_l}\phi(\xi,\eta)}{|\xi|^2}\partial_{\eta_m}q(\xi-\eta,\eta)\big)\hat{g}_{k_1}(t,\xi-\eta)\hat{f}_{n,k_2}(t,\eta)\psi_k(\xi)d\eta dt\bigg\|_{L^2}\label{km2-1.11-2}\\
    &+ \bigg\|\int_{t_1}^{t_2}\int_{\R^3}e^{it\phi(\xi,\eta)}\frac{\xi_m\partial_{\xi_l}\phi(\xi,\eta)}{|\xi|^2}\partial_{\eta_m}q(\xi-\eta,\eta)\hat{g}_{k_1}(t,\xi-\eta)\hat{f}_{n,k_2}(t,\eta)\partial_{\xi_n}\psi_k(\xi)d\eta dt\bigg\|_{L^2}\label{km2-1.12-2}\\
    & + \bigg\|\int_{t_1}^{t_2}\int_{\R^3}e^{it\phi(\xi,\eta)}it\partial_{\xi_l}\phi(\xi,\eta)\frac{\xi_m\partial_{\xi_l}\phi(\xi,\eta)}{|\xi|^2}\partial_{\eta_m}q(\xi-\eta,\eta)\hat{g}_{k_1}(t,\xi-\eta)\hat{f}_{n,k_2}(t,\eta)\psi_k(\xi)d\eta dt\bigg\|_{L^2}\label{km2-1.2-2}.
\end{align}
Integration by parts using the identity $e^{it\phi(\xi,\eta)}=-\sum_{j=1}^3\frac{\xi_j}{it2c_n|\xi|^2}\partial_{\eta_j}e^{it\phi(\xi,\eta)}$ on the terms $\eqref{km2-1.2}$ and $\eqref{km2-1.2-2}$ gives
\begin{align}
    &\int_{t_1}^{t_2}\int_{\R^3}e^{it\phi(\xi,\eta)}it\partial_{\xi_n}\phi(\xi,\eta)\partial_{\eta_m}\frac{\xi_m\partial_{\xi_l}\phi(\xi,\eta)}{|\xi|^2}q(\xi-\eta,\eta)\hat{g}_{k_1}(t,\xi-\eta)\hat{f}_{n,k_2}(t,\eta)\psi_k(\xi)d\eta dt\nonumber\\
    = & \int_{t_1}^{t_2}\int_{\R^3}e^{it\phi(\xi,\eta)}\partial_{\eta_j}\big(\frac{\xi_j\partial_{\xi_n}\phi(\xi,\eta)}{2c_n|\xi|^2}\partial_{\eta_m}\frac{\xi_m\partial_{\xi_l}\phi(\xi,\eta)}{|\xi|^2}q(\xi-\eta,\eta)\big)\hat{g}_{k_1}(t,\xi-\eta)\hat{f}_{n,k_2}(t,\eta)\psi_k(\xi)d\eta dt\label{km2-1.3}\\
    & +\int_{t_1}^{t_2}\int_{\R^3}e^{it\phi(\xi,\eta)}\frac{\xi_j\partial_{\xi_n}\phi(\xi,\eta)}{2c_n|\xi|^2}\partial_{\eta_m}\frac{\xi_m\partial_{\xi_l}\phi(\xi,\eta)}{|\xi|^2}q(\xi-\eta,\eta)\partial_{\eta_j}\hat{g}_{k_1}(t,\xi-\eta)\hat{f}_{n,k_2}(t,\eta)\psi_k(\xi)d\eta dt\label{km2-1.4}\\
    & +\int_{t_1}^{t_2}\int_{\R^3}e^{it\phi(\xi,\eta)}\frac{\xi_j\partial_{\xi_n}\phi(\xi,\eta)}{2c_n|\xi|^2}\partial_{\eta_m}\frac{\xi_m\partial_{\xi_l}\phi(\xi,\eta)}{|\xi|^2}q(\xi-\eta,\eta)\hat{g}_{k_1}(t,\xi-\eta)\partial_{\eta_j}\hat{f}_{n,k_2}(t,\eta)\psi_k(\xi)d\eta dt\label{km2-1.5}
\end{align}
and
\begin{align}
    &\int_{t_1}^{t_2}\int_{\R^3}e^{it\phi(\xi,\eta)}it\partial_{\xi_l}\phi(\xi,\eta)\frac{\xi_m\partial_{\xi_l}\phi(\xi,\eta)}{|\xi|^2}\partial_{\eta_m}q(\xi-\eta,\eta)\hat{g}_{k_1}(t,\xi-\eta)\hat{f}_{n,k_2}(t,\eta)\psi_k(\xi)d\eta dt\nonumber\\
    = & \int_{t_1}^{t_2}\int_{\R^3}e^{it\phi(\xi,\eta)}\partial_{\eta_j}\big(\frac{\xi_j\partial_{\xi_l}\phi(\xi,\eta)}{2c_n|\xi|^2}\frac{\xi_m\partial_{\xi_l}\phi(\xi,\eta)}{|\xi|^2}\partial_{\eta_m}q(\xi-\eta,\eta)\big)\hat{g}_{k_1}(t,\xi-\eta)\hat{f}_{n,k_2}(t,\eta)\psi_k(\xi)d\eta dt\label{km2-1.3-2}\\
    & +\int_{t_1}^{t_2}\int_{\R^3}e^{it\phi(\xi,\eta)}\frac{\xi_j\partial_{\xi_l}\phi(\xi,\eta)}{2c_n|\xi|^2}\frac{\xi_m\partial_{\xi_l}\phi(\xi,\eta)}{|\xi|^2}\partial_{\eta_m}q(\xi-\eta,\eta)\partial_{\eta_j}\hat{g}_{k_1}(t,\xi-\eta)\hat{f}_{n,k_2}(t,\eta)\psi_k(\xi)d\eta dt\label{km2-1.4-2}\\
    & +\int_{t_1}^{t_2}\int_{\R^3}e^{it\phi(\xi,\eta)}\frac{\xi_j\partial_{\xi_l}\phi(\xi,\eta)}{2c_n|\xi|^2}\frac{\xi_m\partial_{\xi_l}\phi(\xi,\eta)}{|\xi|^2}\partial_{\eta_m}q(\xi-\eta,\eta)\hat{g}_{k_1}(t,\xi-\eta)\partial_{\eta_j}\hat{f}_{n,k_2}(t,\eta)\psi_k(\xi)d\eta dt\label{km2-1.5-2}.
\end{align}
We now commence the estimation of all terms. Utilizing the bilinear estimate $L^4\times L^4\rightarrow L^2$, along with Lemma \ref{chi+eta}, Lemma \ref{dualitycomp}, and the bounds in \eqref{q} and \eqref{l2first}, we find
\begin{align*}
    &\eqref{km2-1.1}+\eqref{km2-1.1-2}+\|\eqref{km2-1.4}\|_{L^2}+\|\eqref{km2-1.4-2}\|_{L^2}\\
    \lesssim & 2^{M}\big(\|\F^{-1}\partial_{\eta_m}\frac{\xi_m\partial_{\xi_l}\phi(\xi,\eta)}{|\xi|^2}q(\xi-\eta,\eta)\Tilde{\psi}_k(\xi)\Tilde{\psi}_{k_1}(\xi-\eta)\Tilde{\psi}_{k_2}(\eta)\|_{L^1}+\\
    &+\|\F^{-1}\frac{\xi_m\partial_{\xi_l}\phi(\xi,\eta)}{|\xi|^2}\partial_{\eta_m} q(\xi-\eta,\eta)\Tilde{\psi}_k(\xi)\Tilde{\psi}_{k_1}(\xi-\eta)\Tilde{\psi}_{k_2}(\eta)\|_{L^1}\\
    & + \|\F^{-1}\frac{\xi_j\partial_{\xi_n}\phi(\xi,\eta)}{2c_n|\xi|^2}\partial_{\eta_m}\frac{\xi_m\partial_{\xi_l}\phi(\xi,\eta)}{|\xi|^2}q(\xi-\eta,\eta)\Tilde{\psi}_k(\xi)\Tilde{\psi}_{k_1}(\xi-\eta)\Tilde{\psi}_{k_2}(\eta)\|_{L^1}\\
    &+ \|\F^{-1}\frac{\xi_j\partial_{\xi_n}\phi(\xi,\eta)}{2c_n|\xi|^2}\frac{\xi_m\partial_{\xi_l}\phi(\xi,\eta)}{|\xi|^2} \partial_{\eta_m}q(\xi-\eta,\eta)\Tilde{\psi}_k(\xi)\Tilde{\psi}_{k_1}(\xi-\eta)\Tilde{\psi}_{k_2}(\eta)\|_{L^1}\big)\times\\
    &\times \sup_{t\in[2^{M-1},2^M]}
    \|e^{ic_mt\la}\F^{-1}\nabla_\xi\hat{g}_{k_1}\|_{L^4_x}\|e^{ic_nt\la}f_{n,k_2}\|_{L^4_x}\\
    \lesssim & 2^{\e k_--k-M/2+k_1/4+k_2/4}(1+2^{k_2-k})\sup_{t\in[2^{M-1},2^M]}\|\nabla_\xi\hat{f}_{n,k_2}\|_{L^2_\xi}\|\nabla^2_\xi\hat{g}_{k_1}\|_{L^2_\xi}\\
    \leq & 2^{\e k_--k-M/2-2k_{2,+}+\gamma k_2+k_1/4-k_2/4}(1+2^{k_2-k})\e_1\|\hat{g}\|_{L^\infty_t([2^{M-1},2^M])H^2_\xi}.
\end{align*}
Then, the bilinear estimate $L^2\times L^\infty\rightarrow L^2$, Lemma \ref{op}, \eqref{lpnormseq1} and \eqref{l2} give 
\begin{align*}
    &\eqref{km2-1.11}+\eqref{km2-1.11-2}+\eqref{km2-1.12}+\eqref{km2-1.12-2}+\|\eqref{km2-1.3}\|_{L^2}+\|\eqref{km2-1.3-2}\|_{L^2}\\
    \lesssim &2^{M}\big[\|\F^{-1}\partial_{\xi_n}\big(\partial_{\eta_m}\frac{\xi_m\partial_{\xi_l}\phi(\xi,\eta)}{|\xi|^2}q(\xi-\eta,\eta)\big)\Tilde{\psi}_k(\xi)\Tilde{\psi}_{k_1}(\xi-\eta)\Tilde{\psi}_{k_2}(\eta)\|_{L^1}+\\
    & + \|\F^{-1}\partial_{\xi_n}\big(\frac{\xi_m\partial_{\xi_l}\phi(\xi,\eta)}{|\xi|^2}\partial_{\eta_m}q(\xi-\eta,\eta)\big)\Tilde{\psi}_k(\xi)\Tilde{\psi}_{k_1}(\xi-\eta)\Tilde{\psi}_{k_2}(\eta)\|_{L^1}+\\
    & + \|\nabla_\xi\psi_k\|_{L^\infty}\big(\|\F^{-1} \partial_{\eta_m}\frac{\xi_m\partial_{\xi_l}\phi(\xi,\eta)}{|\xi|^2}q(\xi-\eta,\eta)\Tilde{\psi}_k(\xi)\Tilde{\psi}_{k_1}(\xi-\eta)\Tilde{\psi}_{k_2}(\eta)\|_{L^1}+\\
    &+\|\F^{-1} \frac{\xi_m\partial_{\xi_l}\phi(\xi,\eta)}{|\xi|^2}\partial_{\eta_m}q(\xi-\eta,\eta)\Tilde{\psi}_k(\xi)\Tilde{\psi}_{k_1}(\xi-\eta)\Tilde{\psi}_{k_2}(\eta)\|_{L^1}\big)+\\
    & + \|\F^{-1} \partial_{\eta_j}\big(\frac{\xi_j\partial_{\xi_n}\phi(\xi,\eta)}{2c_n|\xi|^2}\partial_{\eta_m}\frac{\xi_m\partial_{\xi_l}\phi(\xi,\eta)}{|\xi|^2}q(\xi-\eta,\eta)\big)\Tilde{\psi}_k(\xi)\Tilde{\psi}_{k_1}(\xi-\eta)\Tilde{\psi}_{k_2}(\eta)\|_{L^1}\\
    & + \|\F^{-1} \partial_{\eta_j}\big(\frac{\xi_j\partial_{\xi_n}\phi(\xi,\eta)}{2c_n|\xi|^2}\frac{\xi_m\partial_{\xi_l}\phi(\xi,\eta)}{|\xi|^2}\partial_{\eta_m}q(\xi-\eta,\eta)\big)\Tilde{\psi}_k(\xi)\Tilde{\psi}_{k_1}(\xi-\eta)\Tilde{\psi}_{k_2}(\eta)\|_{L^1}\big]\times\\
    &\times\sup_{t\in[2^{M-1},2^M]}\|e^{ic_mt\la}g_{k_1}\|_{L^\infty_x}\|f_{n,k_2}\|_{L^2_x}\\
    \lesssim & 2^{\e k_--2k-M/2}\sup_{t\in[2^{M-1},2^M]}\|f_{n,k_2}\|_{L^2_x}\|g_{k_1}\|_{L^1_x}\\
    \lesssim & 2^{\e k_--2k-M/2}\sup_{t\in[2^{M-1},2^M]}\|f_{n,k_2}\|_{L^2_x}2^{k_1/2}\|\nabla^2_\xi\hat{g}_{k_1}\|_{L^2_\xi}\\
    \lesssim & 2^{\e k_--2k-M/2-2k_{2,+}+\gamma k_2+k_2/2+k_1/2}\e_1\|\hat{g}\|_{L^\infty_t([2^{M-1},2^M])H^2_\xi}
\end{align*}
and
\begin{align*}
    &\|\eqref{km2-1.5}\|_{L^2}+\|\eqref{km2-1.5-2}\|_{L^2}\\
    \lesssim &2^{M}\big(\|\F^{-1}\frac{\xi_j\partial_{\xi_n}\phi(\xi,\eta)}{2c_n|\xi|^2}\partial_{\eta_m}\frac{\xi_m\partial_{\xi_l}\phi(\xi,\eta)}{|\xi|^2}q(\xi-\eta,\eta)\Tilde{\psi}_k(\xi)\Tilde{\psi}_{k_1}(\xi-\eta)\Tilde{\psi}_{k_2}(\eta)\|_{L^1}\\
    &+\|\F^{-1}\frac{\xi_j\partial_{\xi_n}\phi(\xi,\eta)}{2c_n|\xi|^2}\frac{\xi_m\partial_{\xi_l}\phi(\xi,\eta)}{|\xi|^2}\partial_{\eta_m}q(\xi-\eta,\eta)\Tilde{\psi}_k(\xi)\Tilde{\psi}_{k_1}(\xi-\eta)\Tilde{\psi}_{k_2}(\eta)\|_{L^1}\big)\times\\
    &\times\sup_{t\in[2^{M-1},2^M]}\|e^{ic_mt\la}g_{k_1}\|_{L^\infty_x}\|\nabla_\xi\hat{f}_{n,k_2}\|_{L^2_\xi}\\
    \lesssim & 2^{\e k_--k-M/2+k_2-k+k_1/2}\sup_{t\in[2^{M-1},2^M]}\|\nabla_\xi \hat{f}_{n,k_2}\|_{L^2_\xi}\|\nabla^2_\xi\hat{g}_{k_1}\|_{L^2_\xi}\\
    \lesssim & 2^{\e k_--2k-M/2-2k_{2,+}+\gamma k_2+k_2/2+k_1/2}\e_1\|\hat{g}\|_{L^\infty_t([2^{M-1},2^M])H^2_\xi}.
\end{align*}
Furthermore, employing the duality result stated in Lemma \ref{dualitycomp}, we can find a new set of bounds on the previous terms. We compute
\begin{align*}
    &\eqref{km2-1.1}+\eqref{km2-1.1-2}+\|\eqref{km2-1.4}\|_{L^2}+\|\eqref{km2-1.4-2}\|_{L^2}\\
    \lesssim & 2^{\e k_--k+M+3\min\{k,k_2\}/2}(1+2^{k_2-k})\sup_{t\in[2^{M-1},2^M]}\|\hat{f}_{n,k_2}\|_{L^2_\xi}\|\nabla_\xi\hat{g}_{k_1}\|_{L^2_\xi}\\
    \lesssim & 2^{\e k_--k+M+3\min\{k,k_1\}/2-2k_{2,+}+\gamma k_2+k_2/2}(1+2^{k_2-k})\e_1\|\hat{g}\|_{L^\infty_t([2^{M-1},2^M])H^1_\xi},
\end{align*}
\begin{align*}
    &\eqref{km2-1.11}+\eqref{km2-1.11-2}+\eqref{km2-1.12}+\eqref{km2-1.12-2}+\|\eqref{km2-1.3}\|_{L^2}+\|\eqref{km2-1.3-2}\|_{L^2}\\
    \lesssim & 2^{\e k_--2k+M+3\min\{k,k_2\}/2}\sup_{t\in[2^{M-1},2^M]}\|\hat{f}_{n,k_2}\|_{L^2_\xi}\|\hat{g}_{k_1}\|_{L^2_\xi}\\
    \lesssim & 2^{\e k_-+M+3\min\{k,k_1\}/2-2k_{2,+}+\gamma k_2+k_2/2-2k+k_1}\e_1\|\hat{g}\|_{L^\infty_t([2^{M-1},2^M])H^1_\xi},
\end{align*}
and
\begin{align*}
    &\|\eqref{km2-1.5}\|_{L^2}+\|\eqref{km2-1.5-2}\|_{L^2}\\
    \lesssim & 2^{\e k_-+M+3\min\{k,k_2\}/2+k_2-2k}\sup_{t\in[2^{M-1},2^M]}\|\nabla_\xi\hat{f}_{n,k_2}\|_{L^2_\xi}\|\hat{g}_{k_1}\|_{L^2_\xi}\\
    \lesssim & 2^{\e k_-+M+3\min\{k,k_1\}/2-2k_{2,+}+\gamma k_2+k_2/2-2k+k_1}\e_1\|\hat{g}\|_{L^\infty_t([2^{M-1},2^M])H^1_\xi}.
\end{align*}
To summarize, we have
\begin{align}
    \|T^1_{0+iy}\hat{g}\|_{L^2}+\|T^2_{0+iy}\hat{g}\|_{L^2}\lesssim 2^{\e k_--k-M/2-2k_{2,+}+\gamma k_2-k_2/4+k_1/4}\e_1\|\hat{g}\|_{L^\infty_t([2^{M-1},2^M])H^1_\xi}
\end{align}
and
\begin{align}
    \|T^1_{1+iy}\hat{g}\|_{L^2}+\|T^2_{1+iy}\hat{g}\|_{L^2}
    &\lesssim 2^{\e k_--k-M/2-2k_{2,+}+\gamma k_2-k_2/4+k_1/4+k_1-k}\e_1\|\hat{g}\|_{L^\infty_t([2^{M-1},2^M])H^2_\xi}.
\end{align}
Using the interpolation result from Lemma \ref{si}, we get
\begin{align*}
    \|T^1_{\alpha}\hat{g}\|_{L^2}+\|T^2_{\alpha}\hat{g}\|_{L^2}\lesssim 2^{\e k_--k-M/2-2k_{2,+}+\gamma k_2 -k_2/4+k_1/4+(k_1-k)\alpha}\e_1\|\hat{g}\|_{L^\infty_t([2^{M-1},2^M])H^{1+\alpha}_\xi},
\end{align*}
for any $\alpha\in(0,1)$.

If we look at a different set of bounds,
\begin{align}
    \|T^1_{0+iy}\hat{g}\|_{L^2}+\|T^2_{0+iy}\hat{g}\|_{L^2}\lesssim 2^{\e k_--k+M+3\min\{k,k_2\}/2-2k_{2,+}+\gamma k_2+k_2/2}\e_1\|\hat{g}\|_{L^\infty_t([2^{M-1},2^M])L^2_\xi}
\end{align}
and
\begin{align}
    \|T^1_{1+iy}\hat{g}\|_{L^2}+\|T^2_{1+iy}\hat{g}\|_{L^2}\lesssim 2^{\e k_--k+M+3\min\{k,k_2\}/2-2k_{2,+}+\gamma k_2 +k_2/2+k_1-k}\e_1\|\hat{g}\|_{L^\infty_t([2^{M-1},2^M])H^1_\xi},
\end{align}
we obtain for any $\alpha\in(0,1)$,
\begin{align*}
    \|T^1_{\alpha}\hat{g}\|_{L^2}+\|T^2_{\alpha}\hat{g}\|_{L^2}\lesssim 2^{\e k_--k+M+3\min\{k,k_2\}/2-2k_{2,+}+\gamma k_2 +k_2/2+(k_1-k)\alpha}\e_1\|\hat{g}\|_{L^\infty_t([2^{M-1},2^M])H^\alpha_\xi}.
\end{align*}
Thus, combining the two bounds we got on $\|T^1_{\alpha}\hat{g}\|_{L^2}+\|T^2_{\alpha}\hat{g}\|_{L^2}$, we have
\begin{align*}
    &\bigg\|D^\alpha\int_{t_1}^{t_2}\int_{\R^3}e^{it\phi(\xi,\eta)}\partial_{\eta_m}\frac{\xi_m\partial_{\xi_l}\phi(\xi,\eta)}{2c_n|\xi|^2}q(\xi-\eta,\eta)\hat{f}_{m,k_1}(t,\xi-\eta)\hat{f}_{n,k_2}(t,\eta)\psi_k(\xi)d\eta dt\bigg\|_{L^2}\\
    &+\bigg\|D^\alpha\int_{t_1}^{t_2}\int_{\R^3}e^{it\phi(\xi,\eta)}\frac{\xi_m\partial_{\xi_l}\phi(\xi,\eta)}{2c_n|\xi|^2}\partial_{\eta_m}q(\xi-\eta,\eta)\hat{f}_{m,k_1}(t,\xi-\eta)\hat{f}_{n,k_2}(t,\eta)\psi_k(\xi)d\eta dt\bigg\|_{L^2}\\
    \lesssim & 2^{\e k_--k-2k_{2,+}+\gamma k_2+(k_1-k)\alpha}\min\{2^{-M/2-k_2/4+k_1/4}\e_1\|\hat{g}_{k_1}\|_{L^\infty_t([2^{M-1},2^M])H^{1+\alpha}_\xi}
    ,\\
    &\qquad\qquad\qquad 2^{M+3\min\{k,k_2\}/2 +k_2/2}\e_1\|\hat{g}_{k_1}\|_{L^\infty_t([2^{M-1},2^M])H^\alpha_\xi}\}.
\end{align*}
Finally, by substituting $f_m$ for $g$ and employing the bounds on $\|\hat{f}_{m,k}\|_{H^\alpha}$ and $\|\hat{f}_{m,k}\|_{H^{1+\alpha}}$ from equation \eqref{sobolevnorms}, we arrive at the desired result.
\end{proof}
We move on to the estimation of the next term in the splitting of $J^{M,2}_{k,k_1,k_2}$.
\begin{lem}
\label{km2-2}
Suppose $t_1,t_2\in[2^{M-1},2^M]$ and $\sup_{t\in[1,T]}\|f\|_{Z}\leq \e_1$. If either $(k_1,k_2)\in\chi^2_k\cup\chi_k^3$ or $(k_1,k_2)\in\chi^1_k$ and $c_l=c_m$, then
\begin{align*}
    \|\eqref{km2.3}\|_{L^2}
    \lesssim  &2^{\e k_--k-\alpha k-2k_{1,+}+\gamma k_1-2k_{2,+}+\gamma k_2}\min\{2^{-M/2+\gamma M/4-k_1/2+\gamma k_{1,+}/4+\gamma k_{2,+}/4},\\
    &\qquad\qquad\qquad\qquad2^{M+3\min\{k,k_2\}/2-k_1/2+3 k_2/2}\}\e_1^2.
\end{align*}
\end{lem}
\begin{proof}
Define a family $T_z$ of operators on $\{z\in\C:0\leq \Re(z)\leq 1\}$,
\begin{align*}
    T_{z} \hat{g} (\xi)=D^z\int_{t_1}^{t_2}\int_{\R^3}e^{it\phi(\xi,\eta)}\frac{\xi_m\partial_{\xi_l}\phi(\xi,\eta)}{2c_n|\xi|^2}q(\xi-\eta,\eta)\partial_{\eta_m}\hat{g}_{k_1}(t,\xi-\eta)\hat{f}_{n,k_2}(t,\eta)\psi_k(\xi)d\eta dt.
\end{align*}
We want to show
$T_{0+iy}:L^{\infty}_t([2^{M-1},2^M])H^1_\xi\rightarrow L^2_\xi$
and 
$T_{1+iy}:L^{\infty}_t([2^{M-1},2^M])H^2_\xi\rightarrow L^2_\xi$
are bounded for all $y\in\R$.\\
We estimate the operator $T_{0+iy}$ by employing the bilinear estimate $L^{2}\times L^{\infty}\rightarrow L^2$, along with Lemma \ref{dualitycomp} and \eqref{linfinity},
\begin{align*}
    &\|T_{0+it}\hat{g}\|_{L^2}\\
    =&\bigg\|\int_{t_1}^{t_2}\int_{\R^3}e^{it\phi(\xi,\eta)}\frac{\xi_m\partial_{\xi_l}\phi(\xi,\eta)}{2c_n|\xi|^2}q(\xi-\eta,\eta)\partial_{\eta_m}\hat{g}_{k_1}(t,\xi-\eta)\hat{f}_{n,k_2}(t,\eta)\psi_k(\xi)d\eta dt\bigg\|_{L^2}\\
    \lesssim & 2^{M}\sup_{t\in[2^{M-1},2^M]}\|\F^{-1} \frac{\xi_m\partial_{\xi_l}\phi(\xi,\eta)}{2c_n|\xi|^2} q(\xi-\eta,\eta)\Tilde{\psi}_k(\xi)\Tilde{\psi}_{k_1}(\xi-\eta)\Tilde{\psi}_{k_2}(\eta)\|_{L^1}\times\\
    &\times\min\{\|e^{ic_mt\la}\F^{-1}\nabla\hat{g}_{k_1}\|_{L^2_x}\|e^{ic_nt\la}f_{n,k_2}\|_{L^\infty_x},2^{3\min\{k,k_2\}/2}\|e^{ic_mt\la}\F^{-1}\nabla\hat{g}_{k_1}\|_{L^2_x}\|e^{ic_nt\la}f_{n,k_2}\|_{L^2_x}\}\\
    \lesssim & 2^{\e k_-+M+k_2-k}\sup_{t\in[2^{M-1},2^M]}\min\{\|\nabla\hat{g}_{k_1}\|_{L^2_\xi}\|e^{ic_nt\la}f_{k_2}\|_{L^\infty_x},2^{3\min\{k,k_2\}/2}\|\nabla\hat{g}_{k_1}\|_{L^2_x}\|f_{n,k_2}\|_{L^2_x}\}\\
    \lesssim &2^{\e k_--k-2k_{2,+}+\gamma k_2}\min\{2^{-M/2+\gamma M/8+\gamma k_2/4},2^{M+ 3\min\{k,k_2\}/2+3k_2/2}\}\e_1\|\hat{g}_{k_1}\|_{L^\infty_t([2^{M-1},2^M])H^1_\xi}.
\end{align*}
Next, for the operator $T_{1+iy}$, we compute
\begin{align}
    &\|T_{1+it}\hat{g}\|_{L^2}\nonumber\\    =&\bigg\|\partial_{\xi_p}\int_{t_1}^{t_2}\int_{\R^3}e^{it\phi(\xi,\eta)}\frac{\xi_m\partial_{\xi_l}\phi(\xi,\eta)}{2c_n|\xi|^2}q(\xi-\eta,\eta)\partial_{\eta_m}\hat{g}_{k_1}(t,\xi-\eta)\hat{f}_{n,k_2}(t,\eta)\psi_k(\xi)d\eta dt\bigg\|_{L^2}\nonumber\\
    \leq &\bigg\|\int_{t_1}^{t_2}\int_{\R^3}e^{it\phi(\xi,\eta)}\frac{it\partial_{\xi_p}\phi(\xi,\eta)\xi_m\partial_{\xi_l}\phi(\xi,\eta)}{2c_n|\xi|^2}q(\xi-\eta,\eta)\partial_{\eta_m}\hat{g}_{k_1}(t,\xi-\eta)\hat{f}_{n,k_2}(t,\eta)\psi_k(\xi)d\eta dt\bigg\|_{L^2}\label{km2-2.1}\\
    & + \bigg\|\int_{t_1}^{t_2}\int_{\R^3}e^{it\phi(\xi,\eta)}\partial_{\xi_p}\big(\frac{\xi_m\partial_{\xi_l}\phi(\xi,\eta)}{2c_n|\xi|^2}q(\xi-\eta,\eta)\big)\partial_{\eta_m}\hat{g}_{k_1}(t,\xi-\eta)\hat{f}_{n,k_2}(t,\eta)\psi_k(\xi)d\eta dt\bigg\|_{L^2}\label{km2-2.2}\\
    & + \bigg\|\int_{t_1}^{t_2}\int_{\R^3}e^{it\phi(\xi,\eta)}\frac{\xi_m\partial_{\xi_l}\phi(\xi,\eta)}{2c_n|\xi|^2}q(\xi-\eta,\eta)\partial_{\xi_p}\partial_{\eta_m}\hat{g}_{k_1}(t,\xi-\eta)\hat{f}_{n,k_2}(t,\eta)\psi_k(\xi)d\eta dt\bigg\|_{L^2}\label{km2-2.3}\\
    & + \bigg\|\int_{t_1}^{t_2}\int_{\R^3}e^{it\phi(\xi,\eta)}\frac{\xi_m\partial_{\xi_l}\phi(\xi,\eta)}{2c_n|\xi|^2}q(\xi-\eta,\eta)\partial_{\eta_m}\hat{g}_{k_1}(t,\xi-\eta)\hat{f}_{n,k_2}(t,\eta)\partial_{\xi_p}\psi_k(\xi)d\eta dt\bigg\|_{L^2}\label{km2-2.7}.
\end{align}
Immediately, we see
\begin{align*}
    &\|\eqref{km2-2.7}\|_{L^2}\\
    \lesssim &\|T_{0+it}\hat{g}\|_{L^2}\|\nabla_\xi\psi_k\|_{L^\infty}\\
    \lesssim & 2^{-2k+\e k_--2k_{2,+}+\gamma k_2}\min\{2^{-M/2+\gamma M/8+\gamma k_2/4}, 2^{M+3\min\{k,k_2\}/2+3k_2/2}\}\e_1 \|\hat{g}_{k_1}\|_{L^\infty_t([2^{M-1},2^M])H^1_\xi}\\
    \lesssim & 2^{-2k+\e k_--2k_{2,+}+\gamma k_2+k_1}\min\{2^{-M/2+\gamma M/8+\gamma k_2/4}, 2^{M+3\min\{k,k_2\}/2+3k_2/2}\}\e_1 \|\hat{g}_{k_1}\|_{L^\infty_t([2^{M-1},2^M])H^2_\xi},
\end{align*}
where the last line uses \eqref{addder}.\\
To handle the term \eqref{km2-2.1}, we proceed with integration by parts using $e^{it\phi(\xi,\eta)}=-\sum_{j=1}^3\frac{\xi_j}{it2c_n|\xi|^2}\partial_{\eta_j}e^{it\phi(\xi,\eta)}$. Hence, we have
\begin{align}
    &\int_{t_1}^{t_2}\int_{\R^3}\big(-\sum_{j}\frac{\xi_j}{2c_n|\xi|^2}\partial_{\eta_j}\big)e^{it\phi(\xi,\eta)}\frac{\partial_{\xi_p}\phi(\xi,\eta)\xi_m\partial_{\xi_l}\phi(\xi,\eta)}{2c_n|\xi|^2}q(\xi-\eta,\eta)\partial_{\eta_m}\hat{g}_{k_1}(t,\xi-\eta)\hat{f}_{n,k_2}(t,\eta)\psi_k(\xi)d\eta dt\nonumber\\
    =&\int_{t_1}^{t_2}\int_{\R^3}e^{it\phi(\xi,\eta)}\partial_{\eta_j}\big(\frac{\xi_j\partial_{\xi_p}\phi(\xi,\eta)\xi_m\partial_{\xi_l}\phi(\xi,\eta)}{4c_n^2|\xi|^4}q(\xi-\eta,\eta)\big)\partial_{\eta_m}\hat{g}_{k_1}(t,\xi-\eta)\hat{f}_{n,k_2}(t,\eta)\psi_k(\xi)d\eta dt\label{km2-2.4}\\
    &+\int_{t_1}^{t_2}\int_{\R^3}e^{it\phi(\xi,\eta)}\frac{\xi_j\partial_{\xi_p}\phi(\xi,\eta)\xi_m\partial_{\xi_l}\phi(\xi,\eta)}{4c_n^2|\xi|^4}q(\xi-\eta,\eta)\partial_{\eta_j}\partial_{\eta_m}\hat{g}_{k_1}(t,\xi-\eta)\hat{f}_{n,k_2}(t,\eta)\psi_k(\xi)d\eta dt\label{km2-2.5}\\
    &+\int_{t_1}^{t_2}\int_{\R^3}e^{it\phi(\xi,\eta)}\frac{\xi_j\partial_{\xi_p}\phi(\xi,\eta)\xi_m\partial_{\xi_l}\phi(\xi,\eta)}{4c_n^2|\xi|^4}q(\xi-\eta,\eta)\partial_{\eta_m}\hat{g}_{k_1}(t,\xi-\eta)\partial_{\eta_j}\hat{f}_{n,k_2}(t,\eta)\psi_k(\xi)d\eta dt\label{km2-2.6}.
\end{align}
Now, we combine similar terms and perform estimations. Employing the bilinear estimate $L^4\times L^4\rightarrow L^2$,  Lemma \ref{chi+eta}, Lemma \ref{dualitycomp}, together with estimations \eqref{q}, \eqref{addder}, \eqref{l2} and \eqref{l4}, we have
\begin{align*}
    &\|\eqref{km2-2.2}\|_{L^2}+\|\eqref{km2-2.4}\|_{L^2}\\
    \lesssim &2^M\sup_{t\in[2^{M-1},2^M]}\big(\|\F^{-1} \partial_{\xi_p}\big(\frac{\xi_m\partial_{\xi_l}\phi(\xi,\eta)}{2c_n|\xi|^2}q(\xi-\eta,\eta)\big)\Tilde{\psi}_k(\xi)\Tilde{\psi}_{k_1}(\xi-\eta)\Tilde{\psi}_{k_2}(\eta)\|_{L^1}\\
    &+\|\F^{-1} \partial_{\eta_j}\big(\frac{\xi_j\partial_{\xi_p}\phi(\xi,\eta)\xi_m\partial_{\xi_l}\phi(\xi,\eta)}{4c_n^2|\xi|^4}q(\xi-\eta,\eta)\big)\Tilde{\psi}_k(\xi)\Tilde{\psi}_{k_1}(\xi-\eta)\Tilde{\psi}_{k_2}(\eta)\|_{L^1}\big)\times\\
    &\times \min\{\|e^{ic_mt\la}\F^{-1}\nabla\hat{g}_{k_1}\|_{L^4_x}\|e^{ic_nt\la}f_{n,k_2}\|_{L^4_x},2^{3\min\{k,k_2\}/2}\|e^{ic_mt\la}\F^{-1}\nabla\hat{g}_{k_1}\|_{L^2_x}\|e^{ic_nt\la}f_{n,k_2}\|_{L^2_x}\}\\
    \lesssim & 2^{\e k_-+M+k_2-2k}\sup_{t\in[2^{M-1},2^M]}\min\{\|e^{ic_mt\la}\F^{-1}\nabla\hat{g}_{k_1}\|_{L^4_x}\|e^{ic_nt\la}f_{n,k_2}\|_{L^4_x},\\
    &\qquad\qquad\qquad\qquad 2^{3\min\{k,k_2\}/2}\|e^{ic_mt\la}\F^{-1}\nabla\hat{g}_{k_1}\|_{L^2_x}\|e^{ic_nt\la}f_{n,k_2}\|_{L^2_x}\}\\
    \lesssim & 2^{\e k_-+M+k_2-2k}\sup_{t\in[2^{M-1},2^M]}\min\{2^{-3M/4+k_1/4}\|\nabla^2\hat{g}_{k_1}\|_{L^2_\xi}\|e^{ic_nt\la}f_{n,k_2}\|_{L^4_x},\\
    &\qquad\qquad\qquad\qquad 2^{3\min\{k,k_2\}/2}2^{k_1}\|\nabla^2\hat{g}_{k_1}\|_{L^2_\xi}\|f_{n,k_2}\|_{L^2_x}\}\\
    \lesssim & 2^{\e k_--2k-2k_{2,+}+\gamma k_2}\min\{2^{-M/2+k_1/4+3k_2/4},2^{M+3\min\{k,k_2\}/2+k_1+3k_2/2}\}\e_1\|\hat{g}_{k_1}\|_{L^\infty_t([2^{M-1},2^M])H^2_\xi}.
\end{align*}
Applying the bilinear estimate $L^2\times L^\infty\rightarrow L^2$, together with Lemma \ref{chi+eta}, Lemma \ref{dualitycomp}, bounds in \eqref{q}, \eqref{l2} and \eqref{linfinity}, we obtain
\begin{align*}
    &\|\eqref{km2-2.3}\|_{L^2}+\|\eqref{km2-2.5}\|_{L^2}\\
    \lesssim & 2^M\big(\|\F^{-1} \frac{\xi_m\partial_{\xi_l}\phi(\xi,\eta)}{2c_n|\xi|^2} q(\xi-\eta,\eta)\Tilde{\psi}_k(\xi)\Tilde{\psi}_{k_1}(\xi-\eta)\Tilde{\psi}_{k_2}(\eta)\|_{L^1}+\\
    &+\|\F^{-1} \frac{\xi_j\partial_{\xi_p}\phi(\xi,\eta)\xi_m\partial_{\xi_l}\phi(\xi,\eta)}{4c_n^2|\xi|^4} q(\xi-\eta,\eta)\Tilde{\psi}_k(\xi)\Tilde{\psi}_{k_1}(\xi-\eta)\Tilde{\psi}_{k_2}(\eta)\|_{L^1}\big)\times\\
    &\times\sup_{t\in[2^{M-1},2^M]}\min\{\|e^{ic_mt\la}\F^{-1}\nabla^2\hat{g}_{k_1}\|_{L^{2_x}}\|e^{ic_nt\la}f_{n,k_2}\|_{L^\infty_x},\\
    &\qquad\qquad\qquad\qquad 2^{3\min\{k,k_2\}/2}\|e^{ic_mt\la}\F^{-1}\nabla^2\hat{g}_{k_1}\|_{L^2_x}\|e^{ic_nt\la}f_{n,k_2}\|_{L^2_x}\}\\
    \lesssim &2^{\e k_-+M+k_2-k}(1+2^{k_2-k})\sup_{t\in[2^{M-1},2^M]}\min\{\|\nabla^2\hat{g}_{k_1}\|_{L^{2}_\xi}\|e^{ic_nt\la}f_{n,k_2}\|_{L^\infty_x},\\
    &\qquad\qquad\qquad\qquad 2^{3\min\{k,k_2\}/2}\|\nabla^2\hat{g}_{k_1}\|_{L^2_\xi}\|f_{n,k_2}\|_{L^2_x}\}\\
    \lesssim &2^{\e k_-+M+k_2-k-2k_{2,+}+\gamma k_2}(1+2^{k_2-k})\min\{2^{-3M/2+\gamma M/8-k_2+\gamma k_2/4},\\
    &\qquad\qquad\qquad\qquad 2^{3\min\{k,k_2\}/2+k_2/2}\}\e_1\|\nabla^2\hat{g}_{k_1}\|_{L^\infty_t([2^{M-1},2^M])L^{2}_\xi}\\
    \lesssim &2^{\e k_--k-2k_{2,+}+\gamma k_2}(1+2^{k_2-k})\min\{2^{-M/2+\gamma M/8+\gamma k_2/4},2^{M+3\min\{k,k_2\}/2+3k_2/2}\}\e_1\|\hat{g}_{k_1}\|_{L^\infty_t([2^{M-1},2^M])H^{2}_\xi}\\
    \lesssim &2^{\e k_--k-2k_{2,+}+\gamma k_2+k_1-k}\min\{2^{-M/2+\gamma M/8+\gamma k_2/4},2^{M+3\min\{k,k_2\}/2+3k_2/2}\}\e_1\|\hat{g}_{k_1}\|_{L^\infty_t([2^{M-1},2^M])H^{2}_\xi}.
\end{align*}
Lastly, we use the bilinear estimate $L^6\times L^3\rightarrow L^2$ and Lemma \ref{dualitycomp} to get
\begin{align*}
    \|\eqref{km2-2.6}\|_{L^2}
    \lesssim & 2^M\sup_{t\in[2^{M-1},2^M]}\|\F^{-1} \frac{\xi_j\partial_{\xi_p}\phi(\xi,\eta)\xi_m\partial_{\xi_l}\phi(\xi,\eta)}{4c_n^2|\xi|^4} q(\xi-\eta,\eta)\Tilde{\psi}_k(\xi)\Tilde{\psi}_{k_1}(\xi-\eta)\Tilde{\psi}_{k_2}(\eta)\|_{L^1}\times\\
    &\qquad\times \min\{\|e^{ic_mt\la}\F^{-1}\nabla\hat{g}_{k_1}\|_{L^6_x}\|e^{ic_nt\la}\F^{-1}\nabla\hat{f}_{n,k_2}\|_{L^3_x},\\
    &\qquad\qquad\qquad\qquad\qquad 2^{3\min\{k,k_2\}/2}\|e^{ic_mt\la}\F^{-1}\nabla\hat{g}_{k_1}\|_{L^2_x}\|e^{ic_nt\la}\F^{-1}\nabla\hat{f}_{n,k_2}\|_{L^2_x}\}\\
    \lesssim & 2^{\e k_-+M+2k_2-2k}\sup_{t\in[2^{M-1},2^M]}\min\{\|e^{ic_mt\la}\F^{-1}\nabla\hat{g}_{k_1}\|_{L^6_x}\|e^{ic_nt\la}\F^{-1}\nabla\hat{f}_{n,k_2}\|_{L^3_x},\\
    &\qquad\qquad\qquad\qquad\qquad 2^{3\min\{k,k_2\}/2}\|\nabla\hat{g}_{k_1}\|_{L^2_\xi}\|\nabla\hat{f}_{n,k_2}\|_{L^2_\xi}\}.
\end{align*}
In order to estimate the $L^6$ norm, we employ Bernstein's inequality in Lemma \ref{bernstein}, together with Lemma \ref{op} and Lemma \ref{lpnorms} to get
\begin{equation}
\label{l6normforoneder}
    \begin{aligned}
    \|e^{ic_mt\la}\F^{-1}\nabla\hat{g}_{k}\|_{L^6_x}
    \lesssim &2^{\gamma k/8}\|e^{ic_mt\la}\F^{-1}\nabla\hat{g}_{k}\|_{L^\frac{24}{\gamma+4}_x}\\
    \lesssim & 2^{\gamma k/8-M+\gamma M/8}\|\F^{-1}\nabla\hat{g}_{k}\|_{L^\frac{24}{20-\gamma}_x}\\
    \lesssim & 2^{-M+\gamma M/8+\gamma k/4}\|\nabla^2\hat{g}_{k}\|_{L^2_\xi}.
    \end{aligned}
\end{equation}
Hence, using the bound above with the estimations \eqref{l2first} and \eqref{l3nablaf}, we obtain
\begin{align*}
    \|\eqref{km2-2.6}\|_{L^2}
    &\lesssim 2^{\e k_-+M+2k_2-2k}\sup_{t\in[2^{M-1},2^M]}\min\{2^{-M+\gamma M/8+\gamma k_1/4}\|e^{it\la}\F^{-1}\nabla\hat{f}_{k_2}\|_{L^3_x},\\
    &\qquad\qquad\qquad\qquad 2^{3\min\{k,k_2\}/2+k_1}\|\nabla\hat{f}_{k_2}\|_{L^2_\xi}\}\|\nabla^2_\xi\hat{g}_{k_1}\|_{L^2_\xi}\\
    &\lesssim 2^{\e k_--2k_{2,+}+\gamma k_2-2k}\min\{2^{-M/2+\gamma M/4+k_2+\gamma k_2/4 +\gamma k_1/4},\\
    &\qquad\qquad\qquad\qquad 2^{M+3\min\{k,k_2\}/2+k_1+3k_2/2}\}\e_1\|\hat{g}_{k_1}\|_{L^\infty_t([2^{M-1},2^M])H^2_\xi}.
\end{align*}
To conclude, we have
\begin{equation*}
    \begin{aligned}
    \|T_{1+it}\hat{g}\|_{L^2}
    \lesssim &2^{\e k_--2k_{2,+}+\gamma k_2-2k}\min\{2^{-M/2+\gamma M/4+k_1+\gamma k_{2,+}/4 +\gamma k_{1,+}/4},\\
    &\qquad\qquad\qquad\qquad 2^{M+3\min\{k,k_2\}/2+k_1+3k_2/2}\}\e_1\|\hat{g}_{k_1}\|_{L^\infty_t([2^{M-1},2^M])H^2_\xi}
    \end{aligned}
\end{equation*}
and
\begin{align*}
    \|T_{0+it}\hat{g}\|_{L^2}\lesssim 2^{\e k_--k-2k_{2,+}+\gamma k_2}\min\{2^{-M/2+\gamma M/8+\gamma k_2/4},2^{M+3\min\{k,k_2\}/2+3k_2/2}\}\e_1\|\hat{g}_{k_1}\|_{L^\infty_t([2^{M-1},2^M])H^1_\xi}.
\end{align*}
By the interpolation result in Lemma \ref{si},
\begin{align*}
    \|T_{\alpha}\hat{g}\|_{L^2}
    \lesssim & 2^{\e k_--k-2k_{2,+}+\gamma k_2+\alpha(k_1-k)}\min\{2^{-M/2+\gamma M/4+\gamma k_{1,+}/4+\gamma k_{2,+}/4},\\
    &\qquad\qquad\qquad\qquad 2^{M+3\min\{k,k_2\}/2+3 k_2/2}\}\e_1\|\hat{g}\|_{L^\infty_t([2^{M-1},2^M])H^{1+\alpha}_\xi},
\end{align*}
for any $\alpha\in(0,1)$.
Lastly, we plug $g=\F^{-1}\hat{f}_m\Tilde{\psi}_{k_1}$ into the result above and use \eqref{sobolevnorms} to get
\begin{align*}
    &\bigg\|D^\alpha\int_{t_1}^{t_2}\int_{\R^3}e^{it\phi(\xi,\eta)}\frac{\xi_m\partial_{\xi_l}\phi(\xi,\eta)}{2c_n|\xi|^2}q(\xi-\eta,\eta)\partial_{\eta_m}\hat{f}_{m,k_1}(t,\xi-\eta)\hat{f}_{n,k_2}(t,\eta)\psi_k(\xi)d\eta dt\bigg\|_{L^2}\\
    \lesssim &2^{\e k_--k-2k_{2,+}+\gamma k_2+\alpha(k_1-k)}\min\{2^{-M/2+\gamma M/4+\gamma k_{1,+}/4+\gamma k_{2,+}/4},\\
    &\qquad\qquad\qquad\qquad 2^{M+3\min\{k,k_2\}/2+3 k_2/2}\}\e_1\|\hat{f}_{m,k_1}\|_{L^\infty_t([2^{M-1},2^M])H^{1+\alpha}_\xi}\\
    \lesssim & 2^{\e k_--k-\alpha k-2k_{1,+}+\gamma k_1-2k_{2,+}+\gamma k_2}\min\{2^{-M/2+\gamma M/4-k_1/2+\gamma k_{1,+}/4+\gamma k_{2,+}/4},2^{M+3\min\{k,k_2\}/2-k_1/2+3 k_2/2}\}\e_1^2.
\end{align*}
\end{proof}
Now, we look at the final term in the splitting of $J^{M,2}_{k,k_1,k_2}$.
\begin{lem}
\label{km2-3}
Suppose $t_1,t_2\in[2^{M-1},2^M]$ and $\sup_{t\in[1,T]}\|f\|_{Z}\leq \e_1$. If either $(k_1,k_2)\in\chi^2_k\cup\chi_k^3$ or $(k_1,k_2)\in\chi^1_k$ and $c_l=c_m$, then
\begin{align*}
    \|\eqref{km2.4}\|_{L^2}
    \lesssim & 2^{\e k_--k-\alpha k-2k_{1,+}-2k_{2,+}+\gamma k_1+\gamma k_2+\alpha(k_1-k_2)}\min\{2^{-M/2+\gamma M/4-k_1+k_2/2+\gamma k_{1,+}/4+\gamma k_{2,+}/4},\\
    &\qquad\qquad\qquad\qquad 2^{M+3\min\{k,k_2\}/2+k_1/2+k_2/2}\}\e_1^2.
\end{align*}
\end{lem}
\begin{proof}
Since we restrict ourselves to the case when $c_n+c_m=0$ for the term $J^{M,2}_{k,k_1,k_2}$, we indeed have $\partial_{\xi_l}\phi(\xi,\eta)=2(c_l-c_m)\xi_l+2c_m\eta_l$. As we perform a change of variable $\zeta=\xi-\eta$, we can rewrite $\eqref{km2.4}$ into
\begin{align*}
    &\eqref{km2.4}\\
    =&D^\alpha\int_{t_1}^{t_2}\int_{\R^3}e^{it\phi(\xi,\xi-\zeta)}\frac{\xi_m((c_l-c_m)\xi_l+c_m(\xi_l-\zeta_l))}{c_n|\xi|^2}q(\zeta,\xi-\zeta)\hat{f}_{m,k_1}(t,\zeta)\partial_{\eta_m}\hat{f}_{n,k_2}(t,\xi-\zeta)\psi_k(\xi)d\zeta dt.
\end{align*}
Based on this observation, we define a family $T_z$ of operators on $\{z\in\C:0\leq \Re(z)\leq 1\}$ as follows
\begin{align*}
    T_{z} \hat{g} (\xi)=D^z\int_{t_1}^{t_2}\int_{\R^3}e^{it\phi(\xi,\xi-\eta)}\frac{\xi_m((c_l-c_m)\xi_l+c_m(\xi_l-\eta_l))}{2c_n|\xi|^2}q(\eta,\xi-\eta)\hat{f}_{m,k_1}(t,\eta)\partial_{\xi_m}\hat{g}_{k_2}(t,\xi-\eta)\psi_k(\xi)d\eta dt.
\end{align*}
It suffices to show that $T_{0+iy}:L^{\infty}_t([2^{M-1},2^M])H^1_\xi\rightarrow L^2_\xi$
and 
$T_{1+iy}:L^{\infty}_t([2^{M-1},2^M])H^2_\xi\rightarrow L^2_\xi$
are bounded for all $y\in\R$ with suitable bounds.\\
First, we look at the operator $T_{0+iy}$. The $L^2\times L^\infty\rightarrow L^2$ bilinear estimate, Lemma \ref{dualitycomp}, Lemma \ref{chi+eta}, bounds in \eqref{q} and \eqref{linfinity} yields
\begin{align*}
    &\|T_{0+it}\hat{g}\|_{L^2}\\
    =&\bigg\|\int_{t_1}^{t_2}\int_{\R^3}e^{it\phi(\xi,\xi-\eta)}\frac{\xi_m((c_l-c_m)\xi_l+c_m(\xi_l-\eta_l))}{c_n|\xi|^2}q(\eta,\xi-\eta)\hat{f}_{m,k_1}(t,\eta)\partial_{\xi_m}\hat{g}_{k_2}(t,\xi-\eta)\psi_k(\xi)d\eta dt\bigg\|_{L^2}\\
    \lesssim &2^M\sup_{t\in[2^{M-1},2^M]}\|\F^{-1} \frac{\xi_m((c_l-c_m)\xi_l+c_m(\xi_l-\eta_l))}{c_n|\xi|^2}q(\eta,\xi-\eta)\Tilde{\psi}_k(\xi)\Tilde{\psi}_{k_1}(\eta)\Tilde{\psi}_{k_2}(\xi-\eta)\|_{L^1}\times\\
    &\times \min\{\|e^{ic_nt\la}\F^{-1}\nabla\hat{g}_{k_2}\|_{L^{2}_x}\|e^{ic_mt\la}f_{m,k_1}\|_{L^\infty_x},2^{3\min\{k,k_2\}/2}\|e^{ic_nt\la}\F^{-1}\nabla\hat{g}_{k_2}\|_{L^2_x}\|e^{c_mit\la}f_{m,k_1}\|_{L^2_x}\}\\
    \lesssim & 2^{\e k_-+M+k_2-k}\sup_{t\in[2^{M-1},2^M]}\min\{\|\nabla\hat{g}_{k_2}\|_{L^{2}_\xi}\|e^{ic_mt\la}f_{m,k_1}\|_{L^\infty_x},2^{3\min\{k,k_2\}/2}\|\nabla\hat{g}_{k_2}\|_{L^2_\xi}\|f_{m,k_1}\|_{L^2_x}\}\\
    \lesssim & 2^{\e k_-+M+k_2-k-2k_{1,+}+\gamma k_1}\min\{2^{-3M/2+\gamma M/8-k_1+\gamma k_1/4},2^{3\min\{k,k_2\}/2+k_1/2}\}\|\nabla\hat{g}_{k_2}\|_{L^\infty_t([2^{M-1},2^M])L^{2}_\xi}\\
    \lesssim & 2^{\e k_--k-2k_{1,+}+\gamma k_1}\min\{2^{-M/2+\gamma M/8-k_1+k_2+\gamma k_1/4},2^{M+3\min\{k,k_2\}/2+k_1/2+k_2}\}\|\hat{g}_{k_2}\|_{L^\infty_t([2^{M-1},2^M])H^1_\xi}.
\end{align*}
Next, we perform the following computation for $T_{1+iy}$,
\begin{align}
    &\|T_{1+it}\hat{g}\|_{L^2}\nonumber\\
    =&\bigg\|\partial_{\xi_p}\int_{t_1}^{t_2}\int_{\R^3}e^{it\phi(\xi,\xi-\eta)}\frac{\xi_m((c_l-c_m)\xi_l+c_m(\xi_l-\eta_l))}{c_n|\xi|^2}q(\eta,\xi-\eta)\hat{f}_{m,k_1}(t,\eta)\partial_{\xi_m}\hat{g}_{k_2}(t,\xi-\eta)\psi_k(\xi)d\eta dt\bigg\|_{L^2}\nonumber\\
    \leq &\bigg\|\int_{t_1}^{t_2}\int_{\R^3}e^{it\phi(\xi,\xi-\eta)}\frac{it\partial_{\xi_p}\phi(\xi,\xi-\eta)\xi_m((c_l-c_m)\xi_l+c_m(\xi_l-\eta_l))}{c_n|\xi|^2}q(\eta,\xi-\eta)\hat{f}_{m,k_1}(t,\eta)\partial_{\xi_m}\hat{g}_{k_2}(t,\xi-\eta)\psi_k(\xi)d\eta dt\bigg\|_{L^2}\label{km2-3.1}\\
    & + \bigg\|\int_{t_1}^{t_2}\int_{\R^3}e^{it\phi(\xi,\xi-\eta)}\partial_{\xi_p}\big(\frac{\xi_m((c_l-c_m)\xi_l+c_m(\xi_l-\eta_l))}{c_n|\xi|^2}q(\eta,\xi-\eta)\big)\hat{f}_{m,k_1}(t,\eta)\partial_{\xi_m}\hat{g}_{k_2}(t,\xi-\eta)\psi_k(\xi)d\eta dt\bigg\|_{L^2}\label{km2-3.2}\\
    & + \bigg\|\int_{t_1}^{t_2}\int_{\R^3}e^{it\phi(\xi,\xi-\eta)}\frac{\xi_m((c_l-c_m)\xi_l+c_m(\xi_l-\eta_l))}{c_n|\xi|^2}q(\eta,\xi-\eta)\hat{f}_{m,k_1}(t,\eta)\partial_{\xi_p}\partial_{\xi_m}\hat{g}_{k_2}(t,\xi-\eta)\psi_k(\xi)d\eta dt\bigg\|_{L^2}\label{km2-3.3}\\
    & + \bigg\|\int_{t_1}^{t_2}\int_{\R^3}e^{it\phi(\xi,\xi-\eta)}\frac{\xi_m((c_l-c_m)\xi_l+c_m(\xi_l-\eta_l))}{c_n|\xi|^2}q(\eta,\xi-\eta)\hat{f}_{m,k_1}(t,\eta)\partial_{\xi_m}\hat{g}_{k_2}(t,\xi-\eta)\partial_{\xi_p}\psi_k(\xi)d\eta dt\bigg\|_{L^2}\label{km2-3.7}.
\end{align}
The last term can be estimated directly using the results for $T_{0+it}$ and \eqref{addder},
\begin{align*}
    &\|\eqref{km2-3.7}\|_{L^2}\\
    \lesssim &\|\nabla\psi_k\|_{L^\infty}\|T_{0+it}\hat{g}\|_{L^2}\\
    \lesssim &2^{-k+\e k_--k-2k_{1,+}+\gamma k_1}\min\{2^{-M/2+\gamma M/8-k_1+k_2+\gamma k_1/4},2^{M+3\min\{k,k_2\}/2+k_1/2+k_2}\}\|\hat{g}_{k_2}\|_{L^\infty_t([2^{M-1},2^M])H^1_\xi}\\
    \lesssim &2^{\e k_--2k-2k_{1,+}+\gamma k_1}\min\{2^{-M/2+\gamma M/8-k_1+2k_2+\gamma k_1/4},2^{M+3\min\{k,k_2\}/2+k_1/2+2k_2}\}\|\hat{g}_{k_2}\|_{L^\infty_t([2^{M-1},2^M])H^2_\xi}.
\end{align*}
For the term \eqref{km2-3.1}, integration by parts using the identity $e^{it\phi(\xi,\eta)}=-\sum_{j=1}^3\frac{\xi_j}{it2c_n|\xi|^2}\partial_{\eta_j}e^{it\phi(\xi,\eta)}$ gives
\begin{align}
    &\int_{t_1}^{t_2}\int_{\R^3}\big(\sum_{j}\frac{\xi_j}{2c_n|\xi|^2}\partial_{\eta_j}\big)e^{it\phi(\xi,\xi-\eta)}\frac{\partial_{\xi_p}\phi(\xi,\xi-\eta)\xi_m((c_l-c_m)\xi_l+c_m(\xi_l-\eta_l))}{c_n|\xi|^2} q(\eta,\xi-\eta)\times\nonumber\\
    &\qquad\qquad\qquad\qquad \times\hat{f}_{m,k_1}(t,\eta)\partial_{\xi_m}\hat{g}_{k_2}(t,\xi-\eta)\psi_k(\xi)d\eta dt\nonumber\\
    =&-\int_{t_1}^{t_2}\int_{\R^3}e^{it\phi(\xi,\xi-\eta)}\partial_{\eta_j}\big(\frac{\xi_j\partial_{\xi_p}\phi(\xi,\xi-\eta)\xi_m((c_l-c_m)\xi_l+c_m(\xi_l-\eta_l))}{2c_n^2|\xi|^4} q(\eta,\xi-\eta)\big)\times\nonumber\\
    &\qquad\qquad\qquad\qquad \times\hat{f}_{m,k_1}(t,\eta)\partial_{\xi_m}\hat{g}_{k_2}(t,\xi-\eta)\psi_k(\xi)d\eta dt\label{km2-3.4}\\
    &-\int_{t_1}^{t_2}\int_{\R^3}e^{it\phi(\xi,\xi-\eta)}\frac{\xi_j\partial_{\xi_p}\phi(\xi,\xi-\eta)\xi_m((c_l-c_m)\xi_l+c_m(\xi_l-\eta_l))}{2c_n^2|\xi|^4} q(\eta,\xi-\eta)\times\nonumber\\
    &\qquad\qquad\qquad\qquad \times\partial_{\eta_j}\hat{f}_{m,k_1}(t,\eta)\partial_{\xi_m}\hat{g}_{k_2}(t,\xi-\eta)\psi_k(\xi)d\eta dt\label{km2-3.6}\\
    &-\int_{t_1}^{t_2}\int_{\R^3}e^{it\phi(\xi,\xi-\eta)}\frac{\xi_j\partial_{\xi_p}\phi(\xi,\xi-\eta)\xi_m((c_l-c_m)\xi_l+c_m(\xi_l-\eta_l))}{2c_n^2|\xi|^4} q(\eta,\xi-\eta)\times\nonumber\\
    &\qquad\qquad\qquad\qquad \times\hat{f}_{m,k_1}(t,\eta)\partial_{\eta_j}\partial_{\xi_m}\hat{g}_{k_2}(t,\xi-\eta)\psi_k(\xi)d\eta dt\label{km2-3.5}.
\end{align}
Applying the $L^2\times L^\infty\rightarrow L^2$ bilinear estimate, along with Lemma \ref{chi+eta}, Lemma \ref{dualitycomp}, bounds in \eqref{q}, \eqref{addder}, \eqref{l2}, and \eqref{linfinity}, we obtain
\begin{align*}
    &\|\eqref{km2-3.2}\|_{L^2}+\|\eqref{km2-3.4}\|_{L^2}\\
    \lesssim & 2^M\sup_{t\in[2^{M-1},2^M]}\big(\|\F^{-1} \partial_{\xi_p}\big(\frac{\xi_m((c_l-c_m)\xi_l+c_m(\xi_l-\eta_l))}{c_n|\xi|^2}q(\eta,\xi-\eta)\big)\Tilde{\psi}_k(\xi)\Tilde{\psi}_{k_1}(\eta)\Tilde{\psi}_{k_2}(\xi-\eta)\|_{L^1}+\\
    &+\|\F^{-1}\partial_{\eta_j}\big(\frac{\xi_j\partial_{\xi_p}\phi(\xi,\xi-\eta)\xi_m((c_l-c_m)\xi_l+c_m(\xi_l-\eta_l))}{2c_n^2|\xi|^4}q(\eta,\xi-\eta)\big)\Tilde{\psi}_k(\xi)\Tilde{\psi}_{k_1}(\eta)\Tilde{\psi}_{k_2}(\xi-\eta)\|_{L^1}\big)\times\\
    &\times \min\{\|e^{ic_nt\la}\F^{-1}\nabla\hat{g}_{k_2}\|_{L^2_x}\|e^{ic_mt\la}f_{m,k_1}\|_{L^\infty_x},2^{3\min\{k,k_2\}/2}\|e^{ic_nt\la}\F^{-1}\nabla\hat{g}_{k_2}\|_{L^2_x}\|e^{ic_mt\la}f_{m,k_1}\|_{L^2_x}\}\\
    \lesssim &2^{\e k_-+M-k}(1+2^{k_2-k})\sup_{t\in[2^{M-1},2^M]}\min\{\|e^{ic_nt\la}\F^{-1}\nabla\hat{g}_{k_2}\|_{L^2_x}\|e^{ic_mt\la}f_{m,k_1}\|_{L^\infty_x},\\
    &\qquad\qquad\qquad\qquad\qquad 2^{3\min\{k,k_2\}/2}\|\nabla\hat{g}_{k_2}\|_{L^2_\xi}\|f_{m,k_1}\|_{L^2_x}\}\\
    \lesssim &2^{\e k_-+M-k-2k_{1,+}+\gamma k_1}(1+2^{k_2-k})\min\{2^{-3M/2+\gamma M/8-k_1+\gamma k_1/4+k_2},2^{3\min\{k,k_2\}/2+k_2+k_1/2}\}\e_1\times\\
    &\qquad\times\|\nabla^2\hat{g}_{k_2}\|_{L^\infty_t([2^{M-1},2^M])L^2_\xi}\\
    \lesssim &2^{\e k_--2k-2k_{1,+}+\gamma k_1}\min\{2^{-M/2+\gamma M/8+\gamma k_1/4+k_2},2^{M+3\min\{k,k_2\}/2+k_2+3k_1/2}\}\e_1\|\hat{g}_{k_2}\|_{L^\infty_t([2^{M-1},2^M])H^2_\xi}
\end{align*}
and
\begin{align*}
    &\|\eqref{km2-3.3}\|_{L^2}+\|\eqref{km2-3.5}\|_{L^2}\\
    \lesssim & 2^M\sup_{t\in[2^{M-1},2^M]}\big(\|\F^{-1} \frac{\xi_m((c_l-c_m)\xi_l+c_m(\xi_l-\eta_l))}{c_n|\xi|^2} q(\eta,\xi-\eta)\Tilde{\psi}_k(\xi)\Tilde{\psi}_{k_1}(\eta)\Tilde{\psi}_{k_2}(\xi-\eta)\|_{L^1}\\
    &+\|\F^{-1} \frac{\xi_j\partial_{\xi_p}\phi(\xi,\xi-\eta)\xi_m((c_l-c_m)\xi_l+c_m(\xi_l-\eta_l))}{2c_n^2|\xi|^4}  q(\eta,\xi-\eta)\Tilde{\psi}_k(\xi)\Tilde{\psi}_{k_1}(\eta)\Tilde{\psi}_{k_2}(\xi-\eta)\|_{L^1}\big)\times\\
    &\times\min\{\|e^{ic_nt\la}\F^{-1}\nabla^2\hat{g}_{k_2}\|_{L^2_x}\|e^{ic_mt\la}f_{m,k_1}\|_{L^\infty_x},2^{3\min\{k,k_2\}/2}\|e^{ic_nt\la}\F^{-1}\nabla^2\hat{g}_{k_2}\|_{L^2_x}\|e^{ic_mt\la}f_{m,k_1}\|_{L^2_x}\}\\
    \lesssim & 2^{\e k_-+M+k_2-k+k_1-k}\sup_{t\in[2^{M-1},2^M]}\min\{\|\nabla^2\hat{g}_{k_2}\|_{L^2_\xi}\|e^{ic_mt\la}f_{m,k_1}\|_{L^\infty_x},2^{3\min\{k,k_2\}/2}\|\nabla^2\hat{g}_{k_2}\|_{L^2_\xi}\|f_{m,k_1}\|_{L^2_x}\}\\
    \lesssim & 2^{\e k_-+M+k_2-k+k_1-k-2k_{1,+}+\gamma k_1}\min\{2^{-3M/2+\gamma M/8-k_1+\gamma k_1/4},2^{3\min\{k,k_2\}/2+k_1/2}\}\e_1\|\nabla^2\hat{g}_{k_2}\|_{L^\infty_t([2^{M-1},2^M])L^2_\xi}\\
    \lesssim & 2^{\e k_--2k-2k_{1,+}+\gamma k_1}\min\{2^{-M/2+\gamma M/8+k_2+\gamma k_1/4},2^{M+3\min\{k,k_2\}/2+3k_1/2+k_2}\}\e_1\|\hat{g}_{k_2}\|_{L^\infty_t([2^{M-1},2^M])H^2_\xi}.
\end{align*}
From the bilinear estimate $L^6\times L^3\rightarrow L^2$ and Lemma \ref{dualitycomp}, we get
\begin{align*}
    &\|\eqref{km2-3.6}\|_{L^2}\\
    \lesssim & 2^M\sup_{t\in[2^{M-1},2^M]}\|\F^{-1} \frac{\xi_j\partial_{\xi_p}\phi(\xi,\xi-\eta)\xi_m((c_l-c_m)\xi_l+c_m(\xi_l-\eta_l))}{2c_n^2|\xi|^4} q(\eta,\xi-\eta)\Tilde{\psi}_k(\xi)\Tilde{\psi}_{k_1}(\eta)\Tilde{\psi}_{k_2}(\xi-\eta)\|_{L^1}\times\\
    & \qquad\times\min\{\|e^{ic_nt\la}\F^{-1}\nabla\hat{g}_{k_2}\|_{L^6_x}\|e^{ic_mt\la}\F^{-1}\nabla\hat{f}_{m,k_1}\|_{L^3_x},\\
    &\qquad\qquad\qquad\qquad\qquad 2^{3\min\{k,k_2\}/2}\|e^{ic_nt\la}\F^{-1}\nabla\hat{g}_{k_2}\|_{L^2_x}\|e^{ic_mt\la}\F^{-1}\nabla\hat{f}_{m,k_1}\|_{L^2_x}\}\\
    \lesssim & 2^{\e k_-+M+k_2-k+k_1-k}\sup_{t\in[2^{M-1},2^M]}\min\{\|e^{ic_nt\la}\F^{-1}\nabla\hat{g}_{k_2}\|_{L^6_x}\|e^{ic_mt\la}\F^{-1}\nabla\hat{f}_{m,k_1}\|_{L^3_x},\\
    &\qquad\qquad\qquad\qquad 2^{3\min\{k,k_2\}/2}\|\nabla\hat{g}_{k_2}\|_{L^2_\xi}\|\nabla\hat{f}_{m,k_1}\|_{L^2_\xi}\}.
\end{align*}
Therefore, using the $L^6$ norm result \eqref{l6normforoneder} obtained in the previous lemma, together with estimations \eqref{l2first} and \eqref{l3nablaf}, we have
\begin{align*}
    \|\eqref{km2-3.6}\|_{L^2}
    \lesssim & 2^{\e k_-+M-2k_{1,+}+\gamma k_1+k_2-k+k_1-k}\min\{2^{-3M/2+\gamma M/4-k_1+\gamma k_1/4 +\gamma k_2/4}, \\
    &\qquad\qquad\qquad\qquad 2^{3\min\{k,k_2\}/2+k_2-k_1/2}\}\e_1\|\nabla^2\hat{g}_{k_2}\|_{L^\infty_t([2^{M-1},2^M])L^2_\xi}\\
    \lesssim & 2^{\e k_--2k_{1,+}+\gamma k_1-2k}\min\{2^{-M/2+\gamma M/4+k_2+\gamma k_1/4 +\gamma k_2/4},\\
    &\qquad\qquad\qquad\qquad 2^{M+3\min\{k,k_2\}/2+2k_2+k_1/2}\}\e_1\|\hat{g}_{k_2}\|_{L^\infty_t([2^{M-1},2^M])H^2_\xi}.
\end{align*}
In conclusion, we showed
\begin{align*}
    \|T_{1+it}\hat{g}\|_{L^2}
    \lesssim &2^{\e k_--2k-2k_{1,+}+\gamma k_1}\min\{2^{-M/2+\gamma M/4+k_2+\gamma k_{1,+}/4 +\gamma k_{2,+}/4},\\
    &\qquad\qquad\qquad\qquad 2^{M+3\min\{k,k_2\}/2+k_2+3k_1/2}\}\e_1\|\hat{g}\|_{L^\infty_t([2^{M-1},2^M])H^2_\xi}
\end{align*}
and
\begin{align*}
    \|T_{0+it}\hat{g}\|_{L^2}
    \lesssim &2^{\e k_--k-2k_{1,+}+\gamma k_1}\min\{2^{-M/2+\gamma M/8-k_1+k_2+\gamma k_1/4},\\
    &\qquad\qquad\qquad\qquad 2^{M+3\min\{k,k_2\}/2+k_1/2+k_2}\}\|\hat{g}\|_{L^\infty_t([2^{M-1},2^M])H^1_\xi}.
\end{align*}
Applying the interpolation result in Lemma \ref{si} gives us
\begin{align*}
    \|T_{\alpha}\hat{g}\|_{L^2}
    &\lesssim2^{\e k_--k-2k_{1,+}+\gamma k_1+\alpha(k_1-k)}\min\{2^{-M/2+\gamma M/4-k_1+k_2+\gamma k_{1,+}/4+\gamma k_{2,+}/4},\\
    &\qquad\qquad\qquad\qquad2^{M+3\min\{k,k_2\}/2+k_1/2+k_2}\}\e_1\|\hat{g}\|_{L^\infty_t([2^{M-1},2^M])H^{1+\alpha}_\xi},
\end{align*}
for any $\alpha\in(0,1)$.\\
As we plug in $g=\F^{-1}\hat{f}_n\Tilde{\psi}_{k_2}$, we obtain the desired bound
\begin{align*}
    &\bigg\|D^\alpha\int_{t_1}^{t_2}\int_{\R^3}e^{it\phi(\xi,\eta)}\frac{\xi_m\partial_{\xi_l}\phi(\xi,\eta)}{2c_n|\xi|^2}q(\xi-\eta,\eta)\hat{f}_{m,k_1}(t,\xi-\eta)\partial_{\eta_m}\hat{f}_{n,k_2}(t,\eta)\psi_k(\xi)d\eta dt\bigg\|_{L^2}\\
    \lesssim & 2^{\e k_--k-2k_{1,+}+\gamma k_1+\alpha(k_1-k)}\min\{2^{-M/2+\gamma M/4-k_1+k_2+\gamma k_{1,+}/4+\gamma k_{2,+}/4},\\
    &\qquad\qquad\qquad\qquad 2^{M+3\min\{k,k_2\}/2+k_1/2+k_2}\}\e_1\|\hat{f}_{k_2}\|_{L^\infty_t([2^{M-1},2^M])H^{1+\alpha}_\xi}\\
    \lesssim & 2^{\e k_--k-\alpha k-2k_{1,+}-2k_{2,+}+\gamma k_1+\gamma k_2+\alpha(k_1-k_2)}\min\{2^{-M/2+\gamma M/4-k_1+k_2/2+\gamma k_{1,+}/4+\gamma k_{2,+}/4},\\
    &\qquad\qquad\qquad\qquad 2^{M+3\min\{k,k_2\}/2+k_1/2+k_2/2}\}\e_1^2.
\end{align*}
\end{proof}
Combine the results we obtained in Lemma \ref{km2-1}, Lemma \ref{km2-2} and Lemma \ref{km2-3}. Given that $0<\alpha<1/2$ and $k_2\leq k_1+2a+2$, we have $2^{\alpha(k_1-k_2)}\lesssim 2^{(k_1-k_2)/2}$ and
\begin{align*}
    \|D^\alpha_\xi J^{M,2}_{k,k_1,k_2}\|_{L^2}
    \lesssim & 2^{\e k_--k-2k_{1,+}-2k_{2,+}+\gamma k_1+\gamma k_2-\alpha k}\min\{2^{-M/2-k_2/4-k_1/4}
    ,2^{M+3\min\{k,k_2\}/2 +k_2/2+k_1/2}\}\e_1^2\\
    &+2^{\e k_--k-\alpha k-2k_{1,+}+\gamma k_1-2k_{2,+}+\gamma k_2}\min\{2^{-M/2+\gamma M/4-k_1/2+\gamma k_{1,+}/4+\gamma k_{2,+}/4},\\
    &\qquad\qquad\qquad\qquad 2^{M+3\min\{k,k_2\}/2-k_1/2+3 k_2/2}\}\e_1^2\\
    &+2^{\e k_--k-\alpha k-2k_{1,+}-2k_{2,+}+\gamma k_1+\gamma k_2+\alpha(k_1-k_2)}\min\{2^{-M/2+\gamma M/4-k_1+k_2/2+\gamma k_{1,+}/4+\gamma k_{2,+}/4},\\
    &\qquad\qquad\qquad\qquad 2^{M+3\min\{k,k_2\}/2+k_1/2+k_2/2}\}\e_1^2\\
    \lesssim & 2^{\e k_--k-\alpha k-2k_{1,+}-2k_{2,+}+\gamma k_1+\gamma k_2}\big(\min\{2^{-M/2-k_2/4-k_1/4}
    ,2^{M+3\min\{k,k_2\}/2 +k_2/2+k_1/2}\}\\
    &+\min\{2^{-M/2+\gamma M/4-k_1/2+\gamma k_{1,+}/2},2^{M+3\min\{k,k_2\}/2+k_1}\}\big)\e_1^2,
\end{align*}
if either $(k_1,k_2)\in\chi^2_k\cup\chi^3_k$ or $(k_1,k_2)\in\chi^1_k$ and $c_l=c_m$. Therefore, we proved
\begin{equation}
    \label{jm2inchi2}
    \begin{aligned}
    &\sum_{1\leq M\leq\log T}\sup_{2^{M-1}\leq t_1\leq t_2\leq 2^M}2^{2k_+-\gamma k+k/2+\alpha k}\bigg(\sum_{\substack{c_m+c_n= 0\\c_l=c_m}}A_{lmn}\sum_{(k_1,k_2)\in\chi^1_k}\|D^\alpha J^{M,2}_{k,k_1,k_2}\|_{L^2}+\\
    &\qquad\qquad\qquad\qquad +\sum_{\substack{c_m+c_n= 0}}A_{lmn}\sum_{(k_1,k_2)\in\chi^2_k\cup\chi^3_k}\|D^\alpha J^{M,2}_{k,k_1,k_2}\|_{L^2}\bigg)\\
    \lesssim &\sum_{1\leq M\leq \log T}\sum_{(k_1,k_2)\in\chi_k^1\cup\chi_k^2\cup\chi_k^3} 2^{\e k_--k/2-\gamma k+2k_+-2k_{1,+}-2k_{2,+}+\gamma k_1+\gamma k_2}\big(\min\{2^{-M/2-k_2/4-k_1/4}
    ,\\
    &\qquad\qquad\qquad\qquad 2^{M+3\min\{k,k_2\}/2 +k_2/2+k_1/2}\}+\min\{2^{-M/2+\gamma M/4-k_1/2+\gamma k_{1,+}/2},2^{M+3\min\{k,k_2\}/2+k_1}\}\big)\e_1^2\\
    \lesssim & \e_1^2.
    \end{aligned}
\end{equation}

\subsubsection{$D^{\alpha}I^{M,2}_{k,k_1,k_2}$ and $D^{\alpha}J^{M,2}_{k,k_1,k_2}$ where $(k_1,k_2)\in\chi^1_k$}
\label{2-2}
In this section, we focus solely on the cases where $(k_1,k_2)\in\chi^1_k$ and $c_l\neq c_m$. These conditions give us the lower bound $|\phi(\xi,\eta)|\gtrsim 2^{k_1}$ according to \eqref{lbt} and allows us to use the identity
$$e^{it\phi(\xi,\eta)}=\frac{1}{i\phi(\xi,\eta)}\partial_te^{it\phi(\xi,\eta)}.$$
Recall the definition of $I^{m,2}_{k,k_1,k_2}$ and $J^{m,2}_{k,k_1,k_2}$ in \eqref{i2} and \eqref{j2}. We perform integration by parts in the variable $t$ and get
\begin{align}
    D^{\alpha}_\xi I^{m,2}_{k,k_1,k_2}
    =&D^{\alpha}_\xi\int_{\R^3}e^{it_2\phi(\xi,\eta)}\frac{t_2\partial_{\xi_l}\phi(\xi,\eta)}{i\phi(\xi,\eta)}\hat{f}_{m,k_1}(t_2,\xi-\eta)\hat{f}_{n,k_2}(t_2,\eta)\psi_k(\xi)d\eta\label{f1-i}\\
    &-D^{\alpha}_\xi\int_{\R^3}e^{it_1\phi(\xi,\eta)}\frac{t_1\partial_{\xi_l}\phi(\xi,\eta)}{i\phi(\xi,\eta)}\hat{f}_{m,k_1}(t_1,\xi-\eta)\hat{f}_{n,k_2}(t_1,\eta)\psi_k(\xi)d\eta\label{f2-i}\\
    &-D^{\alpha}_\xi\int_{t_1}^{t_2}\int_{\R^3}e^{it\phi(\xi,\eta)}\frac{\partial_{\xi_l}\phi(\xi,\eta)}{i\phi(\xi,\eta)}\hat{f}_{m,k_1}(t,\xi-\eta)\hat{f}_{n,k_2}(t,\eta)\psi_k(\xi)d\eta dt\label{f3-i}\\
    &-D^{\alpha}_\xi\int_{t_1}^{t_2}\int_{\R^3}e^{it\phi(\xi,\eta)}\frac{t\partial_{\xi_l}\phi(\xi,\eta)}{i\phi(\xi,\eta)}\partial_t\hat{f}_{m,k_1}(t,\xi-\eta)\hat{f}_{n,k_2}(t,\eta)\psi_k(\xi)d\eta dt\label{f4-i}\\
    &-D^{\alpha}_\xi\int_{t_1}^{t_2}\int_{\R^3}e^{it\phi(\xi,\eta)}\frac{t\partial_{\xi_l}\phi(\xi,\eta)}{i\phi(\xi,\eta)}\hat{f}_{m,k_1}(t,\xi-\eta)\partial_t\hat{f}_{n,k_2}(t,\eta)\psi_k(\xi)d\eta dt\label{f5-i}
\end{align}
and
\begin{align}
    D^{\alpha}_\xi J^{m,2}_{k,k_1,k_2}
    =&D^{\alpha}_\xi\int_{\R^3}e^{it_2\phi(\xi,\eta)}\frac{t_2\partial_{\xi_l}\phi(\xi,\eta)}{i\phi(\xi,\eta)}q(\eta,\xi-\eta)\hat{f}_{m,k_1}(t_2,\xi-\eta)\hat{f}_{n,k_2}(t_2,\eta)\psi_k(\xi)d\eta\label{f1}\\
    &-D^{\alpha}_\xi\int_{\R^3}e^{it_1\phi(\xi,\eta)}\frac{t_1\partial_{\xi_l}\phi(\xi,\eta)}{i\phi(\xi,\eta)}q(\eta,\xi-\eta)\hat{f}_{m,k_1}(t_1,\xi-\eta)\hat{f}_{n,k_2}(t_1,\eta)\psi_k(\xi)d\eta\label{f2}\\
    &-D^{\alpha}_\xi\int_{t_1}^{t_2}\int_{\R^3}e^{it\phi(\xi,\eta)}\frac{\partial_{\xi_l}\phi(\xi,\eta)}{i\phi(\xi,\eta)}q(\eta,\xi-\eta)\hat{f}_{m,k_1}(t,\xi-\eta)\hat{f}_{n,k_2}(t,\eta)\psi_k(\xi)d\eta dt\label{f3}\\
    &-D^{\alpha}_\xi\int_{t_1}^{t_2}\int_{\R^3}e^{it\phi(\xi,\eta)}\frac{t\partial_{\xi_l}\phi(\xi,\eta)}{i\phi(\xi,\eta)}q(\eta,\xi-\eta)\partial_t\hat{f}_{m,k_1}(t,\xi-\eta)\hat{f}_{n,k_2}(t,\eta)\psi_k(\xi)d\eta dt\label{f4}\\
    &-D^{\alpha}_\xi\int_{t_1}^{t_2}\int_{\R^3}e^{it\phi(\xi,\eta)}\frac{t\partial_{\xi_l}\phi(\xi,\eta)}{i\phi(\xi,\eta)}q(\eta,\xi-\eta)\hat{f}_{m,k_1}(t,\xi-\eta)\partial_t\hat{f}_{n,k_2}(t,\eta)\psi_k(\xi)d\eta dt\label{f5}.
\end{align}
We will divide the resulted terms into two groups, based on whether the derivative in $t$ falls on the product $\hat{f}_{m,k_1}(t,\xi-\eta)\hat{f}_{n,k_2}(t,\xi)$, and estimate them respectively in Lemma \ref{jm2-1} and \ref{jm2-2}. For both groups, we will use Lemma \ref{dalpha} to bound the $\alpha$ fractional derivative.
\begin{lem}
\label{jm2-1}
Suppose $t_1,t_2\in[2^{M-1},2^M]$ and $\sup_{t\in[1,T]}\|f\|_{Z}\leq \e_1$. If $(k_1,k_2)\in\chi^1_k$ and $c_l\neq c_m$, then
\begin{align*}
    &\|\eqref{f1-i}+\eqref{f2-i}+\eqref{f3-i}\|_{L^2}+\|\eqref{f1}+\eqref{f2}+\eqref{f3}\|_{L^2}\\
    \lesssim &  2^{-k-2k_{+}+\gamma k-2k_{2,+}+\gamma k_2-\alpha k_2}\min\{2^{-M/2+\gamma M/8-k+k_2/2+\gamma k/4},2^{M+k/2+2k_2}\}\e_1^2.
\end{align*}
\end{lem}
\begin{proof}
Let
\begin{align*}
    F_1(\xi)=\int_{\R^3}e^{it_j\phi(\xi,\eta)}\frac{t_j\partial_{\xi_l}\phi(\xi,\eta)}{\phi(\xi,\eta)}\hat{f}_{m,k_1}(t_j,\xi-\eta)\hat{f}_{n,k_2}(t_j,\eta)\psi_k(\xi)d\eta,
\end{align*}
\begin{align*}
    F^q_1(\xi)=\int_{\R^3}e^{it_j\phi(\xi,\eta)}\frac{t_j\partial_{\xi_l}\phi(\xi,\eta)}{\phi(\xi,\eta)}q(\eta,\xi-\eta)\hat{f}_{m,k_1}(t_j,\xi-\eta)\hat{f}_{n,k_2}(t_j,\eta)\psi_k(\xi)d\eta
\end{align*}
for $j=1,2$ and
\begin{align*}
    F_2(\xi)=\int_{t_1}^{t_2}\int_{\R^3}e^{it\phi(\xi,\eta)}\frac{\partial_{\xi_l}\phi(\xi,\eta)}{\phi(\xi,\eta)}\hat{f}_{m,k_1}(t,\xi-\eta)\hat{f}_{n,k_2}(t,\eta)\psi_k(\xi)d\eta dt,
\end{align*}
\begin{align*}
    F^q_2(\xi)=\int_{t_1}^{t_2}\int_{\R^3}e^{it\phi(\xi,\eta)}\frac{\partial_{\xi_l}\phi(\xi,\eta)}{\phi(\xi,\eta)}q(\eta,\xi-\eta)\hat{f}_{m,k_1}(t,\xi-\eta)\hat{f}_{n,k_2}(t,\eta)\psi_k(\xi)d\eta dt.
\end{align*}
In order to apply Lemma \ref{dalpha}, we shall obtain bounds on $\|F_1+F_2\|_{L^2}$, $\|F^q_1+F^q_2\|_{L^2}$, $\|\nabla_\xi (F_1+F_2)\|_{L^2}$, and $\|\nabla_\xi (F^q_1+F^q_2)\|_{L^2}$. \\
First, using the bilinear estimate $L^\infty\times L^2\rightarrow L^2$ in Lemma \ref{bilinear}, along with Lemma \ref{chi,eta}, Lemma \ref{dualitycomp}, estimates in \eqref{q}, \eqref{linfinity}, and \eqref{l2}, we obtain
\begin{align*}
    &\|F_1+F_2\|_{L^2}+\|F^q_1+F^q_2\|_{L^2}\\
    \lesssim & 2^{M}\big(\|\F^{-1}\frac{\partial_{\xi_l}\phi(\xi,\eta)}{\phi(\xi,\eta)}\Tilde{\psi}_k(\xi)\Tilde{\psi}_{k_1}(\xi-\eta)\Tilde{\psi}_{k_2}(\eta)\|_{L^1}+\|\F^{-1}\frac{\partial_{\xi_l}\phi(\xi,\eta)}{\phi(\xi,\eta)}q(\xi-\eta,\eta)\Tilde{\psi}_k(\xi)\Tilde{\psi}_{k_1}(\xi-\eta)\Tilde{\psi}_{k_2}(\eta)\|_{L^1}\big)\times\\
    &\times\sup_{t\in[2^{M-1},2^M]}\min\{\|e^{ic_mt\la}f_{m,k_1}\|_{L^\infty_x}\|e^{ic_nt\la}f_{n,k_2}\|_{L^2_x},2^{3k_2/2}\|e^{ic_mt\la}f_{m,k_1}\|_{L^2_x}\|e^{ic_nt\la}f_{n,k_2}\|_{L^2_x}\}\\
    \lesssim & 2^{M-k}\sup_{t\in[2^{M-1},2^M]}\min\{\|e^{ic_mt\la}f_{m,k_1}\|_{L^\infty_x}\|f_{n,k_2}\|_{L^2_x},2^{3k_2/2}\|f_{m,k_1}\|_{L^2_x}\|f_{n,k_2}\|_{L^2_x}\}\\
    \lesssim & 2^{-k-2k_{1,+}+\gamma k_1-2k_{2,+}+\gamma k_2}\min\{2^{-M/2+\gamma M/8-k_1+\gamma k_1/4+k_2/2},2^{M+k_1/2+2k_2}\}\e_1^2.
\end{align*}
Next, we compute the derivatives and get
\begin{align}
    \partial_{\xi_m} F_1
    = &\int_{\R^3}e^{it_j\phi(\xi,\eta)}\frac{t_j\partial_{\xi_l}\phi(\xi,\eta)}{\phi(\xi,\eta)}\partial_{\xi_m}\hat{f}_{m,k_1}(t_j,\xi-\eta)\hat{f}_{n,k_2}(t_j,\eta)\psi_k(\xi)d\eta\label{jm2.1-i}\\
    & + \int_{\R^3}e^{it_j\phi(\xi,\eta)}\partial_{\xi_m}\frac{t_j\partial_{\xi_l}\phi(\xi,\eta)}{\phi(\xi,\eta)}\hat{f}_{m,k_1}(t_j,\xi-\eta)\hat{f}_{n,k_2}(t_j,\eta)\psi_k(\xi)d\eta\label{jm2.11.2-i}\\
    &+ \int_{\R^3}e^{it_j\phi(\xi,\eta)}\frac{t_j\partial_{\xi_l}\phi(\xi,\eta)}{\phi(\xi,\eta)}\hat{f}_{m,k_1}(t_j,\xi-\eta)\hat{f}_{n,k_2}(t_j,\eta)\partial_{\xi_m}\psi_k(\xi)d\eta\label{jm2.12-i}\\
    & + \int_{\R^3}e^{it_j\phi(\xi,\eta)}\frac{i\partial_{\xi_m}\phi(\xi,\eta)t^2_j\partial_{\xi_l}\phi(\xi,\eta)}{\phi(\xi,\eta)}\hat{f}_{m,k_1}(t_j,\xi-\eta)\hat{f}_{n,k_2}(t_j,\eta)\psi_k(\xi)d\eta\label{jm2.2-i},
\end{align}
\begin{align}
    \partial_{\xi_m} F^q_1
    = &\int_{\R^3}e^{it_j\phi(\xi,\eta)}\frac{t_j\partial_{\xi_l}\phi(\xi,\eta)}{\phi(\xi,\eta)}q(\eta,\xi-\eta)\partial_{\xi_m}\hat{f}_{m,k_1}(t_j,\xi-\eta)\hat{f}_{n,k_2}(t_j,\eta)\psi_k(\xi)d\eta\label{jm2.1}\\
    & + \int_{\R^3}e^{it_j\phi(\xi,\eta)}\partial_{\xi_m}\big(\frac{t_j\partial_{\xi_l}\phi(\xi,\eta)}{\phi(\xi,\eta)}q(\eta,\xi-\eta)\big)\hat{f}_{m,k_1}(t_j,\xi-\eta)\hat{f}_{n,k_2}(t_j,\eta)\psi_k(\xi)d\eta\label{jm2.11.2}\\
    &+ \int_{\R^3}e^{it_j\phi(\xi,\eta)}\frac{t_j\partial_{\xi_l}\phi(\xi,\eta)}{\phi(\xi,\eta)}q(\eta,\xi-\eta)\hat{f}_{m,k_1}(t_j,\xi-\eta)\hat{f}_{n,k_2}(t_j,\eta)\partial_{\xi_m}\psi_k(\xi)d\eta\label{jm2.12}\\
    & + \int_{\R^3}e^{it_j\phi(\xi,\eta)}\frac{i\partial_{\xi_m}\phi(\xi,\eta)t^2_j\partial_{\xi_l}\phi(\xi,\eta)}{\phi(\xi,\eta)}q(\eta,\xi-\eta)\hat{f}_{m,k_1}(t_j,\xi-\eta)\hat{f}_{n,k_2}(t_j,\eta)\psi_k(\xi)d\eta\label{jm2.2},
\end{align}
\begin{align}
    \partial_{\xi_m} F_2
    =
    &\int_{t_1}^{t_2}\int_{\R^3}e^{it\phi(\xi,\eta)}\frac{\partial_{\xi_l}\phi(\xi,\eta)}{\phi(\xi,\eta)}\partial_{\xi_m}\hat{f}_{m,k_1}(t,\xi-\eta)\hat{f}_{n,k_2}(t,\eta)\psi_k(\xi)d\eta dt\label{jm2.1-2-i}\\
    &+ \int_{t_1}^{t_2}\int_{\R^3}e^{it\phi(\xi,\eta)}\partial_{\xi_m}\frac{\partial_{\xi_l}\phi(\xi,\eta)}{\phi(\xi,\eta)}\hat{f}_{m,k_1}(t,\xi-\eta)\hat{f}_{n,k_2}(t,\eta)\psi_k(\xi)d\eta dt\label{jm2.11.2-2-i}\\
    &+ \int_{t_1}^{t_2}\int_{\R^3}e^{it\phi(\xi,\eta)}\frac{\partial_{\xi_l}\phi(\xi,\eta)}{\phi(\xi,\eta)}\hat{f}_{m,k_1}(t,\xi-\eta)\hat{f}_{n,k_2}(t,\eta)\partial_{\xi_m}\psi_k(\xi)d\eta dt\label{jm2.12-2-i}\\
    & + \int_{t_1}^{t_2}\int_{\R^3}e^{it\phi(\xi,\eta)}\frac{it\partial_{\xi_m}\phi(\xi,\eta)\partial_{\xi_l}\phi(\xi,\eta)}{\phi(\xi,\eta)}\hat{f}_{m,k_1}(t,\xi-\eta)\hat{f}_{n,k_2}(t,\eta)\psi_k(\xi)d\eta dt\label{jm2.2-2-i},
\end{align}
and
\begin{align}
    \partial_{\xi_m} F^q_2
    =
    &\int_{t_1}^{t_2}\int_{\R^3}e^{it\phi(\xi,\eta)}\frac{\partial_{\xi_l}\phi(\xi,\eta)}{\phi(\xi,\eta)}q(\eta,\xi-\eta)\partial_{\xi_m}\hat{f}_{m,k_1}(t,\xi-\eta)\hat{f}_{n,k_2}(t,\eta)\psi_k(\xi)d\eta dt\label{jm2.1-2}\\
    &+ \int_{t_1}^{t_2}\int_{\R^3}e^{it\phi(\xi,\eta)}\partial_{\xi_m}\big(\frac{\partial_{\xi_l}\phi(\xi,\eta)}{\phi(\xi,\eta)}q(\eta,\xi-\eta)\big)\hat{f}_{m,k_1}(t,\xi-\eta)\hat{f}_{n,k_2}(t,\eta)\psi_k(\xi)d\eta dt\label{jm2.11.2-2}\\
    &+ \int_{t_1}^{t_2}\int_{\R^3}e^{it\phi(\xi,\eta)}\frac{\partial_{\xi_l}\phi(\xi,\eta)}{\phi(\xi,\eta)}q(\eta,\xi-\eta)\hat{f}_{m,k_1}(t,\xi-\eta)\hat{f}_{n,k_2}(t,\eta)\partial_{\xi_m}\psi_k(\xi)d\eta dt\label{jm2.12-2}\\
    & + \int_{t_1}^{t_2}\int_{\R^3}e^{it\phi(\xi,\eta)}\frac{it\partial_{\xi_m}\phi(\xi,\eta)\partial_{\xi_l}\phi(\xi,\eta)}{\phi(\xi,\eta)}q(\eta,\xi-\eta)\hat{f}_{m,k_1}(t,\xi-\eta)\hat{f}_{n,k_2}(t,\eta)\psi_k(\xi)d\eta dt\label{jm2.2-2}.
\end{align}
Furthermore, we perform integration by parts in the space variable on the terms $\eqref{jm2.2-i}$, $\eqref{jm2.2}$, $\eqref{jm2.2-2-i}$, and $\eqref{jm2.2-2}$ using the identity  $e^{it\phi(\xi,\eta)}=\sum_n\frac{\partial_{\eta_n}\phi(\xi,\eta)}{it|\nabla_\eta\phi(\xi,\eta)|^2}\partial_{\eta_n}e^{it\phi(\xi,\eta)}$,
\begin{align}
    \eqref{jm2.2-i}
    =&-\int_{\R^3}e^{it_j\phi(\xi,\eta)}\partial_{\eta_n}\frac{\partial_{\eta_n}\phi(\xi,\eta)\partial_{\xi_m}\phi(\xi,\eta)t_j\partial_{\xi_l}\phi(\xi,\eta)}{|\nabla_\eta\phi(\xi,\eta)|^2\phi(\xi,\eta)}\hat{f}_{m,k_1}(t_j,\xi-\eta)\hat{f}_{n,k_2}(t_j,\eta)\psi_k(\xi)d\eta\label{jm2.3-i}\\
    &-\int_{\R^3}e^{it_j\phi(\xi,\eta)}\frac{\partial_{\eta_n}\phi(\xi,\eta)\partial_{\xi_m}\phi(\xi,\eta)t_j\partial_{\xi_l}\phi(\xi,\eta)}{|\nabla_\eta\phi(\xi,\eta)|^2\phi(\xi,\eta)}\partial_{\eta_n}\hat{f}_{m,k_1}(t_j,\xi-\eta)\hat{f}_{n,k_2}(t_j,\eta)\psi_k(\xi)d\eta\label{jm2.4-i}\\
    &-\int_{\R^3}e^{it_j\phi(\xi,\eta)}\frac{\partial_{\eta_n}\phi(\xi,\eta)\partial_{\xi_m}\phi(\xi,\eta)t_j\partial_{\xi_l}\phi(\xi,\eta)}{|\nabla_\eta\phi(\xi,\eta)|^2\phi(\xi,\eta)}\hat{f}_{m,k_1}(t_j,\xi-\eta)\partial_{\eta_n}\hat{f}_{n,k_2}(t_j,\eta)\psi_k(\xi)d\eta\label{jm2.5-i},
\end{align}
\begin{align}
    \eqref{jm2.2}
    =&-\int_{\R^3}e^{it_j\phi(\xi,\eta)}\partial_{\eta_n}\big(\frac{\partial_{\eta_n}\phi(\xi,\eta)\partial_{\xi_m}\phi(\xi,\eta)t_j\partial_{\xi_l}\phi(\xi,\eta)}{|\nabla_\eta\phi(\xi,\eta)|^2\phi(\xi,\eta)}q(\eta,\xi-\eta)\big)\hat{f}_{m,k_1}(t_j,\xi-\eta)\hat{f}_{n,k_2}(t_j,\eta)\psi_k(\xi)d\eta\label{jm2.3}\\
    &-\int_{\R^3}e^{it_j\phi(\xi,\eta)}\frac{\partial_{\eta_n}\phi(\xi,\eta)\partial_{\xi_m}\phi(\xi,\eta)t_j\partial_{\xi_l}\phi(\xi,\eta)}{|\nabla_\eta\phi(\xi,\eta)|^2\phi(\xi,\eta)}q(\eta,\xi-\eta)\partial_{\eta_n}\hat{f}_{m,k_1}(t_j,\xi-\eta)\hat{f}_{n,k_2}(t_j,\eta)\psi_k(\xi)d\eta\label{jm2.4}\\
    &-\int_{\R^3}e^{it_j\phi(\xi,\eta)}\frac{\partial_{\eta_n}\phi(\xi,\eta)\partial_{\xi_m}\phi(\xi,\eta)t_j\partial_{\xi_l}\phi(\xi,\eta)}{|\nabla_\eta\phi(\xi,\eta)|^2\phi(\xi,\eta)}q(\eta,\xi-\eta)\hat{f}_{m,k_1}(t_j,\xi-\eta)\partial_{\eta_n}\hat{f}_{n,k_2}(t_j,\eta)\psi_k(\xi)d\eta\label{jm2.5},
\end{align}
\begin{align}
    \eqref{jm2.2-2-i}
    =&-\int_{t_1}^{t_2}\int_{\R^3}e^{it\phi(\xi,\eta)}\partial_{\eta_n}\frac{\partial_{\eta_n}\phi(\xi,\eta)\partial_{\xi_m}\phi(\xi,\eta)\partial_{\xi_l}\phi(\xi,\eta)}{|\nabla_\eta\phi(\xi,\eta)|^2\phi(\xi,\eta)}\hat{f}_{m,k_1}(t,\xi-\eta)\hat{f}_{n,k_2}(t,\eta)\psi_k(\xi)d\eta dt\label{jm2.3-2-i}\\
    &-\int_{t_1}^{t_2}\int_{\R^3}e^{it\phi(\xi,\eta)}\frac{\partial_{\eta_n}\phi(\xi,\eta)\partial_{\xi_m}\phi(\xi,\eta)\partial_{\xi_l}\phi(\xi,\eta)}{|\nabla_\eta\phi(\xi,\eta)|^2\phi(\xi,\eta)}\partial_{\eta_n}\hat{f}_{m,k_1}(t,\xi-\eta)\hat{f}_{n,k_2}(t,\eta)\psi_k(\xi)d\eta dt\label{jm2.4-2-i}\\
    &-\int_{t_1}^{t_2}\int_{\R^3}e^{it\phi(\xi,\eta)}\frac{\partial_{\eta_n}\phi(\xi,\eta)\partial_{\xi_m}\phi(\xi,\eta)\partial_{\xi_l}\phi(\xi,\eta)}{|\nabla_\eta\phi(\xi,\eta)|^2\phi(\xi,\eta)}\hat{f}_{m,k_1}(t,\xi-\eta)\partial_{\eta_n}\hat{f}_{n,k_2}(t,\eta)\psi_k(\xi)d\eta dt\label{jm2.5-2-i},
\end{align}
and
\begin{align}
    \eqref{jm2.2-2}
    =&-\int_{t_1}^{t_2}\int_{\R^3}e^{it\phi(\xi,\eta)}\partial_{\eta_n}\big(\frac{\partial_{\eta_n}\phi(\xi,\eta)\partial_{\xi_m}\phi(\xi,\eta)\partial_{\xi_l}\phi(\xi,\eta)}{|\nabla_\eta\phi(\xi,\eta)|^2\phi(\xi,\eta)}q(\eta,\xi-\eta)\big)\hat{f}_{m,k_1}(t,\xi-\eta)\hat{f}_{n,k_2}(t,\eta)\psi_k(\xi)d\eta dt\label{jm2.3-2}\\
    &-\int_{t_1}^{t_2}\int_{\R^3}e^{it\phi(\xi,\eta)}\frac{\partial_{\eta_n}\phi(\xi,\eta)\partial_{\xi_m}\phi(\xi,\eta)\partial_{\xi_l}\phi(\xi,\eta)}{|\nabla_\eta\phi(\xi,\eta)|^2\phi(\xi,\eta)}q(\eta,\xi-\eta)\partial_{\eta_n}\hat{f}_{m,k_1}(t,\xi-\eta)\hat{f}_{n,k_2}(t,\eta)\psi_k(\xi)d\eta dt\label{jm2.4-2}\\
    &-\int_{t_1}^{t_2}\int_{\R^3}e^{it\phi(\xi,\eta)}\frac{\partial_{\eta_n}\phi(\xi,\eta)\partial_{\xi_m}\phi(\xi,\eta)\partial_{\xi_l}\phi(\xi,\eta)}{|\nabla_\eta\phi(\xi,\eta)|^2\phi(\xi,\eta)}q(\eta,\xi-\eta)\hat{f}_{m,k_1}(t,\xi-\eta)\partial_{\eta_n}\hat{f}_{n,k_2}(t,\eta)\psi_k(\xi)d\eta dt\label{jm2.5-2}.
\end{align}
By the bilinear estimate $L^2\times L^\infty\rightarrow L^2$, together with Lemma \ref{dualitycomp}, Lemma \ref{chi,eta}, estimates in \eqref{q}, \eqref{l2}, \eqref{l2first} and \eqref{linfinity}, we have
\begin{align*}
    &\|\eqref{jm2.11.2-i}+\eqref{jm2.11.2-2-i}+\eqref{jm2.12-i}+\eqref{jm2.12-2-i}+\eqref{jm2.3-i}+\eqref{jm2.3-2-i}
    \|_{L^2}+\\
    &+\|\eqref{jm2.11.2}+\eqref{jm2.11.2-2}+\eqref{jm2.12}+\eqref{jm2.12-2}+\eqref{jm2.3}+\eqref{jm2.3-2}\|_{L^2}\\
    \lesssim &2^{M}\big(\|\F^{-1}\partial_{\xi_m}\frac{\partial_{\xi_l}\phi(\xi,\eta)}{\phi(\xi,\eta)}\Tilde{\psi}_k(\xi)\Tilde{\psi}_{k_1}(\xi-\eta)\Tilde{\psi}_{k_2}(\eta)\|_{L^1}+\\
    &+\|\F^{-1}\partial_{\xi_m}\big(\frac{\partial_{\xi_l}\phi(\xi,\eta)}{\phi(\xi,\eta)}q(\eta,\xi-\eta)\big)\Tilde{\psi}_k(\xi)\Tilde{\psi}_{k_1}(\xi-\eta)\Tilde{\psi}_{k_2}(\eta)\|_{L^1}+\\
    & + \|\nabla_\xi\psi_k\|_{L^\infty}(\|\F^{-1}\frac{\partial_{\xi_l}\phi(\xi,\eta)}{\phi(\xi,\eta)}\Tilde{\psi}_k(\xi)\Tilde{\psi}_{k_1}(\xi-\eta)\Tilde{\psi}_{k_2}(\eta)\|_{L^1}
    +\\
    &+\|\F^{-1}\frac{\partial_{\xi_l}\phi(\xi,\eta)}{\phi(\xi,\eta)}q(\eta,\xi-\eta)\Tilde{\psi}_k(\xi)\Tilde{\psi}_{k_1}(\xi-\eta)\Tilde{\psi}_{k_2}(\eta)\|_{L^1})+\\
    & + \|\F^{-1}\partial_{\eta_n}\frac{\partial_{\eta_n}\phi(\xi,\eta)\partial_{\xi_m}\phi(\xi,\eta)\partial_{\xi_l}\phi(\xi,\eta)}{|\nabla_\eta\phi(\xi,\eta)|^2\phi(\xi,\eta)}\Tilde{\psi}_k(\xi)\Tilde{\psi}_{k_1}(\xi-\eta)\Tilde{\psi}_{k_2}(\eta)\|_{L^1}+\\
    &+\|\F^{-1}\partial_{\eta_n}\big(\frac{\partial_{\eta_n}\phi(\xi,\eta)\partial_{\xi_m}\phi(\xi,\eta)\partial_{\xi_l}\phi(\xi,\eta)}{|\nabla_\eta\phi(\xi,\eta)|^2\phi(\xi,\eta)}q(\eta,\xi-\eta)\big)\Tilde{\psi}_k(\xi)\Tilde{\psi}_{k_1}(\xi-\eta)\Tilde{\psi}_{k_2}(\eta)\|_{L^1}\big)\times\\
    &\times\sup_{t\in[2^{M-1},2^M]}\min\{\|e^{ic_mt\la}f_{m,k_1}\|_{L^\infty_x}\|e^{ic_nt\la}{f}_{n,k_2}\|_{L^2_x},2^{3k_2/2}\|e^{ic_mt\la}f_{m,k_1}\|_{L^2_x}\|e^{ic_nt\la}{f}_{n,k_2}\|_{L^2_x}\}\\
    \lesssim & 2^{M-2k}(1+2^{k-k_2})\sup_{t\in[2^{M-1},2^M]}\min\{\|e^{ic_mt\la}f_{m,k_1}\|_{L^\infty_x}\|\hat{f}_{n,k_2}\|_{L^2_\xi},2^{3k_2/2}\|f_{m,k_1}\|_{L^2_x}\|{f}_{n,k_2}\|_{L^2_x}\}\\
    \lesssim & 2^{-k-k_2-2k_{1,+}-2k_{2,+}+\gamma k_1+\gamma k_2}\min\{2^{-M/2+\gamma M/8-k_1+k_2/2+\gamma k_1/4},2^{M+3k_2/2+k_1/2+k_2/2}\}\e_1^2\\
    \lesssim & 2^{-k-2k_{+}-2k_{2,+}+\gamma k+\gamma k_2}\min\{2^{-M/2+\gamma M/8-k-k_2/2+\gamma k/4},2^{M+k_2+k/2}\}\e_1^2.
\end{align*}
Using the bilinear estimate $L^6\times L^3\rightarrow L^2$ instead with the estimate in \eqref{l3nablaf}, we get
\begin{align*}
    &\|\eqref{jm2.1-i}+\eqref{jm2.1-2-i}+\eqref{jm2.4-i}+\eqref{jm2.4-2-i}\|_{L^2}+\|\eqref{jm2.1}+\eqref{jm2.1-2}+\eqref{jm2.4}+\eqref{jm2.4-2}\|_{L^2}\\
    \lesssim & 2^{M}\big(\|\F^{-1}\frac{\partial_{\xi_l}\phi(\xi,\eta)}{\phi(\xi,\eta)}\Tilde{\psi}_k(\xi)\Tilde{\psi}_{k_1}(\xi-\eta)\Tilde{\psi}_{k_2}(\eta)\|_{L^1}+\|\F^{-1}\frac{\partial_{\xi_l}\phi(\xi,\eta)}{\phi(\xi,\eta)}q(\eta,\xi-\eta)\Tilde{\psi}_k(\xi)\Tilde{\psi}_{k_1}(\xi-\eta)\Tilde{\psi}_{k_2}(\eta)\|_{L^1}+\\
    &+\|\F^{-1}\frac{\partial_{\eta_n}\phi(\xi,\eta)\partial_{\xi_m}\phi(\xi,\eta)\partial_{\xi_l}\phi(\xi,\eta)}{|\nabla_\eta\phi(\xi,\eta)|^2\phi(\xi,\eta)}\Tilde{\psi}_k(\xi)\Tilde{\psi}_{k_1}(\xi-\eta)\Tilde{\psi}_{k_2}(\eta)\|_{L^1}+\\
    &+\|\F^{-1}\frac{\partial_{\eta_n}\phi(\xi,\eta)\partial_{\xi_m}\phi(\xi,\eta)\partial_{\xi_l}\phi(\xi,\eta)}{|\nabla_\eta\phi(\xi,\eta)|^2\phi(\xi,\eta)}q(\eta,\xi-\eta)\Tilde{\psi}_k(\xi)\Tilde{\psi}_{k_1}(\xi-\eta)\Tilde{\psi}_{k_2}(\eta)\|_{L^1}\big)\times\\
    &\times\sup_{t\in[2^{M-1},2^M]} \min\{\|e^{ic_mt\la}\F^{-1}\nabla_\xi\hat{f}_{m,k_1}\|_{L^3_x}\|e^{ic_nt\la}f_{n,k_2}\|_{L^6_x},2^{3k_2/2}\|e^{ic_mt\la}\F^{-1}\nabla_\xi\hat{f}_{m,k_1}\|_{L^2_x}\|e^{ic_nt\la}f_{n,k_2}\|_{L^2_x}\}\\
    \lesssim & 2^{M-k}\sup_{t\in[2^{M-1},2^M]}\min\{\|e^{ic_mt\la}\F^{-1}\nabla_\xi\hat{f}_{m,k_1}\|_{L^3_x}\|e^{ic_nt\la}f_{n,k_2}\|_{L^6_x}, 2^{3k_2/2}\|\nabla_\xi\hat{f}_{m,k_1}\|_{L^2_\xi}\|f_{n,k_2}\|_{L^2_x}\}\\
    \lesssim & 2^{-k-2k_{1,+}-2k_{2,+}+\gamma k_1+\gamma k_2}\min\{2^{-M/2+\gamma M/8-k_1+\gamma k_1/4 -k_2/2},2^{M+3k_2/2-k_1/2+k_2/2}\}\e_1^2\\
    \lesssim & 2^{-k-2k_{+}-2k_{2,+}+\gamma k+\gamma k_2}\min\{2^{-M/2+\gamma M/8-k-k_2/2+\gamma k/4},2^{M+2k_2-k/2}\}\e_1^2.
\end{align*}
Lastly, we may bound the remaining terms by the $L^\infty\times L^2\rightarrow L^2$ bilinear estimate, along with Lemma \ref{chi,eta}, Lemma \ref{timel2}, Lemma \ref{dualitycomp}, estimations \eqref{l2}, \eqref{l2first} and \eqref{linfinity},
\begin{align*}
    &\|\eqref{jm2.5-i}+\eqref{jm2.5-2-i}\|_{L^2}+\|\eqref{jm2.5}+\eqref{jm2.5-2}\|_{L^2}\\
    \lesssim &2^{M}\big(\|\F^{-1} \frac{\partial_{\eta_n}\phi(\xi,\eta)\partial_{\xi_m}\phi(\xi,\eta)\partial_{\xi_l}\phi(\xi,\eta)}{|\nabla_\eta\phi(\xi,\eta)|^2\phi(\xi,\eta)}\Tilde{\psi}_k(\xi)\Tilde{\psi}_{k_1}(\xi-\eta)\Tilde{\psi}_{k_2}(\eta)\|_{L^1}+\\
    &+\|\F^{-1} \frac{\partial_{\eta_n}\phi(\xi,\eta)\partial_{\xi_m}\phi(\xi,\eta)\partial_{\xi_l}\phi(\xi,\eta)}{|\nabla_\eta\phi(\xi,\eta)|^2\phi(\xi,\eta)}q(\eta,\xi-\eta)\Tilde{\psi}_k(\xi)\Tilde{\psi}_{k_1}(\xi-\eta)\Tilde{\psi}_{k_2}(\eta)\|_{L^1}\big)\times\\
    &\times\sup_{t\in[2^{M-1},2^M]}\min\{\|e^{ic_mt\la}f_{m,k_1}\|_{L^\infty_x}\|e^{ic_nt\la}\F^{-1}\nabla_\xi\hat{f}_{n,k_2}\|_{L^2_x},2^{3k_2/2}\|e^{ic_mt\la}f_{m,k_1}\|_{L^2_x}\|e^{ic_nt\la}\F^{-1}\nabla_\xi\hat{f}_{n,k_2}\|_{L^2_x}\}\\
    \lesssim & 2^{M-k}\sup_{t\in[2^{M-1},2^M]}\min\{\|e^{ic_mt\la}f_{m,k_1}\|_{L^\infty_x}\|\nabla_\xi\hat{f}_{n,k_2}\|_{L^2_\xi},2^{3k_2/2}\|f_{m,k_1}\|_{L^2_x}\|\nabla_\xi\hat{f}_{n,k_2}\|_{L^2_\xi}\}\\
    \lesssim & 2^{-k-2k_{1,+}-2k_{2,+}+\gamma k_1+\gamma k_2}\min\{2^{-M/2+\gamma M/8-k_1+\gamma k_1/4-k_2/2},2^{M+3k_2/2+k_1/2-k_2/2}\}\e_1^2\\
    \lesssim & 2^{-k-2k_{+}-2k_{2,+}+\gamma k+\gamma k_2}\min\{2^{-M/2+\gamma M/8-k+\gamma k/4-k_2/2},2^{M+k_2+k/2}\}\e_1^2.
\end{align*}
Thus, we showed
\begin{align*}
    &\|\nabla_\xi(F_1+F_2)\|_{L^2}+\|\nabla_\xi(F^q_1+F^q_2)\|_{L^2}\\
    \lesssim & 2^{-k-2k_{+}-2k_{2,+}+\gamma k+\gamma k_2}\min\{2^{-M/2+\gamma M/8-k-k_2/2+\gamma k/4},2^{M+k_2+k/2}\}\e_1^2
\end{align*}
and
\begin{align*}
    &\|F_1+F_2\|_{L^2}+\|F^q_1+F^q_2\|_{L^2}\\
    \lesssim & 2^{-k-2k_{+}+\gamma k-2k_{2,+}+\gamma k_2}\min\{2^{-M/2+\gamma M/8-k+k_2/2+\gamma k/4},2^{M+k/2+2k_2}\}\e_1^2,
\end{align*}
which imply
\begin{align*}
    &\|D^\alpha_\xi(F_1+F_2)\|_{L^2}+\|D^\alpha_\xi(F^q_1+F^q_2)\|_{L^2}\\
    &\leq \|F_1+F_2\|_{L^2}^{1-\alpha}\|\nabla_\xi(F_1+F_2)\|_{L^2}^\alpha+\|F^q_1+F^q_2\|_{L^2}^{1-\alpha}\|\nabla_\xi(F^q_1+F^q_2)\|_{L^2}^\alpha\\
    &\lesssim 2^{-k-2k_{+}+\gamma k-2k_{2,+}+\gamma k_2-\alpha k_2}\min\{2^{-M/2+\gamma M/8-k+k_2/2+\gamma k/4},2^{M+k/2+2k_2}\}\e_1^2,
\end{align*}
according to Lemma \ref{dalpha}.
\end{proof}
The estimation for the second group of terms is more intricate, as it requires bounding the $L^p$ norms for the time derivatives $\partial_t \hat{f}_k(t,\xi)$ and the mixed derivatives $\partial_t\nabla_\xi\hat{f}(t,\xi)$. We will state our results in two cases based on the relative size of $|\xi|\sim 2^k$ compared to $t\sim 2^M$.
\begin{lem}
\label{jm2-2}
Given $t_1,t_2\in[2^{M-1},2^M]$, $(k_1,k_2)\in\chi^1_k$, $c_l\neq c_m$, and $\sup_{t\in[1,T]}\|f\|_{Z}\leq \e_1$, we have
\begin{align*}
    &\|\eqref{f4-i}+\eqref{f5-i}\|_{L^2}+\|\eqref{f4}+\eqref{f5}\|_{L^2}\\
    \lesssim  & 2^{-2k_++\gamma k-2k_{2,+}+\gamma k_2-\alpha k_2}\min\{2^{-M/2-3k/2},2^{M-k/2+2k_2}\}\e_1^2,
\end{align*}
if $M+2k\leq 0$, and 
\begin{align*}
    &\|\eqref{f4-i}+\eqref{f5-i}\|_{L^2}+\|\eqref{f4}+\eqref{f5}\|_{L^2}\\
    \lesssim  & 2^{-2k_++\gamma k-\gamma k_-/2-2k_{2,+}+\gamma k_2-\gamma k_{2,-}/2+\alpha k} \min\{2^{-M/2-2k-k_2/2},2^{M-k/2+k_2}\}\e_1^2,
\end{align*}
if $M+2k>0$.
\end{lem}
\begin{proof}
Let
\begin{align*}
    F_1(\xi)=\int_{t_1}^{t_2}\int_{\R^3}e^{it\phi(\xi,\eta)}\frac{t\partial_{\xi_l}\phi(\xi,\eta)}{\phi(\xi,\eta)}\partial_t\hat{f}_{m,k_1}(t,\xi-\eta)\hat{f}_{n,k_2}(t,\eta)\psi_k(\xi)d\eta dt,
\end{align*}
\begin{align*}
    F^q_1(\xi)=\int_{t_1}^{t_2}\int_{\R^3}e^{it\phi(\xi,\eta)}\frac{t\partial_{\xi_l}\phi(\xi,\eta)}{\phi(\xi,\eta)}q(\eta,\xi-\eta)\partial_t\hat{f}_{m,k_1}(t,\xi-\eta)\hat{f}_{n,k_2}(t,\eta)\psi_k(\xi)d\eta dt,
\end{align*}
\begin{align*}
    F_2(\xi)=\int_{t_1}^{t_2}\int_{\R^3}e^{it\phi(\xi,\eta)}\frac{t\partial_{\xi_l}\phi(\xi,\eta)}{\phi(\xi,\eta)}\hat{f}_{m,k_1}(t,\xi-\eta)\partial_t\hat{f}_{n,k_2}(t,\eta)\psi_k(\xi)d\eta dt,
\end{align*}
and
\begin{align*}
    F^q_2(\xi)=\int_{t_1}^{t_2}\int_{\R^3}e^{it\phi(\xi,\eta)}\frac{t\partial_{\xi_l}\phi(\xi,\eta)}{\phi(\xi,\eta)}q(\eta,\xi-\eta)\hat{f}_{m,k_1}(t,\xi-\eta)\partial_t\hat{f}_{n,k_2}(t,\eta)\psi_k(\xi)d\eta dt.
\end{align*}
We shall use Lemma \ref{dalpha} to bound the $L^2$ norm of $D^\alpha F_i, D^\alpha F^q_i$ for $i=1,2$.\\
First, we find a bound for $\|F_1+F_2\|_{L^2}+\|F^q_1+F^q_2\|_{L^2}$.\\
Using the bilinear estimate $L^6\times L^3\rightarrow L^2$, along with Lemma \ref{chi,eta}, Lemma \ref{dualitycomp}, Lemma \ref{op}, Lemma \ref{lpnorms}, Lemma \ref{timel2}, \eqref{q}, and \eqref{l2}, we can estimate
\begin{align*}
    &\|F_1\|_{L^2}+\|F^q_1\|_{L^2}\\
    \lesssim & 2^{2M}\big(\|\F^{-1}\frac{\partial_{\xi_l}\phi(\xi,\eta)}{\phi(\xi,\eta)}\Tilde{\psi}_k(\xi)\Tilde{\psi}_{k_1}(\xi-\eta)\Tilde{\psi}_{k_2}(\eta)\|_{L^1}+\|\F^{-1}\frac{\partial_{\xi_l}\phi(\xi,\eta)}{\phi(\xi,\eta)}q(\eta,\xi-\eta)\Tilde{\psi}_k(\xi)\Tilde{\psi}_{k_1}(\xi-\eta)\Tilde{\psi}_{k_2}(\eta)\|_{L^1}\big)\times\\
    &\times\sup_{t\in[2^{M-1},2^M]}\min\{\|e^{ic_mt\la}\partial_t{f}_{m,k_1}\|_{L^6_x}\|e^{ic_nt\la}{f}_{n,k_2}\|_{L^3_x},2^{3k_2/2}\|e^{ic_mt\la}\partial_tf_{m,k_1}\|_{L^2_x}\|e^{ic_nt\la}f_{n,k_2}\|_{L^2_x}\}\\
    \lesssim & 2^{2M-k}\sup_{t\in[2^{M-1},2^M]}\min\{\|e^{ic_mt\la}\partial_t{f}_{m,k_1}\|_{L^6_x}\|e^{ic_nt\la}{f}_{n,k_2}\|_{L^3_x},2^{3k_2/2}\|\partial_t f_{m,k_1}\|_{L^2_x}\|f_{n,k_2}\|_{L^2_x}\}\\
    \lesssim & 2^{2M-k-2k_{1,+}+\gamma k_1-2k_{2,+}+\gamma k_2}\min\{2^{-2M-k_1/2-M/2},2^{3k_2/2-M+k_1/2+k_2/2}\}\e_1^2\\
    \lesssim & 2^{-2k_++\gamma k-2k_{2,+}+\gamma k_2}\min\{2^{-M/2-3k/2},2^{M-k/2+2k_2}\}\e_1^2.
\end{align*}
Similarly, by the $L^\infty\times L^2\rightarrow L^2$ bilinear estimate, Lemma \ref{chi,eta}, Lemma \ref{dualitycomp}, together with \eqref{q}, \eqref{linfinity}, and \eqref{l2}, we have
\begin{align*}
    &\|F_2\|_{L^2}+\|F^q_2\|_{L^2}\\
    \lesssim & 2^{2M}\big(\|\F^{-1}\frac{\partial_{\xi_l}\phi(\xi,\eta)}{\phi(\xi,\eta)}\Tilde{\psi}_k(\xi)\Tilde{\psi}_{k_1}(\xi-\eta)\Tilde{\psi}_{k_2}(\eta)\|_{L^1}+\|\F^{-1}\frac{\partial_{\xi_l}\phi(\xi,\eta)}{\phi(\xi,\eta)}q(\eta,\xi-\eta)\Tilde{\psi}_k(\xi)\Tilde{\psi}_{k_1}(\xi-\eta)\Tilde{\psi}_{k_2}(\eta)\|_{L^1}\big)\times\\
    &\times\sup_{t\in[2^{M-1},2^M]}\min\{\|e^{ic_mt\la}f_{m,k_1}\|_{L^\infty_x}\|e^{ic_nt\la}\partial_tf_{n,k_2}\|_{L^2_x},2^{3k_2/2}\|e^{ic_mt\la}f_{m,k_1}\|_{L^2_x}\|e^{ic_nt\la}\partial_tf_{n,k_2}\|_{L^2_x}\}\\
    \lesssim & 2^{2M-k}\sup_{t\in[2^{M-1},2^M]}\min\{\|e^{ic_mt\la}f_{m,k_1}\|_{L^\infty_x}\|\partial_tf_{n,k_2}\|_{L^2_x},2^{3k_2/2}\|f_{m,k_1}\|_{L^2_x}\|\partial_tf_{n,k_2}\|_{L^2_x}\}\\
    \lesssim & 2^{-k-2k_{1,+}+\gamma k_1-2k_{2,+}+\gamma k_2}\min\{2^{-M/2+\gamma M/8-k_1+\gamma k_1/4+k_2/2},2^{M+k_1/2+2k_2}\}\e_1^2\\
    \lesssim & 2^{-2k_++\gamma k-2k_{2,+}+\gamma k_2}\min\{2^{-M/2+\gamma M/8-2k+\gamma k/4+k_2/2},2^{M-k/2+2k_2}\}\e_1^2.
\end{align*}
Hence, we get
\begin{equation}
    \label{jm2-2bound1}
    \begin{aligned}
    \|F_1+F_2\|_{L^2}+\|F^q_1+F^q_2\|_{L^2}\lesssim 2^{-2k_++\gamma k-2k_{2,+}+\gamma k_2}\min\{(1+2^{\gamma M/8+\gamma k/4})2^{-M/2-3k/2},2^{M-k/2+2k_2}\}\e_1^2.
\end{aligned}
\end{equation}
Next, we use integration by parts with the identity $\sum_{m}\frac{\partial_{\eta_m}\phi(\xi,\eta)}{it|\nabla_\eta\phi(\xi,\eta)|^2}\partial_{\eta_m}e^{it\phi(\xi,\eta)}=e^{it\phi(\xi,\eta)}$ to obtain another bound for $\|F_1+F_2\|_{L^2}+\|F^q_1+F^q_2\|_{L^2}$. We compute that
\begin{align}
    F_1(\xi)
    =&\int_{t_1}^{t_2}\int_{\R^3}e^{it\phi(\xi,\eta)}\partial_{\eta_m}\frac{\partial_{\eta_m}\phi(\xi,\eta)\partial_{\xi_l}\phi(\xi,\eta)}{i|\nabla_\eta\phi(\xi,\eta)|^2\phi(\xi,\eta)}\partial_t\hat{f}_{m,k_1}(t,\xi-\eta)\hat{f}_{n,k_2}(t,\eta)\psi_k(\xi)d\eta dt\label{jm20.1.1-i}\\
    &+\int_{t_1}^{t_2}\int_{\R^3}e^{it\phi(\xi,\eta)}\frac{\partial_{\eta_m}\phi(\xi,\eta)\partial_{\xi_l}\phi(\xi,\eta)}{i|\nabla_\eta\phi(\xi,\eta)|^2\phi(\xi,\eta)}\partial_{\eta_m}\partial_t\hat{f}_{m,k_1}(t,\xi-\eta)\hat{f}_{n,k_2}(t,\eta)\psi_k(\xi)d\eta dt\label{jm20.1.3-i}\\
    &+\int_{t_1}^{t_2}\int_{\R^3}e^{it\phi(\xi,\eta)}\frac{\partial_{\eta_m}\phi(\xi,\eta)\partial_{\xi_l}\phi(\xi,\eta)}{i|\nabla_\eta\phi(\xi,\eta)|^2\phi(\xi,\eta)}\partial_t\hat{f}_{m,k_1}(t,\xi-\eta)\partial_{\eta_m}\hat{f}_{n,k_2}(t,\eta)\psi_k(\xi)d\eta dt,\label{jm20.1.4-i}
\end{align}
\begin{align}
    F^q_1(\xi)
    =&\int_{t_1}^{t_2}\int_{\R^3}e^{it\phi(\xi,\eta)}\partial_{\eta_m}\frac{\partial_{\eta_m}\phi(\xi,\eta)\partial_{\xi_l}\phi(\xi,\eta)}{i|\nabla_\eta\phi(\xi,\eta)|^2\phi(\xi,\eta)}q(\eta,\xi-\eta)\partial_t\hat{f}_{m,k_1}(t,\xi-\eta)\hat{f}_{n,k_2}(t,\eta)\psi_k(\xi)d\eta dt\label{jm20.1.1}\\
    +&\int_{t_1}^{t_2}\int_{\R^3}e^{it\phi(\xi,\eta)}\frac{\partial_{\eta_m}\phi(\xi,\eta)\partial_{\xi_l}\phi(\xi,\eta)}{i|\nabla_\eta\phi(\xi,\eta)|^2\phi(\xi,\eta)}\partial_{\eta_m}q(\eta,\xi-\eta)\partial_t\hat{f}_{m,k_1}(t,\xi-\eta)\hat{f}_{n,k_2}(t,\eta)\psi_k(\xi)d\eta dt\label{jm20.1.2}\\
    &+\int_{t_1}^{t_2}\int_{\R^3}e^{it\phi(\xi,\eta)}\frac{\partial_{\eta_m}\phi(\xi,\eta)\partial_{\xi_l}\phi(\xi,\eta)}{i|\nabla_\eta\phi(\xi,\eta)|^2\phi(\xi,\eta)}q(\eta,\xi-\eta)\partial_{\eta_m}\partial_t\hat{f}_{m,k_1}(t,\xi-\eta)\hat{f}_{n,k_2}(t,\eta)\psi_k(\xi)d\eta dt\label{jm20.1.3}\\
    &+\int_{t_1}^{t_2}\int_{\R^3}e^{it\phi(\xi,\eta)}\frac{\partial_{\eta_m}\phi(\xi,\eta)\partial_{\xi_l}\phi(\xi,\eta)}{i|\nabla_\eta\phi(\xi,\eta)|^2\phi(\xi,\eta)}q(\eta,\xi-\eta)\partial_t\hat{f}_{m,k_1}(t,\xi-\eta)\partial_{\eta_m}\hat{f}_{n,k_2}(t,\eta)\psi_k(\xi)d\eta dt,\label{jm20.1.4}
\end{align}
\begin{align}
    F_2(\xi)
    =&\int_{t_1}^{t_2}\int_{\R^3}e^{it\phi(\xi,\eta)}\partial_{\eta_m}\frac{\partial_{\eta_m}\phi(\xi,\eta)\partial_{\xi_l}\phi(\xi,\eta)}{i|\nabla_\eta\phi(\xi,\eta)|^2\phi(\xi,\eta)}\hat{f}_{m,k_1}(t,\xi-\eta)\partial_t\hat{f}_{n,k_2}(t,\eta)\psi_k(\xi)d\eta dt\label{jm20.2.1-i}\\
    &+\int_{t_1}^{t_2}\int_{\R^3}e^{it\phi(\xi,\eta)}\frac{\partial_{\eta_m}\phi(\xi,\eta)\partial_{\xi_l}\phi(\xi,\eta)}{i|\nabla_\eta\phi(\xi,\eta)|^2\phi(\xi,\eta)}\partial_{\eta_m}\hat{f}_{m,k_1}(t,\xi-\eta)\partial_t\hat{f}_{n,k_2}(t,\eta)\psi_k(\xi)d\eta dt\label{jm20.2.3-i}\\
    &+\int_{t_1}^{t_2}\int_{\R^3}e^{it\phi(\xi,\eta)}\frac{\partial_{\eta_m}\phi(\xi,\eta)\partial_{\xi_l}\phi(\xi,\eta)}{i|\nabla_\eta\phi(\xi,\eta)|^2\phi(\xi,\eta)}\hat{f}_{m,k_1}(t,\xi-\eta)\partial_{\eta_m}\partial_t\hat{f}_{n,k_2}(t,\eta)\psi_k(\xi)d\eta dt,\label{jm20.2.4-i}
\end{align}
and
\begin{align}
    F^q_2(\xi)
    =&\int_{t_1}^{t_2}\int_{\R^3}e^{it\phi(\xi,\eta)}\partial_{\eta_m}\frac{\partial_{\eta_m}\phi(\xi,\eta)\partial_{\xi_l}\phi(\xi,\eta)}{i|\nabla_\eta\phi(\xi,\eta)|^2\phi(\xi,\eta)}q(\eta,\xi-\eta)\hat{f}_{m,k_1}(t,\xi-\eta)\partial_t\hat{f}_{n,k_2}(t,\eta)\psi_k(\xi)d\eta dt\label{jm20.2.1}\\
    +&\int_{t_1}^{t_2}\int_{\R^3}e^{it\phi(\xi,\eta)}\frac{\partial_{\eta_m}\phi(\xi,\eta)\partial_{\xi_l}\phi(\xi,\eta)}{i|\nabla_\eta\phi(\xi,\eta)|^2\phi(\xi,\eta)}\partial_{\eta_m}q(\eta,\xi-\eta)\hat{f}_{m,k_1}(t,\xi-\eta)\partial_t\hat{f}_{n,k_2}(t,\eta)\psi_k(\xi)d\eta dt\label{jm20.2.2}\\
    &+\int_{t_1}^{t_2}\int_{\R^3}e^{it\phi(\xi,\eta)}\frac{\partial_{\eta_m}\phi(\xi,\eta)\partial_{\xi_l}\phi(\xi,\eta)}{i|\nabla_\eta\phi(\xi,\eta)|^2\phi(\xi,\eta)}q(\eta,\xi-\eta)\partial_{\eta_m}\hat{f}_{m,k_1}(t,\xi-\eta)\partial_t\hat{f}_{n,k_2}(t,\eta)\psi_k(\xi)d\eta dt\label{jm20.2.3}\\
    &+\int_{t_1}^{t_2}\int_{\R^3}e^{it\phi(\xi,\eta)}\frac{\partial_{\eta_m}\phi(\xi,\eta)\partial_{\xi_l}\phi(\xi,\eta)}{i|\nabla_\eta\phi(\xi,\eta)|^2\phi(\xi,\eta)}q(\eta,\xi-\eta)\hat{f}_{m,k_1}(t,\xi-\eta)\partial_{\eta_m}\partial_t\hat{f}_{n,k_2}(t,\eta)\psi_k(\xi)d\eta dt\label{jm20.2.4}.
\end{align}
The bilinear estimate $L^6\times L^3\rightarrow L^2$, Lemma \ref{chi,eta}, Lemma \ref{dualitycomp}, Lemma \ref{op}, Lemma \ref{lpnorms}, Lemma \ref{timel2}, along with the estimates in \eqref{q} and \eqref{l2} imply
\begin{align*}
    &\|\eqref{jm20.1.1-i}+\eqref{jm20.2.1-i}\|_{L^2}+\|\eqref{jm20.1.1}+\eqref{jm20.2.1}\|_{L^2}\\
    \lesssim & 2^{M}\big(\|\F^{-1}\partial_{\eta_m}\frac{\partial_{\eta_m}\phi(\xi,\eta)\partial_{\xi_l}\phi(\xi,\eta)}{i|\nabla_\eta\phi(\xi,\eta)|^2\phi(\xi,\eta)}\Tilde{\psi}_k(\xi)\Tilde{\psi}_{k_1}(\xi-\eta)\Tilde{\psi}_{k_2}(\eta)\|_{L^1}+\\
    &+\|\F^{-1}\partial_{\eta_m}\frac{\partial_{\eta_m}\phi(\xi,\eta)\partial_{\xi_l}\phi(\xi,\eta)}{i|\nabla_\eta\phi(\xi,\eta)|^2\phi(\xi,\eta)}q(\eta,\xi-\eta)\Tilde{\psi}_k(\xi)\Tilde{\psi}_{k_1}(\xi-\eta)\Tilde{\psi}_{k_2}(\eta)\|_{L^1}\big)\times\\
    &\times\sup_{t\in[2^{M-1},2^M]}\big(\min\{\|e^{ic_mt\la}\partial_t{f}_{m,k_1}\|_{L^6_x}\|e^{ic_nt\la}{f}_{n,k_2}\|_{L^3_x},2^{3k_2/2}\|e^{ic_mt\la}\partial_tf_{m,k_1}\|_{L^2_x}\|e^{ic_nt\la}f_{n,k_2}\|_{L^2_x}\}\\
    &\qquad + \min\{\|e^{ic_mt\la}{f}_{m,k_1}\|_{L^6_x}\|e^{ic_nt\la}\partial_t{f}_{n,k_2}\|_{L^3_x},2^{3k_2/2}\|e^{ic_mt\la}f_{m,k_1}\|_{L^2_x}\|e^{ic_nt\la}\partial_tf_{n,k_2}\|_{L^2_x}\}\big)\\
    \lesssim & 2^{M-3k}\sup_{t\in[2^{M-1},2^M]}\big(\min\{\|e^{ic_mt\la}\partial_t{f}_{m,k_1}\|_{L^6_x}\|e^{ic_nt\la}{f}_{n,k_2}\|_{L^3_x},2^{3k_2/2}\|\partial_t f_{m,k_1}\|_{L^2_x}\|f_{n,k_2}\|_{L^2_x}\}\\
    &\qquad + \min\{\|e^{ic_mt\la}{f}_{m,k_1}\|_{L^6_x}2^{k_2/2}\|e^{ic_nt\la}\partial_t{f}_{n,k_2}\|_{L^2_x},2^{3k_2/2}\|f_{m,k_1}\|_{L^2_x}\|\partial_tf_{n,k_2}\|_{L^2_x}\}\big)\\
    \lesssim & 2^{M-3k-2k_{1,+}+\gamma k_1-2k_{2,+}+\gamma k_2}\big(\min\{2^{-2M-k_1/2-M/2},2^{3k_2/2-M+k_1/2+k_2/2}\}\\
    &\qquad + \min\{2^{-M-k_1/2+k_2/2-M+k_2/2},2^{3k_2/2+k_1/2-M+k_2/2}\}\big)\e_1^2\\
    \lesssim & (1+2^{M/2+k})2^{\e k_--2k_++\gamma k-2k_{2,+}+\gamma k_2}\min\{2^{-3M/2-7k/2},2^{-5k/2+2k_2}\}\e_1^2,
\end{align*}
\begin{align*}
    &\|\eqref{jm20.1.2}\|_{L^2}+\|\eqref{jm20.2.2}\|_{L^2}\\
    \lesssim & 2^{M}\|\F^{-1}\frac{\partial_{\eta_m}\phi(\xi,\eta)\partial_{\xi_l}\phi(\xi,\eta)}{i|\nabla_\eta\phi(\xi,\eta)|^2\phi(\xi,\eta)}\partial_{\eta_m}q(\eta,\xi-\eta)\Tilde{\psi}_k(\xi)\Tilde{\psi}_{k_1}(\xi-\eta)\Tilde{\psi}_{k_2}(\eta)\|_{L^1}\times\\
    &\times\sup_{t\in[2^{M-1},2^M]}\big(\min\{\|e^{ic_mt\la}\partial_t{f}_{m,k_1}\|_{L^6_x}\|e^{ic_nt\la}{f}_{n,k_2}\|_{L^3_x},2^{3k_2/2}\|e^{ic_mt\la}\partial_tf_{m,k_1}\|_{L^2_x}\|e^{ic_nt\la}f_{n,k_2}\|_{L^2_x}\}\\
    &\qquad +\min\{\|e^{ic_mt\la}{f}_{m,k_1}\|_{L^6_x}\|e^{ic_nt\la}\partial_t{f}_{n,k_2}\|_{L^3_x},2^{3k_2/2}\|e^{ic_mt\la}f_{m,k_1}\|_{L^2_x}\|e^{ic_nt\la}\partial_tf_{n,k_2}\|_{L^2_x}\}\big)\\
    \lesssim & 2^{M-2k-k_2}\sup_{t\in[2^{M-1},2^M]}\big(\min\{\|e^{ic_mt\la}\partial_t{f}_{m,k_1}\|_{L^6_x}2^{k_2/2}\|e^{ic_nt\la}{f}_{n,k_2}\|_{L^2_x},2^{3k_2/2}\|\partial_t f_{m,k_1}\|_{L^2_x}\|f_{n,k_2}\|_{L^2_x}\}\\
    &\qquad + \min\{\|e^{ic_mt\la}{f}_{m,k_1}\|_{L^6_x}2^{k_2/2}\|e^{ic_nt\la}\partial_t{f}_{n,k_2}\|_{L^2_x},2^{3k_2/2}\|f_{m,k_1}\|_{L^2_x}\|\partial_tf_{n,k_2}\|_{L^2_x}\}\big)\\
    \lesssim & 2^{M-2k-k_2-2k_{1,+}+\gamma k_1-2k_{2,+}+\gamma k_2}\big(\min\{2^{-2M-k_1/2+k_2/2+k_2/2},2^{3k_2/2-M+k_1/2+k_2/2}\}\\
    &\qquad + \min\{2^{-M-k_1/2+k_2/2-M+k_2/2},2^{3k_2/2+k_1/2-M+k_2/2}\}\big)\e_1^2\\
    \lesssim & 2^{-2k_++\gamma k-2k_{2,+}+\gamma k_2} \min\{2^{-M-5k/2},2^{-3k/2+k_2}\}\e_1^2,
\end{align*}
\begin{align*}
    &\|\eqref{jm20.1.3-i}+\eqref{jm20.2.4-i}\|_{L^2}+\|\eqref{jm20.1.3}+\eqref{jm20.2.4}\|_{L^2}\\
    \lesssim & 2^{M}\big(\|\F^{-1}\frac{\partial_{\eta_m}\phi(\xi,\eta)\partial_{\xi_l}\phi(\xi,\eta)}{i|\nabla_\eta\phi(\xi,\eta)|^2\phi(\xi,\eta)} \Tilde{\psi}_k(\xi)\Tilde{\psi}_{k_1}(\xi-\eta)\Tilde{\psi}_{k_2}(\eta)\|_{L^1}+\\
    &+\|\F^{-1}\frac{\partial_{\eta_m}\phi(\xi,\eta)\partial_{\xi_l}\phi(\xi,\eta)}{i|\nabla_\eta\phi(\xi,\eta)|^2\phi(\xi,\eta)} q(\xi-\eta,\eta)\Tilde{\psi}_k(\xi)\Tilde{\psi}_{k_1}(\xi-\eta)\Tilde{\psi}_{k_2}(\eta)\|_{L^1}\big)\times\\
    &\times\sup_{t\in[2^{M-1},2^M]}\big(\min\{\|e^{ic_mt\la}\F^{-1}\nabla_\xi\partial_t\hat{f}_{m,k_1}\|_{L^3_x}\|e^{ic_nt\la}{f}_{n,k_2}\|_{L^6_x},\\
    &\qquad\qquad\qquad\qquad\qquad 2^{3k_2/2}\|e^{ic_mt\la}\F^{-1}\nabla_\xi\partial_t\hat{f}_{m,k_1}\|_{L^2_x}\|e^{ic_nt\la}f_{n,k_2}\|_{L^2_x}\}+\\
    &\qquad\qquad\qquad\qquad +\min\{\|e^{ic_mt\la}{f}_{m,k_1}\|_{L^6_x}\|e^{ic_nt\la}\F^{-1}\nabla_\xi\partial_t\hat{f}_{n,k_2}\|_{L^3_x},\\
    &\qquad\qquad\qquad\qquad\qquad 2^{3k_2/2}\|e^{ic_mt\la}f_{m,k_1}\|_{L^2_x}\|e^{ic_nt\la}\F^{-1}\nabla_\xi\partial_t\hat{f}_{n,k_2}\|_{L^2_x}\}\big)\\
    \lesssim & 2^{M-2k}\sup_{t\in[2^{M-1},2^M]}\big(\min\{\|e^{ic_mt\la}\F^{-1}\nabla_\xi\partial_t\hat{f}_{m,k_1}\|_{L^3_x}\|e^{ic_nt\la}{f}_{n,k_2}\|_{L^6_x},2^{3k_2/2}\|\nabla_\xi\partial_t\hat{f}_{m,k_1}\|_{L^2_\xi}\|f_{n,k_2}\|_{L^2_x}\}+\\
    &\qquad\qquad\qquad\qquad+\min\{\|e^{ic_mt\la}{f}_{m,k_1}\|_{L^6_x}\|e^{ic_nt\la}\F^{-1}\nabla_\xi\partial_t\hat{f}_{n,k_2}\|_{L^3_x},2^{3k_2/2}\|f_{m,k_1}\|_{L^2_x}\|\nabla_\xi\partial_t\hat{f}_{n,k_2}\|_{L^2_\xi}\}\big)\\
    \lesssim & 2^{M-2k-2k_{1,+}+\gamma k_1-2k_{2,+}+\gamma k_2}\big(\min\{2^{-M-\gamma k_1-M-k_2/2},(1+2^{M/2+k})2^{3k_2/2-M-k_1/2-\gamma k_1+k_2/2}\}+\\
    &\qquad + \min\{2^{-M-k_1/2-M-\gamma k_2},(1+2^{M/2+k_2})2^{3k_2/2+k_1/2-M-k_2/2}\}\big)\e_1^2\\
    \lesssim & (1+2^{M/2+k})2^{-2k_++\gamma k-2k_{2,+}+\gamma k_2} \min\{2^{-3M/2-3k-\gamma k_{-}-k_2/2},2^{-3k/2-\gamma k_-+k_2}\}\e_1^2,
\end{align*}
and
\begin{align*}
    &\|\eqref{jm20.1.4-i}+\eqref{jm20.2.3-i}\|_{L^2}+\|\eqref{jm20.1.4}+\eqref{jm20.2.3}\|_{L^2}\\
    \lesssim & 2^{M}\big(\|\F^{-1}\frac{\partial_{\eta_m}\phi(\xi,\eta)\partial_{\xi_l}\phi(\xi,\eta)}{i|\nabla_\eta\phi(\xi,\eta)|^2\phi(\xi,\eta)} \Tilde{\psi}_k(\xi)\Tilde{\psi}_{k_1}(\xi-\eta)\Tilde{\psi}_{k_2}(\eta)\|_{L^1}+\\
    &+\|\F^{-1}\frac{\partial_{\eta_m}\phi(\xi,\eta)\partial_{\xi_l}\phi(\xi,\eta)}{i|\nabla_\eta\phi(\xi,\eta)|^2\phi(\xi,\eta)} q(\xi-\eta,\eta)\Tilde{\psi}_k(\xi)\Tilde{\psi}_{k_1}(\xi-\eta)\Tilde{\psi}_{k_2}(\eta)\|_{L^1}\big)\times\\
    &\times\sup_{t\in[2^{M-1},2^M]}\big(\min\{\|e^{ic_mt\la}\partial_t{f}_{m,k_1}\|_{L^6_x}\|e^{ic_nt\la}\F^{-1}\nabla_\xi\hat{f}_{n,k_2}\|_{L^3_x},\\
    &\qquad\qquad\qquad\qquad\qquad 2^{3k_2/2}\|e^{ic_mt\la}\partial_tf_{m,k_1}\|_{L^2_x}\|e^{ic_nt\la}\F^{-1}\nabla_\xi\hat{f}_{n,k_2}\|_{L^2_x}\}+\\
    &\qquad\qquad\qquad\qquad +\min\{\|e^{ic_mt\la}\F^{-1}\nabla_\xi\hat{f}_{m,k_1}\|_{L^3_x}\|e^{ic_nt\la}\partial_t{f}_{n,k_2}\|_{L^6_x},\\
    &\qquad\qquad\qquad\qquad\qquad 2^{3k_2/2}\|e^{ic_mt\la}\F^{-1}\nabla_\xi\hat{f}_{m,k_1}\|_{L^2_x}\|e^{ic_nt\la}\partial_tf_{n,k_2}\|_{L^2_x}\}\big)\\
    \lesssim & 2^{M-2k}\sup_{t\in[2^{M-1},2^M]}\big(\min\{\|e^{ic_mt\la}\partial_t{f}_{m,k_1}\|_{L^6_x}2^{k_2/2}\|e^{ic_nt\la}\F^{-1}\nabla_\xi\hat{f}_{n,k_2}\|_{L^2_x},2^{3k_2/2}\|\partial_t f_{m,k_1}\|_{L^2_x}\|\nabla_\xi\hat{f}_{n,k_2}\|_{L^2_\xi}\}+\\
    &\qquad\qquad\qquad\qquad + \min\{\|e^{ic_mt\la}\F^{-1}\nabla_\xi\hat{f}_{m,k_1}\|_{L^3_x}\|e^{ic_nt\la}\partial_t{f}_{n,k_2}\|_{L^6_x},2^{3k_2/2}\|\nabla_\xi\hat{f}_{m,k_1}\|_{L^2_\xi}\|\partial_tf_{n,k_2}\|_{L^2_x}\}\big)\\
    \lesssim & 2^{M-2k-2k_{1,+}+\gamma k_1-2k_{2,+}+\gamma k_2}\big(\min\{2^{-2M-k_1/2+k_2/2-k_2/2},2^{3k_2/2-M+k_1/2-k_2/2}\}+\\
    &\qquad + \min\{2^{-M/2+\gamma M/8-k_1+\gamma k_1/4-2M-k_2/2},2^{3k_2/2-k_1/2-M+k_2/2}\}\big)\e_1^2\\
    \lesssim & 2^{-2k_++\gamma k-2k_{2,+}+\gamma k_2}\big(\min\{2^{-M-5k/2},2^{-3k/2+k_2}\}+ \min\{2^{-3M/2+\gamma M/8-3k+\gamma k/4-k_2/2},2^{-5k/2+2k_2}\}\big)\e_1^2\\
    \lesssim & (1+2^{M/2+k})2^{-2k_++\gamma k-2k_{2,+}+\gamma k_2} \min\{2^{-3M/2-3k-k_2/2},2^{-3k/2+k_2}\}\e_1^2.
\end{align*}
Thus, we found a new bound 
\begin{equation}
\label{jm2-2bound2}
    \begin{aligned}
    \|F_1+F_2\|_{L^2}+\|F^q_1+F^q_2\|_{L^2}\lesssim (1+2^{M/2+k})2^{-2k_++\gamma k-\gamma k_--2k_{2,+}+\gamma k_2} \min\{2^{-3M/2-+3k-k_2/2},2^{-3k/2+k_2}\}\e_1^2.
    \end{aligned}
\end{equation}
Next, we want to estimate $\|\nabla_\xi(F_1+F_2)\|_{L^2}+\|\nabla_\xi(F^q_1+F^q_2)\|_{L^2}$. Observe
\begin{align}
    \partial_{\xi_m} F_1
    = &\int_{t_1}^{t_2}\int_{\R^3}e^{it\phi(\xi,\eta)}\frac{t\partial_{\xi_l}\phi(\xi,\eta)}{\phi(\xi,\eta)}\partial_{\xi_m}\partial_t\hat{f}_{m,k_1}(t,\xi-\eta)\hat{f}_{n,k_2}(t,\eta)\psi_k(\xi)d\eta dt\label{jm2-2.1-i}\\
    & + \int_{t_1}^{t_2}\int_{\R^3}e^{it\phi(\xi,\eta)}\partial_{\xi_m}\frac{t\partial_{\xi_l}\phi(\xi,\eta)}{\phi(\xi,\eta)}\partial_t\hat{f}_{m,k_1}(t,\xi-\eta)\hat{f}_{n,k_2}(t,\eta)\psi_k(\xi)d\eta dt\label{jm2-2.11.2-i}\\
    &+ \int_{t_1}^{t_2}\int_{\R^3}e^{it\phi(\xi,\eta)}\frac{t\partial_{\xi_l}\phi(\xi,\eta)}{\phi(\xi,\eta)}\partial_t\hat{f}_{m,k_1}(t,\xi-\eta)\hat{f}_{n,k_2}(t,\eta)\partial_{\xi_m}\psi_k(\xi)d\eta dt\label{jm2-2.12-i}\\
    & + \int_{t_1}^{t_2}\int_{\R^3}e^{it\phi(\xi,\eta)}\frac{it^2\partial_{\xi_m}\phi(\xi,\eta)\partial_{\xi_l}\phi(\xi,\eta)}{\phi(\xi,\eta)}\partial_t\hat{f}_{m,k_1}(t,\xi-\eta)\hat{f}_{n,k_2}(t,\eta)\psi_k(\xi)d\eta dt\label{jm2-2.2-i},
\end{align}
\begin{align}
    \partial_{\xi_m} F^q_1
    = &\int_{t_1}^{t_2}\int_{\R^3}e^{it\phi(\xi,\eta)}\frac{t\partial_{\xi_l}\phi(\xi,\eta)}{\phi(\xi,\eta)}q(\eta,\xi-\eta)\partial_{\xi_m}\partial_t\hat{f}_{m,k_1}(t,\xi-\eta)\hat{f}_{n,k_2}(t,\eta)\psi_k(\xi)d\eta dt\label{jm2-2.1}\\
    & + \int_{t_1}^{t_2}\int_{\R^3}e^{it\phi(\xi,\eta)}\partial_{\xi_m}\big(\frac{t\partial_{\xi_l}\phi(\xi,\eta)}{\phi(\xi,\eta)}q(\eta,\xi-\eta)\big)\partial_t\hat{f}_{m,k_1}(t,\xi-\eta)\hat{f}_{n,k_2}(t,\eta)\psi_k(\xi)d\eta dt\label{jm2-2.11.2}\\
    &+ \int_{t_1}^{t_2}\int_{\R^3}e^{it\phi(\xi,\eta)}\frac{t\partial_{\xi_l}\phi(\xi,\eta)}{\phi(\xi,\eta)}q(\eta,\xi-\eta)\partial_t\hat{f}_{m,k_1}(t,\xi-\eta)\hat{f}_{n,k_2}(t,\eta)\partial_{\xi_m}\psi_k(\xi)d\eta dt\label{jm2-2.12}\\
    & + \int_{t_1}^{t_2}\int_{\R^3}e^{it\phi(\xi,\eta)}\frac{it^2\partial_{\xi_m}\phi(\xi,\eta)\partial_{\xi_l}\phi(\xi,\eta)}{\phi(\xi,\eta)}q(\eta,\xi-\eta)\partial_t\hat{f}_{m,k_1}(t,\xi-\eta)\hat{f}_{n,k_2}(t,\eta)\psi_k(\xi)d\eta dt\label{jm2-2.2},
\end{align}
\begin{align}
    \partial_{\xi_m} F_2
    = &\int_{t_1}^{t_2}\int_{\R^3}e^{it\phi(\xi,\eta)}\frac{t\partial_{\xi_l}\phi(\xi,\eta)}{\phi(\xi,\eta)}\partial_{\xi_m}\hat{f}_{m,k_1}(t,\xi-\eta)\partial_t\hat{f}_{n,k_2}(t,\eta)\psi_k(\xi)d\eta dt\label{jm2-2.1-2-i}\\
    & + \int_{t_1}^{t_2}\int_{\R^3}e^{it\phi(\xi,\eta)}\partial_{\xi_m}\frac{t\partial_{\xi_l}\phi(\xi,\eta)}{\phi(\xi,\eta)}\hat{f}_{m,k_1}(t,\xi-\eta)\partial_t\hat{f}_{n,k_2}(t,\eta)\psi_k(\xi)d\eta dt\label{jm2-2.11.2-2-i}\\
    &+ \int_{t_1}^{t_2}\int_{\R^3}e^{it\phi(\xi,\eta)}\frac{t\partial_{\xi_l}\phi(\xi,\eta)}{\phi(\xi,\eta)}\hat{f}_{m,k_1}(t,\xi-\eta)\partial_t\hat{f}_{n,k_2}(t,\eta)\partial_{\xi_m}\psi_k(\xi)d\eta dt\label{jm2-2.12-2-i}\\
    & + \int_{t_1}^{t_2}\int_{\R^3}e^{it\phi(\xi,\eta)}\frac{it^2\partial_{\xi_m}\phi(\xi,\eta)\partial_{\xi_l}\phi(\xi,\eta)}{\phi(\xi,\eta)}\hat{f}_{m,k_1}(t,\xi-\eta)\partial_t\hat{f}_{n,k_2}(t,\eta)\psi_k(\xi)d\eta dt\label{jm2-2.2-2-i},
\end{align}
and
\begin{align}
    \partial_{\xi_m} F^q_2
    = &\int_{t_1}^{t_2}\int_{\R^3}e^{it\phi(\xi,\eta)}\frac{t\partial_{\xi_l}\phi(\xi,\eta)}{\phi(\xi,\eta)}q(\eta,\xi-\eta)\partial_{\xi_m}\hat{f}_{m,k_1}(t,\xi-\eta)\partial_t\hat{f}_{n,k_2}(t,\eta)\psi_k(\xi)d\eta dt\label{jm2-2.1-2}\\
    & + \int_{t_1}^{t_2}\int_{\R^3}e^{it\phi(\xi,\eta)}\partial_{\xi_m}\big(\frac{t\partial_{\xi_l}\phi(\xi,\eta)}{\phi(\xi,\eta)}q(\eta,\xi-\eta)\big)\hat{f}_{m,k_1}(t,\xi-\eta)\partial_t\hat{f}_{n,k_2}(t,\eta)\psi_k(\xi)d\eta dt\label{jm2-2.11.2-2}\\
    &+ \int_{t_1}^{t_2}\int_{\R^3}e^{it\phi(\xi,\eta)}\frac{t\partial_{\xi_l}\phi(\xi,\eta)}{\phi(\xi,\eta)}q(\eta,\xi-\eta)\hat{f}_{m,k_1}(t,\xi-\eta)\partial_t\hat{f}_{n,k_2}(t,\eta)\partial_{\xi_m}\psi_k(\xi)d\eta dt\label{jm2-2.12-2}\\
    & + \int_{t_1}^{t_2}\int_{\R^3}e^{it\phi(\xi,\eta)}\frac{it^2\partial_{\xi_m}\phi(\xi,\eta)\partial_{\xi_l}\phi(\xi,\eta)}{\phi(\xi,\eta)}q(\eta,\xi-\eta)\hat{f}_{m,k_1}(t,\xi-\eta)\partial_t\hat{f}_{n,k_2}(t,\eta)\psi_k(\xi)d\eta dt\label{jm2-2.2-2}.
\end{align}
For the terms \eqref{jm2-2.2-i}, \eqref{jm2-2.2}, \eqref{jm2-2.2-2-i} and \eqref{jm2-2.2-2}, we perform integration by parts in the space variable using the identity $e^{it\phi(\xi,\eta)}=\sum_n\frac{\partial_{\eta_n}\phi(\xi,\eta)}{it|\nabla_\eta\phi(\xi,\eta)|^2}\partial_{\eta_n}e^{it\phi(\xi,\eta)}$. We get
\begin{align}
    &\eqref{jm2-2.2-i}\nonumber\\
    =&\int_{t_1}^{t_2}\int_{\R^3}e^{it\phi(\xi,\eta)}\partial_{\eta_n}\frac{t\partial_{\eta_n}\phi(\xi,\eta)\partial_{\xi_m}\phi(\xi,\eta)\partial_{\xi_l}\phi(\xi,\eta)}{|\nabla_\eta\phi(\xi,\eta)|^2\phi(\xi,\eta)}\partial_t\hat{f}_{m,k_1}(t,\xi-\eta)\hat{f}_{n,k_2}(t,\eta)\psi_k(\xi)d\eta dt\label{jm2-2.3-i}\\
    &+\int_{t_1}^{t_2}\int_{\R^3}e^{it\phi(\xi,\eta)}\frac{t\partial_{\eta_n}\phi(\xi,\eta)\partial_{\xi_m}\phi(\xi,\eta)\partial_{\xi_l}\phi(\xi,\eta)}{|\nabla_\eta\phi(\xi,\eta)|^2\phi(\xi,\eta)}\partial_{\eta_n}\partial_t\hat{f}_{m,k_1}(t,\xi-\eta)\hat{f}_{n,k_2}(t,\eta)\psi_k(\xi)d\eta dt\label{jm2-2.4-i}\\
    &+\int_{t_1}^{t_2}\int_{\R^3}e^{it\phi(\xi,\eta)}\frac{t\partial_{\eta_n}\phi(\xi,\eta)\partial_{\xi_m}\phi(\xi,\eta)\partial_{\xi_l}\phi(\xi,\eta)}{|\nabla_\eta\phi(\xi,\eta)|^2\phi(\xi,\eta)}\partial_t\hat{f}_{m,k_1}(t,\xi-\eta)\partial_{\eta_n}\hat{f}_{n,k_2}(t,\eta)\psi_k(\xi)d\eta dt\label{jm2-2.5-i},
\end{align}
\begin{align}
    &\eqref{jm2-2.2}\nonumber\\
    =&\int_{t_1}^{t_2}\int_{\R^3}e^{it\phi(\xi,\eta)}\partial_{\eta_n}\big(\frac{t\partial_{\eta_n}\phi(\xi,\eta)\partial_{\xi_m}\phi(\xi,\eta)\partial_{\xi_l}\phi(\xi,\eta)}{|\nabla_\eta\phi(\xi,\eta)|^2\phi(\xi,\eta)}q(\eta,\xi-\eta)\big)\partial_t\hat{f}_{m,k_1}(t,\xi-\eta)\hat{f}_{n,k_2}(t,\eta)\psi_k(\xi)d\eta dt\label{jm2-2.3}\\
    &+\int_{t_1}^{t_2}\int_{\R^3}e^{it\phi(\xi,\eta)}\frac{t\partial_{\eta_n}\phi(\xi,\eta)\partial_{\xi_m}\phi(\xi,\eta)\partial_{\xi_l}\phi(\xi,\eta)}{|\nabla_\eta\phi(\xi,\eta)|^2\phi(\xi,\eta)}q(\eta,\xi-\eta)\partial_{\eta_n}\partial_t\hat{f}_{m,k_1}(t,\xi-\eta)\hat{f}_{n,k_2}(t,\eta)\psi_k(\xi)d\eta dt\label{jm2-2.4}\\
    &+\int_{t_1}^{t_2}\int_{\R^3}e^{it\phi(\xi,\eta)}\frac{t\partial_{\eta_n}\phi(\xi,\eta)\partial_{\xi_m}\phi(\xi,\eta)\partial_{\xi_l}\phi(\xi,\eta)}{|\nabla_\eta\phi(\xi,\eta)|^2\phi(\xi,\eta)}q(\eta,\xi-\eta)\partial_t\hat{f}_{m,k_1}(t,\xi-\eta)\partial_{\eta_n}\hat{f}_{n,k_2}(t,\eta)\psi_k(\xi)d\eta dt\label{jm2-2.5},
\end{align}
\begin{align}
    &\eqref{jm2-2.2-2-i}\nonumber\\
    =&\int_{t_1}^{t_2}\int_{\R^3}e^{it\phi(\xi,\eta)}\partial_{\eta_n}\frac{t\partial_{\eta_n}\phi(\xi,\eta)\partial_{\xi_m}\phi(\xi,\eta)\partial_{\xi_l}\phi(\xi,\eta)}{|\nabla_\eta\phi(\xi,\eta)|^2\phi(\xi,\eta)}\hat{f}_{m,k_1}(t,\xi-\eta)\partial_t\hat{f}_{n,k_2}(t,\eta)\psi_k(\xi)d\eta dt\label{jm2-2.3-2-i}\\
    &+\int_{t_1}^{t_2}\int_{\R^3}e^{it\phi(\xi,\eta)}\frac{t\partial_{\eta_n}\phi(\xi,\eta)\partial_{\xi_m}\phi(\xi,\eta)\partial_{\xi_l}\phi(\xi,\eta)}{|\nabla_\eta\phi(\xi,\eta)|^2\phi(\xi,\eta)}\partial_{\eta_n}\hat{f}_{m,k_1}(t,\xi-\eta)\partial_t\hat{f}_{n,k_2}(t,\eta)\psi_k(\xi)d\eta dt\label{jm2-2.4-2-i}\\
    &+\int_{t_1}^{t_2}\int_{\R^3}e^{it\phi(\xi,\eta)}\frac{t\partial_{\eta_n}\phi(\xi,\eta)\partial_{\xi_m}\phi(\xi,\eta)\partial_{\xi_l}\phi(\xi,\eta)}{|\nabla_\eta\phi(\xi,\eta)|^2\phi(\xi,\eta)}\hat{f}_{m,k_1}(t,\xi-\eta)\partial_{\eta_n}\partial_t\hat{f}_{n,k_2}(t,\eta)\psi_k(\xi)d\eta dt\label{jm2-2.5-2-i},
\end{align}
and
\begin{align}
    &\eqref{jm2-2.2-2}\nonumber\\
    =&\int_{t_1}^{t_2}\int_{\R^3}e^{it\phi(\xi,\eta)}\partial_{\eta_n}\big(\frac{t\partial_{\eta_n}\phi(\xi,\eta)\partial_{\xi_m}\phi(\xi,\eta)\partial_{\xi_l}\phi(\xi,\eta)}{|\nabla_\eta\phi(\xi,\eta)|^2\phi(\xi,\eta)}q(\eta,\xi-\eta)\big)\hat{f}_{m,k_1}(t,\xi-\eta)\partial_t\hat{f}_{n,k_2}(t,\eta)\psi_k(\xi)d\eta dt\label{jm2-2.3-2}\\
    &+\int_{t_1}^{t_2}\int_{\R^3}e^{it\phi(\xi,\eta)}\frac{t\partial_{\eta_n}\phi(\xi,\eta)\partial_{\xi_m}\phi(\xi,\eta)\partial_{\xi_l}\phi(\xi,\eta)}{|\nabla_\eta\phi(\xi,\eta)|^2\phi(\xi,\eta)}q(\eta,\xi-\eta)\partial_{\eta_n}\hat{f}_{m,k_1}(t,\xi-\eta)\partial_t\hat{f}_{n,k_2}(t,\eta)\psi_k(\xi)d\eta dt\label{jm2-2.4-2}\\
    &+\int_{t_1}^{t_2}\int_{\R^3}e^{it\phi(\xi,\eta)}\frac{t\partial_{\eta_n}\phi(\xi,\eta)\partial_{\xi_m}\phi(\xi,\eta)\partial_{\xi_l}\phi(\xi,\eta)}{|\nabla_\eta\phi(\xi,\eta)|^2\phi(\xi,\eta)}q(\eta,\xi-\eta)\hat{f}_{m,k_1}(t,\xi-\eta)\partial_{\eta_n}\partial_t\hat{f}_{n,k_2}(t,\eta)\psi_k(\xi)d\eta dt\label{jm2-2.5-2}.
\end{align}
Employing bilinear estimates $L^2\times L^\infty\rightarrow L^2$ and $L^3\times L^6\rightarrow L^2$, along with Lemma \ref{dualitycomp}, Lemma \ref{chi,eta}, Lemma \ref{timel2}, \eqref{q}, \eqref{l2}, \eqref{l2first}, and \eqref{linfinity}, we obtain
\begin{align*}
    &\|\eqref{jm2-2.11.2-i}+\eqref{jm2-2.12-i}+\eqref{jm2-2.3-i}+\eqref{jm2-2.11.2-2-i}+\eqref{jm2-2.12-2-i}
    +\eqref{jm2-2.3-2-i}\|_{L^2}+\\
    &+\|\eqref{jm2-2.11.2}+\eqref{jm2-2.12}+\eqref{jm2-2.3}+\eqref{jm2-2.11.2-2}+\eqref{jm2-2.12-2}
    +\eqref{jm2-2.3-2}\|_{L^2}\\
    \lesssim &2^{2M}\big(\|\F^{-1}\partial_{\xi_m}\frac{\partial_{\xi_l}\phi(\xi,\eta)}{\phi(\xi,\eta)}\Tilde{\psi}_k(\xi)\Tilde{\psi}_{k_1}(\xi-\eta)\Tilde{\psi}_{k_2}(\eta)\|_{L^1}+\\
    &\qquad +\|\F^{-1}\partial_{\xi_m}\big(\frac{\partial_{\xi_l}\phi(\xi,\eta)}{\phi(\xi,\eta)}q(\eta,\xi-\eta)\big)\Tilde{\psi}_k(\xi)\Tilde{\psi}_{k_1}(\xi-\eta)\Tilde{\psi}_{k_2}(\eta)\|_{L^1}+\\
    &\qquad + \|\nabla_\xi\psi_k\|_{L^\infty}\big(\|\F^{-1}\frac{\partial_{\xi_l}\phi(\xi,\eta)}{\phi(\xi,\eta)}\Tilde{\psi}_k(\xi)\Tilde{\psi}_{k_1}(\xi-\eta)\Tilde{\psi}_{k_2}(\eta)\|_{L^1}+\\
    &\qquad + \|\F^{-1}\frac{\partial_{\xi_l}\phi(\xi,\eta)}{\phi(\xi,\eta)}q(\eta,\xi-\eta)\Tilde{\psi}_k(\xi)\Tilde{\psi}_{k_1}(\xi-\eta)\Tilde{\psi}_{k_2}(\eta)\|_{L^1}\big)+\\
    &\qquad +\|\F^{-1}\partial_{\eta_n}\frac{\partial_{\eta_n}\phi(\xi,\eta)\partial_{\xi_m}\phi(\xi,\eta)\partial_{\xi_l}\phi(\xi,\eta)}{|\nabla_\eta\phi(\xi,\eta)|^2\phi(\xi,\eta)}\Tilde{\psi}_k(\xi)\Tilde{\psi}_{k_1}(\xi-\eta)\Tilde{\psi}_{k_2}(\eta)\|_{L^1}+\\
    &\qquad +\|\F^{-1}\partial_{\eta_n}\big(\frac{\partial_{\eta_n}\phi(\xi,\eta)\partial_{\xi_m}\phi(\xi,\eta)\partial_{\xi_l}\phi(\xi,\eta)}{|\nabla_\eta\phi(\xi,\eta)|^2\phi(\xi,\eta)}q(\eta,\xi-\eta)\big)\Tilde{\psi}_k(\xi)\Tilde{\psi}_{k_1}(\xi-\eta)\Tilde{\psi}_{k_2}(\eta)\|_{L^1}\big)\times\\
    &\times\sup_{t\in[2^{M-1},2^M]}\big(\min\{\|e^{ic_mt\la}\partial_t{f}_{m,k_1}\|_{L^6_x}\|e^{ic_nt\la}{f}_{n,k_2}\|_{L^3_x},2^{3k_2/2}\|e^{ic_mt\la}\partial_tf_{m,k_1}\|_{L^2_x}\|e^{ic_nt\la}{f}_{n,k_2}\|_{L^2_x}\}+\\
    &\qquad+\min\{\|e^{ic_mt\la}f_{m,k_1}\|_{L^\infty_x}\|e^{ic_nt\la}\partial_t{f}_{n,k_2}\|_{L^2_x},2^{3k_2/2}\|e^{ic_mt\la}f_{m,k_1}\|_{L^2_x}\|e^{ic_nt\la}\partial_t{f}_{n,k_2}\|_{L^2_x}\}\big)\\
    \lesssim & 2^{2M-k-k_2}\sup_{t\in[2^{M-1},2^M]}\big(\min\{\|e^{ic_mt\la}\partial_t{f}_{m,k_1}\|_{L^6_x}2^{k_2/2}\|e^{ic_nt\la}{f}_{n,k_2}\|_{L^2_x},2^{3k_2/2}\|\partial_tf_{m,k_1}\|_{L^2_x}\|{f}_{n,k_2}\|_{L^2_x}\}\\
    &\qquad+\min\{\|e^{ic_mt\la}f_{m,k_1}\|_{L^\infty_x}\|\partial_t{f}_{n,k_2}\|_{L^2_x},2^{3k_2/2}\|f_{m,k_1}\|_{L^2_x}\|\partial_t{f}_{n,k_2}\|_{L^2_x}\}\big)\\
    \lesssim & 2^{2M-k-k_2-2k_{2,+}+\gamma k_2-2k_{1,+}+\gamma k_1}\big(\min\{2^{-2M-k_1/2+k_2/2+k_2/2}, 2^{-M+3k_2/2+k_1/2+k_2/2}\}\\ 
    &\qquad+\min\{2^{-5M/2+\gamma M/8-k_1+k_2/2+\gamma k_1/4},2^{-M+3k_2/2+k_1/2+k_2/2}\}\big)\e_1^2\\
    \lesssim & 2^{-k-2k_{2,+}+\gamma k_2-2k_{1,+}+\gamma k_1}\big(\min\{2^{-k_1/2}, 2^{M+k_1/2+k_2}\}\\
    &\qquad+\min\{2^{-M/2+\gamma M/8-k_1-k_2/2+\gamma k_1/4},2^{M+k_1/2+k_2}\}\big)\e_1^2\\
    \lesssim & (1+2^{M/2+k}) 2^{-2k_{2,+}+\gamma k_2-2k_{+}+\gamma k}\min\{2^{-M/2-2k-k_2/2}, 2^{M-k/2+k_2}\}\e_1^2,
\end{align*}
\begin{align*}
    &\|\eqref{jm2-2.1-i}+\eqref{jm2-2.4-i}+\eqref{jm2-2.1-2-i}+\eqref{jm2-2.4-2-i}\|_{L^2}+\|\eqref{jm2-2.1}+\eqref{jm2-2.4}+\eqref{jm2-2.1-2}+\eqref{jm2-2.4-2}\|_{L^2}\\
    \lesssim & 2^{2M}\big(\|\F^{-1}\frac{\partial_{\xi_l}\phi(\xi,\eta)}{\phi(\xi,\eta)}\Tilde{\psi}_k(\xi)\Tilde{\psi}_{k_1}(\xi-\eta)\Tilde{\psi}_{k_2}(\eta)\|_{L^1}+\|\F^{-1}\frac{\partial_{\xi_l}\phi(\xi,\eta)}{\phi(\xi,\eta)}q(\eta,\xi-\eta)\Tilde{\psi}_k(\xi)\Tilde{\psi}_{k_1}(\xi-\eta)\Tilde{\psi}_{k_2}(\eta)\|_{L^1}+\\
    &+\|\F^{-1}\frac{\partial_{\eta_n}\phi(\xi,\eta)\partial_{\xi_m}\phi(\xi,\eta)\partial_{\xi_l}\phi(\xi,\eta)}{|\nabla_\eta\phi(\xi,\eta)|^2\phi(\xi,\eta)}\Tilde{\psi}_k(\xi)\Tilde{\psi}_{k_1}(\xi-\eta)\Tilde{\psi}_{k_2}(\eta)\|_{L^1}+\\
    &+\|\F^{-1}\frac{\partial_{\eta_n}\phi(\xi,\eta)\partial_{\xi_m}\phi(\xi,\eta)\partial_{\xi_l}\phi(\xi,\eta)}{|\nabla_\eta\phi(\xi,\eta)|^2\phi(\xi,\eta)}q(\eta,\xi-\eta)\Tilde{\psi}_k(\xi)\Tilde{\psi}_{k_1}(\xi-\eta)\Tilde{\psi}_{k_2}(\eta)\|_{L^1}\big)\times\\
    &\times \sup_{t\in[2^{M-1},2^M]}\big(\min\{\|e^{ic_mt\la}\F^{-1}\nabla_\xi\partial_t\hat{f}_{m,k_1}\|_{L^3_x}\|e^{ic_nt\la} f_{n,k_2}\|_{L^6_x},\\
    &\qquad\qquad\qquad\qquad\qquad 2^{3k_2/2}\|e^{ic_mt\la}\F^{-1}\nabla_\xi\partial_t\hat{f}_{m,k_1}\|_{L^2_x}\|e^{ic_nt\la} f_{n,k_2}\|_{L^2_x}\}+\\
    &\qquad\qquad\qquad\qquad+\min\{\|e^{ic_mt\la}\F^{-1}\nabla_\xi\hat{f}_{m,k_1}\|_{L^3_x}\|e^{ic_nt\la}\partial_t f_{n,k_2}\|_{L^6_x},\\
    &\qquad\qquad\qquad\qquad\qquad 2^{3k_2/2}\|e^{ic_mt\la}\F^{-1}\nabla_\xi\hat{f}_{m,k_1}\|_{L^2_x}\|e^{ic_nt\la}\partial_t f_{n,k_2}\|_{L^2_x}\}\big)\\
    \lesssim & 2^{2M-k-2k_{1,+}-2k_{2,+}+\gamma k_2+\gamma k_1}\big(\min\{2^{-M-\gamma k_1-M-k_2/2},(1+2^{M/2+k})2^{3k_2/2-M-k_1/2-\gamma k_1+k_2/2}\}+\\
    &\qquad+\min\{2^{-M/2+\gamma M/8-k_1+\gamma k_1/4-2M-k_2/2},2^{3k_2/2-k_1/2-M+k_2/2}\}\big)\e_1^2\\
    \lesssim & 2^{-k-2k_{1,+}-2k_{2,+}+\gamma k_1+\gamma k_2}\big(\min\{2^{-\gamma k_1-k_2/2},(1+2^{M/2+k})2^{M-k_1/2-\gamma k_1+2k_2}\}+\\
    &\qquad +\min\{2^{-M/2+\gamma M/8-k_1+\gamma k_1/4-k_2/2},2^{M-k_1/2+2k_2}\}\big)\e_1^2\\
    \lesssim & (1+2^{M/2+k})2^{-2k_{+}-2k_{2,+}+\gamma k-\gamma k_-+\gamma k_2}\min\{2^{-M/2-2k-k_2/2},2^{M-3k/2+2k_2}\}\e_1^2,
\end{align*}
and
\begin{align*}
    &\|\eqref{jm2-2.5-i}+\eqref{jm2-2.5-2-i}\|_{L^2}+\|\eqref{jm2-2.5}+\eqref{jm2-2.5-2}\|_{L^2}\\
    \lesssim &2^{2M}\big(\|\F^{-1} \frac{\partial_{\eta_n}\phi(\xi,\eta)\partial_{\xi_m}\phi(\xi,\eta)\partial_{\xi_l}\phi(\xi,\eta)}{|\nabla_\eta\phi(\xi,\eta)|^2\phi(\xi,\eta)}\Tilde{\psi}_k(\xi)\Tilde{\psi}_{k_1}(\xi-\eta)\Tilde{\psi}_{k_2}(\eta)\|_{L^1}+\\
    &+\|\F^{-1} \frac{\partial_{\eta_n}\phi(\xi,\eta)\partial_{\xi_m}\phi(\xi,\eta)\partial_{\xi_l}\phi(\xi,\eta)}{|\nabla_\eta\phi(\xi,\eta)|^2\phi(\xi,\eta)}q(\eta,\xi-\eta)\Tilde{\psi}_k(\xi)\Tilde{\psi}_{k_1}(\xi-\eta)\Tilde{\psi}_{k_2}(\eta)\|_{L^1}\big)\times\\
    &\times\sup_{t\in[2^{M-1},2^M]}\big(\min\{\|e^{ic_mt\la}\partial_t f_{m,k_1}\|_{L^6_x}\|e^{ic_nt\la}\F^{-1}\nabla_\xi\hat{f}_{n,k_2}\|_{L^3_x},\\
    &\qquad\qquad\qquad\qquad\qquad 2^{3k_2/2}\|e^{ic_mt\la}\partial_t f_{m,k_1}\|_{L^2_x}\|e^{ic_nt\la}\F^{-1}\nabla_\xi\hat{f}_{n,k_2}\|_{L^2_x}\}+\\
    &\qquad\qquad\qquad\qquad+\min\{\|e^{ic_mt\la}f_{m,k_1}\|_{L^6_x}\|e^{ic_nt\la}\F^{-1}\nabla_\xi\partial_t\hat{f}_{n,k_2}\|_{L^3_x},\\
    &\qquad\qquad\qquad\qquad\qquad 2^{3k_2/2}\|e^{ic_mt\la}f_{m,k_1}\|_{L^2_x}\|e^{ic_nt\la}\F^{-1}\nabla_\xi\partial_t\hat{f}_{n,k_2}\|_{L^2_x}\}\big)\\
    \lesssim &2^{2M-k}\sup_{t\in[2^{M-1},2^M]}\big(\min\{\|e^{ic_mt\la}\partial_t f_{m,k_1}\|_{L^6_x}2^{k_2/2}\|e^{ic_nt\la}\F^{-1}\nabla_\xi\hat{f}_{n,k_2}\|_{L^2_x},2^{3k_2/2}\|\partial_t f_{m,k_1}\|_{L^2_x}\|\nabla_\xi\hat{f}_{n,k_2}\|_{L^2_\xi}\}+\\
    &\qquad\qquad\qquad\qquad+\min\{\|e^{ic_mt\la}f_{m,k_1}\|_{L^6_x}\|e^{ic_nt\la}\F^{-1}\nabla_\xi\partial_t\hat{f}_{n,k_2}\|_{L^3_x},2^{3k_2/2}\|f_{m,k_1}\|_{L^2_x}\|\nabla_\xi\partial_t\hat{f}_{n,k_2}\|_{L^2_\xi}\}\big)\\
    \lesssim & 2^{2M-k-2k_{1,+}-2k_{2,+}+\gamma k_2+\gamma k_1}\big(\min\{2^{-2M-k_1/2+k_2/2-k_2/2},2^{3k_2/2-k_2/2-M+k_1/2}\}+\\
    &\qquad\qquad\qquad\qquad +\min\{2^{-M-k_1/2-M-\gamma k_2},(1+2^{M/2+k_2})2^{-M+3k_2/2+k_1/2-k_2/2-\gamma k_2}\}\big)\e_1^2\\
    \lesssim & 2^{-k-2k_{1,+}-2k_{2,+}+\gamma k_2+\gamma k_1}\big(\min\{2^{-k_1/2},2^{M+k_1/2+k_2}\}
    +\min\{2^{-k_1/2-\gamma k_2},(1+2^{M/2+k_2})2^{M+k_1/2+k_2-\gamma k_2}\}\big)\e_1^2\\
    \lesssim & (1+2^{M/2+k})2^{-2k_++\gamma k-2k_{2,+}+\gamma k_2-\gamma k_{2,-}}\min\{2^{-M/2-5k/2},2^{M-k/2+k_2}\}\e_1^2.
\end{align*}
Thus, we showed
\begin{align*}
    &\|\nabla_\xi(F_1+F_2)\|_{L^2}+\|\nabla_\xi(F^q_1+F^q_2)\|_{L^2}\\
    \lesssim & (1+2^{M/2+k})2^{-2k_++\gamma k-2k_{2,+}+\gamma k_2-\gamma k_{2,-}}\min\{2^{-M/2-2k-k_2/2},2^{M-k/2+k_2}\}\e_1^2.
\end{align*}
Using Lemma \ref{dalpha} and \eqref{jm2-2bound1}, when $M+2k\leq 0$
\begin{align*}
    &\|D^\alpha_\xi(F_1+F_2)\|_{L^2}+\|D^\alpha_\xi(F^q_1+F^q_2)\|_{L^2}\\
    \leq & \|F_1+F_2\|_{L^2}^{1-\alpha}\|\nabla_\xi(F_1+F_2)\|_{L^2}^\alpha+\|F^q_1+F^q_2\|_{L^2}^{1-\alpha}\|\nabla_\xi(F^q_1+F^q_2)\|_{L^2}^\alpha\\
    \lesssim & 2^{-2k_++\gamma k-2k_{2,+}+\gamma k_2}\min\{(1+2^{\gamma M/8+\gamma k/4})2^{-M/2-3k/2-\alpha k/2-\alpha k_2/2},2^{M-k/2+2k_2-\alpha k_2}\}\e_1^2\\
    \lesssim & 2^{-2k_++\gamma k-2k_{2,+}+\gamma k_2-\alpha k_2}\min\{2^{-M/2-3k/2},2^{M-k/2+2k_2}\}\e_1^2.
\end{align*}
When $M+2k>0$, we use \eqref{jm2-2bound2} and get
\begin{align*}
    &\|D^\alpha_\xi(F_1+F_2)\|_{L^2}+\|D^\alpha_\xi(F^q_1+F^q_2)\|_{L^2}\\
    \lesssim & 2^{-2k_++\gamma k-\gamma k_-/2-2k_{2,+}+\gamma k_2-\gamma k_{2,-}/2} \min\{2^{-M+\alpha M-2k+\alpha k-k_2/2},2^{M/2+\alpha M-k/2+\alpha k+k_2}\}\e_1^2\\
    \lesssim & 2^{-2k_++\gamma k-\gamma k_-/2-2k_{2,+}+\gamma k_2-\gamma k_{2,-}/2+\alpha k} \min\{2^{-M/2-2k-k_2/2},2^{M-k/2+k_2}\}\e_1^2,
\end{align*}
since $\alpha<1/2$.
\end{proof}
As a result, Lemma \ref{jm2-1} implies
\begin{equation}
    \label{jm2-1result}
    \begin{aligned}
    &\sum_{1\leq M\leq \log T}\sum_{(k_1,k_2)\in\chi^1_k}\sup_{2^{M-1}\leq t_1<t_2\leq 2^M} 2^{-\gamma k+2k_++k/2+\alpha k}\times\\
    &\qquad\qquad\qquad\qquad\times\|\eqref{f1-i}+\eqref{f2-i}+\eqref{f3-i}+\eqref{f1}+\eqref{f2}+\eqref{f3}\|_{L^2}\\
    \lesssim &\sum_{1\leq M\leq \log T}\sum_{(k_1,k_2)\in\chi^1_k} 2^{-2k_{2,+}+\gamma k_2+\alpha k-\alpha k_2}\min\{2^{-M/2+\gamma M/8-k+\gamma k/4},2^{M+2k_2}\}\e_1^2\\
    \lesssim &\sum_{1\leq M\leq \log T}\sum_{k_2\leq k} 2^{-2k_{2,+}+\gamma k_2}\min\{2^{-M/2+\gamma M/8-k/2-k_2/2+\gamma k/4},2^{M+k/2+3k_2/2}\}\e_1^2\\
    \lesssim &\e_1^2,
\end{aligned}
\end{equation}
and from Lemma \ref{jm2-2}, we get
\begin{equation}
\label{jm2-2result}
    \begin{aligned}
    &\sum_{1\leq M\leq \log T}\sum_{(k_1,k_2)\in\chi^1_k}\sup_{2^{M-1}\leq t_1<t_2\leq 2^M} 2^{-\gamma k+2k_++k/2+\alpha k}\|\eqref{f4-i}+\eqref{f5-i}+\eqref{f4}+\eqref{f5}\|_{L^2}\\
    \lesssim &\sum_{1\leq M\leq \min\{-2k,\log T\}}\sum_{k_2\leq k} 2^{-2k_{2,+}+\gamma k_2+\alpha k-\alpha k_2}\min\{2^{-M/2-k},2^{M+2k_2}\}\e_1^2\\
    &+\sum_{-2k\leq M\leq \log T}\sum_{k_2\leq k} 2^{-\gamma k_-/2-2k_{2,+}+\gamma k_2-\gamma k_{2,-}/2+2\alpha k} \min\{2^{-M/2-3k/2-k_2/2},2^{M+k_2}\}\e_1^2\\
    \lesssim &\e_1^2.
\end{aligned}
\end{equation}
Lastly, combining \eqref{jm2-1result} and \eqref{jm2-2result} gives
\begin{equation}
    \label{iandjm2inchi1}
    \begin{aligned}
    &\sum_{1\leq M\leq\log T}\sup_{2^{M-1}\leq t_1\leq t_2\leq 2^M}2^{2k_+-\gamma k+k/2+\alpha k}\bigg(\sum_{\substack{c_m+c_n\neq 0\\c_l\neq c_m}}A_{lmn}\sum_{(k_1,k_2)\in\chi^1_k}\|D^\alpha I^{M,2}_{k,k_1,k_2}\|_{L^2}+\\
    &\qquad\qquad\qquad\qquad +\sum_{\substack{c_m+c_n= 0\\c_l\neq c_m}}A_{lmn}\sum_{(k_1,k_2)\in\chi^1_k}\|D^\alpha J^{M,2}_{k,k_1,k_2}\|_{L^2}\bigg)\\
    \lesssim & \e_1^2
    \end{aligned}
\end{equation}
as desired.

\subsubsection{$D^{\alpha}I^{M,2}_{k,k_1,k_2}$ where $(k_1,k_2)\in\chi^3_k$}
\label{2-3}
Unlike the situation where we only consider $c_m+c_n=0$ for the term $J^{M,2}_{k,k_1,k_2}$, which allows use to handle the cases when $(k_1,k_2)\in\chi^3_k$ together with the cases when $(k_1,k_2)\in\chi^2_k$ in Section \ref{2-1}. For $I^{M,2}_{k,k_1,k_2}$, we failed to obtain a lower bound for $|\nabla_\eta\phi(\xi,\eta)|$ when $(k_1,k_2)\in\chi^3_k$. Thus, we treat this case separately here in this section.

Similarly to \autocite[Theorem~2]{germain}{}, we define $P(\xi,\eta)= \frac{1}{2}(c_lc_m+c_lc_n-c_mc_n+1)(\frac{c_m}{c_m+c_n}\xi-\eta)$, then we have the identity
\begin{equation}
\label{chi3id}
    (\partial_t+\frac{P(\xi,\eta)}{t}\cdot\nabla_\eta) e^{it\phi(\xi,\eta)}=iZ(\xi,\eta)e^{it\phi(\xi,\eta)},
\end{equation}
where
\begin{align*}
    Z(\xi,\eta)
    = & \phi(\xi,\eta)+P(\xi,\eta)\cdot\nabla_\eta\phi(\xi,\eta)\\
    = & \frac{c_lc_m+c_lc_n-c_mc_n}{c_m+c_n}\big((1+c^2_m)|\xi|^2-2c_m(c_m+c_n)\xi\cdot\eta+(c_m+c_n)^2|\eta|^2\big).
\end{align*}
This polynomial $P$ is chosen to ensure the $Z$ bounded from below when $(k_1,k_2)\in\chi^3_k$.
Since 
$$2|c_m(c_m+c_n)\xi\cdot\eta|\leq c_m^2|\xi|^2+(c_m+c_n)^2|\eta|^2,$$
we indeed have $|Z(\xi,\eta)|\gtrsim |\xi|^2\sim 2^{2k}$. However, we notice that this is only true given $c_lc_m+c_lc_n-c_mc_n\neq 0$. Hence, we will assume this for now and address the case when $\frac{1}{c_l}=\frac{1}{c_m}+\frac{1}{c_n}$ at the end of this section.

Now, we use the identity \eqref{chi3id} to integrate $I^{M,2}_{k,k_1,k_2}$ by parts and get
\begin{align}
    I^{M,2}_{k,k_1,k_2}
    =&\int_{t_1}^{t_2}\int_{\R^3}(\partial_t+\frac{P(\xi,\eta)}{t}\cdot\nabla_\eta)e^{it\phi(\xi,\eta)}\frac{t\partial_{\xi_l}\phi(\xi,\eta)}{Z(\xi,\eta)}\hat{f}_{m,k_1}(t,\xi-\eta)\hat{f}_{n,k_2}(t,\eta)\psi_k(\xi)d\eta dt\nonumber\\
    =&\sum_{j=1,2}(-1)^j\int_{\R^3}e^{it_j \phi(\xi,\eta)}\frac{t_j\partial_{\xi_l}\phi(\xi,\eta)}{Z(\xi,\eta)}\hat{f}_{m,k_1}(t_j,\xi-\eta)\hat{f}_{n,k_2}(t_j,\eta)\psi_k(\xi)d\eta \label{chi3.1}\\
    &-\int_{t_1}^{t_2}\int_{\R^3}e^{it\phi(\xi,\eta)}\frac{\partial_{\xi_l}\phi(\xi,\eta)}{Z(\xi,\eta)}\hat{f}_{m,k_1}(t,\xi-\eta)\hat{f}_{n,k_2}(t,\eta)\psi_k(\xi)d\eta dt\label{chi3.2}\\
    &-\int_{t_1}^{t_2}\int_{\R^3}e^{it\phi(\xi,\eta)}\frac{t\partial_{\xi_l}\phi(\xi,\eta)}{Z(\xi,\eta)}\partial_t\hat{f}_{m,k_1}(t,\xi-\eta)\hat{f}_{n,k_2}(t,\eta)\psi_k(\xi)d\eta dt\label{chi3.3.1}\\
    &-\int_{t_1}^{t_2}\int_{\R^3}e^{it\phi(\xi,\eta)}\frac{t\partial_{\xi_l}\phi(\xi,\eta)}{Z(\xi,\eta)}\hat{f}_{m,k_1}(t,\xi-\eta)\partial_t\hat{f}_{n,k_2}(t,\eta)\psi_k(\xi)d\eta dt\label{chi3.3.2}\\
    &-\int_{t_1}^{t_2}\int_{\R^3}e^{it\phi(\xi,\eta)}\partial_{\eta_n}\frac{P_n(\xi,\eta)\partial_{\xi_l}\phi(\xi,\eta)}{Z(\xi,\eta)}\hat{f}_{m,k_1}(t,\xi-\eta)\hat{f}_{n,k_2}(t,\eta)\psi_k(\xi)d\eta dt\label{chi3.4}\\
    &-\int_{t_1}^{t_2}\int_{\R^3}e^{it\phi(\xi,\eta)}\frac{P_n(\xi,\eta)\partial_{\xi_l}\phi(\xi,\eta)}{Z(\xi,\eta)}\partial_{\eta_n}\hat{f}_{m,k_1}(t,\xi-\eta)\hat{f}_{n,k_2}(t,\eta)\psi_k(\xi)d\eta dt\label{chi3.5.1}\\
    &-\int_{t_1}^{t_2}\int_{\R^3}e^{it\phi(\xi,\eta)}\frac{P_n(\xi,\eta)\partial_{\xi_l}\phi(\xi,\eta)}{Z(\xi,\eta)}\hat{f}_{m,k_1}(t,\xi-\eta)\partial_{\eta_n}\hat{f}_{n,k_2}(t,\eta)\psi_k(\xi)d\eta dt\label{chi3.5.2}.
\end{align}
We observe that the power of $t$ in $I^{M,2}_{k,k_1,k_2}$ is lowered by $1$ while $I^{M,2}_{k,k_1,k_2}$ is expanded into a sum of terms. In the lemma below, we will first address the terms where the derivatives of $t$ and $\eta$ did not fall on the product $\hat{f}_{m,k_1}(t,\xi-\eta)\hat{f}_{n,k_2}(t,\eta)$ using Lemma \ref{dalpha}.
\begin{lem}
\label{im2-2-1}
Given $t_1,t_2\in[2^{M-1},2^M]$, $(k_1,k_2)\in\chi^3_k$, $\frac{1}{c_l}\neq\frac{1}{c_m}+\frac{1}{c_n}$, and $\sup_{t\in[1,T]}\|f\|_{Z}\leq \e_1$, we have
\begin{align*}
    &\|D^\alpha \eqref{chi3.1}+D^\alpha \eqref{chi3.2}+D^\alpha \eqref{chi3.4}\|_{L^2}
    \lesssim  2^{-4k_++2\gamma k}\min\{2^{-M/2-3k/2+\alpha M+\alpha k},2^{M+3k/2-\alpha k}\}\e_1^2.
\end{align*}
\end{lem}
\begin{proof}
In order to use Lemma \ref{dalpha}, we need bounds on $\| \eqref{chi3.1}+\eqref{chi3.2}\|_{L^2}$ and $\|\nabla  \eqref{chi3.1}+\nabla \eqref{chi3.2}\|_{L^2}$ when $(k_1,k_2)\in\chi^3_k=\{|k_1-k_2|\leq a, |k-k_1|\leq a+2\}$.\\
First, using the bilinear estimate $L^4\times L^4\rightarrow L^2$ in Lemma \ref{bilinear}, Lemma \ref{chi,eta} Lemma \ref{dualitycomp}, together with the estimations in \eqref{l4}, and \eqref{l2}, we get
\begin{align*}
    &\|\eqref{chi3.1}+\eqref{chi3.2}+\eqref{chi3.4}\|_{L^2}\\
    =&\bigg\|\int_{\R^3}e^{it_j \phi(\xi,\eta)}\frac{t_j\partial_{\xi_l}\phi(\xi,\eta)}{Z(\xi,\eta)}\hat{f}_{m,k_1}(t_j,\xi-\eta)\hat{f}_{n,k_2}(t_j,\eta)\psi_k(\xi)d\eta \bigg\|_{L^2}\\
    &+\bigg\|\int_{t_1}^{t_2}\int_{\R^3}e^{it\phi(\xi,\eta)}\frac{\partial_{\xi_l}\phi(\xi,\eta)}{Z(\xi,\eta)}\hat{f}_{m,k_1}(t,\xi-\eta)\hat{f}_{n,k_2}(t,\eta)\psi_k(\xi)d\eta dt\bigg\|_{L^2}\\
    &+\bigg\|\int_{t_1}^{t_2}\int_{\R^3}e^{it\phi(\xi,\eta)}\partial_{\eta_n}\frac{P_n(\xi,\eta)\partial_{\xi_l}\phi(\xi,\eta)}{Z(\xi,\eta)}\hat{f}_{m,k_1}(t,\xi-\eta)\hat{f}_{n,k_2}(t,\eta)\psi_k(\xi)d\eta dt\bigg\|_{L^2}\\
    \lesssim & 2^{M}\big(\|\F^{-1}\frac{\partial_{\xi_l}\phi(\xi,\eta)}{Z(\xi,\eta)}\Tilde{\psi}_{k_2}(\eta)\Tilde{\psi}_k(\xi)\|_{L^1}+\|\F^{-1}\partial_{\eta_n}\frac{P_n(\xi,\eta)\partial_{\xi_l}\phi(\xi,\eta)}{Z(\xi,\eta)}\Tilde{\psi}_{k_2}(\eta)\Tilde{\psi}_k(\xi)\|_{L^1}\big)\times\\
    &\times \sup_{t\in[2^{M-1},2^M]}\min\{\|e^{ic_mt\la}f_{m,k_1}\|_{L^4_x}\|e^{ic_nt\la}f_{n,k_2}\|_{L^4_x},2^{3\min\{k,k_2\}/2}\|e^{ic_mt\la}f_{m,k_1}\|_{L^2_x}\|e^{ic_nt\la}f_{n,k_2}\|_{L^2_x}\}\\
    \lesssim & 2^{M-k-2k_{1,+}+\gamma k_1-2k_{2,+}+\gamma k_2}\min\{2^{-3M/2-k_1/4-k_2/4},2^{3\min\{k,k_2\}/2+k_1/2+k_2/2}\}\e_1^2\\
    \lesssim &2^{-4k_++2\gamma k}\min\{2^{-M/2-3k/2},2^{M+3k/2}\}\e_1^2.
\end{align*}
Next, we compute
\begin{align}
    \partial_{\xi_m} \eqref{chi3.1}
    = &\partial_{\xi_m}\int_{\R^3}e^{it_j \phi(\xi,\eta)}\frac{t_j\partial_{\xi_l}\phi(\xi,\eta)}{Z(\xi,\eta)}\hat{f}_{m,k_1}(t_j,\xi-\eta)\hat{f}_{n,k_2}(t_j,\eta)\psi_k(\xi)d\eta\nonumber\\
    = &\int_{\R^3}e^{it_j \phi(\xi,\eta)}\frac{t_j\partial_{\xi_l}\phi(\xi,\eta)}{Z(\xi,\eta)}\partial_{\xi_m}\hat{f}_{m,k_1}(t_j,\xi-\eta)\hat{f}_{n,k_2}(t_j,\eta)\psi_k(\xi)d\eta\label{im2-2.1}\\
    &+\int_{\R^3}e^{it_j \phi(\xi,\eta)}\frac{t_j\partial_{\xi_l}\phi(\xi,\eta)}{Z(\xi,\eta)}\hat{f}_{m,k_1}(t_j,\xi-\eta)\hat{f}_{n,k_2}(t_j,\eta)\partial_{\xi_m}\psi_k(\xi)d\eta\label{im2-2.12}\\
    &+\int_{\R^3}e^{it_j \phi(\xi,\eta)}\partial_{\xi_m}\frac{t_j\partial_{\xi_l}\phi(\xi,\eta)}{Z(\xi,\eta)}\hat{f}_{m,k_1}(t_j,\xi-\eta)\hat{f}_{n,k_2}(t_j,\eta)\psi_k(\xi)d\eta\label{im2-2.12*}\\
    & + \int_{\R^3}e^{it_j \phi(\xi,\eta)}\frac{it_j^2\partial_{\xi_m}\phi(\xi,\eta)\partial_{\xi_l}\phi(\xi,\eta)}{Z(\xi,\eta)}\hat{f}_{m,k_1}(t_j,\xi-\eta)\hat{f}_{n,k_2}(t_j,\eta)\psi_k(\xi)d\eta\label{im2-2.2},
\end{align}
\begin{align}
    \partial_{\xi_m} \eqref{chi3.2}
    = &\partial_{\xi_m}\int_{t_1}^{t_2}\int_{\R^3}e^{it\phi(\xi,\eta)}\frac{\partial_{\xi_l}\phi(\xi,\eta)}{Z(\xi,\eta)}\hat{f}_{m,k_1}(t,\xi-\eta)\hat{f}_{n,k_2}(t,\eta)\psi_k(\xi)d\eta dt  \nonumber\\
    = &\int_{t_1}^{t_2}\int_{\R^3}e^{it\phi(\xi,\eta)}\frac{\partial_{\xi_l}\phi(\xi,\eta)}{Z(\xi,\eta)}\partial_{\xi_m}\hat{f}_{m,k_1}(t,\xi-\eta)\hat{f}_{n,k_2}(t,\eta)\psi_k(\xi)d\eta dt\label{im2-2.1-2}\\
    &+ \int_{t_1}^{t_2}\int_{\R^3}e^{it\phi(\xi,\eta)}\frac{\partial_{\xi_l}\phi(\xi,\eta)}{Z(\xi,\eta)}\hat{f}_{m,k_1}(t,\xi-\eta)\hat{f}_{n,k_2}(t,\eta)\partial_{\xi_m}\psi_k(\xi)d\eta dt\label{im2-2.12-2}\\
    &+ \int_{t_1}^{t_2}\int_{\R^3}e^{it\phi(\xi,\eta)}\partial_{\xi_m}\frac{\partial_{\xi_l}\phi(\xi,\eta)}{Z(\xi,\eta)}\hat{f}_{m,k_1}(t,\xi-\eta)\hat{f}_{n,k_2}(t,\eta)\psi_k(\xi)d\eta dt\label{im2-2.12-2*}\\
    & +\int_{t_1}^{t_2}\int_{\R^3}e^{it\phi(\xi,\eta)}\frac{it\partial_{\xi_m}\phi(\xi,\eta)\partial_{\xi_l}\phi(\xi,\eta)}{Z(\xi,\eta)}\hat{f}_{m,k_1}(t,\xi-\eta)\hat{f}_{n,k_2}(t,\eta)\psi_k(\xi)d\eta dt\label{im2-2.2-2},
\end{align}
and
\begin{align}
    \partial_{\xi_m} \eqref{chi3.4}
    = &\partial_{\xi_m}\int_{t_1}^{t_2}\int_{\R^3}e^{it\phi(\xi,\eta)}\partial_{\eta_n}\frac{P_n(\xi,\eta)\partial_{\xi_l}\phi(\xi,\eta)}{Z(\xi,\eta)}\hat{f}_{m,k_1}(t,\xi-\eta)\hat{f}_{n,k_2}(t,\eta)\psi_k(\xi)d\eta dt \nonumber\\
    = &\int_{t_1}^{t_2}\int_{\R^3}e^{it\phi(\xi,\eta)}\partial_{\eta_n}\frac{P_n(\xi,\eta)\partial_{\xi_l}\phi(\xi,\eta)}{Z(\xi,\eta)}\partial_{\xi_m}\hat{f}_{m,k_1}(t,\xi-\eta)\hat{f}_{n,k_2}(t,\eta)\psi_k(\xi)d\eta dt\label{im2-2.1-3}\\
    &+ \int_{t_1}^{t_2}\int_{\R^3}e^{it\phi(\xi,\eta)}\partial_{\eta_n}\frac{P_n(\xi,\eta)\partial_{\xi_l}\phi(\xi,\eta)}{Z(\xi,\eta)}\hat{f}_{m,k_1}(t,\xi-\eta)\hat{f}_{n,k_2}(t,\eta)\partial_{\xi_m}\psi_k(\xi)d\eta dt\label{im2-2.12-3}\\
    &+ \int_{t_1}^{t_2}\int_{\R^3}e^{it\phi(\xi,\eta)}\partial_{\xi_m}\partial_{\eta_n}\frac{P_n(\xi,\eta)\partial_{\xi_l}\phi(\xi,\eta)}{Z(\xi,\eta)}\hat{f}_{m,k_1}(t,\xi-\eta)\hat{f}_{n,k_2}(t,\eta)\psi_k(\xi)d\eta dt\label{im2-2.12-3*}\\
    & +\int_{t_1}^{t_2}\int_{\R^3}e^{it\phi(\xi,\eta)}it\partial_{\xi_m}\phi(\xi,\eta)\partial_{\eta_n}\frac{P_n(\xi,\eta)\partial_{\xi_l}\phi(\xi,\eta)}{Z(\xi,\eta)}\hat{f}_{m,k_1}(t,\xi-\eta)\hat{f}_{n,k_2}(t,\eta)\psi_k(\xi)d\eta dt\label{im2-2.2-3}.
\end{align}
Applying the bilinear estimate $L^{18/7}\times L^9\rightarrow L^2$, along with Lemma \ref{chi,eta}, Lemma \ref{op}, Lemma \ref{dualitycomp}, Lemma \ref{lpnorms}, \eqref{l2}, and \eqref{l2first}, we obtain
\begin{align*}
    &\|\eqref{im2-2.1}\|_{L^2}+\|\eqref{im2-2.1-2}\|_{L^2}+\|\eqref{im2-2.1-3}\|_{L^2}\\
    \lesssim & 2^{M}\big(\|\F^{-1}\frac{\partial_{\xi_l}\phi(\xi,\eta)}{Z(\xi,\eta)}\Tilde{\psi}_{k_1}(\xi-\eta)\Tilde{\psi}_{k_2}(\eta)\Tilde{\psi}_k(\xi)\|_{L^1}+\|\F^{-1}\partial_{\eta_n}\frac{P_n(\xi,\eta)\partial_{\xi_l}\phi(\xi,\eta)}{Z(\xi,\eta)}\Tilde{\psi}_{k_1}(\xi-\eta)\Tilde{\psi}_{k_2}(\eta)\Tilde{\psi}_k(\xi)\|_{L^1}\big)\times\\
    &\times \sup_{t\in[2^{M-1},2^M]}\min\{\|e^{ic_mt\la}\F^{-1}\nabla_\xi\hat{f}_{m,k_1}\|_{L^{18/7}_x}\|e^{ic_nt\la}f_{n,k_2}\|_{L^9_x},\\&\qquad\qquad\qquad\qquad 2^{3k/2}\|e^{ic_mt\la}\F^{-1}\nabla_\xi\hat{f}_{m,k_1}\|_{L^2_x}\|e^{ic_nt\la}f_{n,k_2}\|_{L^2_x}\}\\
    \lesssim & 2^{M-k}\sup_{t\in[2^{M-1},2^M]}\min\{2^{-3M/2}\|\F^{-1}\nabla_\xi\hat{f}_{m,k_1}\|_{L^{18/11}_x}\|f_{n,k_2}\|_{L^{9/8}_x},2^{3k/2}\|\nabla_\xi\hat{f}_{m,k_1}\|_{L^2_\xi}\|f_{n,k_2}\|_{L^2_x}\}\\
    \lesssim & 2^{M-k_2-2k_{1,+}+\gamma k_1-2k_{2,+}+\gamma k_2}\min\{2^{-3M/2-5k_1/6-2k_2/3},2^{3k/2-k_1/2+k_2/2}\}\e_1^2\\
    \lesssim & 2^{-4k_{+}+2\gamma k}\min\{2^{-M/2-5k/2},2^{M+k/2}\}\e_1^2
\end{align*}
and
\begin{align*}
    &\|\eqref{im2-2.12}\|_{L^2}+\|\eqref{im2-2.12*}\|_{L^2}+\|\eqref{im2-2.12-2}\|_{L^2}+\|\eqref{im2-2.12-2*}\|_{L^2}+\|\eqref{im2-2.12-3}\|_{L^2}+\|\eqref{im2-2.12-3*}\|_{L^2}+\\
    &+\|\eqref{im2-2.2}\|_{L^2}+\|\eqref{im2-2.2-2}\|_{L^2}+\|\eqref{im2-2.2-3}\|_{L^2}\\
    \lesssim & 
    \big(2^{M}\|\nabla\psi_k\|_{L^\infty}\|\F^{-1}\frac{\partial_{\xi_l}\phi(\xi,\eta)}{Z(\xi,\eta)}\Tilde{\psi}_{k_1}(\xi-\eta)\Tilde{\psi}_{k_2}(\eta)\Tilde{\psi}_k(\xi)\|_{L^1}+\\
    &+2^{M}\|\F^{-1}\partial_{\xi_m}\frac{\partial_{\xi_l}\phi(\xi,\eta)}{Z(\xi,\eta)}\Tilde{\psi}_{k_1}(\xi-\eta)\Tilde{\psi}_{k_2}(\eta)\Tilde{\psi}_k(\xi)\|_{L^1}+\\
    &+2^{M}\|\nabla\psi_k\|_{L^\infty}\|\F^{-1}\partial_{\eta_n}\frac{P_n(\xi,\eta)\partial_{\xi_l}\phi(\xi,\eta)}{Z(\xi,\eta)}\Tilde{\psi}_{k_1}(\xi-\eta)\Tilde{\psi}_{k_2}(\eta)\Tilde{\psi}_k(\xi)\|_{L^1}+\\
    &+2^{M}\|\F^{-1}\partial_{\xi_m}\partial_{\eta_n}\frac{P_n(\xi,\eta)\partial_{\xi_l}\phi(\xi,\eta)}{Z(\xi,\eta)}\Tilde{\psi}_{k_1}(\xi-\eta)\Tilde{\psi}_{k_2}(\eta)\Tilde{\psi}_k(\xi)\|_{L^1}+\\
    &+2^{2M}\|\F^{-1}\frac{\partial_{\xi_m}\phi(\xi,\eta)\partial_{\xi_l}\phi(\xi,\eta)}{Z(\xi,\eta)}\Tilde{\psi}_{k_1}(\xi-\eta)\Tilde{\psi}_{k_2}(\eta)\Tilde{\psi}_k(\xi)\|_{L^1}+\\
    &+2^{2M}\|\F^{-1}\partial_{\xi_m}\phi(\xi,\eta)\partial_{\eta_n}\frac{P_n(\xi,\eta)\partial_{\xi_l}\phi(\xi,\eta)}{Z(\xi,\eta)}\Tilde{\psi}_{k_1}(\xi-\eta)\Tilde{\psi}_{k_2}(\eta)\Tilde{\psi}_k(\xi)\|_{L^1}\big)\times\\
    &\times \sup_{t\in[2^{M-1},2^M]}\min\{\|e^{ic_mt\la}f_{m,k_1}\|_{L^{18/7}_x}\|e^{ic_nt\la}f_{n,k_2}\|_{L^9_x},2^{3k/2}\|e^{ic_mt\la}f_{m,k_1}\|_{L^2_x}\|e^{ic_nt\la}f_{n,k_2}\|_{L^2_x}\}\\
    \lesssim & (2^{M-2k}+2^{2M})\sup_{t\in[2^{M-1},2^M]}\min\{2^{-3M/2}\|f_{m,k_1}\|_{L^{18/11}_x}\|f_{n,k_2}\|_{L^{9/8}_x},2^{3k/2}\|f_{m,k_1}\|_{L^2_x}\|f_{n,k_2}\|_{L^2_x}\}\\
    \lesssim & (1+2^{M+2k})2^{M-2k-2k_{1,+}+\gamma k_1-2k_{2,+}+\gamma k_2}\min\{2^{-3M/2+k_1/6-2k_2/3},2^{3k/2+k_1/2+k_2/2}\}\e_1^2\\
    \lesssim & (1+2^{M+2k})2^{-4k_{+}+2\gamma k}\min\{2^{-M/2-5k/2},2^{M+k/2}\}\e_1^2.
\end{align*}
Thus, we have when $M\leq -2k$,
\begin{align*}
    \|\eqref{chi3.1}+\eqref{chi3.2}+\eqref{chi3.4}\|_{L^2}
    \lesssim 2^{-4k_++2\gamma k+M+3k/2}\e_1^2
\end{align*}
and
\begin{align*}
    \|\nabla \eqref{chi3.1}+\nabla\eqref{chi3.2}+\nabla\eqref{chi3.4}\|_{L^2}
    \lesssim 2^{-4k_{+}+2\gamma k+M+k/2}\e_1^2.
\end{align*}
Lemma \ref{dalpha} implies
\begin{align*}
    \|D^{\alpha}\eqref{chi3.1}+D^\alpha \eqref{chi3.2}+D^\alpha\eqref{chi3.4}\|_{L^2}
    \lesssim 2^{-4k_++2\gamma k+M+3k/2-\alpha k}\e_1^2.
\end{align*}
Furthermore when $M>-2k$,
\begin{align*}
    \|\eqref{chi3.1}+\eqref{chi3.2}+\eqref{chi3.4}\|_{L^2}
    \lesssim 2^{-4k_++2\gamma k-M/2-3k/2}\e_1^2
\end{align*}
and
\begin{align*}
    \|\nabla \eqref{chi3.1}+\nabla\eqref{chi3.2}+\nabla\eqref{chi3.4}\|_{L^2}
    \lesssim 2^{-4k_{+}+2\gamma k+M/2-k/2}\e_1^2.
\end{align*}
By Lemma \ref{dalpha}, we get
\begin{align*}
    \|D^{\alpha}\eqref{chi3.1}+D^\alpha \eqref{chi3.2}+D^\alpha\eqref{chi3.4}\|_{L^2}
    \lesssim 2^{-4k_++2\gamma k-M/2-3k/2+\alpha M+\alpha k}\e_1^2.
\end{align*}
\end{proof}
As a result of Lemma \ref{im2-2-1}, we have
\begin{equation}
\label{im2-2-1r}
    \begin{aligned}
    &\sum_{1\leq M\leq \log T}\sum_{(k_1,k_2)\in\chi^3_k}2^{2k_+-\gamma k+k/2+\alpha k}\sup_{2^{M-1}\leq t_1\leq t_2\leq 2^M}\big\|D^\alpha\big[\eqref{chi3.1}+\eqref{chi3.2}+\eqref{chi3.4}\big]\big\|_{L^2}\\
    \lesssim &\sum_{1\leq M\leq \log T}2^{-2k_++\gamma k}\min\{2^{-M/2-k+\alpha M+2\alpha k},2^{M+2k}\}\e_1^2\\
    \lesssim &\sum_{M\leq -2k}2^{\gamma k+M+2k}\e_1^2+ \sum_{-2k <M\leq \log T} 2^{\alpha(M+2k)-2k_{+}+\gamma k-M/2-k}\e_1^2\\
    \lesssim & \e_1^2.
\end{aligned}
\end{equation}
For the remaining terms in the expansion of $I^{M,2}_{k,k_1,k_2}$, we will estimate directly by applying Lemma \ref{chi3*}.\\
Let $\hat{g}=\partial_t\hat{f}_l$ and $m_d=\frac{\partial_{\xi_l}\phi(\xi,\eta)}{Z(\xi,\eta)}$. Then, we have 
\begin{align*}
    &\|D^\alpha\eqref{chi3.3.1}\|_{L^2}\\
    \lesssim & (1+2^{\alpha(M+2k)})2^{-2k_{+}+\gamma k}\min\{2^{M/2-2k+(1-\alpha)(\gamma M/8+\gamma k/4)},2^{2M+k}\}\e_1\|\partial_t\hat{f}_{m,k_1}\|_{L^\infty_t([2^{M-1},2^M])H^\alpha_\xi}
\end{align*}
and
\begin{align*}
    &\|D^\alpha\eqref{chi3.3.2}\|_{L^2}\\
    \lesssim & (1+2^{\alpha(M+2k)})2^{-2k_{+}+\gamma k}\min\{2^{M/2-2k+(1-\alpha)(\gamma M/8+\gamma k/4)},2^{2M+k}\}\e_1\|\partial_t\hat{f}_{n,k_2}\|_{L^\infty_t([2^{M-1},2^M])H^\alpha_\xi}.
\end{align*}
From Lemma \ref{timel2} and Lemma \ref{mixder}, we know $\|\partial_t\hat{f}_{l,k}\|_{H^\alpha}\lesssim 2^{-M-2k_++(1-\alpha)\gamma k+k/2-\alpha k}\e_1^2$. Thus, 
\begin{align*}
    &\|D^\alpha\eqref{chi3.3.1}\|_{L^2}+\|D^\alpha\eqref{chi3.3.2}\|_{L^2}\\
    \lesssim & 2^{-2k_++(1-\alpha)\gamma k+k/2-\alpha k} (1+2^{\alpha(M+2k)})2^{-2k_{+}+\gamma k}\min\{2^{-M/2-2k+(1-\alpha)(\gamma M/8+\gamma k/4)},2^{M+k}\}\e_1^3,
\end{align*}
for $(k_1,k_2)\in\chi^3_k$.\\
Then, take $\hat{g}=\partial_{\xi_n}\hat{f}_l$ and $m_d=\frac{P_n(\xi,\eta)\partial_{\xi_l}\phi(\xi,\eta)}{Z(\xi,\eta)}$. Employing Lemma \ref{chi3*}, we obtain
\begin{align*}
    &\|D^\alpha\eqref{chi3.5.1}\|_{L^2}\\
    \lesssim & (1+2^{\alpha(M+2k)})2^{-2k_{+}+\gamma k}\min\{2^{-M/2-k+(1-\alpha)(\gamma M/8+\gamma k/4)},2^{M+2k}\}\e_1\|\hat{f}_{m,k_1}\|_{L^\infty_t([2^{M-1},2^M])H^{1+\alpha}_\xi}
\end{align*}
and
\begin{align*}
    &\|D^\alpha\eqref{chi3.5.2}\|_{L^2}\\
    \lesssim & (1+2^{\alpha(M+2k)})2^{-2k_{+}+\gamma k}\min\{2^{-M/2-k+(1-\alpha)(\gamma M/8+\gamma k/4)},2^{M+2k}\}\e_1\|\hat{f}_{n,k_2}\|_{L^\infty_t([2^{M-1},2^M])H^{1+\alpha}_\xi}.
\end{align*}
Using the bound in \eqref{sobolevnorms} for the Sobolev norms, we have 
\begin{align*}
    &\|D^\alpha\eqref{chi3.5.1}\|_{L^2}+\|D^\alpha\eqref{chi3.5.2}\|_{L^2}\\
    \lesssim &2^{-2k_++\gamma k-k/2-\alpha k}(1+2^{\alpha(M+2k)})2^{-2k_{+}+\gamma k}\min\{2^{-M/2-k+(1-\alpha)(\gamma M/8+\gamma k/4)},2^{M+2k}\}\e_1^3,
\end{align*}
for $(k_1,k_2)\in\chi^3_k$.
Hence, we may conclude
\begin{equation}
\label{im2-2-2r}
    \begin{aligned}
    &\sum_{1\leq M\leq \log T}\sum_{(k_1,k_2)\in\chi^3_k}2^{2k_+-\gamma k+k/2+\alpha k}\sup_{2^{M-1}\leq t_1\leq t_2\leq 2^M}\big\|D^\alpha\big[\eqref{chi3.3.1}+\eqref{chi3.3.2}+\eqref{chi3.5.1}+\eqref{chi3.5.2}\big]\big\|_{L^2}\\
    \lesssim &\sum_{1\leq M\leq \log T}(1+2^{\alpha(M+2k)}) 2^{-2k_{+}+(1-\alpha)\gamma k}\min\{2^{-M/2-k+(1-\alpha)(\gamma M/8+\gamma k/4)},2^{M+2k}\}\e_1^3\\
    &+\sum_{1\leq M\leq \log T}(1+2^{\alpha(M+2k)})2^{-2k_{+}+\gamma k}\min\{2^{-M/2-k+(1-\alpha)(\gamma M/8+\gamma k/4)},2^{M+2k}\}\e_1^3\\
    \lesssim &\sum_{M\leq -2k}(1+2^{\alpha\gamma k})2^{(1-\alpha)\gamma k+M+2k}\e_1^2+ \sum_{M > -2k}(1+2^{\alpha\gamma k})2^{\alpha(M+2k)-2k_{+}+(1-\alpha)\gamma k-M/2-k+(1-\alpha)(\gamma M/8+\gamma k/4)}\e_1^2\\
    \lesssim & \e_1^2,
\end{aligned}
\end{equation}
given $\alpha+(1-\alpha)\gamma/8<1/2$, i.e. $\alpha<1/2-\gamma/(16-2\gamma)$.

Finally, we deal with the case when $\frac{1}{c_l}=\frac{1}{c_m}+\frac{1}{c_n}$. Since we have
$$\nabla_\xi\phi(\xi,\eta)=\frac{c_l-c_m}{c_m}\nabla_\eta\phi(\xi,\eta),$$
where $\frac{c_l-c_m}{c_m}\neq 0$, we can rewrite $I^{M,2}_{k,k_1,k_2}$ as
\begin{align}
I^{M,2}_{k,k_1,k_2}
=&\int_{t_1}^{t_2}\int_{\R^3}e^{it\phi(\xi,\eta)}it\partial_{\xi_l}\phi(\xi,\eta) \hat{f}_{m,k_1}(t,\xi-\eta)\hat{f}_{n,k_2}(t,\eta)\psi_k(\xi)d\eta dt\nonumber\\
=&\frac{c_l-c_m}{c_m}\int_{t_1}^{t_2}\int_{\R^3}e^{it\phi(\xi,\eta)}it\partial_{\eta_l}\phi(\xi,\eta) \hat{f}_{m,k_1}(t,\xi-\eta)\hat{f}_{n,k_2}(t,\eta)\psi_k(\xi)d\eta dt\nonumber\\
=&-\frac{c_l-c_m}{c_m}\int_{t_1}^{t_2}\int_{\R^3}e^{it\phi(\xi,\eta)}\partial_{\eta_l} \hat{f}_{m,k_1}(t,\xi-\eta)\hat{f}_{n,k_2}(t,\eta)\psi_k(\xi)d\eta dt\label{im2chi3extra1}\\
&-\frac{c_l-c_m}{c_m}\int_{t_1}^{t_2}\int_{\R^3}e^{it\phi(\xi,\eta)} \hat{f}_{m,k_1}(t,\xi-\eta)\partial_{\eta_l}\hat{f}_{n,k_2}(t,\eta)\psi_k(\xi)d\eta dt\label{im2chi3extra2}.
\end{align}
Furthermore, as a result of Lemma \ref{chi3*} and \eqref{sobolevnorms}, we know for $(k_1,k_2)\in\chi^3_k$,
\begin{align*}
    \|D^\alpha I^{M,2}_{k,k_1,k_2}\|_{L^2}
    \leq & \|D^\alpha\eqref{im2chi3extra1}\|_{L^2}+\|D^\alpha\eqref{im2chi3extra2}\|_{L^2}\\
    \lesssim & (1+2^{\alpha(M+2k)})2^{-2k_++\gamma k}\min\{2^{-M/2-k+(1-\alpha)(\gamma M/8+\gamma k/4)},2^{M+2k}\}\e_1\times\\
    &\times(\|\nabla_\xi\hat{f}_{m,k_1}\|_{L^\infty_t([2^{M-1},2^M])H^\alpha_\xi}+\|\nabla_\xi\hat{f}_{n,k_2}\|_{L^\infty_t([2^{M-1},2^M])H^\alpha_\xi})\\
    \lesssim & (1+2^{\alpha(M+2k)})2^{-4k_++2\gamma k-k/2-\alpha k}\min\{2^{-M/2-k+(1-\alpha)(\gamma M/8+\gamma k/4)},2^{M+2k}\}\e_1^2.
\end{align*}
Hence, we obtain
\begin{equation}
\label{im2chi3extra}
    \begin{aligned}
    &\sum_{1\leq M\leq \log T}\sum_{(k_1,k_2)\in\chi^3_k}2^{2k_+-\gamma k+k/2+\alpha k}\sup_{2^{M-1}\leq t_1\leq t_2\leq 2^M}\big\|D^\alpha\big[\eqref{im2chi3extra1}+\eqref{im2chi3extra2}\big]\big\|_{L^2}\\
    \lesssim &\sum_{1\leq M\leq \log T}(1+2^{\alpha(M+2k)})2^{-2k_{+}+\gamma k}\min\{2^{-M/2-k+(1-\alpha)(\gamma M/8+\gamma k/4)},2^{M+2k}\}\e_1^2\\
    \lesssim &\sum_{M\leq -2k}2^{\gamma k+M+2k}\e_1^2+ \sum_{M > -2k}2^{\alpha(M+2k)-2k_{+}+\gamma k-M/2-k+(1-\alpha)(\gamma M/8+\gamma k/4)}\e_1^2\\
    \lesssim & \e_1^2,
\end{aligned}
\end{equation}
given that $\alpha+(1-\alpha)\gamma/8<1/2$.\\
Therefore, combining the results in \eqref{im2-2-1r} and \eqref{im2-2-2r} with \eqref{im2inchi2}, \eqref{jm2inchi2}, \eqref{iandjm2inchi1}, and \eqref{im2chi3extra}, we have
\begin{align*}
    &\sum_{1\leq M\leq \log T}\sup_{2^{M-1}\leq t_1\leq t_2\leq 2^M}2^{2k_+-\gamma k+k/2+\alpha k}\bigg(\sum_{c_m+c_n\neq 0}A_{lmn}\sum_{(k_1,k_2)\in\chi^3_k}\|D^\alpha I^{M,2}_{k,k_1,k_2}\|_{L^2}\\
    &+\big(\sum_{\substack{c_m+c_n\neq 0\\c_m\neq c_l}}+\sum_{\substack{c_m+c_n\neq 0\\c_m=c_l}}\big)A_{lmn}\sum_{(k_1,k_2)\in\chi^1_k}\|D^\alpha I^{M,2}_{k,k_1,k_2}\|_{L^2}
    +\sum_{c_m+c_n\neq 0}A_{lmn}\sum_{(k_1,k_2)\in\chi^2_k}\|D^\alpha I^{M,2}_{k,k_1,k_2}\|_{L^2}\\
    &+\big(\sum_{\substack{c_m+c_n= 0\\c_m\neq c_l}}+\sum_{\substack{c_m+c_n= 0\\c_m=c_l}}\big)A_{lmn}\sum_{(k_1,k_2)\in\chi^1_k}\|D^\alpha J^{M,2}_{k,k_1,k_2}\|_{L^2}
    +\sum_{c_m+c_n= 0}A_{lmn}\sum_{(k_1,k_2)\in\chi^2_k\cup\chi^3_k}\|D^\alpha J^{M,2}_{k,k_1,k_2}\|_{L^2}\\
    \lesssim & \e_1^2
\end{align*}
as desired.

\printbibliography
\end{document}